\documentclass[titlepage,a4paper,twoside,final]{scrreprt}
\usepackage[T1]{fontenc}
\usepackage[utf8]{inputenc}
\usepackage[a4paper,margin=3cm,includehead]{geometry}
\usepackage{amsmath,amssymb,amsthm,bbm}
\usepackage{amsfonts}
\usepackage{enumerate}
\usepackage[notref,notcite]{showkeys}
\usepackage{mathptmx}
\usepackage{helvet}
\usepackage{courier}

\usepackage{makeidx}         
\usepackage{graphicx}        
\usepackage{multicol}        
\usepackage[bottom]{footmisc}

\usepackage{eucal,mathrsfs,dsfont}
\usepackage{stmaryrd}
\usepackage{verbatim}
\usepackage{mathtools}
\usepackage{url}

\usepackage{fancyhdr}
\fancypagestyle{plain}{%
  \fancyhf{}
  
  \fancyfoot[CE,CO]{\sffamily\thepage}
}

\newtheorem{theorem}{Theorem}[chapter]
\newtheorem{corollary}[theorem]{Corollary}
\newtheorem{lemma}[theorem]{Lemma}
\newtheorem{proposition}[theorem]{Proposition}
\theoremstyle{definition}
\newtheorem{definition}[theorem]{Definition}
\newtheorem*{algorithm}{Algorithm}
\theoremstyle{remark}
\newtheorem{remark}[theorem]{Remark}
\newtheorem{example}[theorem]{Example}

\makeindex             

\newenvironment{nindex}[1][51pt]%
    {\begin{list}{}%
        {%
            \setlength{\labelwidth}{#1}%
            \setlength{\leftmargin}{\labelwidth+\labelsep}%
            \setlength{\itemsep}{2.5pt}%
            \setlength{\parsep}{0pt}%
            \setlength{\rightmargin}{0pt}%
        }%
    }%
    {\end{list}}

\newcommand{\et}{\quad\text{and}\quad}
\newcommand{\fa}{\quad\text{for all\ \ }}
\newcommand{\omu}[3]{\overset{#1}{\underset{#3}{#2}}}
\newcommand{\pomu}[3]{\overset{\phantom{#1}}{\underset{\phantom{#3}}{#2}}}

\newcommand{\Exp}{\mathsf{Exp}}
\newcommand{\Poi}{\mathsf{Poi}}
\newcommand{\Normal}{\mathsf{N}}
\newcommand{\Xt}{(X_t)_{t\geq 0}}

\newcommand{\id}{\operatorname{id}}
\newcommand{\sgn}{\operatorname{sgn}}
\newcommand{\supp}{\operatorname{supp}}
\newcommand{\eup}{\mathrm{e}}
\newcommand{\iup}{\mathop{\textup{i}}}

\newcommand{\entier}[1]{\lfloor#1\rfloor}

\newcommand{\I}{\mathds{1}}
\newcommand{\comp}{\mathds{C}}
\newcommand{\Cov}{\mathds{C}\mathrm{ov}}
\newcommand{\Vv}{\mathds{V}}
\newcommand{\Ee}{\mathds{E}}
\newcommand{\Pp}{\mathds{P}}
\newcommand{\nat}{\mathds{N}}
\newcommand{\real}{\mathds{R}}
\newcommand{\rd}{\mathds{R}^d}
\newcommand{\integer}{\mathds{Z}}

\newcommand{\Ecirc}{E^\circ}
\newcommand{\Nhat}{\widehat{N}} 
\newcommand{\Ncirc}{N^\circ} 
\newcommand{\Ncirccont}{\mu^\circ} 

\newcommand{\Eecirc}{\Escr^\circ}
\newcommand{\Rrcirc}{\Rscr^\circ}

\newcommand{\Ntilde}{\widetilde{N}}
\newcommand{\Ncont}{\mu}

\newcommand{\Ascr}{\mathscr{A}}
\newcommand{\Bscr}{\mathscr{B}}
\newcommand{\Escr}{\mathscr{E}}
\newcommand{\Fscr}{\mathscr{F}}
\newcommand{\Iscr}{\mathscr{I}}
\newcommand{\Pscr}{\mathscr{P}}
\newcommand{\Rscr}{\mathscr{R}}
\newcommand{\Sscr}{\mathscr{S}}
\newcommand{\Xscr}{\mathscr{X}}
\newcommand{\Afrak}{\mathfrak{A}}
\newcommand{\Bcal}{\mathcal{B}}
\newcommand{\cont}{\mathcal{C}}
\newcommand{\dom}{\mathcal{D}}
\newcommand{\Fcal}{\mathcal{F}}
\newcommand{\Scal}{\mathcal{S}}

\newcommand{\bbot}{\mathop{\perp\!\!\!\!\!\perp}}

\newcommand{\emphh}[1]{\textbf{\upshape#1}}

\renewcommand{\Re}{\ensuremath{\operatorname{Re}}}
\renewcommand{\Im}{\ensuremath{\operatorname{Im}}}
\renewcommand{\geq}{\geqslant}
\renewcommand{\leq}{\leqslant}

\renewcommand{\le}{\leqslant}

\DeclareFontFamily{U}{mathx}{\hyphenchar\font45}
\DeclareFontShape{U}{mathx}{m}{n}{
      <5> <6> <7> <8> <9> <10>
      <10.95> <12> <14.4> <17.28> <20.74> <24.88>
      mathx10
      }{}
\DeclareSymbolFont{mathx}{U}{mathx}{m}{n}
\DeclareFontSubstitution{U}{mathx}{m}{n}
\DeclareMathAccent{\widecheck}{0}{mathx}{"71}
\DeclareMathAccent{\wideparen}{0}{mathx}{"75}

\makeatletter
\providecommand*{\dcup}{%
  \mathbin{%
    \mathpalette\@dcup{}%
  }%
}
\newcommand*{\@dcup}[2]{%
  \ooalign{%
    $\m@th#1\cup$\cr
    \sbox0{$#1\cup$}%
    \dimen@=\ht0 %
    \sbox0{$\m@th#1\cdot$}%
    \advance\dimen@ by -\ht0 %
    \dimen@=.5\dimen@
    \hidewidth\raise\dimen@\box0\hidewidth
  }%
}

\providecommand*{\bigdcup}{%
  \mathop{%
    \vphantom{\bigcup}%
    \mathpalette\@bigdcup{}%
  }%
}
\newcommand*{\@bigdcup}[2]{%
  \ooalign{%
    $\m@th#1\bigcup$\cr
    \sbox0{$#1\bigcup$}%
    \dimen@=\ht0 %
    \advance\dimen@ by -\dp0 %
    \sbox0{\scalebox{2}{$\m@th#1\cdot$}}%
    \advance\dimen@ by -\ht0 %
    \dimen@=.5\dimen@
    \hidewidth\raise\dimen@\box0\hidewidth
  }%
}
\makeatother

\begin{document}
\pagestyle{fancy}
\renewcommand{\chaptermark}[1]{\markboth{#1}{#1}}
%

\title{\bfseries An Introduction to L\'evy and Feller Processes}
\subtitle{-- Advanced Courses in Mathematics - CRM Barcelona 2014 --}

\author{Ren\'e L.\ Schilling}
\date{\normalsize
    TU Dresden, Institut f\"ur Mathematische Stochastik,\\01062 Dresden, Germany\\\mbox{}\\
    \url{rene.schilling@tu-dresden.de}
    \url{http://www.math.tu-dresden.de/sto/schilling}\\\mbox{}\\
    \small
    These course notes will be published, together Davar Khoshnevisan's notes on
    \emph{Invariance and Comparison Principles for Parabolic Stochastic Partial Differential Equations} as \emph{From L\'{e}vy-Type Processes to Parabolic SPDEs} by the CRM, Barcelona and Birk\"{a}user, Cham 2017 (ISBN: 978-3-319-34119-4).
    The arXiv-version and the published version may differ in layout, pagination and wording, but not in content.}

\maketitle
\setcounter{page}{1}
\tableofcontents
\cleardoublepage

\chapter*{Preface}
\addcontentsline{toc}{chapter}{Preface}
\markboth{Preface}{Preface}

These lecture notes are an extended version of my lectures on L\'evy and L\'evy-type processes given at the \emph{Second Barcelona Summer School on Stochastic Analysis} organized by the \emph{Centre de Recerca Matem\`atica} (CRM). The lectures are aimed at advanced graduate and PhD students. In order to read these notes, one should have sound knowledge of measure theoretic probability theory and some background in stochastic processes, as it is covered in my books \emph{Measures, Integals and Martingales} \cite{schilling-mims} and \emph{Brownian Motion} \cite{schilling-bm}.

\medskip
My purpose in these lectures is to give an introduction to L\'evy processes, and to show how one can extend this approach to space inhomogeneous processes which behave locally like L\'evy processes. After a brief overview (Chapter~\ref{orient}) I introduce L\'evy processes, explain how to characterize them (Chapter~\ref{levy}) and discuss the quintessential examples of L\'evy processes (Chapter~\ref{exlp}). The  Markov (loss of memory) property of L\'evy processes is studied in Chapter~\ref{lpmp}. A short analytic interlude (Chapter~\ref{semi}) gives an introduction to operator semigroups, resolvents and their generators from a probabilistic perspective. Chapter~\ref{gen} brings us back to generators of L\'evy processes which are identified as pseudo differential operators whose symbol is the characteristic exponent of the L\'evy process. As a by-product we obtain the L\'evy--Khintchine formula.

Continuing this line, we arrive at the first construction of L\'evy processes in Chapter~\ref{cons}. Chapter~\ref{spe} is devoted to two very special L\'evy processes: (compound) Poisson processes and Brownian motion. We give elementary constructions of both processes and show how and why they are special L\'evy processes, indeed. This is also the basis for the next chapter (Chapter~\ref{rm}) where we construct a random measure from the jumps of a L\'evy process. This can be used to provide a further construction of L\'evy processes, culminating in the famous L\'evy--It\^o decomposition and yet another proof of the L\'evy--Khintchine formula.

A second interlude (Chapter~\ref{si}) embeds these random measures into the larger theory of random orthogonal measures. We show how we can use random orthogonal measures to develop an extension of It\^o's theory of stochastic integrals for square-integrable (not necessarily continuous) martingales, but we restrict ourselves to the bare bones, i.e.\ the $L^2$-theory. In Chapter~\ref{feller} we introduce Feller processes as the proper spatially inhomogeneous brethren of L\'evy processes, and we show how our proof of the L\'evy--Khintchine formula carries over to this setting. We will see, in particular, that Feller processes have a \emphh{symbol} which is the state-space dependent analogue of the characteristic exponent of a L\'evy process. The symbol describes the process and its generator. A probabilistic way to calculate the symbol and some first consequences (in particular the semimartingale decomposition of Feller processes) is discussed in Chapter~\ref{symb}; we also show that the symbol contains information on global properties of the process, such as conservativeness. In the final Chapter~\ref{den}, we summarize (mostly without proofs) how other path properties of a Feller process can be obtained via the symbol. In order to make these notes self-contained, we collect in the appendix some material which is not always included in standard graduate probability courses.

\medskip
It is now about time to thank many individuals who helped to bring this enterprise on the way. I am grateful to the scientific board and the organizing committee for the kind invitation to deliver these lectures at the \emph{Centre de Recerca Matem\`atica} in Barcelona. The CRM is a wonderful place to teach and to do research, and I am very happy to acknowledge their support and hospitality. I would like to thank the students who participated in the CRM course as well as all students and readers who were exposed to earlier (temporally \& spatially inhomogeneous\dots) versions of my lectures; without your input these notes would look different!

I am greatly indebted to Ms.\ Franziska K\"{u}hn for her interest in this topic; her valuable comments pinpointed many mistakes and helped to make the presentation much clearer.

And, last and most, I thank my wife for her love, support and forbearance while these notes were being prepared.

\bigskip
\noindent Dresden, September 2015 \hfill \emph{Ren\'e L.\ Schilling}

\chapter*{Symbols and notation}
\addcontentsline{toc}{chapter}{Symbols and notation}
\markboth{Symbols and notation}{Symbols and notation}

\setlength{\columnseprule}{.5pt}
\noindent
This index is intended to aid cross-referencing, so notation that is
specific to a single chapter is generally not listed. Some symbols are
used locally, without ambiguity, in senses other than those given below;
numbers following an entry are page numbers.

Unless otherwise stated,
functions are real-valued and binary operations
between functions such as
$f\pm g$, $f\cdot g$, $f\wedge g$, $f\vee g$, comparisons $f\leq g$, $f
< g$ or limiting relations $f_n \xrightarrow{n\to\infty} f$, $\lim_n
f_n$, $\liminf_n f_n$, $\limsup_n f_n$, $\sup_n f_n$ or $\inf_n f_n$
are understood pointwise.

\begin{multicols}{2}
\raggedright
\parindent=0pt
\begin{nindex}\itemsep3pt\sloppy
\item[\sffamily\bfseries General notation: analysis\vphantom{$\int\limits_a$}]
\item[positive]                 always in the sense $\geq 0$
\item[negative]                 always in the sense $\leq 0$
\item[$\nat$]                   $1,2,3,\dots$
\item[$\inf\emptyset$]          $\inf\emptyset = +\infty$
\item[$a\vee b$]                maximum of $a$ and $b$
\item[$a\wedge b$]              minimum of $a$ and $b$
\item[$\entier x$]              largest integer $n\leq x$
\item[$|x|$]                    norm in $\rd$: $|x|^2=x_1^2+\cdots+x_d^2$
\item[$x\cdot y$]               scalar product in $\rd$: $\sum_{j=1}^d x_jy_j$
\item[$\I_A$]                   $\I_A(x) =
                                \begin{cases}
                                1, & x\in A\\
                                0, & x\notin A
                                \end{cases}$
\item[$\delta_x$]               point mass at $x$
\item[$\dom$]                   domain
\item[$\Delta$]                 Laplace operator
\item[$\partial_j$]             partial derivative $\frac{\partial}{\partial x_j}$
\item[$\nabla$, $\nabla_x$]     gradient $\big(\frac{\partial}{\partial x_1}, \ldots, \frac{\partial}{\partial x_d}\big)^\top$
\item[$\Fcal f, \;\widehat f$]    Fourier transform
                                $(2\pi)^{-d}\int \eup^{-\iup x\cdot\xi} f(x)\,dx$
\item[$\Fcal^{-1} f,\; \widecheck f$]    inverse Fourier transform
                                $\int \eup^{\iup x\cdot\xi} f(x)\,dx$
\item[$\eup_\xi(x)$]            $\eup^{-\iup x\cdot\xi}$

\itemsep3pt
\medskip
\item[\sffamily\bfseries General notation: probability\vphantom{$\int\limits_a$}]

\item[$\sim$]                   `is distributed as'
\item[$\bbot$]                  `is stochastically independent'
\item[a.s.]                     almost surely (w.\,r.\,t.\ $\Pp$)
\item[iid]                      independent and identically distributed
\item[$\Normal, \Exp, \Poi$]    normal, exponential, Poisson distribution
\item[$\Pp, \Ee$]               probability, expectation
\item[$\Vv, \Cov$]              variance, covariance
\item[\eqref{Lnull}--\eqref{Lcont}] definition of a L\'evy process,  \pageref{Lnull}
\item[\eqref{LindepAlt}]              \pageref{LindepAlt}

\smallskip
\item[\sffamily\bfseries Sets and $\sigma$-algebras\vphantom{$\int\limits_a$}]

\item[$A^c$]                        complement of the set $A$
\item[$\overline{A}$]               closure of the set $A$
\item[$A\dcup B$]                   disjoint union, i.e.\ $A\cup B$ for disjoint sets $A\cap B=\emptyset$
\item[$B_r(x)$]                     open ball,\linebreak centre $x$, radius $r$
\item[$\supp f$]                    support, $\overline{\{f\neq 0\}}$
\item[$\Bscr(E)$]                   Borel sets of $E$
\item[$\Fscr_t^X$]                  canonical filtration $\sigma(X_s\,:\, s\leq t)$
\item[$\Fscr_{\infty}$]             $\sigma\left(\bigcup_{t\geq 0}\Fscr_t\right)$
\item[$\Fscr_\tau$]                 \pageref{feller-05} 
\item[$\Fscr_{\tau+}$]              \pageref{lpmp-31} 
\item[$\Pscr$]                      predictable $\sigma$-algebra, \pageref{app-61}

\medskip

\item[\sffamily\bfseries Stochastic processes\vphantom{$\int\limits_a$}]

\item[$\Pp^x, \Ee^x$]           law and mean of a Markov process starting at $x$, \pageref{lpmp-07}
\item[$X_{t-}$]                 left limit $\lim_{s\uparrow t} X_s$
\item[$\Delta X_t$]             jump at time $t$: $X_t-X_{t-}$
\item[$\sigma, \tau$]           stopping times: $\{\sigma\leq t\}\in\Fscr_t$, $t\geq 0$
\item[$\tau_r^x, \tau_r$]       $\inf\{t>0\::\: |X_t-X_0|\geq r\}$, first exit time from the open ball $B_r(x)$ centered at $x=X_0$
\item[c\`adl\`ag]               right continuous on $[0,\infty)$ with finite left limits on $(0,\infty)$

\itemsep3pt
\medskip
\item[\sffamily\bfseries Spaces of functions\vphantom{$\int\limits_a$}]

\item[$\Bcal(E)$]          Borel functions on $E$
\item[$\Bcal_b(E)$]        ~--  ~-- , bounded
\item[$\cont(E)$]           continuous functions\ on $E$
\item[$\cont_b(E)$]         ~--  ~-- , bounded
\item[$\cont_\infty(E)$]    ~--  ~-- , $\lim\limits_{|x|\to\infty} f(x)=0$
\item[$\cont_c(E)$]         ~--  ~-- , compact support

\item[$\cont^n(E)$]         $n$ times continuously diff'ble functions on $E$
\item[$\cont^n_b(E)$]       ~--  ~-- , bounded (with all derivatives)
\item[$\cont^n_\infty(E)$]  ~--  ~-- , $0$ at infinity (with all derivatives)
\item[$\cont^n_c(E)$]       ~--  ~-- , compact support

\item[$L^p(E,\mu), L^p(\mu), L^p(E)$] \hfill\mbox{}
                            $L^p$ space w.\,r.\,t.\ the measure space $(E,\Ascr,\mu)$
\item[$\Scal(\rd)$]         rapidly decreasing smooth functions on $\rd$, \pageref{gen}
\end{nindex}
\end{multicols}


\chapter{Orientation}\label{orient}

Stochastic processes with stationary and independent increments are classical examples of Markov processes. Their importance both in theory and for applications justifies to study these processes and their history.

The origins of processes with independent increments reach back to the late 1920s and they are closely connected with the notion of infinite divisibility and the genesis of the L\'evy--Khintchine formula. Around this time, the limiting behaviour of sums of independent random variables
$$
    X_0:= 0 \et X_n := \xi_1+\xi_2+\dots+\xi_n,\quad n\in\nat,
$$
was well understood through the contributions of Borel, Markov, Cantelli, Lindeberg, Feller, de Finetti, Khintchine, Kolmogorov and, of course, L\'evy;  two new developments emerged, on the one hand the study of dependent random variables and, on the other, the study of continuous-time analogues of sums of independent random variables. In order to pass from $n\in\nat$ to a continuous parameter $t\in [0,\infty)$ we need to replace the steps $\xi_k$ by increments $X_t-X_s$. It is not hard to see that $X_t$, $t\in\nat$, with iid (independent and identically distributed) steps $\xi_k$ enjoys the following properties:
\index{independent increments \eqref{Lindep}}%
\index{stationary increments \eqref{Lstat}}%
\index{L@\eqref{Lnull}, \eqref{Lstat}, \eqref{Lindep}, \eqref{Lcont}}%
\begin{flalign}
&                                       & X_0 &= 0 \quad\text{a.s.} &&\tag{L0}\label{Lnull}\\
&\text{\emphh{stationary increments}}   & X_t-X_s &\sim X_{t-s}-X_0\qquad\forall s\leq t &&\tag{L1}\label{Lstat}\\
&\text{\emphh{independent increments}}  & X_t-X_s &\bbot \sigma(X_r, r\leq s)\quad \forall s\leq t &&\tag{L2}\label{Lindep}
\intertext{where `$\sim$' stands for `same distribution' and `$\bbot$' for stochastic independence. In the non-discrete setting we will also require a mild regularity condition}
&\text{\emphh{continuity in probability}}  & \lim_{t\to 0} \Pp(&|X_t-X_0| >\epsilon) = 0 \qquad\forall \epsilon>0 &&\tag{L3}\label{Lcont}
\index{continuous in probability}%
\end{flalign}
which rules out fixed discontinuities of the path $t\mapsto X_t$. Under \eqref{Lnull}--\eqref{Lindep} one has that
\begin{equation}\label{orient-e05}
    X_t = \sum_{k=1}^n \xi_{k,n}(t)
    \et \xi_{k,n}(t) = (X_{\frac{kt}{n}}-X_{\frac{(k-1)t}{n}})\text{\ \ are iid}
\end{equation}
for every $n\in\nat$. Letting $n\to\infty$ shows that $X_t$ arises as a suitable limit of (a triangular array of) iid random variables which transforms the problem into a question of limit theorems and \emphh{infinite divisibility}.
\index{L\'evy process!infinite divisibility}%

This was first observed in 1929 by de Finetti \cite{definetti29} who introduces (without naming it, the name is due to Bawly \cite{bawly36} and Khintchine \cite{khintchine37b}) the concept of infinite divisibility of a random variable $X$
\index{infinitely divisible}%
\begin{equation}\label{orient-e07}
    \forall n\quad \exists \text{\ iid random variables\ }\xi_{i,n}\::\: X\sim \sum_{i=1}^n \xi_{i,n}
\end{equation}
and asks for the general structure of infinitely divisible random variables. His paper contains two remarkable results on the characteristic function $\chi(\xi) = \Ee\,\eup^{\iup\xi\cdot X}$ of an infinite divisible random variable (taken from \cite{mai-rog06}):
\begin{description}\index{de Finetti's theorem}%
\item[\normalfont\bfseries De Finetti's first theorem.]\itshape
    A random variable $X$ is infinitely divisible if, and only if, its characteristic function is of the form $\chi(\xi) = \lim_{n\to\infty} \exp\big[-p_n(1-\phi_n(\xi))\big]$ where $p_n\geq 0$ and $\phi_n$ is a characteristic function.
\item[\normalfont\bfseries De Finetti's second theorem.]\itshape
    The characteristic function of an infinitely divisible random variable $X$ is the limit of finite products of Poissonian characteristic functions $$\chi_n(\xi) = \exp\big[-p_n(1-\eup^{\iup h_n\xi})\big],$$ and the converse is also true. In particular, all infinitely divisible laws are limits of convolutions of Poisson distributions.
\end{description}
Because of \eqref{orient-e05}, $X_t$ is infinitely divisible and as such one can construct, \emph{in principle}, all indepen\-dent-\/increment processes $X_t$ as limits of sums of Poisson random variables. The contributions of Kolmogorov \cite{kolmogorov32}, L\'evy \cite{levy34} and Khintchine \cite{khintchine37} show the exact form of the characteristic function of an infinitely divisible random variable
\begin{equation}\label{orient-e09}
    -\log \Ee\,\eup^{\iup\xi\cdot X}
    = -\iup l\cdot\xi + \frac 12\xi\cdot Q\xi + \int_{y\neq 0}\left(1-\eup^{\iup y\cdot\xi}+\iup\xi\cdot y\I_{(0,1)}(|y|)\right)\nu(dy)
\end{equation}
\index{L\'evy--Khintchine formula}%
where $l\in\rd$, $Q\in\real^{d\times d}$ is a positive semidefinite symmetric matrix, and $\nu$ is a measure on $\rd\setminus\{0\}$ such that $\int_{y\neq 0} \min\{1,|y|^2\}\,\nu(dy)<\infty$. This is the famous L\'evy--Khintchine formula. The exact knowledge of \eqref{orient-e09} makes it possible to find the approximating Poisson variables in de Finetti's theorem explicitly, thus leading to a construction of $X_t$.

A little later, and without knowledge of de Finetti's results, L\'evy came up in his seminal paper \cite{levy34} (see also \cite[Chap.~VII]{levy37}) with a decomposition of $X_t$ in four independent components: a deterministic drift, a Gaussian part, the compensated small jumps and the large jumps $\Delta X_s := X_s - X_{s-}$. This is now known as L\'evy--It\^o decomposition:
\index{L\'evy--It\^o decomposition}%
\begin{align}\label{orient-e17}
        X_t
        &= tl + \sqrt{Q}W_t
        + \lim_{\epsilon\to 0}
                \raisebox{-6pt}{\Bigg(}
                    \sum_{\begin{subarray}{c}0<s\leq t\\\epsilon\leq |\Delta X_s|<1\end{subarray}}
                        \Delta X_s - t\!\!\! \raisebox{-2pt}{$\displaystyle\int\limits_{\epsilon\leq |y| < 1}$}\!\!\!y\,\nu(dy)
                \raisebox{-6pt}{\Bigg)}
        + \sum_{\begin{subarray}{c}0<s\leq t\\|\Delta X_s| \geq 1\end{subarray}} \Delta X_s\\
    \label{orient-e18}
        &= tl + \sqrt{Q}W_t + \iint\limits_{(0,t]\times B_1(0)} y\,(N(ds, dy) - ds\,\nu(dy)) + \iint\limits_{(0,t]\times B_1(0)^c} y\,N(ds, dy).
\end{align}
L\'evy uses results from the convergence of random series, notably Kolmogorov's three series theorem, in order to explain the convergence of the series appearing in \eqref{orient-e17}. A rigorous proof based on the representation \eqref{orient-e18} are due to It\^o \cite{ito42} who completed L\'evy's programme to construct $X_t$. The coefficients $l, Q, \nu$ are the same as in \eqref{orient-e09}, $W$ is a $d$-dimensional standard Brownian motion, and $N_\omega((0,t]\times B)$ is the random measure $\#\{s\in (0,t]\,:\, X_s(\omega)-X_{s-}(\omega)\in B\}$ counting the jumps of $X$; it is a Poisson random variable with intensity $\Ee N((0,t]\times B) = t\nu(B)$ for all Borel sets $B\subset\rd\setminus\{0\}$ such that $0\notin\overline{B}$.

Nowadays there are at least six possible approaches to constructing processes with (stationary and) independent increments $X=(X_t)_{t\geq 0}$.

\paragraph{The de Finetti--L\'evy(--Kolmogorov--Khintchine) construction.} The starting point is the observation that each $X_t$ satisfies \eqref{orient-e05} and is, therefore, infinitely divisible. Thus, the characteristic exponent $\log \Ee\,\eup^{\iup\xi\cdot X_t}$ is given by the L\'evy--Khintchine formula \eqref{orient-e09}, and using the triplet $(l,Q,\nu)$ one can construct a drift $lt$, a Brownian motion $\sqrt{Q}W_t$ and compound Poisson processes, i.e.\ Poisson processes whose intensities $y\in\rd$ are mixed with respect to the finite measure $\nu_\epsilon(dy):=\I_{[\epsilon,\infty)}(|y|)\nu(dy)$. Using a suitable compensation (in the spirit of Kolmogorov's three series theorem) of the small jumps, it is possible to show that the limit $\epsilon\to 0$ exists locally uniformly in $t$. A very clear presentation of this approach can be found in Breiman \cite[Chapter 14.7--8]{breiman68}, see also Chapter~\ref{cons}.

\paragraph{The L\'evy--It\^o construction.} This is currently the most popular approach to in\-de\-pen\-dent-in\-crement processes, see e.g.\ Applebaum \cite[Chapter 2.3--4]{applebaum09} or Kyprianou \cite[Chapter 2]{kyprianou06}. Originally the idea is due to L\'evy \cite{levy34}, but It\^o \cite{ito42} gave the first rigorous construction. It is based on the observation that the jumps of a process with stationary and independent increments define a Poisson random measure $N_\omega([0,t]\times B)$ and this can be used to obtain the L\'evy--It\^o decomposition \eqref{orient-e18}. The L\'evy--Khintchine formula is then a corollary of the pathwise decomposition. Some of the best presentations can be found in Gikhman--Skorokhod \cite[Chapter VI]{gik-sko69}, It\^o \cite[Chapter 4.3]{ito61} and Bretagnolle \cite{bretagnolle73}. A proof based on additive functionals and martingale stochastic integrals is due to Kunita \& Watanabe \cite[Section 7]{kun-wat67}. We follow this approach in Chapter~\ref{rm}.

\paragraph{Variants of the L\'evy--It\^o construction.}
The L\'evy--It\^o decomposition \eqref{orient-e18} is, in fact, the semimartingale decomposition of a process with stationary and independent increments. Using the general theory of semimartingales -- which heavily relies on general random measures -- we can identify processes with independent increments as those semimartingales whose semimartingale characteristics are deterministic, cf.\ Jacod \& Shiryaev \cite[Chapter II.4c]{jacod-shiryaev87}. A further interesting derivation of the L\'evy--It\^o decomposition is based on stochastic integrals driven by martingales. The key is It\^o's formula and, again, the fact that the jumps of a process with stationary and independent increments defines a Poisson point process which can be used as a good stochastic integrator; this unique approach\footnote{It reminds of the elegant use of It\^o's formula in Kunita-and-Watanabe's proof of L\'evy's characterization of Brownian motion, see e.g.\ Schilling \& Partzsch \cite[Chapter 18.2]{schilling-bm}.} can be found in Kunita \cite[Chapter 2]{kunita04}.

\paragraph{Kolmogorov's construction.} This is the classic construction of stochastic processes starting from the finite-dimensional distributions. For a process with stationary and independent increments these are given as iterated convolutions of the form
\begin{align*}
    \Ee f(&X_{t_0},\dots,X_{t_n})\\
    &= \int\!\!\cdots\!\!\int f(y_0,y_0+y_1,\dots,y_0+\dots+y_n)\,p_{t_0}(dy_0)p_{t_1-t_0}(dy_1)\dots p_{t_n-t_{n-1}}(dy_n)
\end{align*}
with $p_t(dy) = \Pp(X_t\in dy)$ or $\int \eup^{\iup\xi\cdot y}p_t(dy) = \exp\left[-t\psi(\xi)\right]$ where $\psi$ is the characteristic exponent \eqref{orient-e09}. Particularly nice presentations are those of Sato \cite[Chapter 2.10--11]{sato99} and Bauer \cite[Chapter 37]{bauer96}.

\paragraph{The invariance principle.} Just as for a Brownian motion, it is possible to construct L\'evy processes as limits of (suitably interpolated) random walks. For finite dimensional distributions this is done in Gikhman \& Skorokhod \cite[Chapter IX.6]{gik-sko69}; for the whole trajectory, i.e.\ in the space of c\`adl\`ag\footnote{A french acronym meaning `right-continuous and finite limits from the left'.} functions $D[0,1]$ equipped with the Skorokhod topology, the proper references are Prokhorov \cite{prokhorov56} and Grimvall \cite{grimvall72}.

\paragraph{Random series constructions.} A series representation of an independent-increment process $(X_t)_{t\in [0,1]}$ is an expression of the form
$$
    X_t = \lim_{n\to\infty} \sum_{k=1}^n \left(J_k\I_{[0,t]}(U_k) - tc_k\right)
    \quad\text{a.s.}
$$
The random variables $J_k$ represent the jumps, $U_k$ are iid uniform random variables and $c_k$ are suitable deterministic centering terms. Compared with the L\'evy--It\^o decomposition \eqref{orient-e17}, the main difference is the fact that the jumps are summed over a deterministic index set $\{1,2,\dots n\}$ while the summation in \eqref{orient-e17} extends over the random set $\{s\,:\,|\Delta X_s|>1/n\}$. In order to construct a process with characteristic exponent \eqref{orient-e09} where $l=0$ and $Q=0$, one considers a disintegration
$$
    \nu(dy) = \int_0^\infty \sigma(r,dy)\,dr.
$$
It is possible, cf.\ Rosi\'nski \cite{rosinski01}, to choose $\sigma(r, dy) = \Pp(H(r,V_k)\in dy)$ where $V=(V_k)_{k\in\nat}$ is any sequence of $d$-dimensional iid random variables and $H:(0,\infty)\times \rd\to\rd$ is measurable. Now let $\Gamma=(\Gamma_k)_{k\in\nat}$ be a sequence of partial sums of iid standard exponential random variables and $U=(U_k)_{k\in\nat}$ iid uniform random variables on $[0,1]$ such that $U,V,\Gamma$ are independent. Then
$$
    J_k := H(\Gamma_k,V_k)
    \et
    c_k = \int_{k-1}^{k}\int_{|y|<1} y\,\sigma(r,dy)\,dr
$$
is the sought-for series representation, cf.\ Rosi\'nski \cite{rosinski01} and \cite{rosinski90}. This approach is important if one wants to simulate independent-increment processes. Moreover, it still holds for Banach space valued random variables.

\chapter{L\'evy processes}\label{levy}

Throughout this chapter, $(\Omega,\Ascr,\Pp)$ is a fixed probability space, $t_0=0 \leq t_1 \leq\dots\leq t_n$ and $0\leq s<t$ are positive real numbers, and $\xi_k,\eta_k$, $k=1,\dots,n$, denote vectors from $\rd$; we write $\xi\cdot\eta$ for the Euclidean scalar product.

\begin{definition}\label{levy-03}
\index{L\'evy process}%
    A \emphh{L\'evy process} $X=(X_t)_{t\geq 0}$ is a stochastic process $X_t:\Omega\to\rd$ satisfying \eqref{Lnull}--\eqref{Lcont}; this is to say that $X$ starts at zero, has stationary and independent increments and is continuous in probability. 
\end{definition}

One should understand L\'evy processes as continuous-time versions of sums of iid random variables. This can easily be seen from the telescopic sum
\begin{equation}\label{levy-e00}
    X_t - X_s = \sum_{k=1}^n \big(X_{t_k}-X_{t_{k-1}}\big),\quad s<t,\; n\in\nat,
\end{equation}
where $t_k = s+ \tfrac kn(t-s)$. Since the increments $X_{t_k}-X_{t_{k-1}}$ are iid random variables, we see that all $X_t$ of a L\'evy process are \emphh{infinitely divisible},
\index{L\'evy process!infinite divisibility}%
i.e.\ \eqref{orient-e07} holds. Many properties of a L\'evy process will, therefore, resemble those of sums of iid random variables.

Let us briefly discuss the conditions \eqref{Lnull}--\eqref{Lcont}.

\begin{remark}\label{levy-09}
\index{canonical filtration}%
    We have stated \eqref{Lindep} using the \emphh{canonical filtration} $\Fscr_t^X := \sigma(X_r,\; r\leq t)$ of the process $X$. Often this condition is written in the following way
\index{L1@\eqref{LindepAlt}}%
    \begin{equation}\label{LindepAlt}\tag{$\text{\ref{Lindep}}'$}
    \begin{gathered}
        X_{t_n}-X_{t_{n-1}}, \dots, X_{t_1}-X_{t_0}
        \quad\text{are independent random variables}\\
        \text{for all\ \ } n\in\nat,\; t_0=0 < t_1 < \dots < t_n.
    \end{gathered}
    \end{equation}
    It is easy to see that this is actually equivalent to \eqref{Lindep}: From
    $$
        (X_{t_1},\ldots,X_{t_n}) \xleftrightarrow{\:\;\text{bi-measurable}\;\:} (X_{t_1}-X_{t_0},\ldots,X_{t_n}-X_{t_{n-1}})
    $$
    it follows that
    \begin{equation}\label{levy-e02}
    \begin{aligned}
    		\Fscr_t^X
    		&= \sigma\big((X_{t_1},\dots,X_{t_n}),\; 0\leq t_1 \leq \ldots \leq t_n \leq t\big) \\
    		&= \sigma\big((X_{t_1}-X_{t_0},\dots,X_{t_n}-X_{t_{n-1}}),\; 0=t_0\leq t_1 \leq \ldots \leq t_n \leq t\big) \\
    		&= \sigma\big(X_u-X_v,\; 0 \leq v \leq u \leq t\big),
    \end{aligned}
    \end{equation}
    and we conclude that \eqref{Lindep} and \eqref{LindepAlt} are indeed equivalent.

    The condition \eqref{Lcont} is equivalent to either of the following
    \index{continuous in probability}%
    \begin{itemize}
    \item `$t\mapsto X_t$ is continuous in probability';
    \item `$t\mapsto X_t$ is a.s.\ c\`adl\`ag'\footnote{`Right-continuous and finite limits from the left'} (up to a modification of the process).
    \end{itemize}
    The equivalence with the first claim, and the direction `$\Leftarrow$' of the second claim are easy:
    \begin{equation}\label{levy-e04}
        \lim_{u\to t}\Pp(|X_u-X_t|> \epsilon)
        = \lim_{|t-u|\to 0}\Pp(|X_{|t-u|}|> \epsilon)
        \leq \lim_{h\to 0} \frac 1\epsilon \Ee(|X_h|\wedge \epsilon),
    \end{equation}
    but
    \index{L\'evy process!c\`adl\`ag modification}%
    it takes more effort to show that continuity in probatility \eqref{Lcont} guarantees that almost all paths are c\`adl\`ag.\footnote{More precisely: that there exists a modification of $X$ which has almost surely c\`adl\`ag paths.} Usually, this is proved by controlling the oscillations of the paths of a L\'evy process, cf.\ Sato \cite[Theorem 11.1]{sato99}, or by the fundamental regularization theorem for submartingales, see Revuz \& Yor \cite[Theorem II.(2.5)]{rev-yor05} and Remark~\ref{feller-04}; in contrast to the general martingale setting
    \cite[Theorem II.(2.9)]{rev-yor05}, we do not need to augment the natural filtration because of \eqref{Lstat} and \eqref{Lcont}. Since our construction of L\'evy processes gives directly a c\`adl\`ag version, we do not go into further detail.
\end{remark}

The condition \eqref{Lcont} has another consequence. Recall that the \emphh{Cauchy--Abel functional equations}
\index{Cauchy--Abel functional equation}%
have unique solutions if, say, \emphh{$\phi$, $\psi$ and $\theta$ are (right-)continuous}:
\begin{align}
       \phi(s+t) &= \phi(s)\cdot \phi(t) & \phi(t) &= \phi(1)^t, \notag\\
       \psi(s+t) &= \psi(s)+\psi(t) \quad     (s,t\geq 0)\!\!\!\!\!\!\!\!\!\!\!\!\!\!\!\!\!\!\!\!\!\!\!\!\!\!\!\!\!\!\!\!&\implies\quad \psi(t)&=\psi(1)\cdot t\label{levy-e06}\\
       \theta(st) &= \theta(s)\cdot\theta(t)      & \theta(t)&= t^c,\; c\geq 0.\notag
\end{align}
The first equation is treated in Theorem~\ref{app-31} in the appendix. For a thorough discussion on conditions ensuring uniqueness we refer to Aczel \cite[Chapter~2.1]{aczel66}.

\begin{proposition}\label{levy-11}
    Let $\Xt$ be a L\'evy process in $\rd$. Then
    \begin{equation}\label{levy-e10}
        \Ee\,\eup^{\iup\xi\cdot X_t}
        = \big[\Ee\,\eup^{\iup\xi\cdot X_1}\big]^t,\quad t\geq 0,\; \xi\in\rd.
    \end{equation}
\end{proposition}
\begin{proof}
	Fix $s,t \geq 0$. We get
    \begin{align*}
		\Ee\,\eup^{\iup \xi\cdot (X_{t+s}-X_s)+ \iup \xi\cdot X_s}
		\omu{\eqref{Lindep}}{=}{} \Ee\,\eup^{\iup \xi\cdot (X_{t+s}-X_s)} \, \Ee\,\eup^{\iup \xi\cdot X_s}
		\omu{\eqref{Lstat}}{=}{} \Ee\,\eup^{\iup\xi\cdot X_t} \, \Ee\,\eup^{\iup\xi\cdot X_s},
	\end{align*}
    or $\phi(t+s)=\phi(t)\cdot\phi(s)$, if we write $\phi(t)=\Ee\,\eup^{\iup\xi\cdot X_t}$. Since $x\mapsto \eup^{\iup\xi\cdot x}$ is continuous, there is for every $\epsilon>0$ some $\delta>0$ such that
    $$
        |\phi(t)-\phi(s)|
        \leq \Ee \big|\eup^{\iup\xi\cdot(X_t-X_s)}-1\big|
        \leq \epsilon + 2\Pp(|X_t-X_s|\geq \delta)
        = \epsilon + 2\Pp(|X_{|t-s|}|\geq \delta).
    $$
    Thus, \eqref{Lcont} guarantees that $t\mapsto \phi(t)$ is continuous, and the claim follows from \eqref{levy-e06}.
\end{proof}

Notice that any solution $f(t)$ of \eqref{levy-e06} also satisfies \eqref{Lnull}--\eqref{Lindep}; by Proposition~\ref{levy-11} $X_t+f(t)$ is a L\'evy process if, and only if, $f(t)$ is continuous. On the other hand, Hamel, cf.\ \cite[p.~35]{aczel66}, constructed discontinuous (non-measurable and locally unbounded) solutions to \eqref{levy-e06}. Thus, \eqref{Lcont} means that $t\mapsto X_t$ has no fixed discontinuities,
\index{L\'evy process!no fixed discontinuities}%
i.e.\ all jumps occur at random times.

\begin{corollary}\label{levy-13}
	The finite-dimensional distributions $\Pp(X_{t_1} \in dx_1,\ldots,X_{t_n} \in dx_n)$ of a L\'evy process are uniquely determined by
\index{L\'evy process!finite-dimensional distributions}%
    \begin{equation}\label{levy-e12}
		\Ee \exp \left( \iup \sum_{k=1}^n \xi_k\cdot X_{t_k} \right)
		= \prod_{k=1}^n \Big[ \Ee \exp \left(\iup (\xi_k+ \dots +\xi_n)\cdot X_1 \right) \Big]^{t_k-t_{k-1}}
	\end{equation}
	for all $\xi_1,\ldots,\xi_n \in \rd$, $n\in\nat$ and $0 = t_0 \leq t_1 \leq \dots \leq t_n$.
\end{corollary}
\begin{proof}
    The left-hand side of \eqref{levy-e12} is just the characteristic function of $(X_{t_1},\dots,X_{t_n})$. Consequently, the assertion follows from \eqref{levy-e12}. Using Proposition~\ref{levy-11}, we have
    \begin{align*}
		\Ee\exp \left(\iup  \sum_{k=1}^n \xi_k\cdot X_{t_k} \right)
        &\pomu{\eqref{Lindep}}{=}{\eqref{Lstat}}
            \Ee \exp \left(\iup \sum_{k=1}^{n-2} \xi_k\cdot X_{t_k} + \iup (\xi_n+\xi_{n-1})\cdot X_{t_{n-1}} + \iup \xi_n\cdot (X_{t_n}-X_{t_{n-1}})\right)\\
		&\omu{\eqref{Lindep}}{=}{\eqref{Lstat}}
            \Ee\exp\left(\iup  \sum_{k=1}^{n-2} \xi_k\cdot X_{t_k} + \iup (\xi_n+\xi_{n-1})\cdot X_{t_{n-1}}  \right)
        \big(\Ee\,\eup^{\iup \xi_n\cdot X_1}\big)^{t_n-t_{n-1}}.
	\end{align*}
    Since the first half of the right-hand side has the same structure as the original expression, we can iterate this calculation and obtain \eqref{levy-e12}.
\end{proof}
It is not hard to invert the Fourier transform in \eqref{levy-e12}. Writing $p_t(dx) := \Pp(X_t\in dx)$ we get
\begin{align}\label{levy-e14}
    \Pp (X_{t_1} \in B_1,\dots,X_{t_n} \in B_n)
    &= \idotsint \prod_{k=1}^n 1_{B_k}(x_1+\dots+x_k) p_{t_k-t_{k-1}}(dx_k)\\
    \label{levy-e16}
    &= \idotsint \prod_{k=1}^n 1_{B_k}(y_k) p_{t_k-t_{k-1}}(dy_k - y_{k-1}).
\end{align}

Let us discuss the structure of the characteristic function $\chi(\xi)=\Ee\,\eup^{\iup\xi\cdot X_1}$ of $X_1$.
From \eqref{levy-e00} we see that each random variable $X_t$ of a L\'evy process is infinitely divisible.
Clearly, $|\chi(\xi)|^2$ is the (real-valued) characteristic function of the symmetrization $\widetilde X_1 = X_1-X_1'$ ($X_1'$ is an independent copy of $X_1$) and $\widetilde X_1$ is again infinitely divisible:
$$
    \widetilde X_1
    = \sum_{k=1}^n (\widetilde X_{\frac kn}-\widetilde X_{\frac{k-1}{n}})
    = \sum_{k=1}^n \left[(X_{\frac kn}-X_{\frac{k-1}{n}}) - (X_{\frac kn}'-X_{\frac{k-1}{n}}')\right].
$$
In particular, $|\chi|^2 = |\chi_{1/n}|^{2n}$ where $|\chi_{1/n}|^2$ is the characteristic function of $\widetilde X_{1/n}$. Since everything is real and $|\chi(\xi)|\leq 1$, we get
$$
    \theta(\xi):=\lim_{n\to\infty} |\chi_{1/n}(\xi)|^2 = \lim_{n\to\infty} |\chi(\xi)|^{2/n},\quad \xi\in\rd,
$$
which is $0$ or $1$ depending on $|\chi(\xi)|=0$ or $|\chi(\xi)|>0$, respectively. As $\chi(\xi)$ is continuous at $\xi=0$ with $\chi(0)=1$, we have $\theta\equiv 1$ in a neighbourhood $B_r(0)$ of $0$. Now we can use L\'evy's continuity theorem (Theorem~\ref{app-85}) and conclude that the limiting function $\theta(\xi)$ is continuous everywhere, hence $\theta\equiv 1$. In particular, $\chi(\xi)$ has no zeroes.

\begin{corollary}\label{levy-15}
\index{characteristic exponent}%
\index{L\'evy process!characteristic exponent}%
	Let $\Xt$ be a L\'evy process in $\rd$. There exists a unique continuous function $\psi:\rd \to \comp$ such that
    $$
		\Ee\exp(\iup   \xi\cdot X_t) =\eup^{-t\psi(\xi)},\quad t\geq 0,\;\xi\in\rd.
	$$
    The function $\psi$ is called the \emphh{characteristic exponent}.
\end{corollary}
\begin{proof}
    In view of Proposition~\ref{levy-11} it is enough to consider $t=1$.
    Set $\chi(\xi):= \Ee\exp(\iup\xi\cdot X_1)$. An obvious candidate for the exponent is $\psi(\xi) = -\log \chi(\xi)$, but with complex logarithms there is always the trouble which branch of the logarithm one should take. Let us begin with the uniqueness:
    $$
        \eup^{-\psi} = \eup^{-\phi}
        \implies
        \eup^{-(\psi-\phi)}=1
        \implies
        \psi(\xi)-\phi(\xi) = 2\pi \iup k_\xi
    $$
    for some integer $k_\xi\in\integer$. Since $\phi, \psi$ are continuous and $\phi(0)=\psi(0)=1$, we get $k_\xi\equiv 0$.

    To prove the existence of the logarithm, it is not sufficient to take the principal branch of the logarithm. As we have seen above, $\chi(\xi)$ is continuous and has no zeroes, i.e.\ $\inf_{|\xi|\leq r} |\chi(\xi)|> 0$ for any $r>0$; therefore, there is a `distinguished', continuous\footnote{A very detailed argument is given in Sato \cite[Lemma 7.6]{sato99}, a completely different proof can be found in Dieudonn\'e \cite[Chapter IX, Appendix 2]{dieudonne69}.} version of the argument $\arg_\circ\chi(\xi)$ such that $\arg_\circ\chi(0)=0$.

    This allows us to take a continuous version of $\log\chi(\xi) = \log|\chi(\xi)|+\arg_\circ\chi(\xi)$.
\end{proof}

\begin{corollary}\label{levy-17}
\index{L\'evy process!infinite divisibility}
	Let $Y$ be an infinitely divisible random variable. Then there exists at most one\footnote{We will see in Chapter~\ref{cons} how to construct this process. It is unique in the sense that its finite-dimensional distributions are uniquely determined by $Y$.} L\'evy process $\Xt$ such that $X_1\sim Y$.
\end{corollary}
\begin{proof}
    Since $X_1\sim Y$, infinite divisibility is a necessary requirement for $Y$. On the other hand, Proposition~\ref{levy-11} and Corollary~\ref{levy-13} show how to construct the finite-dimensional distributions of a L\'evy process, hence the process, from $X_1$.
\end{proof}
So far, we have seen the following one-to-one correspondences
$$
    \Xt\text{\ \ L\'evy process}
    \xleftrightarrow{1:1} \Ee\,\eup^{\iup\xi\cdot X_1}
    \xleftrightarrow{1:1} \psi(\xi) = -\log \Ee\,\eup^{\iup\xi\cdot X_1}
$$
and the next step is to find \emphh{all possible characteristic exponents}. This will lead us to the L\'evy--Khintchine formula.

\chapter{Examples}\label{exlp}

We begin with a useful alternative characterisation of L\'evy processes.
\begin{theorem}\label{exlp-03}
\index{L\'evy process!characterization}%
    Let $X=\Xt$ be a stochastic process with values in $\rd$, $\Pp(X_0=0)=1$ and $\Fscr_t = \Fscr_t^X = \sigma(X_r,\; r\leq t)$. The process $X$ is a L\'evy process if, and only if, there exists an exponent $\psi:\rd\to\comp$ such that
    \begin{equation}\label{exlp-e02}
        \Ee \left( \eup^{\iup\xi\cdot(X_t-X_s)} \:\middle|\: \Fscr_s\right)
        = \eup^{-(t-s)\psi(\xi)} \fa s<t,\; \xi\in\rd.
    \end{equation}
\end{theorem}
\begin{proof}
    If $X$ is a L\'evy process, we get
    $$
        \Ee \left( \eup^{\iup\xi\cdot(X_t-X_s)} \:\middle|\: \Fscr_s\right)
        \omu{\eqref{Lindep}}{=}{}
        \Ee\,\eup^{\iup\xi\cdot(X_t-X_s)}
        \omu{\eqref{Lstat}}{=}{}
        \Ee\,\eup^{\iup\xi\cdot X_{t-s}}
        \omu{\text{Cor.~\ref{levy-15}}}{=}{}
        \eup^{-(t-s)\psi(\xi)}.
    $$

    Conversely, assume that $X_0=0$ a.s.\ and \eqref{exlp-e02} holds. Then
    $$
        \Ee\,\eup^{\iup\xi\cdot(X_t-X_s)}
        = \eup^{-(t-s)\psi(\xi)}
        = \Ee\,\eup^{\iup\xi\cdot(X_{t-s}-X_0)}
    $$
    which shows $X_{t}-X_s\sim X_{t-s}-X_0 = X_{t-s}$, i.e.\ \eqref{Lstat}.

    For any $F\in\Fscr_s$ we find from the tower property of conditional expectation
    \begin{align}\label{exlp-e04}
        \Ee \left( \I_F \cdot \eup^{\iup\xi\cdot(X_t-X_s)}\right)
        = \Ee \left( \I_F  \Ee\left[ \eup^{\iup\xi\cdot(X_t-X_s)} \:\middle|\: \Fscr_s\right]\right)
        = \Ee \I_F\cdot \eup^{-(t-s)\psi(\xi)}.
    \end{align}
    Observe that $\eup^{\iup u \I_F} = \I_{F^c} + \eup^{\iup u}\I_F$ for any $u\in\real$; since both $F$ and $F^c$ are in $\Fscr_s$, we get
    \begin{align*}
        \Ee\left(\eup^{\iup u\I_F}\eup^{\iup\xi\cdot(X_t-X_s)}\right)
        &\pomu{\eqref{exlp-e04}}{=}{} \Ee\left(\I_{F^c} \eup^{\iup\xi\cdot(X_t-X_s)}\right) + \Ee\left(\I_{F} \eup^{\iup u} \eup^{\iup\xi\cdot(X_t-X_s)}\right)\\
        &\omu{\eqref{exlp-e04}}{=}{} \Ee\left(\I_{F^c} + \eup^{\iup u}\I_F\right) \eup^{-(t-s)\psi(\xi)}\\
        &\omu{\eqref{exlp-e04}}{=}{} \Ee\,\eup^{\iup u\I_F}\, \Ee\,\eup^{\iup\xi\cdot(X_t-X_s)}.
    \end{align*}
    Thus, $\I_F \bbot (X_t-X_s)$ for any $F\in\Fscr_s$, and \eqref{Lindep} follows.

    Finally, $\lim_{t\to 0} \Ee\,\eup^{\iup\xi\cdot X_t} = \lim_{t\to 0} \eup^{-t\psi(\xi)} = 1$ proves that $X_t \to 0$ in distribution, hence in probability. This gives \eqref{Lcont}.
\end{proof}

Theorem~\ref{exlp-03} allows us to give concrete examples of L\'evy processes.
\begin{example}\label{exlp-05}
The following processes are L\'evy processes.

\medskip\noindent a)
    \emphh{Drift} in direction $l/|l|$, $l\in\rd$, with speed $|l|$: $X_t = tl$ and $\psi(\xi) = -\iup l\cdot\xi$.

\medskip\noindent b) \index{Brownian motion}%
    \emphh{Brownian motion} with (positive semi-definite) covariance matrix $Q\in\real^{d\times d}$: Let $(W_t)_{t\geq 0}$ be a standard Wiener process on $\rd$ and set $X_t := \sqrt{Q}W_t$.

    \noindent
    Then $\psi(\xi) = \frac 12 \xi\cdot Q\xi$ and $\Pp(X_t\in dy) = (2\pi t)^{-d/2} (\det Q)^{-1/2} \exp(-y\cdot Q^{-1}y / 2t)\,dy$.

\medskip\noindent c)
\index{Poisson process}%
    \emphh{Poisson process} in $\real$ with jump height $1$ and intensity $\lambda$. This is an integer-valued counting process $(N_t)_{t\geq 0}$ which increases by $1$ after an independent exponential waiting time with mean $\lambda$. Thus,
    $$
        N_t = \sum_{k=1}^\infty \I_{[0,t]}(\tau_k),\quad
        \tau_k = \sigma_1+\dots+\sigma_k,\quad
        \sigma_k\sim\Exp(\lambda)\text{\ \ iid}.
    $$
    Using this definition, it is a bit messy to show that $N$ is indeed a L\'evy process (see e.g.\ \c{C}inlar \cite[Chapter 4]{cinlar75}). We will give a different proof in Theorem~\ref{exlp-13} below. Usually, the first step is to show that its law is a Poisson distribution
    $$
        \Pp(N_t = k) = \eup^{-t\lambda}\,\frac{(\lambda t)^k}{k!},\quad k=0,1,2,\dots
    $$
    (thus the name!) and from this one can calculate the characteristic exponent
    \begin{align*}
        \Ee\,\eup^{\iup uN_t}
        = \sum_{k=0}^\infty \eup^{\iup uk}\eup^{-t\lambda}\,\frac{(\lambda t)^k}{k!}
        = \eup^{-t\lambda} \exp\big[\lambda t \eup^{\iup u}\big]
        = \exp\big[-t\lambda(1-\eup^{\iup u})\big],
    \end{align*}
    i.e.\ $\psi(u) = \lambda (1-\eup^{\iup u})$. Mind that this is strictly weaker than \eqref{exlp-e02} and does not prove that $N$ is a L\'evy process.

\medskip\noindent d)
\index{compound Poisson process}%
    \emphh{Compound Poisson process} in $\rd$ with jump distribution $\mu$ and intensity $\lambda$. Let $N=(N_t)_{t\geq 0}$ be a Poisson process with intensity $\lambda$ and replace the jumps of size $1$ by independent iid jumps of random height $H_1, H_2, \dots$ with values in $\rd$ and $H_1\sim\mu$. This is a compound Poisson process:
    $$
        C_t = \sum_{k=1}^{N_t} H_k,\quad
        H_k\sim\mu \text{\ \ iid and independent of $(N_t)_{t\geq 0}$}.
    $$
    We will see in Theorem~\ref{exlp-13} that compound Poisson processes are L\'evy processes.
\end{example}

\index{Campbell's formula|(}%
Let us show that the Poisson and compound Poisson processes are L\'evy processes. For this we need the following auxiliary result. Since $t\mapsto C_t$ is a step function, the Riemann--Stieltjes integral $\int f(u)\,dC_u$ is well-defined.
\begin{lemma}[Campbell's formula]\label{exlp-11}
    Let $C_t = H_1+\dots +H_{N_t}$ be a compound Poisson process as in Example~\ref{exlp-05}.\textup{d)} with iid jumps $H_k\sim\mu$ and an independent Poisson process $(N_t)_{t\geq 0}$ with intensity $\lambda$. Then
    \begin{equation}\label{exlp-e12}
		\Ee\exp \left(\iup \int_0^{\infty} f(t+s) \, dC_t \right)
        = \exp \left(\lambda \int_0^{\infty} \!\! \int_{y \neq 0} (\eup^{\iup  y f(s+t)}-1) \, \mu(dy) \, dt \right)
	\end{equation}
    holds for all $s\geq 0$ and bounded measurable functions $f:[0,\infty)\to\rd$ with compact support.
\end{lemma}
\index{Campbell's formula|)}%
\begin{proof}
	Set $\tau_k = \sigma_1+\dots+\sigma_k$ where $\sigma_k \sim \Exp(\lambda)$ are iid. Then
    \begin{align*}
		\phi(s)
		\pomu{\text{iid}}{:=}{}&
            \Ee\exp \left(\iup \int_0^{\infty} f(s+t) \, dC_t \right) \\
		\pomu{\text{iid}}{=}{}&
            \Ee\exp \left(\iup \sum_{k=1}^\infty f(s+\sigma_1+\dots+\sigma_k) H_k \right) \\
		\omu{\text{iid}}{=}{}&
            \int_0^{\infty} \underbrace{\Ee\exp \left(\iup \sum_{k=2}^\infty f(s+x+\sigma_2+\dots+\sigma_k)H_k \right)}_{=\phi(s+x)}
            \underbrace{\Ee\exp(\iup  f(s+x) H_1)\vphantom{\left(\sum_{k=2}^\infty\right)}}_{=:\gamma(s+x)}
            \underbrace{\Pp(\sigma_1\in dx)\vphantom{\left(\sum_{k=2}^\infty\right)}}_{=\lambda \eup^{-\lambda x} \, dx} \\
		\pomu{\text{iid}}{=}{}&
            \lambda \int_0^{\infty} \phi(s+x) \gamma(s+x)\eup^{-\lambda x} \, dx \\
		\pomu{\text{iid}}{=}{}&
            \lambda \eup^{\lambda s} \int_s^{\infty} \gamma(t) \phi(t) \eup^{-\lambda t} \, dt.
	\end{align*}
	This is equivalent to
    $$
		\eup^{-\lambda s} \phi(s) = \lambda \int_s^{\infty} (\phi(t)\eup^{-\lambda t}) \gamma(t) \, dt
	$$
    and $\phi(\infty)=1$ since $f$ has compact support. This integral equation has a unique solution; it is now a routine exercise to verify that the right-hand side of \eqref{exlp-e12} is indeed a solution.
\end{proof}

\begin{theorem}\label{exlp-13}
\index{Poisson process!is a L\'evy process}%
\index{compound Poisson process!is a L\'evy process}%
    Let $C_t = H_1+\dots +H_{N_t}$ be a compound Poisson process as in Example~\ref{exlp-05}.\textup{d)} with iid jumps $H_k\sim\mu$ and an independent Poisson process $(N_t)_{t\geq 0}$ with intensity $\lambda$. Then $(C_t)_{t \geq 0}$ \textup{(}and also $(N_t)_{t \geq 0}$\textup{)} is a $d$-dimensional L\'evy process with characteristic exponent
    \begin{equation}\label{exlp-e14}
		\psi(\xi) = \lambda \int_{y \neq 0} (1-\eup^{\iup  y\cdot \xi}) \, \mu(dy).
	\end{equation}
\end{theorem}
\begin{proof}
    Since the trajectories of $t\mapsto C_t$ are c\`adl\`ag step functions with $C_0=0$, the properties \eqref{Lnull} and \eqref{Lcont}, see \eqref{levy-e04}, are satisfied. We will show \eqref{Lstat} and \eqref{Lindep}. Let $\xi_k \in \rd$, $0=t_0  \leq \dots \leq t_n$, and $a<b$. Then the Riemann--Stieltjes integral
    \begin{align*}
		\int_0^{\infty} \I_{(a,b]}(t) \, dC_t = \sum_{k=1}^\infty \I_{(a,b]}(\tau_k) H_k = C_b-C_a
	\end{align*}
    exists. We apply the Campbell formula \eqref{exlp-e12} to the function
	$$
		f(t) := \sum_{k=1}^n \xi_k \I_{(t_{k-1},t_k]}(t)
	$$
	and with $s=0$. Then the left-hand side of \eqref{exlp-e12} becomes the characteristic function of the increments
    $$
		\Ee\exp \left(\iup \sum_{k=1}^n \xi_k \cdot(C_{t_k}-C_{t_{k-1}}) \right),
	$$
    while the right-hand side is equal to
    \begin{align*}
		\exp \left[ \lambda \int_{y \neq 0} \sum_{k=1}^n \int_{t_{k-1}}^{t_k} (\eup^{\iup  \xi_k \cdot y}-1) \, dt \, \mu(dy) \right]
		&= \prod_{k=1}^n \exp \left[ \lambda (t_k-t_{k-1}) \int_{y \neq 0} (\eup^{\iup  \xi_k\cdot y}-1) \, \mu(dy) \right] \\
		&= \prod_{k=1}^n \Ee\exp \left[\iup   \xi_k\cdot C_{t_k-t_{k-1}} \right]
	\end{align*}
    (use Campbell's formula with $n=1$ for the last equality).
    This shows that the increments are independent, i.e.\ \eqref{LindepAlt} holds, as well as \eqref{Lstat}: $C_{t_k}-C_{t_{k-1}}\sim C_{t_k-t_{k-1}}$.

    If $d=1$ and $H_k\sim\delta_1$, $C_t$ is a Poisson process.
\end{proof}

Denote by $\mu^{*k}$ the $k$-fold convolution of the measure $\mu$; as usual, $\mu^{*0}:= \delta_0$.
\begin{corollary}\label{exlp-15}
\index{compound Poisson process!distribution}%
    Let $(N_t)_{t\geq 0}$ be a Poisson process with intensity $\lambda$ and $C_t = H_1+\dots +H_{N_t}$ a compound Poisson process with iid jumps $H_k\sim\mu$. Then, for all $t\geq 0$,
    \begin{align}
    \label{exlp-e16}
        \Pp(N_t = k)    &= \eup^{-\lambda t}\,\frac{(\lambda t)^k}{k!}, \quad k=0,1,2,\dots \\
    \label{exlp-e18}
        \Pp(C_t \in B)  &= \eup^{-\lambda t} \sum_{k=0}^\infty \frac{(\lambda t)^k}{k!}\,\mu^{*k}(B),\quad B\subset\rd\text{\ \ Borel}.
    \end{align}
\end{corollary}
\begin{proof}
    If we use Theorem~\ref{exlp-13} for $d=1$ and $\mu = \delta_1$, we see that the characteristic function of $N_t$ is $\chi_t(u) = \exp[-\lambda t(1-\eup^{\iup u})]$. Since this is also the characteristic function of the Poisson distribution (i.e.\ the r.h.s.\ of \eqref{exlp-e16}), we get $N_t\sim\Poi(\lambda t)$.

    Since $(H_k)_{k\in\nat}\bbot (N_t)_{t\geq 0}$, we have for any Borel set $B$
    \begin{align*}
        \Pp(C_t\in B)
        &= \sum_{k=0}^\infty \Pp(C_t\in B,\: N_t=k)\\
        &= \delta_0(B)\Pp(N_t=0) + \sum_{k=1}^\infty \Pp(H_1+\dots+H_k\in B)\Pp(N_t=k)\\
        &= \eup^{-\lambda t} \sum_{k=0}^\infty \frac{(\lambda t)^k}{k!}\mu^{*k}(B).
    \qedhere
    \end{align*}
\end{proof}

Example~\ref{exlp-05} contains the basic L\'evy processes which will also be the building blocks for all L\'evy processes. In order to define more specialized L\'evy processes, we need further assumptions on the distributions of the random variables $X_t$.
\begin{definition}\label{exlp-31}
    Let $\Xt$ be a stochastically continuous process in $\rd$. It is called \emphh{self-similar}, if
    \begin{equation}\label{exlp-e32}
        \forall a\geq 0\quad\exists b=b(a) \::\: (X_{at})_{t\geq 0} \sim (bX_t)_{t\geq 0}
    \end{equation}
    in the sense that both sides have the same finite-dimensional distributions.
\end{definition}
\begin{lemma}[Lamperti]\label{exlp-33}
\index{index of self-similarity}%
    If $\Xt$ is self-similar and non-degenerate, then there exists a unique \emphh{index of self-similarity} $H\geq 0$ such that $b(a)=a^H$. If $\Xt$ is a self-similar L\'evy process, then $H\geq \frac 12$.
\end{lemma}
\begin{proof}
    Since $\Xt$ is self-similar, we find for $a, a'\geq 0$ and each $t>0$
    $$
        b(aa')X_t \sim X_{aa't} \sim b(a)X_{a't} \sim b(a)b(a')X_t,
    $$
    and so $b(aa')=b(a)b(a')$ as $X_t$ is non-degenerate.\footnote{We use here that $bX\sim cX\implies b=c$ if $X$ is non-degenerate. To see this, set $\chi(\xi)=\Ee\,\eup^{\iup\xi\cdot X}$ and notice
    $$
        |\chi(\xi)| = \big|\chi\big(\tfrac bc\xi\big)\big| = \dots = \big|\chi\big(\big(\tfrac bc\big)^n\xi\big)\big|.
    $$
    If $b<c$, the right-hand side converges for $n\to\infty$ to $\chi(0)=1$, hence $|\chi|\equiv 1$, contradicting the fact that $X$ is non-degenerate. Since $b,c$ play symmetric roles, we conclude that $b=c$.} By the convergence of types theorem (Theorem~\ref{app-03}) and the continuity in probability of $t\mapsto X_t$ we see that $a\mapsto b(a)$ is continuous. Thus, the Cauchy functional equation $b(aa') = b(a)b(a')$ has the unique continuous solution $b(a) = a^H$ for some $H\geq 0$.

    Assume now that $\Xt$ is a L\'evy process. We are going to show that $H \geq\frac 12$.
    Using self-similarity and the properties \eqref{Lstat}, \eqref{Lindep} we get (primes always denote iid copies of the respective random variables)
    \begin{equation}\label{exlp-e34}
        (n+m)^H X_1 \sim X_{n+m} = (X_{n+m}-X_m) + X_m \sim X_n'' + X_m' \sim n^H X_1'' + m^H X_1'.
    \end{equation}
    Any standard normal random variable $X_1$ satisfies \eqref{exlp-e34} with $H=\frac 12$. On the other hand, if $X_1$ has a second moment, we get $
        (n+m)\Vv X_1 = \Vv X_{n+m} = \Vv X_n'' + \Vv X_m' = n\Vv X_1'' + m\Vv X_1'
    $
    by Bienaym\'es identity for variances, i.e.\ \eqref{exlp-e34} can only hold with $H=\frac 12$. Thus, any self-similar $X_1$ with finite second moment has to satisfy \eqref{exlp-e34} with $H=\frac 12$. If we can show that $H<\frac 12$ implies the existence of a second moment, we have reached a contradiction.

    If $X_n$ is symmetric and $H<\frac 12$, we find because of $X_n\sim n^H X_1$ some $u>0$ such that
    $$
        \Pp(|X_n|> u n^H) = \Pp(|X_1|> u) < \frac 14.
    $$
    By the symmetrization inequality (Theorem~\ref{app-13}),
    $$
        \frac 12\left(1-\exp\{-n\Pp(|X_1|> u n^H)\}\right)
        \leq \Pp(|X_n|> u n^H)
        < \frac 14
    $$
    which means that $n\Pp(|X_1|>u n^H)\leq c$ for all $n\in\nat$. Thus, $\Pp(|X_1|>x)\leq c'x^{-1/H}$ for all $x>u+1$, and so
    $$
        \Ee |X_1|^2
        = 2\int_0^\infty x\, \Pp(|X_1|>x)\,dx
        \leq 2(u+1) + 2c'\int_{u+1}^\infty x^{1- 1/H}\,dx < \infty
    $$
    as $H<\frac 12$. If $X_n$ is not symmetric, we use its symmetrization $X_n-X_n'$ where $X_n'$ are iid copies of $X_n$.
\end{proof}

\begin{definition}\label{exlp-35}
\index{stable random variable}%
    A random variable $X$ is called \emphh{stable} if
    \begin{equation}\label{exlp-e36}
        \forall n\in\nat\quad
        \exists b_n\geq 0,\; c_n\in\rd\::\:
        X_1'+\dots+X_n' \sim b_n X + c_n
    \end{equation}
    where $X_1',\dots,X_n'$ are iid copies of $X$. If \eqref{exlp-e36} holds with $c_n=0$, the random variable is called \emphh{strictly stable}.
    A L\'evy process $\Xt$ is (strictly) stable if $X_1$ is a (strictly) stable random variable.
\end{definition}
Note that the symmetrization $X-X'$ of a stable random variable is strictly stable. Setting $\chi(\xi)=\Ee\,\eup^{\iup\xi\cdot X}$ it is easy to see that  \eqref{exlp-e36} is equivalent to
    \begin{gather}\tag{$\text{\ref{exlp-e36}}'$}\label{exlpE36Alt}
        \forall n\in\nat\quad
        \exists b_n\geq 0,\; c_n\in\rd\::\:
        \chi(\xi)^n = \chi(b_n\xi) \eup^{\iup c_n\cdot\xi}.
    \end{gather}
\begin{example}
\medskip\noindent a) \emphh{Stable processes.}
\index{stable process}\index{L\'evy process!stable process}%
    By definition, any stable random variable is infinitely divisible, and for every stable $X$ there is a unique L\'evy process on $\rd$ such that $X_1\sim X$, cf.\ Corollary~\ref{levy-17}.

    A L\'evy process $\Xt$ is stable if, and only if, all random variables $X_t$ are stable. This follows at once from \eqref{exlpE36Alt} if we use $\chi_t(\xi):= \Ee\,\eup^{\iup\xi\cdot X_t}$:
    $$
        \chi_t(\xi)^n
        \omu{\eqref{levy-e10}}{=}{} \big(\chi_1(\xi)^n\big)^t
        \omu{\eqref{exlpE36Alt}}{=}{} \chi_1(b_n\xi)^t \eup^{\iup(t c_n)\cdot\xi}
        \omu{\eqref{levy-e10}}{=}{} \chi_t(b_n\xi) \eup^{\iup(t c_n)\cdot\xi}.
    $$
    It is possible to determine the characteristic exponent of a stable process, cf.\ Sato \cite[Theorem 14.10]{sato99} and \eqref{exlp-e38} further down.

\medskip\noindent b) \emphh{Self-similar processes.}
\index{self-similar process}\index{L\'evy process!stable process}%
    Assume that $\Xt$ is a self-similar L\'evy process. Then
    $$
        \forall n\in\nat\::\: b(n) X_1 \sim X_n = \sum_{k=1}^n (X_{k}-X_{k-1}) \sim X'_{1,n}+\dots+X'_{n,n}
    $$
    where the $X_{k,n}'$ are iid copies of $X_1$. This shows that $X_1$, hence $\Xt$, is strictly stable. In fact, the converse is also true:

\medskip\noindent c)
    \emphh{A strictly stable L\'evy process is self-similar}. We have already seen in b) that self-similar L\'evy processes are strictly stable.
    Assume now that $\Xt$ is strictly stable. Since $X_{nt}\sim b_nX_t$ we get
    $$
        \eup^{-nt\psi(\xi)} = \Ee\,\eup^{\iup\xi\cdot X_{nt}} = \Ee\,\eup^{\iup b_n\xi\cdot X_t} = \eup^{-t\psi(b_n\xi)}.
    $$
    Taking $n=m$, $t\rightsquigarrow t/m$ and $\xi \rightsquigarrow b_m^{-1}\xi$ we see
    $$
        \eup^{-\frac tm\psi(\xi)} = \eup^{-t\psi(b_m^{-1}\xi)}.
    $$
    From these equalities we obtain for $q=n/m\in\mathds{Q}^+$ and $b(q):= b_n/b_m$
    $$
        \eup^{- qt\psi(\xi)} = \eup^{-t\psi(b(q)\xi)}
        \implies
        X_{qt} \sim b(q) X_t
        \implies
        X_{at} \sim b(a) X_t
    $$
    for all $t\geq 0$ because of the continuity in probability of $\Xt$. Since, by Corollary~\ref{levy-13}, the finite-dimensional distributions are determined by the one-dimensional distributions, we conclude that \eqref{exlp-e32} holds.

    This means, in particular, that strictly stable L\'evy processes have an index of self-similarity $H\geq\frac 12$. It is common to call $\alpha = 1/H \in (0,2]$ the \emphh{index of stability}
\index{index of stability}%
    of $\Xt$, and we have $X_{nt}\sim n^{1/\alpha}X_t$.

    If $X$ is `only' stable, its symmetrization is strictly stable and, thus, every stable L\'evy process has an index $\alpha\in (0,2]$. It plays an important role for the characteristic exponent. For a general stable process the characteristic exponent is of the form
\index{characteristic exponent!stable process@of a stable process}
    \begin{align}\label{exlp-e38}
        \psi(\xi) =
        \begin{cases}\displaystyle
            \int_{\mathds S^d}|z\cdot\xi|^\alpha \big(1-\iup\sgn(z\cdot\xi)\tan\tfrac {\alpha\pi}2\big)\,\sigma(dz) - \iup\mu\cdot\xi, & (\alpha \neq 1),\\[10pt]
        \displaystyle
            \int_{\mathds S^d}|z\cdot\xi| \big(1+\tfrac 2\pi \iup\sgn(z\cdot\xi)\log |z\cdot\xi|\big)\,\sigma(dz) - \iup\mu\cdot\xi, & (\alpha = 1),
        \end{cases}
    \end{align}
    where $\sigma$ is a finite measure on $\mathds S^d$ and $\mu\in\rd$. The strictly stable exponents have $\mu=0$ (if $\alpha\neq 1$) and $\int_{\mathds S^d} z_k\,\sigma(dz)=0$, $k=1,\dots, d$ (if $\alpha=1$). These formulae can be derived from the general L\'evy--Khintchine formula; a good reference is the monograph by Samorodnitsky \& Taqqu \cite[Chapters 2.3--4]{sam-taq94}.

    If $X$ is strictly stable such that the distribution of $X_t$ is rotationally invariant, it is clear that $\psi(\xi) = c|\xi|^\alpha$. If $X_t$ is symmetric, i.e.\ $X_t\sim -X_t$, then $\psi(\xi) = \int_{\mathds S^d} |z\cdot\xi|^\alpha\,\sigma(dz)$ for some finite, symmetric measure $\sigma$ on the unit sphere $\mathds S^d\subset\rd$.
\end{example}

Let us finally show Kolmogorov's proof of the L\'evy--Khintchine formula for one-dimensional L\'evy processes admitting second moments. We need the following auxiliary result.
\begin{lemma}\label{exlp-51}
\index{L\'evy process!moments}%
    Let $\Xt$ be a L\'evy process on $\real$. If $\Vv X_1<\infty$, then $\Vv X_t < \infty$ for all $t>0$ and
    $$
        \Ee X_t = t\Ee X_1 =: t\mu
        \et
        \Vv X_t = t\Vv X_1 =: t\sigma^2.
    $$
\end{lemma}
\begin{proof}
    If $\Vv X_1<\infty$, then $\Ee |X_1|<\infty$. With Bienaym\'e's identity, we get
    $$
        \Vv X_m = \sum_{k=1}^m \Vv (X_k-X_{k-1}) = m \Vv X_1
        \et
        \Vv X_1 = n \Vv X_{1/n}.
    $$
    In particular, $\Vv X_m, \Vv X_{1/n} < \infty$. This, and a similar argument for the expectation, show
    $$
        \Vv X_q = q\Vv X_1
        \et
        \Ee X_q = q\Ee X_1
        \fa q\in\mathds Q^+.
    $$
    Moreover, $\Vv (X_q - X_r) = \Vv X_{q-r} = (q-r)\Vv X_1$ for all rational numbers $r\leq q$, and this shows that $X_q - \Ee X_q = X_q - q\mu$ converges in $L^2$ as $q\to t$. Since $t\mapsto X_t$ is continuous in probability, we can identify the limit and find $X_q - q\mu \to X_t - t\mu$. Consequenctly, $\Vv X_t = t\sigma^2$ and $\Ee X_t = t\mu$.
\end{proof}
We have seen in Proposition~\ref{levy-11} that the characteristic function of a L\'evy process is of the form
$$
    \chi_t(\xi) = \Ee\,\eup^{\iup\xi X_t} = \big[\Ee\,\eup^{\iup\xi X_1}\big]^t = \chi_1(\xi)^t.
$$
Let us assume that $X$ is real-valued and has finite (first and) second moments $\Vv X_1=\sigma^2$ and $\Ee X_1=\mu$. By Taylor's formula
\begin{align*}
    \Ee\,\eup^{\iup\xi (X_t-t\mu)}
    &= \Ee \left[ 1 + \iup\xi (X_t-t\mu) - \int_0^1 \xi^2 (X_t-t\mu)^2\, (1-\theta)\eup^{\iup\theta\xi(X_t-t\mu)}\,d\theta \right]\\
    &= 1 - \Ee \left[ \xi^2 (X_t-t\mu)^2 \int_0^1 (1-\theta)\eup^{\iup\theta\xi(X_t-t\mu)}\,d\theta \right].
\end{align*}
Since
$$
    \left|\int_0^1 (1-\theta)\eup^{\iup\theta\xi(X_t-t\mu)}\,d\theta\right|
    \leq \int_0^1 (1-\theta)\,d\theta
    = \frac 12\,,
$$
we get
$$
    \big|\Ee\,\eup^{\iup\xi X_t}\big|
    = \big|\Ee\,\eup^{\iup\xi (X_t-t\mu)}\big|
    \geq 1- \frac{\xi^2}2\,t \sigma^2.
$$
Thus, $\chi_{1/n}(\xi)\neq 0$ if $n\geq N(\xi)\in\nat$ is large, hence $\chi_1(\xi)=\chi_{1/n}(\xi)^n\neq 0$. For $\xi\in\real$ we find (using a suitable branch of the complex logarithm)
\begin{align}
    \psi(\xi)
    := -\log\chi_1(\xi)
    \notag&=-\frac{\partial}{\partial t} \big[\chi_1(\xi)\big]^t\bigg|_{t=0} \\
    \notag&=\lim_{t\to 0} \frac{1-\Ee\,\eup^{\iup\xi X_t}}{t}\\
    \notag&= \lim_{t\to 0} \frac 1t\int_{-\infty}^{\infty} \big(1-\eup^{\iup y\xi}+\iup y\xi\big)\,p_t(dy) - \iup\xi\mu\\
    \label{exlp-e52}&= \lim_{t\to 0} \int_{-\infty}^{\infty} \frac{1-\eup^{\iup y\xi}+\iup y\xi}{y^2}\,\pi_t(dy) - \iup\xi\mu
\end{align}
where $p_t(dy)=\Pp(X_t\in dy)$ and $\pi_t(dy) := y^2 \,t^{-1}p_t(dy)$. Yet another application of Taylor's theorem shows that the integrand in the above integral is bounded, vanishes at infinity, and admits a continuous extension onto the whole real line if we choose the value $\frac 12\,\xi^2$ at $y=0$. The family $(\pi_t)_{t\in (0,1]}$ is uniformly bounded,
$$
    \frac 1t\int y^2\,p_t(dy)
    = \frac 1t\,\Ee (X_t^2) = \frac 1t\Big(\Vv X_t + \big[\Ee X_t\big]^2\Big)
    = \sigma^2 + t\mu^2\xrightarrow{t\to 0}\sigma^2,
$$
hence sequentially vaguely relatively compact (see Theorem~\ref{app-81}). We conclude that every sequence $(\pi_{t(n)})_{n\in\nat} \subset (\pi_t)_{t\in (0,1]}$ with $t(n)\to 0$ as $n\to\infty$ has a vaguely convergent subsequence. But since the limit \eqref{exlp-e52} exists, all subsequential limits coincide which means\footnote{Note that $\eup^{\iup y\xi} = \partial_\xi^2 \big(1-\eup^{\iup y\xi}+\iup y\xi\big)/y^2$, i.e.\ the kernel appearing in \eqref{exlp-e52} is indeed measure-determining.} that $\pi_t$ converges vaguely to a finite measure $\pi$ on $\real$. This proves that
\begin{align*}
    \psi(\xi)
    = -\log\chi_1(\xi)
    = \int_{-\infty}^{\infty} \frac{1-\eup^{\iup y\xi}+\iup y\xi}{y^2}\,\pi(dy) - \iup\xi\mu
\end{align*}
for some finite measure $\pi$ on $(-\infty,\infty)$ with total mass $\pi(\real) = \sigma^2$. This is sometimes called the \emphh{de Finetti--Kolmogorov formula}. If we set $\nu(dy):= y^{-2}\I_{\{y\neq 0\}}\,\pi(dy)$ and $\sigma_0^2:=\pi\{0\}$, we obtain the \emphh{L\'evy--Khintchine formula}
\index{L\'evy--Khintchine formula}%
\begin{align*}
    \psi(\xi)
    = -\iup\mu\xi + \frac 12\, \sigma_0^2\,\xi^2 + \int_{y\neq 0} \big(1-\eup^{\iup y\xi}+\iup y\xi\big)\,\nu(dy)
\end{align*}
where $\sigma^2 = \sigma_0^2 + \int_{y\neq 0} y^2\,\nu(dy)$.

\chapter{On the Markov property}\label{lpmp}

Let $(\Omega,\Ascr,\Pp)$ be a probability space with some filtration $(\Fscr_t)_{t\geq 0}$ and a $d$-dimensional \emphh{adapted} stochastic process \index{adapted}%
$X=(X_t)_{t\geq 0}$, i.e.\ each $X_t$ is $\Fscr_t$ measurable. We write $\Bscr(\rd)$ for the Borel sets and set $\Fscr_\infty := \sigma(\bigcup_{t\geq 0}\Fscr_t)$.

The process $X$ is said to be a \emphh{simple Markov process},
\index{Markov process}%
if
\begin{equation}\label{lpmp-e04}
    \Pp(X_t\in B\mid \Fscr_s) = \Pp(X_t\in B\mid X_s),\quad s\leq t,\; B\in\Bscr(\rd),
\end{equation}
holds true. This is pretty much the most general definition of a Markov process, but it is usually too general to work with. It is more convenient to consider Markov families.
\begin{definition}\label{lpmp-03} \index{transition function}%
    A (temporally homogeneous) \emphh{Markov transition function} is a measure kernel $p_t(x,B)$, $t\geq 0$, $x\in\rd$, $B\in \Bscr(\rd)$ such that
    \begin{enumerate}
    \item[\textup{a)}] $B\mapsto p_s(x,B)$ is a probability measure for every $s\geq 0$ and $x\in\rd$;
    \item[\textup{b)}] $(s,x)\mapsto p_s(x,B)$ is a Borel measurable function for every $B\in\Bscr(\rd)$;
    \item[\textup{c)}] the \emphh{Chapman--Kolmogorov equations}
    \index{Chapman--Kolmogorov equations}%
    hold
    \begin{equation}\label{lpmp-e06}
        p_{s+t}(x,B) = \int p_t(y,B)\,p_s(x,dy)
        \fa s,t\geq 0,\; x\in\rd,\; B\in\Bscr(\rd).
    \end{equation}
    \end{enumerate}
\end{definition}

\begin{definition}\label{lpmp-05}
    A stochastic process $\Xt$ is called a (temporally homogeneous) \emphh{Mar\-kov process with transition function} if there exists a Mar\-kov transition function $p_t(x,B)$ such that
    \begin{equation}\label{lpmp-e08}
        \Pp(X_t\in B \mid \Fscr_s) = p_{t-s}(X_s,B)\quad \text{a.s.\ for all\ \ } s\leq t,\; B\in\Bscr(\rd).
    \end{equation}
\end{definition}
Conditioning w.r.t.\ $\sigma(X_s)$ and using the tower property of conditional expectation shows that \eqref{lpmp-e08} implies the simple Markov property \eqref{lpmp-e04}. Nowadays the following definition of a Markov process is commonly used.

\begin{definition}\label{lpmp-07}
\index{Markov process!universal Markov process}\index{Markov process}%
    A \emphh{(universal) Markov process} is a tuple $(\Omega,\Ascr,\Fscr_t,X_t,t\geq 0,\Pp^x,x\in\rd)$ such that $p_t(x,B)=\Pp^x(X_t\in B)$ is a Markov transition function and $\Xt$ is for each $\Pp^x$ a Markov process in the sense of Definition~\ref{lpmp-05} such that $\Pp^x(X_0=x)=1$. In particular,
    \begin{equation}\label{lpmp-e10}
        \Pp^x(X_t\in B \mid \Fscr_s) = \Pp^{X_s}(X_{t-s}\in B)\quad \Pp^x\text{-a.s.\ for all\ \ } s\leq t,\; B\in\Bscr(\rd).
    \end{equation}
\end{definition}

We are going to show that a L\'evy process is a (universal) Markov process. Assume that $\Xt$ is a L\'evy process and set $\Fscr_t := \Fscr_t^X = \sigma(X_r,\: r\leq t)$. Define probability measures
$$
    \Pp^x(X_\bullet \in \Gamma)
    := \Pp(X_\bullet+x\in\Gamma),\quad x\in\rd,
$$
where $\Gamma$ is a Borel set of the path space $(\rd)^{[0,\infty)} = \{w \mid w:[0,\infty)\to\rd\}$.\footnote{Recall that $\Bscr((\rd)^{[0,\infty)})$ is the smallest $\sigma$-algebra containing the cylinder sets $Z = \mathop{\mbox{\large$\boldsymbol {\times}$}}_{t\geq 0} B_t$ where $B_t\in\Bscr(\rd)$ and only finitely many $B_t\neq\rd$.}
We set $\Ee^x := \int \dots d\Pp^x$. By construction, $\Pp=\Pp^0$ and $\Ee = \Ee^0$.

Note that $X^x_t := X_t + x$ satisfies the conditions \eqref{Lstat}--\eqref{Lcont}, and it is common to call $(X_t^x)_{t\geq 0}$ a \emphh{L\'evy process starting from $x$}.
\index{L\'evy process!starting from $x$}%
\begin{lemma}\label{lpmp-09}
    Let $\Xt$ be a L\'evy process on $\rd$. Then
    $$
        p_t(x,B) := \Pp^x(X_t\in B) := \Pp(X_t+x\in B),\quad t\geq 0,\; x\in\rd,\; B\in\Bscr(\rd),
    $$
    is a Markov transition function.
\end{lemma}
\begin{proof}
    Since $p_t(x,B)=\Ee \I_B(X_t+x)$ (the proof of) Fubini's theorem shows that $x\mapsto p_t(x,B)$ is a measurable function and $B\mapsto p_t(x,B)$ is a probability measure. The Chapman--Kolmogorov equations follow from
    \begin{align*}
        p_{s+t}(x,B)
        =\Pp(X_{s+t}+x\in B)
        &\pomu{\eqref{Lindep}}{=}{} \Pp( (X_{s+t}-X_t) + x + X_t\in B)\\
        &\omu{\eqref{Lindep}}{=}{} \int_{\rd} \Pp( y + X_t \in B)\,\Pp((X_{s+t}-X_t) + x\in dy)\\
        &\omu{\eqref{Lstat}}{=}{\phantom{\eqref{Lindep}}} \int_{\rd} \Pp( y + X_t \in B)\,\Pp(X_s + x\in dy)\\
        &\pomu{\eqref{Lindep}}{=}{} \int_{\rd} p_t(y,B)\,p_s(x,dy).
    \qedhere
    \end{align*}
\end{proof}

\begin{remark}\label{lpmp-11}
\index{convolution semigroup}%
\index{translation invariant}%
\index{L\'evy process!translation invariance}%
    The proof of Lemma~\ref{lpmp-09} shows a bit more: From
    $$
        p_t(x,B)
        =\int \I_B(x+y)\,\Pp(X_t\in dy)
        =\int \I_{B-x}(y)\,\Pp(X_t\in dy)
        = p_t(0,B-x)
    $$
    we see that the kernels $p_t(x,B)$ are invariant under shifts in $\rd$ (\emphh{translation invariant}). In slight abuse of notation we write $p_t(x,B)=p_t(B-x)$. From this it becomes clear that the Chapman--Kolmogorov equations are convolution identities $p_{t+s}(B)=p_t*p_s(B)$, and $(p_t)_{t\geq 0}$ is a \emphh{convolution semigroup} of probability measures; because of \eqref{Lcont}, this semigroup is weakly continuous at $t=0$, i.e.\ $p_t \to \delta_0$ as $t\to 0$, cf.\ Theorem~\ref{app-81} \emph{et seq.}\ for the weak convergence of measures.
\end{remark}

L\'evy processes enjoy an even stronger version of the above Markov property.
\begin{theorem}[Markov property for L\'evy processes]\label{lpmp-13}
\index{L\'evy process!Markov property}%
    Let $X$ be a $d$-dimensional L\'evy process and set $Y:=(X_{t+a}-X_a)_{t \geq 0}$ for some fixed $a\geq 0$. Then $Y$ is again a L\'evy process satisfying
    \begin{enumerate}
	\item[\upshape a)] $Y \bbot (X_r)_{r \leq a}$, i.e.\ $\Fscr_{\infty}^Y \bbot \Fscr_a^X$.
    \item[\upshape b)] $Y \sim X$, i.e.\ $X$ and $Y$ have the same finite dimensional distributions.
	\end{enumerate}
\end{theorem}
\begin{proof}
    Observe that $\Fscr_s^Y = \sigma(X_{r+a}-X_a,\, r\leq s)\subset\Fscr_{s+a}^X$.
    Using Theorem~\ref{exlp-03} and the tower property of conditional expectation yields for all $s\leq t$
    \begin{align*}
        \Ee \left( \eup^{\iup\xi\cdot(Y_t-Y_s)} \:\middle|\: \Fscr_s^Y\right)
        &= \Ee \left[\Ee \left( \eup^{\iup\xi\cdot(X_{t+a}-X_{s+a})} \:\middle|\: \Fscr_{s+a}^X\right)\:\middle|\: \Fscr_s^Y\right]
        = \eup^{-(t-s)\psi(\xi)}.
    \end{align*}
    Thus, $(Y_t)_{t\geq 0}$ is a L\'evy process with the same characteristic function as $\Xt$. The property \eqref{LindepAlt} for $X$ gives
    $$
        X_{t_n+a}-X_{t_{n-1}+a}, \; X_{t_{n-1}+a}-X_{t_{n-2}+a}, \;\dots,\; X_{t_1+a}-X_{a} \bbot \Fscr_a^X.
    $$
    As $\sigma(Y_{t_1},\dots,Y_{t_n}) = \sigma(Y_{t_n}-Y_{t_{n-1}}, \dots, Y_{t_1}-Y_{t_0})\bbot\Fscr_a^X$ for all $t_0 = 0 <t_1<\dots<t_n$, we get
    \begin{gather*}
		\Fscr_\infty^Y
        = \sigma \left( \bigcup_{t_1 < \dots < t_n,\:n\in\nat} \sigma(Y_{t_1},\dots,Y_{t_n}) \right) \bbot \Fscr_a^X.
    \qedhere
	\end{gather*}
\end{proof}
Using the Markov transition function $p_t(x,B)$ we can define a linear operator on the bounded Borel measurable functions $f:\rd\to\real$:
\begin{equation}\label{lpmp-e22}
    P_tf(x) := \int f(y)\,p_t(x,dy) = \Ee^x f(X_t),
    \quad f\in\Bcal_b(\rd),\; t\geq 0,\; x\in\rd.
\end{equation}
For
\index{convolution semigroup}%
a L\'evy process, cf.\ Remark~\ref{lpmp-11}, we have $p_t(x,B)=p_t(B-x)$ and the operators $P_t$ are actually \emphh{convolution operators}:
\begin{equation}\label{lpmp-e24}
    P_tf(x) = \Ee f(X_t+x) = \int f(y+x)\,p_t(dy) = f*\widetilde p_t(x)
    \quad\text{where}\quad
    \widetilde p_t(B) := p_t(-B).
\end{equation}
\begin{definition}
\label{lpmp-21}
    Let $P_t$, $t\geq 0$, be defined by \eqref{lpmp-e22}. The operators are said to be
    \begin{enumerate}
    \item[\upshape a)]
        \emphh{acting on} $\Bcal_b(\rd)$, if $P_t:\Bcal_b(\rd) \to \Bcal_b(\rd)$.
    \item[\upshape b)]
        an operator \emphh{semigroup}, if $P_{t+s}=P_t\circ P_s$ for all $s,t\geq 0$ and $P_0=\id$.
        \index{semigroup}%
    \item[\upshape c)]
        \emphh{sub-Markovian} if $0\leq f\leq 1 \implies 0\leq P_tf\leq 1$.
        \index{semigroup!sub-Markovian}%
    \item[\upshape d)]
        \emphh{contractive} if $\|P_tf\|_\infty \leq \|f\|_\infty$ for all $f\in\Bcal_b(\rd)$.
        \index{semigroup!contractive}%
    \item[\upshape e)]
        \emphh{conservative} if $P_t1=1$.
        \index{semigroup!conservative}%
    \item[\upshape f)]
        \emphh{Feller operators}, if $P_t:\cont_\infty(\rd)\to \cont_\infty(\rd)$.\footnote{$\cont_\infty(\rd)$ denotes the space of continuous functions vanishing at infinity. It is a Banach space when equipped with the uniform norm $\|f\|_\infty = \sup_{x\in\rd}|f(x)|$.}
        \index{semigroup!Feller property}%
        \index{Feller property}%
    \item[\upshape g)]
        \emphh{strongly continuous} on $\cont_\infty(\rd)$, if $\lim_{t\to 0}\|P_tf - f\|_\infty=0$ for all $f\in\cont_\infty(\rd)$.
        \index{semigroup!strongly continuous}\index{strong continuity}%
    \item[\upshape h)]
    \index{strong Feller property}%
        \emphh{strong Feller operators}, if $P_t:\Bcal_b(\rd)\to \cont_b(\rd)$.
        \index{semigroup!strong Feller property}%
\end{enumerate}
\end{definition}

\begin{lemma}\label{lpmp-23}
    Let $(P_t)_{t\geq 0}$ be defined by \eqref{lpmp-e22}. The properties \textup{\ref{lpmp-21}.a)--e)} hold for any Markov process, \textup{\ref{lpmp-21}.a)--g)} hold for any L\'evy process, and \textup{\ref{lpmp-21}.a)--h)} hold for any L\'evy process such that all transition probabilities $p_t(dy) = \Pp(X_t\in dy)$, $t>0$, are absolutely continuous w.r.t.\ Lebesgue measure.
\end{lemma}
\begin{proof}
    We only show the assertions about L\'evy processes $\Xt$.
    \begin{enumerate}
    \item[\upshape a)] Since $P_t f(x) = \Ee f(X_t+x)$, the boundedness of $P_tf$ is obvious, and the measurability in $x$ follows from (the proof of) Fubini's theorem.
    \item[\upshape b)] By the tower property of conditional expectation, we get for $s,t\geq 0$
    \begin{align*}
		P_{t+s}f(x)
		= \Ee^x f(X_{t+s})
		&\pomu{\eqref{lpmp-e10}}{=}{} \Ee^x\big(\Ee^x[f(X_{t+s}) \mid \Fscr_s]\big) \\
		&\omu{\eqref{lpmp-e10}}{=}{} \Ee^x\big(\Ee^{X_s}f(X_t)\big)
        = P_s \circ P_t f(x).
	\end{align*}
	For the Markov transition functions this is the Chapman--Kolmogorov identity \eqref{lpmp-e06}.
	\item[\upshape c)] and d), e) follow directly from the fact that $B\mapsto p_t(x,B)$ is a probability measure.
    \item[\upshape f)] Let $f\in\cont_\infty(\rd)$. Since $x\mapsto f(x+X_t)$ is continuous and bounded, the claim follows from dominated convergence as $P_tf(x)=\Ee f(x+X_t)$.
    \item[\upshape g)] $f \in \cont_{\infty}$ is uniformly continuous, i.e.\ for every $\epsilon>0$ there is some $\delta>0$ such that $$|x-y|\leq \delta\implies |f(x)-f(y)|\leq \epsilon.$$ Hence,
    \begin{align*}
		\|P_tf-f\|_{\infty}
		&\leq \sup_{x \in \rd} \int |f(X_t)-f(x)| \, d\Pp^x \\
		&= \sup_{x \in \rd} \left(\int_{|X_t-x| \leq \delta} |f(X_t)-f(x)| \, d\Pp^x
            + \int_{|X_t-x|>\delta} |f(X_t)-f(x)| \, d\Pp^x\right)\\
		&\leq \epsilon + 2\|f\|_{\infty} \sup_{x \in \rd} \Pp ^x(|X_t-x|>\delta) \\
		&= \epsilon + 2\|f\|_{\infty} \,\Pp(|X_t|>\delta)
        \xrightarrow[\;t\to 0\;]{\eqref{Lcont}}\epsilon.
	\end{align*}
    Since $\epsilon>0$ is arbitrary, the claim follows. Note that this proof shows that \emphh{uniform continuity in probability} is responsible for the strong continuity of the semigroup.
    \item[\upshape h)] see Lemma~\ref{lpmp-25}.
    \qedhere
    \end{enumerate}
\end{proof}

\begin{lemma}[Hawkes]\label{lpmp-25}
\index{strong Feller property!of a L\'evy process}
    Let $X=(X_t)_{t\geq 0}$ be a L\'evy process on $\rd$. Then the operators $P_t$ defined by \eqref{lpmp-e22} are strong Feller if, and only if, $X_t\sim p_t(y)\,dy$ for all $t>0$.
\end{lemma}
\begin{proof}
    `$\Leftarrow$': Let $X_t \sim p_t(y) \, dy$. Since $p_t \in L^1$ and since convolutions have a smoothing property (e.g.\
    \cite[Theorem 14.8]{schilling-mims} or \cite[Satz 18.9]{schilling-mi}), we get with $\widetilde p_t(y)=p_t(-y)$
    $$
	   P_t f
        = f*\widetilde{p}_t \in L^{\infty}*L^1 \subset \cont_b(\rd).
	$$

    `$\Rightarrow$': We show that $p_t(dy) \ll dy$. Let $N \in \Bscr(\rd)$ be a Lebesgue null set $\lambda^d(N)=0$ and $g \in \Bcal_b(\rd)$. Then, by the Fubini--Tonelli theorem
    \begin{align*}
		\int g(x) P_t\I_N(x)\,dx
		&= \iint g(x) \I_N(x+y) \, p_t(dy)\,dx \\
		&= \int \underbrace{\int g(x) \I_N(x+y)\,dx}_{=0} \: p_t(dy) =0.
	\end{align*}
	Take $g = P_t \I_N$, then the above calculation shows
    $$
		\int (P_t \I_N(x))^2 \, dx =0.
	$$
	Hence, $P_t\I_N = 0$ Lebesgue-a.e. By the strong Feller property, $P_t\I_N$ is continuous, and so $P_t\I_N\equiv 0$, hence
    \begin{gather*}
		p_t(N) = P_t\I_N(0)=0.
    \qedhere
	\end{gather*}
\end{proof}

\begin{remark}\label{lpmp-27}
    The existence and smoothness of densities for a L\'evy process are time-dependent properties, cf.\ Sato~\cite[Chapter V.23]{sato99}. The typical example is the \emphh{Gamma process}.
\index{Gamma process}%
    This is a (one-dimensional) L\'evy process with characteristic exponent
    $$
        \psi(\xi) = \frac 12\log(1+|\xi|^2) -\iup \arctan\xi,\quad\xi\in\real,
    $$
    and this process has the transition density
    $$
    	p_t(x) = \frac{1}{\Gamma(t)}\, x^{t-1} \eup^{-x} \,\I_{(0,\infty)}(x), \quad t>0.
    $$
    The factor $x^{t-1}$ gives a time-dependent condition for the property $p_t\in L^p(dx)$. One can show, cf.\ \cite{kno-sch13}, that
    $$
        \lim_{|\xi|\to\infty} \frac{\Re\psi(\xi)}{\log(1+|\xi|^2)} = \infty \implies \forall t>0\quad\exists p_t\in \cont^\infty(\rd).
    $$
    The converse direction remains true if $\psi(\xi)$ is rotationally invariant or if it is replaced by its symmetric rearrangement.
\end{remark}

\begin{remark}\label{lpmp-29}
\index{Feller process}%
\index{Feller semigroup}%
    If $(P_t)_{t\geq 0}$ is a \emphh{Feller semigroup}, i.e.\ a semigroup satisfying the conditions~\ref{lpmp-21}.a)-g), then there exists a unique stochastic process (a \emphh{Feller process}) with $(P_t)_{t\geq 0}$ as transition semigroup. The idea is to use Kolmogorov's consistency theorem for the following family of finite-dimensional distributions
    $$	
		p^x_{t_1,\dots,t_n}(B_1 \times \dots \times B_n)
        = P_{t_1} \Big( \I_{B_1} P_{t_2-t_1} \big( \I_{B_2} P_{t_3-t_2}( \dots P_{t_n-t_{n-1}}(\I_{B_n})) \big) \Big)(x)
	$$
	Here $X_{t_0} = X_0 = x$ a.s. Note: It is \emphh{not enough} to have a semigroup on $L^p$ as we need pointwise evaluations.

    If the operators $P_t$ are not \emph{a priori} given on $\Bcal_b(\rd)$ but only on $\cont_\infty(\rd)$, one still can use the Riesz representation theorem to construct Markov kernels $p_t(x,B)$ representing and extending $P_t$ onto $\Bcal_b(\rd)$, cf.\ Lemma~\ref{semi-05}.
\end{remark}

Recall that a \emphh{stopping time} is a random time $\tau : \Omega\to [0,\infty]$ such that $\{\tau \leq t\}\in\Fscr_t$ for all $t\geq 0$. It is not hard to see that $\tau_n := (\entier{2^n \tau} + 1)2^{-n}$, $n\in\nat$, is a sequence of stopping times with values $k 2^{-n}$, $k=1,2,\dots$, such that
$$
    \tau_1\geq \tau_2 \geq \dots \geq \tau_n \downarrow \tau = \inf_{n\in\nat}\tau_n.
$$
This approximation is the key ingredient to extend the Markov property (Theorem~\ref{lpmp-13}) to random times.
\begin{theorem}[Strong Markov property for L\'evy processes]\label{lpmp-31}
\index{strong Markov property!of a L\'evy process}%
\index{L\'evy process!strong Markov property}%
\index{Markov process!strong Markov property}%
    Let $X$ be a L\'evy process on $\rd$ and set $Y:=(X_{t+\tau}-X_\tau)_{t \geq 0}$ for some a.s.\ finite stopping time $\tau$. Then $Y$ is again a L\'evy process satisfying
    \begin{enumerate}
	\item[\upshape a)]
        $Y \bbot (X_r)_{r \leq \tau}$, i.e.\ $\Fscr_{\infty}^Y \bbot \Fscr_{\tau+}^X := \big\{F\in\Fscr_\infty^X\,:\, F\cap\{\tau< t\}\in\Fscr_t^X\;\:\forall t\geq 0\big\}$.
    \item[\upshape b)]
        $Y \sim X$, i.e.\ $X$ and $Y$ have the same finite dimensional distributions.
	\end{enumerate}
\end{theorem}
\begin{proof}
    Let $\tau_n := (\entier{2^n \tau}+1) 2^{-n}$. For all $0\leq s < t$, $\xi\in\rd$ and $F\in\Fscr^X_{\tau+}$ we find by the right-continuity of the sample paths (or by the continuity in probability \eqref{Lcont})
    \begin{align*}
        \Ee\left[ \eup^{\iup \xi\cdot (X_{t+\tau}-X_{s+\tau})}\, \I_F\right]
        &= \lim_{n\to\infty}\Ee\left[ \eup^{\iup \xi\cdot (X_{t+\tau_n}-X_{s+\tau_n})}\, \I_F\right]\\
        &= \lim_{n\to\infty}\sum_{k=1}^\infty \Ee\left[ \eup^{\iup \xi\cdot (X_{t+k2^{-n}}-X_{s+k2^{-n}})}\,
        \I_{\{\tau_n = k2^{-n}\}} \cdot \I_F\right]\\
        &= \lim_{n\to\infty}\sum_{k=1}^\infty
            \Ee\Big[\;
            \underbrace{ \eup^{\iup \xi\cdot (X_{t+k2^{-n}}-X_{s+k2^{-n}})}\vphantom{\I_{\left\{(k-1)2^{-n}\leq \tau < k2^{-n}\right\}}\, \I_F}}_{\mathclap{\bbot\,\Fscr^X_{k2^{-n}}\text{\ by \eqref{Lindep}}}} \,
            \underbrace{\I_{\left\{(k-1)2^{-n}\leq \tau < k2^{-n}\right\}}\, \I_F}_{\mathclap{\in\,\Fscr^X_{k2^{-n}}\text{\ as\ } F\in\Fscr^X_{\tau+}}}
            \;\Big]\\
        &= \lim_{n\to\infty}\sum_{k=1}^\infty \Ee\big[ \eup^{\iup \xi\cdot X_{t-s}}\big]\,
           \Pp\left(\big\{(k-1)2^{-n}\leq \tau < k2^{-n}\big\}\cap F\right)\\
        &= \Ee\big[ \eup^{\iup\xi\cdot X_{t-s}}\big] \,\Pp(F).
    \end{align*}
    In the last equality we use $\bigdcup_{k=1}^\infty\{(k-1)2^{-n}\leq \tau < k2^{-n}\} = \{\tau<\infty\}$ for all $n\geq 1$.

    The same calculation applies to finitely many increments. Let $F\in\Fscr^X_{\tau+}$, $t_0=0 < t_1 < \dots < t_n$ and $\xi_1, \dots, \xi_n\in\rd$. Then
    $$
        \Ee\left[\eup^{\iup \sum_{k=1}^n \xi_k\cdot (X_{t_k+\tau}-X_{t_{k-1}+\tau})}\, \I_F\right]
        = \prod_{k=1}^n \Ee\left[\eup^{\iup \xi_k\cdot X_{t_k-t_{k-1}}}\right] \Pp(F).
    $$
    This shows that the increments $X_{t_k+\tau }-X_{t_{k-1}+\tau}$ are independent and distributed like $X_{t_k-t_{k-1}}$. Moreover, all increments are independent of $F\in\Fscr^X_{\tau+}$.

    Therefore, all random vectors of the form $(X_{t_1+\tau}-X_{\tau}, \ldots, X_{ t_n+\tau}-X_{t_{n-1}+\tau})$ are independent of $\Fscr_{\tau+}^X$, and we conclude that $\Fscr_\infty^Y = \sigma(X_{t+\tau}-X_\tau,\, t\geq 0)\bbot\Fscr^X_{\tau+}$.
\end{proof}

\chapter{A digression: semigroups}\label{semi}

We have seen that the Markov kernel $p_t(x,B)$ of a L\'evy or Markov process induces a semigroup of linear operators $(P_t)_{t\geq 0}$. In this chapter we collect a few tools from functional analysis for the study of operator semigroups. By $\Bcal_b(\rd)$ we denote the bounded Borel  functions $f:\rd\to\real$, and $\cont_\infty(\rd)$ are the continuous functions vanishing at infinity, i.e.\ $\lim_{|x|\to\infty}f(x)=0$; when equipped with the uniform norm $\|\cdot\|_\infty$ both sets become Banach spaces.

\begin{definition}\label{semi-03}
\index{Feller semigroup}%
    A \emphh{Feller semigroup} is a family of linear operators $P_t:\Bcal_b(\rd)\to\Bcal_b(\rd)$ satisfying the properties a)--g) of Definition~\ref{lpmp-21}: $(P_t)_{t\geq 0}$ is a semigroup of conservative, sub-Markovian operators which enjoy the \emphh{Feller property}
\index{Feller property}\index{semigroup!Feller property}%
    $P_t(\cont_\infty(\rd))\subset\cont_\infty(\rd)$ and which are strongly continuous on $\cont_\infty(\rd)$.
\end{definition}
Notice that $(t,x)\mapsto P_tf(x)$ is for every $f\in\cont_\infty(\rd)$ continuous. This follows from
\begin{align*}
    |P_tf(x)-P_sf(y)|
    &\leq |P_tf(x)-P_tf(y)| + |P_tf(y)-P_sf(y)|\\
    &\leq |P_tf(x)-P_tf(y)| + \|P_{|t-s|} f- f\|_\infty,
\end{align*}
the Feller property~\ref{lpmp-21}.f) and the strong continuity~\ref{lpmp-21}.g).

\begin{lemma}\label{semi-05}
\index{Feller semigroup!extension to $\Bcal_b(\rd)$}%
    If $(P_t)_{t\geq 0}$ is a Feller semigroup, then there is a Markov transition function $p_t(x,dy)$ \textup{(}Definition~\ref{lpmp-03}\textup{)} such that $P_tf(x) = \int f(y)\,p_t(x,dy)$.
\end{lemma}
\begin{proof}
By the Riesz representation theorem we see that the operators $P_t$ are of the form $$P_t f(x) = \int f(y)\,p_t(x,dy)$$ where $p_t(x,dy)$ is a Markov kernel. The tricky part is to show the joint measurability of the transition function $(t,x)\mapsto p_t(x,B)$ and the Chapman--Kolmogorov identities \eqref{lpmp-e06}.

For every compact set $K\subset\rd$ the functions defined by
$$
    f_n(x) := \frac{d(x,U_n^c)}{d(x,K)+d(x,U_n^c)},\quad
    d(x,A) := \inf_{a\in A}|x-a|,\quad
    U_n := \{y\,:\, d(y,K) < 1/n\},
$$
are in $\cont_\infty(\rd)$ and $f_n\downarrow\I_K$. By monotone convergence, $p_t(x,K) = \inf_{n\in\nat} P_tf_n(x)$ which proves the joint measurability in $(t,x)$ for all compact sets.

By the same, the semigroup property $P_{t+s}f_n = P_sP_tf_n$ entails the Chapman--Kol\-mo\-go\-rov identities for compact sets: $p_{t+s}(x,K) = \int p_t(y,K)\,p_s(x,dy)$.
Since
$$
    \mathscr{D} := \left\{B\in\Bscr(\rd)\:\middle|\:
    \begin{aligned} &(t,x)\mapsto p_t(x,B)\text{\ \ is measurable \&\ }\\
                    &p_{t+s}(x,B) = \int p_t(y,B)\,p_s(x,dy)
    \end{aligned}\right\}
$$
is a Dynkin system containing the compact sets, we have $\mathscr{D}=\Bscr(\rd)$.
\end{proof}

To get an intuition for semigroups it is a good idea to view the semigroup property
$$
    P_{t+s} = P_s\circ P_t \et P_0 = \id
$$
as an operator-valued Cauchy functional equation. If $t\mapsto P_t$ is---in a suitable sense---continuous, the unique solution will be of the form $P_t = \eup^{tA}$ for some \emphh{operator} $A$. This can be easily made rigorous for matrices $A,P_t\in \real^{n\times n}$
since the matrix exponential is well defined by the uniformly convergent series
$$
    P_t = \exp(tA) := \sum_{k=0}^\infty \frac{t^k A^k}{k!}
    \et
    A = \frac{d}{dt} P_t\Big|_{t=0}
$$
with $A^0 := \id$ and $A^k = A\circ A\circ \dots\circ A$ ($k$ times). With a bit more care, this can be made to work also in general settings.
\begin{definition}\label{semi-07}
\index{generator}%
\index{Feller process!generator}%
    Let $(P_t)_{t\geq 0}$ be a Feller semigroup. The (infinitesimal) \emphh{generator} is a linear operator defined by
    \begin{align}
    \label{semi-e02}
        \dom(A) &:= \left\{f \in \cont_\infty(\rd)\;\:\middle|\;\: \exists g \in \cont_\infty(\rd) \::\: \lim_{t\to 0}\left\| \frac{P_t f-f}{t} -g \right\|_{\infty} = 0 \right\}\\
    \label{semi-e04}
          Af &:= \lim_{t \to 0} \frac{P_t f-f}{t},\quad f\in\dom(A).
	\end{align}
\end{definition}

The following lemma is the rigorous version for the symbolic notation `$P_t = \eup^{tA}$'.
\begin{lemma}\label{semi-09}
    Let $(P_t)_{t\geq 0}$ be a Feller semigroup with infinitesimal generator $(A,\dom(A))$. Then $P_t(\dom(A))\subset\dom(A)$ and
    \begin{equation}\label{semi-e06}
        \frac{d}{dt } P_t f = AP_t f = P_tA f
        \fa f\in \dom(A),\; t\geq 0.
    \end{equation}
    Moreover, $\int_0^t P_s f\,ds\in \dom(A)$ for any $f\in\cont_\infty(\rd)$, and
    \begin{align}
        P_tf - f
        &= A\int_0^t P_s f\,ds, \quad f\in \cont_\infty(\rd),\; t>0\label{semi-e08}\\
        &= \int_0^t P_s Af\,ds, \quad f\in \dom(A),\; t>0.\label{semi-e10}
    \end{align}
\end{lemma}
\begin{proof}
    Let $0<\epsilon<t$ and $f\in \dom(A)$. The semigroup and contraction properties give
    \begin{align*}
        \left\|\frac{P_tf-P_{t-\epsilon} f}{\epsilon} - P_t Af\right\|_\infty
        &\leq \left\|P_{t-\epsilon}\frac{P_\epsilon f-f}{\epsilon} - P_{t-\epsilon} Af\right\|_\infty
        + \big\|P_{t-\epsilon} Af - P_{t-\epsilon}P_\epsilon Af\big\|_\infty\\
        &\leq \left\|\frac{P_\epsilon f-f}{\epsilon} - Af\right\|_\infty
        + \big\|Af - P_\epsilon Af\big\|_\infty
        \xrightarrow[\epsilon\to 0]{} 0
    \end{align*}
    where we use the strong continuity in the last step.
    This shows $\frac{d^-}{dt }P_tf =AP_tf=P_tAf$; a similar (but simpler) calculation proves this also for $\frac{d^+}{dt }P_tf$.

\medskip
    Let $f\in \cont_\infty(\rd)$ and $t,\epsilon>0$. By Fubini's theorem and the representation of $P_t$ with a Markov transition function (Lemma~\ref{semi-05}) we get
    \begin{align*}
        P_\epsilon \int_0^t P_s f(x)\,ds
        = \int_0^t P_\epsilon P_s f(x)\,ds,
    \end{align*}
    and so,
    \begin{align*}
        \frac{P_\epsilon-\id}{\epsilon}\int_0^t P_sf(x)\,ds
        &= \frac 1\epsilon\int_0^t \big(P_{s+\epsilon}f(x)-P_sf(x)\big)\,ds\\
        &= \frac 1\epsilon\int_t^{t+\epsilon} P_sf(x)\,ds - \frac 1\epsilon \int_0^\epsilon P_sf(x)\,ds.
    \end{align*}
    Since $t\mapsto P_tf(x)$ is continuous, the fundamental theorem of calculus applies, and we get $$\lim_{\epsilon\to 0}\frac 1\epsilon\int_r^{r+\epsilon} P_sf(x)\,ds = P_rf(x)$$ for $r\geq 0$.
    This shows that $\int_0^t P_sf\,ds\in\dom(A)$ as well as \eqref{semi-e08}. If $f\in\dom(A)$, then we deduce \eqref{semi-e10} from
    \begin{gather*}
        \int_0^t P_sAf(x)\,ds
        \omu{\eqref{semi-e06}}{=}{} \int_0^t \frac{d}{ds}\,P_sf(x)\,ds
        \pomu{\eqref{semi-e06}}{=}{} P_tf(x)-f(x)
        \omu{\eqref{semi-e08}}{=}{} A\int_0^t P_sf(x)\,ds.
    \qedhere
    \end{gather*}
\end{proof}

\begin{remark}[Consequences of Lemma~\ref{semi-09}]\label{semi-11} Write $\cont_\infty := \cont_\infty(\rd)$.
    \begin{enumerate}
        \item[\upshape a)] \eqref{semi-e08} shows that $\dom(A)$ is dense in $\cont_\infty$, since $\dom(A)\ni t^{-1}\int_0^t P_s f\,ds \xrightarrow[t\to 0]{}f$ for any $f\in\cont_\infty$.
        \item[\upshape b)] \eqref{semi-e10} shows that $A$ is a \emphh{closed operator}, i.e.
\index{closed operator}%
        $$
            f_n\in\dom(A),\; (f_n,Af_n) \xrightarrow[n\to\infty]{\;\text{uniformly}\;} (f,g)\in\cont_\infty\times\cont_\infty
            \implies
            f \in \dom(A)\;\;\&\;\; Af=g.
        $$
    	\item[\upshape c)] \eqref{semi-e06} means that $A$ determines $(P_t)_{t \geq 0}$ uniquely.
    \end{enumerate}
\end{remark}

Let us now consider the Laplace transform of $(P_t)_{t\geq 0}$.
\begin{definition}\label{semi-21}
\index{resolvent}%
	Let $(P_t)_{t\geq 0}$ be a Feller semigroup. The \emphh{resolvent} is a linear operator on $\Bcal_b(\rd)$ given by
    \begin{equation}\label{semi-e22}
		R_{\lambda} f(x) := \int_0^{\infty} \eup^{-\lambda t} P_t f(x) \, dt, \quad f \in \Bcal_b(\rd),\;x\in\rd,\; \lambda>0.
	\end{equation}
\end{definition}

The following formal calculation can easily be made rigorous. Let $(\lambda - A) := (\lambda\id - A)$ for $\lambda>0$ and $f\in\dom(A)$. Then
\begin{align*}
	(\lambda-A) R_{\lambda} f
	&\pomu{\eqref{semi-e08},\eqref{semi-e10}}{=}{\phantom{\text{parts}}}
        (\lambda-A) \int_0^{\infty} \eup^{-\lambda t} P_t f \, dt \\
	&\omu{\eqref{semi-e08},\eqref{semi-e10}}{=}{\phantom{\text{parts}}}
        \int_0^{\infty} \eup^{-\lambda t} (\lambda-A) P_t f \, dt \\
	&\pomu{\eqref{semi-e08},\eqref{semi-e10}}{=}{\phantom{\text{parts}}}
        \lambda \int_0^{\infty}  \eup^{-\lambda t} P_t f\, dt - \int_0^{\infty} \eup^{-\lambda t} \left( \frac{d}{dt} P_t f \right)  dt \\
	&\omu{\text{parts}}{=}{\phantom{\eqref{semi-e08},\eqref{semi-e10}}}
        \lambda \int_0^{\infty}  \eup^{-\lambda t} P_t f \, dt - \lambda \int_0^{\infty}  \eup^{-\lambda t} P_t f \, dt - [\eup^{-\lambda t} P_t f]_{t=0}^{\infty}
	= f.
\end{align*}
A similar calculation for $R_\lambda (\lambda-A)$ gives
\begin{theorem}\label{semi-23}
    Let $(A,\dom(A))$ and $(R_\lambda)_{\lambda>0}$ be the generator and the resolvent of a Feller semigroup. Then
    $$
        R_{\lambda} = (\lambda-A)^{-1}\fa \lambda >0.
    $$
\end{theorem}

Since $R_\lambda$ is the Laplace transform of $(P_t)_{t\geq 0}$, the properties of $(R_\lambda)_{\lambda>0}$ can be found from $(P_t)_{t\geq 0}$ and vice versa. With some effort one can even invert the (operator-valued) Laplace transform which leads to the familiar expression for $\eup^x$:
\begin{equation}\label{semi-e24}
	\left(\frac{n}{t} R_{\frac{n}{t}} \right)^n
    = \left( \id - \frac{t}{n} A\right)^{-n}
    \xrightarrow[n\to\infty]{\;\text{strongly}\;} \eup^{tA} = P_t
\end{equation}
(the notation $\eup^{tA}= P_t$ is, for unbounded operators $A$, formal), see Pazy \cite[Chapter 1.8]{pazy83}.

\begin{lemma}\label{semi-24}
    Let $(R_\lambda)_{\lambda>0}$ be the resolvent of a Feller\footnote{This Lemma only needs that the operators $P_t$ are strongly continuous and contractive, Definition~\ref{lpmp-21}.g), d).} semigroup $(P_t)_{t\geq 0}$. Then
    \begin{align}\label{semi-e25}
        \frac{d^n}{d\lambda^n} R_\lambda = n!(-1)^n R_\lambda^{n+1}
        \quad n\in\nat_0.
    \end{align}
\end{lemma}
\begin{proof}
    Using a symmetry argument we see
    $$
        t^n = \int_0^t\dots\int_0^t dt_1\dots dt_n = n!\int_0^t\int_0^{t_n}\dots\int_0^{t_2} dt_1\dots dt_n.
    $$
    Let $f\in\cont_\infty(\rd)$ and $x\in\rd$. Then
    \begin{align*}
        (-1)^n\frac{d^n}{d\lambda^n} R_\lambda f(x)
        &= \int_0^\infty (-1)^n \frac{d^n}{d\lambda^n} \eup^{-\lambda t}P_tf(x)\,dt
        = \int_0^\infty t^n \eup^{-\lambda t}P_tf(x)\,dt\\
        &= n!\int_0^\infty \int_0^t \int_0^{t_n}\dots\int_0^{t_2} \eup^{-\lambda t}P_tf(x)\,dt_1 \dots dt_n\,dt\\
        &= n!\int_0^\infty \int_{t_n}^\infty\dots\int_{t_1}^\infty \eup^{-\lambda t}P_tf(x)\,dt\,dt_1 \dots dt_n\\
        &= n!\int_0^\infty \int_0^\infty\dots\int_0^\infty \eup^{-\lambda (t+t_1+\dots+t_n)}P_{t+t_1+\dots+t_n}f(x)\,dt\,dt_1 \dots dt_n\\
        &= n!\,R_\lambda^{n+1}f(x).
    \qedhere
    \end{align*}
\end{proof}

The key result identifying the generators of Feller semigroups is the following theorem due to Hille, Yosida and Ray, a proof can be found in Pazy \cite[Chapter 1.4]{pazy83} or Ethier \& Kurtz \cite[Chapter 4.2]{eth-kur86}; a probabilistic approach is due to It\^o \cite{ito93}.

\begin{theorem}[Hille--Yosida--Ray]\label{semi-25}
\index{generator!Hille--Yosida--Ray theorem}%
\index{Hille--Yosida--Ray theorem}%
    A linear operator $(A,\dom(A))$ on $\cont_{\infty}(\rd)$ generates a Feller semigroup $(P_t)_{t \geq 0}$ if, and only if,
    \begin{enumerate}
	\item[\upshape a)] $\dom(A) \subset \cont_{\infty}(\rd)$ dense.
    \item[\upshape b)] $A$ is \emphh{dissipative}, i.e.\ $\|\lambda f -Af\|_{\infty} \geq \lambda \|f\|_{\infty}$ for some \textup{(}or all\textup{)} $\lambda>0$.
        \index{dissipative operator}%
    \item[\upshape c)] $(\lambda-A)(\dom(A)) = \cont_{\infty}(\rd)$ for some \textup{(}or all\textup{)} $\lambda>0$.
    \item[\upshape d)] $A$ satisfies the \emphh{positive maximum principle}:
    \index{generator!positive maximum principle}%
    \index{positive maximum principle (PMP)}%
    \index{PMP (positive maximum principle)}%
    \begin{equation}\label{PMP} \tag{\text{PMP}}
	       f\in \dom(A), \; f(x_0) = \sup_{x \in \rd} f(x) \geq 0 \implies Af(x_0) \leq 0.
	\end{equation}
	\end{enumerate}
\end{theorem}

This variant of the Hille--Yosida theorem is not the standard version from functional analysis since we are interested in positivity preserving (sub-Markov) semigroups. Let us briefly discuss the role of the positive maximum principle.
\begin{remark}\label{semi-27}
\index{semigroup!strongly continuous}\index{semigroup!contractive}\index{strong continuity}%
    Let $(P_t)_{t\geq 0}$ be a strongly continuous contraction semigroup on $\cont_\infty(\rd)$, i.e.\ $$\|P_t f\|_\infty\leq \|f\|_\infty \et \lim_{t\to 0}\|P_tf-f\|_\infty=0,$$ cf.\ Definition~\ref{lpmp-21}.d),g).\footnote{These properties are essential for the existence of a generator and the resolvent on $\cont_\infty(\rd)$.}
    \begin{description}
    \item[\normalfont $1^\circ$ Sub-Markov $\Rightarrow$ \eqref{PMP}.]
    Assume that $f\in\dom(A)$ is such that $f(x_0) = \sup f \geq 0$. Then
    \begin{gather*}
		P_t f(x_0)-f(x_0)
		\omu{f \leq f^+}{\leq}{} P_t^+f(x_0)-f^+(x_0)
		\leq \|f^+\|_{\infty}-f^+(x_0) = 0. \\
		\implies Af(x_0)
		= \lim_{t \to 0} \frac{P_t f(x_0)-f(x_0)}{t} \leq 0.
	\end{gather*}
    Thus, \eqref{PMP} holds.

    \item[\normalfont $2^\circ$ \eqref{PMP} $\Rightarrow$ dissipativity.]
    Assume that \eqref{PMP} holds and let $f\in\dom(A)$. Since $f\in\cont_\infty(\rd)$, we may assume that $f(x_0) = |f(x_0)| = \sup|f|$ (otherwise $f\rightsquigarrow -f$). Then
    $$
		\smash[b]{\|\lambda f-Af\|_{\infty}
		\geq \lambda f(x_0)- \underbrace{Af(x_0)}_{\leq 0}
		\geq \lambda f(x_0)  = \lambda \|f\|_{\infty}.}
	$$

    \item[\normalfont $3^\circ$ \eqref{PMP} $\Rightarrow$ sub-Markov.]
    Since $P_t$ is contractive, we have $P_tf(x) \leq \|P_tf\|_\infty \leq \|f\|_\infty \leq 1$ for all $f\in\cont_\infty(\rd)$ such that $|f|\leq 1$.
    In order to see positivity, let $f\in\cont_\infty(\rd)$ be non-negative. We distinguish between two cases:
    \begin{description}
    \item[\normalfont $1^\circ\ $  $R_\lambda f$ does not attain its infimum.] Since $R_\lambda f\in\cont_\infty(\rd)$ vanishes at infinity, we have necessarily $R_\lambda f\geq 0$.

    \item[\normalfont $2^\circ\ $ $\exists x_0\::\: R_\lambda f(x_0) = \inf R_\lambda f$.] Because of the \eqref{PMP} we find
    \begin{gather*}
        \lambda R_\lambda f(x_0)-f(x_0) = AR_\lambda f(x_0) \geq 0\\
        \implies
        \lambda R_\lambda f(x) \geq \inf \lambda R_\lambda f = \lambda R_\lambda f(x_0) \geq f(x_0) \geq 0.
    \end{gather*}
    \end{description}
    This proves that $f\geq 0\implies \lambda R_\lambda f \geq 0$. From \eqref{semi-e25} we see that $\lambda\mapsto R_\lambda f(x)$ is completely monotone, hence it is the Laplace transform of a positive measure. Since $R_\lambda f(x)$ has the integral representation \eqref{semi-e22}, we conclude that $P_tf(x)\geq 0$ (for all $t\geq 0$ as $t\mapsto P_t f$ is continuous).

    Using the Riesz representation theorem (as in Lemma~\ref{semi-05}) we can extend $P_t$ as a sub-Markov operator onto $\Bcal_b(\rd)$.
    \end{description}
\end{remark}

In order to determine the domain $\dom(A)$ of the generator the following `maximal dissipativity' result is handy.
\begin{lemma}[Dynkin, Reuter]\label{semi-29}
\index{generator!no proper extension}%
\index{Dynkin--Reuter lemma}%
    Assume that $(A,\dom(A))$ generates a Feller semigroup and that $(\mathfrak{A},\dom(\mathfrak{A}))$ extends $A$, i.e.\ $\dom(A) \subset \dom(\mathfrak{A})$ and $\mathfrak{A}|_{\dom(A)}=A$. If
    \begin{equation}\label{semi-e26}
        u\in\dom(\Afrak),\; u-\Afrak u=0 \implies u=0,
    \end{equation}
    then $(A,\dom(A))=(\mathfrak{A},\dom(\mathfrak{A}))$.
\end{lemma}
\begin{proof}
    Since $A$ is a generator, $(\id - A):\dom(A)\to\cont_\infty(\rd)$ is bijective. On the other hand, the relation \eqref{semi-e26} means that $(\id-\mathfrak{A})$ is injective, but $(\id-A)$ cannot have a proper injective extension.
\end{proof}

\begin{theorem}\label{semi-31}
	Let $(P_t)_{t \geq 0}$ be a Feller semigroup with generator $(A,\dom(A))$. Then
\index{generator!domain}%
\index{Feller process!generator}%
    \begin{equation}\label{semi-e28}
        \dom(A)
        = \left\{f \in \cont_{\infty}(\rd)\:\:\middle|\:\: \exists g \in \cont_{\infty}(\rd) \;\: \forall x\::\: \lim_{t\to 0} \frac{P_t f(x)-f(x)}{t} = g(x) \right\}.
	\end{equation}
\end{theorem}
\begin{proof}
	Denote by $\dom(\mathfrak{A})$ the right-hand side of \eqref{semi-e28} and define
    $$
		\mathfrak{A}f(x) := \lim_{t \to 0} \frac{P_t f(x)-f(x)}{t}
        \fa f\in\dom(\mathfrak{A}), \; x\in\rd.
    $$
    Obviously, $(\mathfrak{A},\dom(\mathfrak{A}))$ is a linear operator which extends $(A,\dom(A))$. Since \eqref{PMP} is, essentially, a pointwise assertion (see Remark~\ref{semi-27}, $1^\circ$), $\mathfrak{A}$ inherits \eqref{PMP}; in particular, $\mathfrak{A}$ is dissipative (see Remark~\ref{semi-27}, $2^\circ$):
    $$
		\|\mathfrak{A} f-\lambda f\|_{\infty} \geq \lambda\|f\|_{\infty}.
	$$
	This implies \eqref{semi-e26}, and the claim follows from Lemma~\ref{semi-29}.
\end{proof}

\chapter{The generator of a L\'evy process}\label{gen}

We want to study the structure of the generator of (the semigroup corresponding to) a L\'evy process $X=\Xt$. This will also lead to a proof of the L\'evy--Khintchine formula.

Our approach uses some Fourier analysis. We denote by $\cont_c^\infty(\rd)$ and $\Scal(\rd)$ the smooth, compactly supported functions and the smooth, rapidly decreasing `Schwartz functions'.\footnote{To be precise, $f\in\Scal(\rd)$, if $f\in\cont^\infty(\rd)$ and if $\sup_{x\in\rd} (1+|x|^N) |\partial^\alpha f(x)|\leq c_{N,\alpha}$ for any $N\in\nat_0$ and any multiindex $\alpha\in\nat_0^d$.} The \emphh{Fourier transform}
\index{Fourier transform}%
is denoted by
$$
    \widehat f(\xi) = \Fcal f(\xi) := (2\pi)^{-d}\int_{\rd} f(x)\,\eup^{-\iup \xi\cdot x}\,dx,\quad f\in L^1(dx).
$$
Observe that $\Fcal f$ is chosen in such a way that the \emphh{characteristic function becomes the inverse Fourier transform}.

We have seen in Proposition~\ref{levy-11} and its Corollaries~\ref{levy-13} and~\ref{levy-15} that $X$ is completely characterized by the characteristic exponent $\psi:\rd\to\comp$
$$
	\Ee\,\eup^{\iup  \xi\cdot X_t} = \big[\Ee\,\eup^{\iup  \xi\cdot X_1}\big]^t = \eup^{-t \psi(\xi)},\quad t\geq 0,\;\xi\in\rd.
$$
We need a few more properties of $\psi$ which result from the fact that $\chi(\xi) = \eup^{-\psi(\xi)}$ is a characteristic function.
\begin{lemma}\label{gen-03}
    Let $\chi(\xi)$ be any characteristic function of a probability measure $\mu$. Then
    \begin{equation}\label{gen-e02}
        |\chi(\xi+\eta) -\chi(\xi)\chi(\eta)|^2
        \leq (1-|\chi(\xi)|^2)(1-|\chi(\eta)|^2),\quad \xi,\eta\in\rd.
    \end{equation}
\end{lemma}
\begin{proof}
    Since $\mu$ is a probability measure, we find from the definition of $\chi$
    \begin{align*}
        \chi(\xi+\eta)-\chi(\xi)\chi(\eta)
        &= \iint \big(\eup^{\iup x\cdot\xi}\eup^{\iup x\cdot\eta} - \eup^{\iup x\cdot\xi}\eup^{\iup y\cdot\eta}\big)\,\mu(dx)\,\mu(dy)\\
        &= \frac 12\iint \big(\eup^{\iup x\cdot\xi}-\eup^{\iup y\cdot\xi}\big) \big(\eup^{\iup x\cdot\eta}-\eup^{\iup y\cdot\eta}\big)\,\mu(dx)\,\mu(dy).
    \end{align*}
    In the last equality we use that the integrand is symmetric in $x$ and $y$, which allows us to interchange the variables. Using the elementary formula $|\eup^{\iup a}-\eup^{\iup b}|^2 = 2-2\cos(b-a)$ and the Cauchy--Schwarz inequality yield
    \begin{align*}
        |\chi&(\xi+\eta)-\chi(\xi)\chi(\eta)|\\
        &\leq \frac 12\iint \big|\eup^{\iup x\cdot\xi}-\eup^{\iup y\cdot\xi}\big| \cdot \big|\eup^{\iup x\cdot\eta}-\eup^{\iup y\cdot\eta}\big|\,\mu(dx)\,\mu(dy)\\
        &= \iint \sqrt{1-\cos(y-x)\cdot\xi} \sqrt{1-\cos (y-x)\cdot\eta}\:\mu(dx)\,\mu(dy)\\
        &\leq \sqrt{\iint \big(1-\cos(y-x)\cdot\xi\big)\,\mu(dx)\,\mu(dy)}\sqrt{\iint \big(1-\cos(y-x)\cdot\eta\big)\,\mu(dx)\,\mu(dy)} .
    \end{align*}
    This finishes the proof as
    \begin{gather*}
        \iint \cos(y-x)\cdot\xi\,\mu(dx)\,\mu(dy)
        = \Re \left[\int \eup^{\iup y\cdot\xi}\,\mu(dy)\int \eup^{-\iup x\cdot\xi}\,\mu(dx)\right]
        = |\chi(\xi)|^2.
    \qedhere
    \end{gather*}
\end{proof}

\begin{theorem}\label{gen-05}
    Let $\psi: \rd \to \mathbb{C}$ be the characteristic exponent of a L\'evy process. Then the function $\xi\mapsto\sqrt{|\psi(\xi)|}$ is subadditive and
\index{characteristic exponent!subadditive@is subadditive}%
\index{characteristic exponent!polynomially bounded@is polynomially bounded}%
    \begin{equation}\label{gen-e04}
        |\psi(\xi)| \leq c_\psi (1+|\xi|^2),\quad\xi \in \rd.
    \end{equation}
\end{theorem}
\begin{proof}
    We use \eqref{gen-e02} with $\chi = \eup^{-t\psi}$, divide by $t>0$ and let $t\to 0$. Since $|\chi| = \eup^{-t\Re\psi}$, this gives
    $$
        |\psi(\xi+\eta)-\psi(\xi) - \psi(\eta)|^2
        \leq 4\Re\psi(\xi)\Re\psi(\eta)
        \leq 4|\psi(\xi)|\cdot|\psi(\eta)|.
    $$
    By the lower triangle inequality,
    $$
        |\psi(\xi+\eta)|-|\psi(\xi)|-|\psi(\eta)| \leq 2\sqrt{|\psi(\xi)|}\sqrt{|\psi(\eta)|}
    $$
    and this is the same as subadditivity: $\sqrt{|\psi(\xi+\eta)|}\leq \sqrt{|\psi(\xi)|} + \sqrt{|\psi(\eta)|}$.

    In particular, $|\psi(2\xi)| \leq 4|\psi(\xi)|$. For any $\xi\neq 0$ there is some integer $n=n(\xi)\in\integer$ such that $2^{n-1}\leq |\xi|\leq 2^{n}$, so
    \begin{gather*}
        |\psi(\xi)|
        = |\psi(2^{n} 2^{-n}\xi)|
        \leq \max\{1,2^{2n}\} \sup_{|\eta|\leq 1}|\psi(\eta)|
        \leq 2\sup_{|\eta|\leq 1}|\psi(\eta)|(1+|\xi|^2).
    \qedhere
    \end{gather*}
\end{proof}

\begin{lemma}\label{gen-11}
\index{generator!domain}%
	Let $\Xt$ be a L\'evy process and denote by $(A,\dom(A))$ its infinitesimal generator. Then $\cont_c^{\infty}(\rd) \subset \dom(A)$.
\end{lemma}
\begin{proof}
    Let $f\in\cont_c^\infty(\rd)$. By definition, $P_t f(x) = \Ee f(X_t+x)$. Using the differentiation lemma for parameter-dependent integrals
    (e.g.\ \cite[Theorem 11.5]{schilling-mims} or \cite[12.2]{schilling-mi}) it is not hard to see that $P_t: \Scal(\rd) \to \Scal(\rd)$. Obviously,
    \begin{align}\label{gen-e12}
		\eup^{-t \psi(\xi)} = \Ee\,\eup^{\iup  \xi\cdot X_t} = \Ee^x \eup^{\iup  \xi\cdot (X_t-x)} = \eup_{-\xi}(x) P_t \eup_{\xi}(x)
	\end{align}
	for $\eup_{\xi}(x) := \eup^{\iup  \xi\cdot x}$. Recall that the Fourier transform of $f \in \Scal(\rd)$ is again in $\Scal(\rd)$. From
    \begin{align}
		P_t f
		\pomu{\eqref{gen-e12}}{=}{} P_t \int \widehat{f}(\xi) \eup_{\xi}(\cdot) \, d\xi
		&\pomu{\eqref{gen-e12}}{=}{} \int \widehat{f}(\xi) P_t \eup_{\xi}(\cdot) \, d\xi \label{gen-e13} \\
		&\omu{\eqref{gen-e12}}{=}{} \int \widehat{f}(\xi) \eup_{\xi}(\cdot) \eup^{-t \psi(\xi)} \, d\xi\notag
	\end{align}
	we conclude that $\widehat{P_t f} = \widehat{f} \eup^{-t\psi}$. Hence,
    \begin{equation}\label{gen-e14}
		P_t f = \Fcal^{-1}(\widehat{f} \eup^{-t \psi}).
	\end{equation}
	Consequently,
    \begin{align*}
		&\frac{\widehat{P_t f}-\widehat{f}}{t}
		= \frac{\eup^{-t\psi} \widehat{f}-\widehat{f}}{t}
		\xrightarrow[\;t\to 0\;]{} -\psi\widehat{f} \\
		\xRightarrow{\;\:\widehat{f} \in \Scal(\rd)\;\:}\quad &\frac{P_t f(x)-f(x)}{t}
		\xrightarrow[\;t\to 0\;]{} g(x) := \Fcal^{-1}(-\psi \widehat{f})(x).
	\end{align*}
    Since $\psi$ grows at most polynomially (Lemma~\ref{gen-05}) and $\widehat f\in\Scal(\rd)$, we see $\psi\widehat f\in L^1(dx)$ and, by the Riemann--Lebesgue lemma, $g\in\cont_{\infty}(\rd)$. Using Theorem~\ref{semi-31} it follows that $f\in\dom(A)$.
\end{proof}

\begin{definition}\label{gen-13}
\index{symbol}%
	Let $L : \cont_b^2(\rd)\to \cont_b(\rd)$ be a linear operator. Then
    \begin{equation}\label{gen-e16}
		L(x,\xi) := \eup_{-\xi}(x) L_x \eup_{\xi}(x)
	\end{equation}
	is the \emphh{symbol} of the operator $L=L_x$,  where $\eup_\xi(x) := \eup^{\iup \xi\cdot x}$.
\end{definition}

The proof of Lemma~\ref{gen-11} actually shows that we can recover an operator $L$ from its symbol $L(x,\xi)$ if, say, $L:\cont_b^2(\rd) \to \cont_b(\rd)$ is continuous:\footnote{As usual, $\cont_b^2(\rd)$ is endowed with the norm $\|u\|_{(2)} = \sum_{0 \leq |\alpha| \leq 2} \|\partial^{\alpha} u \|_{\infty}$.} Indeed, for all $u\in\cont_c^\infty(\rd)$
\begin{align*}
	Lu(x) &= L \int \widehat{u}(\xi) \eup_{\xi}(x) \, d\xi \\
	&= \int \widehat{u}(\xi) L_x \eup_{\xi}(x) \, d\xi \\
	&= \int \widehat{u}(\xi) L(x,\xi) \eup_{\xi}(x) d\xi = \Fcal^{-1}(L(x,\cdot)\Fcal u(\cdot))(x).
\end{align*}

\begin{example}\label{gen-15}
\index{Brownian motion!generator}%
\index{generator!Brownian motion@of Brownian motion}%
    A typical example would be the Laplace operator (i.e.\ the generator of a Brownian motion)
    $$
        \tfrac 12\Delta f(x)
        = -\tfrac 12(\tfrac 1\iup\partial_x)^2f(x)
        = \int \widehat f(\xi) \big(-\tfrac 12|\xi|^2\big) \eup^{\iup \xi\cdot x}\,d\xi,
        \quad\text{i.e.}\quad L(x,\xi) = -\frac 12|\xi|^2,
    $$
    or the fractional Laplacian of order $\frac 12 \alpha\in (0,1)$ which generates a rotationally symmetric $\alpha$-stable L\'evy process
    $$
        -(-\Delta)^{\alpha/2} f(x)
        = \int \widehat f(\xi) \big(-|\xi|^\alpha\big) \eup^{\iup \xi\cdot x}\,d\xi,
        \quad\text{i.e.}\quad L(x,\xi) = -|\xi|^\alpha.
    $$
    More generally, if $P(x,\xi)$ is a polynomial in $\xi$, then the corresponding operator is obtained by replacing $\xi$ by $\frac 1\iup \nabla_x$ and formally expanding the powers.
\end{example}

\begin{definition}\label{gen-17}
	An operator of the form
    \begin{equation}\label{gen-e18}
		L(x,D) f(x) = \int \widehat{f}(\xi) L(x,\xi) \eup^{\iup   x\cdot\xi}\, d\xi, \quad f \in \Scal(\rd),
	\end{equation}
	is called (if defined) a \emphh{pseudo differential operator} with (non-classical) symbol $L(x,\xi)$.
\index{pseudo differential operator}%
\index{symbol}%
\end{definition}

\begin{remark}\label{gen-19}
\index{L\'evy process!symbol}%
\index{symbol!of a L\'evy process}%
\index{translation invariant}%
\index{L\'evy process!translation invariance}%
    The \emphh{symbol of a L\'evy process} does not depend on $x$, i.e.\ $L(x,\xi)=L(\xi)$. This is a consequence of the spatial homogeneity of the process which is encoded in the translation invariance of the semigroup (cf.\ \eqref{lpmp-e24} and Lemma~\ref{lpmp-09}):
    $$
        P_tf(x) = \Ee f(X_t+x)
        \implies
        P_tf(x) = \vartheta_{x}(P_tf)(0) = P_t(\vartheta_{x}f)(0)
    $$
    where $\vartheta_{x}u(y) = u(y+x)$ is the shift operator. This property is obviously inherited by the generator, i.e.
    $$
        Af(x) = \vartheta_{x}(Af)(0) = A(\vartheta_{x}f)(0),\quad f\in\dom(A).
    $$

    As a matter of fact, the converse is also true: If $L:\cont_c^\infty(\rd)\to\cont(\rd)$ is a linear operator satisfying $\vartheta_x(Lf) = L(\vartheta_xf)$, then $Lf = f*\lambda$ where $\lambda$ is a distribution, i.e.\ a continuous linear functional $\lambda:\cont_c^\infty(\rd)\to\real$, cf.\ Theorem~\ref{app-91}.
\end{remark}

\begin{theorem}\label{gen-21}
\index{L\'evy process!generator}%
\index{generator!L\'evy process@of a L\'evy process}%
\index{L\'evy--Khintchine formula}%
	Let $\Xt$ be a L\'evy process with generator $A$. Then
    \begin{equation}\label{gen-e22}
		Af(x)
        = l\cdot \nabla f(x) + \frac{1}{2} \nabla\cdot Q \nabla f(x) + \int\limits_{y \neq 0} \big[f(x+y)-f(x)- \nabla f(x)\cdot y\I_{(0,1)}(|y|)\big]\, \nu(dy)
    \end{equation}
    for any $f\in\cont_c^\infty(\rd)$, where $l\in\rd$, $Q\in\real^{d\times d}$ is a positive semidefinite matrix, and $\nu$ is a measure on $\rd\setminus\{0\}$ such that $\int_{y\neq 0} \min\{1,|y|^2\}\,\nu(dy)<\infty$.

    Equivalently, $A$ is a pseudo differential operator
\index{pseudo differential operator}%
\index{symbol!of a L\'evy process}%
    \begin{equation}\label{gen-e34}
		Au(x) = - \psi(D) u(x) = - \int \widehat{u}(\xi) \psi(\xi) \eup^{\iup   x\cdot\xi} \, d\xi, \quad u \in \cont_c^\infty(\rd),
	\end{equation}
    whose symbol is the characteristic exponent $-\psi$ of the L\'evy process. It is given by the L\'evy--Khintchine formula
    \begin{equation}\label{gen-e36}
        \psi(\xi)
        = -\iup l\cdot\xi + \frac 12\xi\cdot Q\xi + \int_{y\neq 0}\left(1-\eup^{\iup y\cdot\xi}+\iup\xi\cdot y\I_{(0,1)}(|y|)\right)\nu(dy)
    \end{equation}
    where the triplet $(l,Q,\nu)$ as above.
\end{theorem}

I learned the following proof from Francis Hirsch; it is based on arguments by Courr\`ege \cite{courrege64} and Herz \cite{herz64}. The presentation below follows the version in B\"{o}ttcher, Schilling \& Wang \cite[Section 2.3]{schilling-lm}.
\begin{proof}
The proof is divided into several steps.
\begin{enumerate}
\item[$1^\circ$]
    We have seen in Lemma~\ref{gen-11} that $\cont_c^{\infty}(\rd) \subset \dom(A)$.

\item[$2^\circ$]
    Set $A_0f := (Af)(0)$ for $f \in \cont_c^{\infty}(\rd)$. This is a linear functional on $\cont_c^{\infty}$. Observe that
    $$
        f\in\cont_c^\infty(\rd),\quad f\geq 0,\quad f(0)=0
        \quad\xRightarrow{\eqref{PMP}}\quad A_0f \geq 0.
    $$

\item[$3^\circ$]  By $2^\circ$, $f \mapsto A_{00}f := A_0(|\cdot|^2 \cdot f)$ is a positive linear functional on $\cont_c^{\infty}(\rd)$. Therefore it is bounded. Indeed, let $f\in \cont_c^\infty(K)$ for a compact set $K\subset\rd$ and let $\phi\in\cont_c^\infty(\rd)$ be a cut-off function such that $\I_K\leq\phi\leq 1$. Then
    $$
        \|f\|_\infty\phi \pm f \geq 0.
    $$
    By linearity and positivity $\|f\|_\infty A_{00}\phi \pm A_{00}f\geq 0$ which shows $|A_{00}f|\leq C_K \|f\|_\infty$ with the constant $C_K = A_{00}\phi$.

    By Riesz' representation theorem, there exists a Radon measure\footnote{A Radon measure on a topological space $E$ is a Borel measure which is finite on compact subsets of $E$ and regular: for all open sets $\mu(U) = \sup_{K\subset U}\mu(K)$ and for all Borel sets $\mu(B) = \inf_{U\supset B}\mu(U)$ ($K$, $U$ are generic compact and open sets, respectively).} $\mu$ such that
    \begin{align*}
		\smash[b]{A_{0} (|\cdot|^2f)
		= \int f(y) \, \mu(dy)
		= \int |y|^2 f(y) \underbrace{\frac{\mu(dy)}{|y|^2}}_{=:\nu(dy)}
		= \int |y|^2 f(y) \, \nu(dy).}
	\end{align*}
	This implies that
    $$
		A_0 f_0 = \int_{y \neq 0} f_0(y) \, \nu(dy) \fa f_0\in \cont_c^{\infty}(\rd \setminus \{0\});
	$$
    since any compact subset of $\rd\setminus\{0\}$ is contained in an annulus $B_R(0)\setminus B_\epsilon(0)$, we have $\supp f_0 \cap B_\epsilon(0) = \emptyset$ for some sufficiently small $\epsilon>0$. The measure $\nu$ is a Radon measure on $\rd\setminus\{0\}$.
	
\item[$4^\circ$]  Let $f,g \in \cont_c^{\infty}(\rd)$, $0 \leq f,g \leq 1$, $\supp f \subset \overline{B_1(0)}$, $\supp g \subset \overline{ B_1(0)}^c$ and $f(0)=1$. From
    $$
		\sup_{y\in\rd} \big( \|g\|_{\infty} f(y)+ g(y) \big)
        = \|g\|_{\infty}
        = \|g\|_{\infty}f(0) + g(0)
	$$
	and \eqref{PMP}, it follows that $A_0(\|g\|_{\infty} f+g) \leq 0$. Consequently,
    $$
		A_0 g \leq - \|g\|_{\infty} A_0 f.
	$$
	If $g \uparrow 1-\I_{\overline{B_1(0)}}$, then this shows
    $$
		\int_{|y|> 1} \nu(dy) \leq - A_0 f < \infty.
	$$
	Hence, $\int_{y\neq 0} (|y|^2 \wedge 1) \, \nu(dy)<\infty$.

\item[$5^\circ$]  Let $f \in \cont_c^{\infty}(\rd)$ and $\phi(y) = \I_{(0,1)}(|y|)$. Define
    \begin{equation}\label{gen-e26}
		S_0 f := \int_{y \neq 0} \big[f(y)-f(0)- y\cdot\nabla f(0) \phi(y)\big] \, \nu(dy).
	\end{equation}
	By Taylor's formula, there is some $\theta\in (0,1)$ such that
    \begin{gather*}
		f(y)-f(0)- y\cdot \nabla f(0) \phi(y)
        = \frac 12\sum_{k,l=1}^d \frac{\partial^{2} f(\theta y)}{\partial x_k\partial x_l} y_ky_l.
    \end{gather*}
    Using the elementary inequality $2y_ky_l \leq y_k^2+y_l^2\leq |y|^2$, we obtain
    \begin{align*}
		|f(y)-f(0)- y\cdot\nabla f(0) \phi(y)|
        &\leq \begin{cases} \frac 14 \sum_{k,l=1}^d \big\|\frac{\partial^2}{\partial x_k \partial x_l} f\big\|_{\infty} |y|^2, & |y| < 1\\
        2 \|f\|_{\infty}, & |y|\geq 1 \end{cases} \\
		&\leq 2 \|f\|_{(2)} (|y|^2 \wedge 1).
	\end{align*}
    This means that $S_0$ defines a distribution (generalized function) of order $2$.

\item[$6^\circ$]
    Set $L_0 := A_0-S_0$. The steps $2^\circ$ and $5^\circ$ show that $A_0$ is a distribution of order $2$. Moreover,
    $$
		L_0f_0 = \int_{y \neq 0} \big[f_0(0)- y\cdot\nabla f_0(0) \phi(y)\big] \, \nu(dy) = 0
	$$
	for any $f_0 \in \cont_c^{\infty}(\rd)$ with $f_0|_{B_\epsilon(0)}=0$ for some $\epsilon>0$. Hence, $\supp(L_0) \subset \{0\}$.

	Let us show that $L_0$ is \emphh{almost positive}
    \index{almost positive (PP)}%
    \index{PP (almost positive)}%
    (also: `fast positiv', `pr\`esque positif'):
    \begin{equation}\label{PP} \tag{PP}
        f_0 \in \cont_c^{\infty}(\rd),\quad f_0(0)=0,\quad f_0 \geq 0
        \quad\implies\quad L_0f_0 \geq 0.
	\end{equation}
    Indeed: Pick $0 \leq \phi_n \in \cont_c^{\infty}(\rd \setminus \{0\})$, $\phi_n \uparrow \I_{\rd \setminus \{0\}}$ and let $f_0$ be as in \eqref{PP}. Then
    \begin{eqnarray*}
		L_0 f_0
		&\omu{\supp L_0 \subset \{0\}}{=}{}& L_0[(1-\phi_n) f_0)] \\
		&=& A_0[(1-\phi_n) f_0] - S_0[(1-\phi_n) f_0] \\
		&\omu{f_0(0)=0}{=}{\nabla f_0(0)=0}& A_0[(1-\phi_n)f_0] - \int_{y \neq 0} (1-\phi_n(y)) f_0(y) \, \nu(dy) \\
		&\omu{2^\circ}{\geq}{\eqref{PMP}}& - \int (1-\phi_n(y)) f_0(y) \, \nu(dy) \xrightarrow[n\to\infty]{} 0
	\end{eqnarray*}
	by the monotone convergence theorem.

\item[$7^\circ$ ]  As in $5^\circ$ we find with Taylor's formula for $f \in \cont_c^{\infty}(\rd)$, $\supp f\subset K$ and $\phi\in \cont_c^\infty(\rd)$ satisfying $\I_{K}\leq\phi\leq 1$
    $$
		\left(f(y)-f(0)-\nabla f(0) \cdot y\right)\phi(y) \leq 2 \|f\|_{(2)} |y|^2\phi(y).
    $$
    (As usual, $\|f\|_{(2)} = \sum_{0\leq|\alpha|\leq 2} \|\partial^\alpha f\|_\infty$.) Therefore,
    $$
        2 \|f\|_{(2)} |y|^2\phi(y) + f(0)\phi(y) + \nabla f(0) \cdot y\phi(y) - f(y) \geq 0,
    $$
    and \eqref{PP} implies
	$$
        L_0 f \leq f(0) L_0\phi + |\nabla f(0)| L_0(|\cdot|\phi)+ 2\|f\|_{(2)} L_0(|\cdot|^2\phi)
        \leq C_K \|f\|_{(2)}.
	$$

\item[$8^\circ$]
    We have seen in $6^\circ$ that $L_0$ is of order $2$ and $\supp L_0 \subset \{0\}$. Therefore,
    \begin{equation}\label{gen-e28}
        L_0 f
        = \frac{1}{2} \sum_{k,l=1}^d q_{kl} \frac{\partial^2 f(0)}{\partial x_k\partial x_l} + \sum_{k=1}^d l_k \frac{\partial f(0)}{\partial x_k} - c f(0).
	\end{equation}
    We will show that $(q_{kl})_{k,l}$ is positive semidefinite. Set $g(y) := (y\cdot \xi)^2 f(y)$ where $f\in\cont_c^\infty(\rd)$ is such that $\I_{B_1(0)}\leq f\leq 1$. By \eqref{PP}, $L_0 g \geq 0$. It is not difficult to see that this implies
    $$
		\sum_{k,l=1}^d q_{kl} \xi_k \xi_l \geq 0, \fa \xi = (\xi_1,\dots,\xi_d)\in\rd.
	$$

\item[$9^\circ$]
    Since L\'evy processes and their semigroups are invariant under translations, cf.\ Remark~\ref{gen-19}, we get $Af(x) = A_0[f(x+\cdot)]$. If we replace $f$ by $f(x+\cdot)$, we get
    \begin{equation}\label{gen-e22-alt}\tag{$\text{\ref{gen-e22}}'$}\begin{aligned}
		Af(x)
        &= cf(x) + l\cdot \nabla f(x) + \frac{1}{2} \nabla\cdot Q \nabla f(x) \\
        &\qquad\mbox{} + \int\limits_{y \neq 0} \big[f(x+y)-f(x)- y\cdot\nabla f(x)\I_{(0,1)}(|y|)\big]\, \nu(dy).
    \end{aligned}\end{equation}
    We will show in the next step that $c=0$.

\item[$10^\circ$]
    So far, we have seen in $5^\circ, 7^\circ$ and $9^\circ$ that
    $$
        \|Af\|_\infty \leq C \|f\|_{(2)} = C \sum_{|\alpha|\leq 2} \|\partial^\alpha f\|_\infty,\quad f\in \cont_c^\infty(\rd),
    $$
    which means that $A$ (has an extension which) is continuous as an operator from $\cont_b^2(\rd)$ to $\cont_b(\rd)$.
    Therefore, $(A,\cont_c^\infty(\rd))$ is a pseudo differential operator with symbol
    $$
    	-\psi(\xi) = \eup_{-\xi}(x) A_x \eup_{\xi}(x), \quad \eup_\xi(x) = \eup^{\iup \xi\cdot x}.
    $$
    Inserting $\eup_\xi$ into \eqref{gen-e22-alt} proves \eqref{gen-e36} and, as $\psi(0)=0$, $c=0$.
\qedhere
\end{enumerate}
\end{proof}

\begin{remark}\label{gen-23}
    In step $8^\circ$ of the proof of Theorem~\ref{gen-21} one can use the \eqref{PMP} to show that the coefficient $c$ appearing in \eqref{gen-e22-alt} is positive. For this, let $(f_n)_{n \in \nat} \subset \cont_c^{\infty}(\rd)$, $f_n \uparrow 1$ and $f_n|_{B_1(0)}=1$. By \eqref{PMP}, $A_0 f_n \leq 0$. Moreover, $\nabla f_n(0)=0$ and, therefore, $$
		S_0 f_n = - \int (1-f_n(y)) \, \nu(dy) \xrightarrow[n \to \infty]{} 0.
	$$
	Consequently,
    $$
		\limsup_{n \to \infty} L_0 f_n = \limsup_{n \to \infty} (A_0 f_n-S_0 f_n) \leq 0
        \implies
        c\geq 0.
	$$

    For L\'evy processes we have $c=\psi(0)=0$ and this is a consequence of the \emphh{infinite life-time} of the process:
\index{L\'evy process!infinite life-time}
    $$
        \Pp(X_t\in\rd) = P_t 1 = 1\fa t\geq 0,
    $$
    and we can use the formula $P_tf - f = A\int_0^t P_s f\,ds$, cf.\ Lemma~\ref{semi-09}, for $f\equiv 1$ to show that $$c= A1 =0 \iff P_t1 = 1.$$
\end{remark}

\begin{definition}\label{gen-25}
    A \emphh{L\'evy measure} is a Radon measure $\nu$ on $\rd\setminus\{0\}$ s.t.\ $\int_{y\neq 0} (|y|^2\wedge 1)\,\nu(dy)<\infty$. A \emphh{L\'evy triplet} is a triplet $(l,Q,\nu)$ consisting of a vector $l\in\rd$, positive semi-definite matrix $Q\in\real^{d\times d}$ and a L\'evy measure $\nu$.
\index{L\'evy measure}%
\index{L\'evy triplet}%
\end{definition}

The proof of Theorem~\ref{gen-21} incidentally shows that the L\'evy triplet defining the exponent \eqref{gen-e36} or the generator \eqref{gen-e22} is unique. The following corollary can easily be checked using the representation \eqref{gen-e22}.
\begin{corollary}\label{gen-27}
    Let $A$ be the generator and $(P_t)_{t\geq 0}$ the semigroup of a L\'evy process. Then the L\'evy triplet is given by
    \begin{equation}\label{gen-e38}
    \begin{gathered}
        \int f_0\,d\nu = Af_0(0) = \lim_{t\to 0} \frac{P_tf_0(0)}{t}\quad\forall f_0\in \cont_c^\infty(\rd\setminus\{0\}),\\
        l_k = A\phi_k(0) - \int y_k\big[\phi(y)-\I_{(0,1)}(|y|)\big]\,\nu(dy),\quad k=1,\dots,d, \\
        q_{kl} = A(\phi_k\phi_l)(0) - \int_{y\neq 0} \phi_k(y)\phi_l(y)\,\nu(dy),\quad k,l=1,\dots d,
    \end{gathered}
    \end{equation}
    where $\phi\in \cont_c^\infty(\rd)$ satisfies $\I_{B_1(0)}\leq\phi\leq 1$ and $\phi_k(y):= y_k\phi(y)$.
    In particular, $(l,Q,\nu)$ is uniquely determined by $A$ or the characteristic exponent $\psi$.
\end{corollary}
We will see an alternative uniqueness proof in the next Chapter~\ref{cons}.

\begin{remark}\label{gen-29}
Setting $p_t(dy) = \Pp(X_t\in dy)$, we can recast the formula for the L\'evy measure as
$$
    \nu(dy) = \lim_{t\to 0}\frac{p_t(dy)}{t}
    \qquad (\text{vague limit of measures on the set $\rd\setminus\{0\}$}).
$$
Moreover, a direct calculation using the L\'evy--Khintchine formula \eqref{gen-e38} gives the following alternative representation for the $q_{kl}$:
$$
    \frac 12\xi\cdot Q\xi
    = \lim_{n\to\infty} \frac{\psi(n\xi)}{n^2},\quad \xi\in\rd.
$$
\end{remark}

\chapter{Construction of L\'evy processes}\label{cons}

Our starting point is now the L\'evy--Khintchine formula for the characteristic exponent $\psi$ of a L\'evy process
\begin{equation}\label{cons-e02}
    \psi(\xi)
        = -\iup l\cdot\xi + \frac 12\xi\cdot Q\xi + \int_{y\neq 0}\big[1-\eup^{\iup y\cdot\xi}+\iup\xi\cdot y\I_{(0,1)}(|y|)\big]\,\nu(dy)
\end{equation}
where $(l,Q,\nu)$ is a L\'evy triplet in the sense of Definition~\ref{gen-25}; a proof of \eqref{cons-e02} is contained in Theorem~\ref{gen-21}, but the exposition below is independent of this proof, see however Remark~\ref{cons-21} at the end of this chapter.

What will be needed is that a compound Poisson process is a L\'evy process with c\`adl\`ag paths and characteristic exponent of the form
\index{compound Poisson process}%
\begin{equation}\label{cons-e04}
    \phi(\xi) = \int_{y\neq 0} \big[1-\eup^{\iup y\cdot\xi}\big]\,\rho(dy)
\end{equation}
($\rho$ is any finite measure), see Example~\ref{exlp-05}.d), where $\rho(dy) = \lambda\cdot\mu(dy)$.

Let $\nu$ be a L\'evy measure and denote by $A_{a,b} = \{y\,:\, a\leq |y|<b\}$ an annulus. Since we have $\int_{y\neq 0}\big[|y|^2\wedge 1\big]\,\nu(dy)<\infty$, the measure $\rho(B) := \nu(B\cap A_{a,b})$ is a finite measure, and there is a corresponding compound Poisson process. Adding a drift with $l = -\int y\,\rho(dy)$ shows that for every exponent
\begin{equation}\label{cons-e06}
    \psi^{a,b}(\xi) = \int_{a\leq |y|<b} \big[1-\eup^{\iup y\cdot\xi} + \iup y\cdot\xi\big]\,\nu(dy)
\end{equation}
there is some L\'evy process $X^{a,b} = (X_t^{a,b})_{t\geq 0}$. In fact,
\begin{lemma}\label{cons-03}
    Let $0<a<b\leq\infty$ and $\psi^{a,b}$ given by \eqref{cons-e06}. Then the corresponding L\'evy process $X^{a,b}$ is an $L^2(\Pp)$-martingale with c\`adl\`ag paths such that
    $$
		\Ee \big[X_t^{a,b}\big]=0
    \et
        \Ee\big[X_t^{a,b} \cdot (X_t^{a,b})^\top\big] = \left( t\int_{a\leq |y| < b} y_k y_l \, \nu(dy) \right)_{k,l}.
	$$
\end{lemma}
\begin{proof}
    Set $X_t := X_t^{a,b}$, $\psi := \psi^{a,b}$ and $\Fscr_t:= \sigma(X_r,\,r\leq t)$.
    Using the differentiation lemma for parameter-dependent integrals we see that $\psi$ is twice continuously differentiable and
    $$
        \frac{\partial\psi(0)}{\partial \xi_k} = 0
        \et
        \frac{\partial^2\psi(0)}{\partial \xi_k\partial \xi_l} = \int_{a\leq |y|<b} y_ky_l\,\nu(dy).
    $$
    Since the characteristic function $\eup^{-t\psi(\xi)}$ is twice continuously differentiable, $X$ has first and second moments, cf.\ Theorem~\ref{app-21}, and these can be obtained by differentiation:
    \begin{gather*}
    	\Ee X_t^{(k)}
        = \frac{1}{\iup}\frac{\partial}{\partial \xi_k} \Ee\,\eup^{\iup   \xi\cdot X_t} \bigg|_{\xi=0}
        = \iup t \frac{\partial\psi(0)}{\partial \xi_k} = 0
    \intertext{and}
    \begin{aligned}
		\Ee(X_t^{(k)} X_t^{(l)})
        = - \frac{\partial^2}{\partial \xi_k \partial \xi_l} \Ee\,\eup^{\iup \xi\cdot X_t} \bigg|_{\xi=0}
        &= \frac{t\partial^2\psi(0)}{\partial \xi_k\partial \xi_l} - \frac{t\partial\psi(0)}{\partial \xi_k}\frac{t\partial\psi(0)}{\partial \xi_l}\\
        &= t\int_{a\leq |y|<b} y_ky_l\,\nu(dy).
    \end{aligned}
    \end{gather*}
    The martingale property now follows from the independence of the increments: Let $s\leq t$, then
    \begin{gather*}
        \Ee(X_t\mid\Fscr_s)
        =\Ee(X_t-X_s + X_s\mid\Fscr_s)
        =\Ee(X_t-X_s \mid\Fscr_s) + X_s
        \omu{\eqref{Lindep}}{=}{\eqref{Lstat}} \Ee(X_{t-s}) + X_s
        = X_s.
    \qedhere
    \end{gather*}
\end{proof}

We will use the processes from Lemma~\ref{cons-03} as main building blocks for the L\'evy process. For this we need some preparations.
\begin{lemma}\label{cons-05}
\index{L\'evy process!sum of two LP}%
    Let $(X_t^k)_{t \geq 0}$ be L\'evy processes with characteristic exponents $\psi_k$. If $X^1 \bbot X^2$, then $X:=X^1+X^2$ is a L\'evy process with characteristic exponent $\psi=\psi_1+\psi_2$.
\end{lemma}
\begin{proof}
    Set $\Fscr_t^k := \sigma(X_s^k,\; s\leq t)$ and $\Fscr_t = \sigma(\Fscr_t^1, \Fscr_t^2)$. Since $X^1 \bbot X^2$, we get for $F=F_1\cap F_2$, $F_k\in\Fscr_s^k$,
    \begin{align*}
        \Ee\left( \eup^{\iup \xi\cdot (X_t-X_s)} \I_F\right)
        &\pomu{\eqref{Lindep}}{=}{} \Ee\left( \eup^{\iup \xi\cdot (X_t^1-X_s^1)} \I_{F_1}\cdot \eup^{\iup \xi\cdot (X_t^2-X_s^2)} \I_{F_2}\right)\\
        &\pomu{\eqref{Lindep}}{=}{} \Ee\left( \eup^{\iup \xi\cdot (X_t^1-X_s^1)} \I_{F_1}\right) \Ee\left(\eup^{\iup \xi\cdot (X_t^2-X_s^2)} \I_{F_2}\right)\\
        &\omu{\eqref{Lindep}}{=}{} \eup^{-(t-s)\psi_1(\xi)}\Pp(F_1)\cdot \eup^{-(t-s)\psi_2(\xi)}\Pp(F_2)\\
        &\pomu{\eqref{Lindep}}{=}{} \eup^{-(t-s)(\psi_1(\xi)+\psi_2(\xi))}\Pp(F).
    \end{align*}
    As $\{F_1\cap F_2\,:\, F_k\in\Fscr_s^k\}$ is a $\cap$-stable generator of $\Fscr_s$, we find
    \begin{align*}
        \Ee\left( \eup^{\iup \xi\cdot (X_t-X_s)} \:\middle|\: \Fscr_s\right)
        &= \eup^{-(t-s)\psi(\xi)}.
    \end{align*}
    Observe that $\Fscr_s$ could be larger than the canonical filtration $\Fscr_s^X$. Therefore, we first condition w.r.t.\ $\Ee (\dots\mid\Fscr_s^X)$ and then use Theorem~\ref{exlp-03}, to see that $X$ is a L\'evy process with characteristic exponent $\psi = \psi_1+\psi_2$.
\end{proof}

\begin{lemma}\label{cons-07}
\index{L\'evy process!limit of LPs}%
    Let $(X^n)_{n \in\nat}$ be a sequence of L\'evy processes with characteristic exponents $\psi_n$. Assume that $X^n_t\to X_t$ converges in probability for every $t\geq 0$. If
    \begin{description}
    \item[\normalfont\itshape either:] the convergence is uniform in probability, i.e.
    $$
		\forall\epsilon>0\;\;\forall t\geq 0\::\: \lim_{n \to \infty} \Pp  \left( \sup_{s \leq t} |X^n_s-X_s| > \epsilon \right)=0,
	$$
    \item[\normalfont\itshape or:] the limiting process $X$ has c\`adl\`ag paths,
    \end{description}
    then $X$ is a L\'evy process with characteristic exponent $\psi:= \lim_{n \to \infty} \psi_n$.
\end{lemma}
\begin{proof}
    Let $0=t_0<t_1<\dots <t_m$ and $\xi_1,\dots,\xi_m\in\rd$. Since the $X^n$ are L\'evy processes,
    $$
         \Ee \exp\left[\iup \sum_{k=1}^m \xi_k\cdot(X_{t_k}^n - X_{t_{k-1}}^n)\right]
         \omu{\eqref{Lindep}, \eqref{Lstat}}{=}{}
         \prod_{k=1}^m \Ee \exp\left[\iup \xi_k\cdot X_{t_k-t_{k-1}}^n\right]
    $$
    and, because of convergence in probability, this equality is inherited by the limiting process $X$. This proves that $X$ has independent \eqref{LindepAlt} and stationary \eqref{Lstat} increments.

    The condition \eqref{Lcont} follows either from the uniformity of the convergence in probability or the c\`adl\`ag property. Thus, $X$ is a L\'evy process. From
    $$
			\lim_{n\to\infty} \Ee\,\eup^{\iup  \xi\cdot X_1^n}
			= \Ee\,\eup^{\iup  \xi\cdot X_1}
    $$
    we get that the limit $\lim_{n\to\infty}\psi_n = \psi$ exists.
\end{proof}

\begin{lemma}[Notation of Lemma~\ref{cons-03}]\label{cons-09}
    Let $(a_n)_{n\in\nat}$ be a sequence $a_1 > a_2 > \dots$ decreasing to zero and assume that the processes $(X^{a_{n+1},a_n})_{n \in\nat}$ are independent L\'evy processes with characteristic exponents $\psi_{a_{n+1},a_n}$. Then $X := \sum_{n=1}^{\infty} X^{a_{n+1},a_n}$ is a L\'evy process with exponent $\psi := \sum_{n=1}^{\infty} \psi_{a_{n+1},a_n}$ and c\`adl\`ag paths. Moreover, $X$ is an $L^2(\Pp)$-martingale.
\end{lemma}
\begin{proof}
    Lemmas~\ref{cons-03},~\ref{cons-05} show that $X^{a_{n+m},a_n} = \sum_{k=1}^m X^{a_{n+k},a_{n+k-1}}$ is a L\'evy process with characteristic exponent $\psi_{a_{n+m},a_n} = \sum_{k=1}^m \psi_{a_{n+k},a_{n+k-1}}$, and $X^{a_{n+m},a_n}$ is an $L^2(\Pp)$-martingale. By Doob's inequality and Lemma~\ref{cons-03}
    \begin{align*}
		\Ee \left( \sup_{s \leq t} |X_s^{a_{n+m},a_n}|^2 \right)
		&\leq 4\, \Ee\big(|X_t^{a_{n+m},a_n}|^2\big)\\
		&= 4t \int_{a_{n+m} \leq |y|<a_n} y^2 \, \nu(dy)
		\xrightarrow[m,n\to\infty]{\text{dom.\ convergence}} 0.
	\end{align*}
    Hence, the limit $X=\lim_{n \to \infty} X^{a_n,a_1}$ exists (uniformly in $t$) in $L^2$, i.e.\ $X$ is an $L^2(\Pp)$-martingale;
    since the convergence is also uniform in probability, Lemma~\ref{cons-07} shows that $X$ is a L\'evy process with exponent $\psi = \sum_{n=1}^\infty\psi_{a_{n+1},a_n}$. Taking a uniformly convergent subsequence, we also see that the limit inherits the c\`adl\`ag property from the approximating L\'evy processes $X^{a_n,a_1}$.
\end{proof}

We can now prove the main result of this chapter.
\begin{theorem}\label{cons-11}
    Let $(l,Q,\nu)$ be a L\'evy triplet and $\psi$ be given by \eqref{cons-e02}. Then there exists a L\'evy process $X$ with c\`adl\`ag paths and characteristic exponent $\psi$.
\end{theorem}
\begin{proof}
	Because of Lemma~\ref{cons-03},~\ref{cons-05} and~\ref{cons-09} we can construct $X$ piece by piece.
\begin{enumerate}
\item[$1^\circ$]
    Let $(W_t)_{t \geq 0}$ be a Brownian motion and set
    $$
        X_t^c := tl + \sqrt{Q} W_t
        \et
        \psi_c(\xi) := -\iup  l\cdot\xi + \frac{1}{2} \xi\cdot Q \xi.
    $$

\item[$2^\circ$]
    Write $\rd \setminus \{0\} = \bigdcup_{n=0}^{\infty} A_n$ with $A_0 := \{|y| \geq 1\}$ and $A_n := \left\{ \frac{1}{n+1} \leq |y| < \frac{1}{n} \right\}$; set $\lambda_n := \nu(A_n)$ and $\mu_n := \nu(\cdot \cap A_n)/\nu(A_n)$.

\item[$3^\circ$]
\index{compound Poisson process!building block of LP}%
    Construct, as in Example~\ref{exlp-05}.d), a compound Poisson process comprising the large jumps
    $$
        X_t^0 := X^{1,\infty}_t \et
        \psi_0(\xi)
        := \int_{1\leq |y| <\infty} \big[1-\eup^{\iup y\cdot \xi}\big] \, \nu(dy)
	$$
    and compensated compound Poisson processes taking account of all small jumps
	$$
        X_t^n := X^{a_{n+1},a_n}_t,\quad a_n := \frac 1n
        \et
        \psi_n(\xi) := \int_{A_n} \big[1-\eup^{\iup   y\cdot \xi}+\iup  y\cdot \xi\big] \, \nu(dy).
    $$
    We can construct the processes $X^n$ stochastically independent (just choose independent jump time processes and independent iid jump heights when constructing the compound Poisson processes) and independent of the Wiener process $W$.

\item[$4^\circ$]
    Setting $\psi = \psi_0+\psi_c + \sum_{n = 1}^\infty \psi_n$, Lemma~\ref{cons-05} and~\ref{cons-09} prove the theorem. Since all approximating processes have c\`adl\`ag paths, this property is inherited by the sums and the limit (Lemma~\ref{cons-09}).
\qedhere
\end{enumerate}
\end{proof}

The proof of Theorem~\ref{cons-11} also implies the following pathwise decomposition of a L\'evy process. We write $\Delta X_t := X_t-X_{t-}$ for the jump at time $t$. From the construction we know that
\begin{flalign}
\label{cons-e12}
&\text{(large jumps)} &    J_t^{[1,\infty)} &= \sum_{s\leq t}\Delta X_s \I_{[1,\infty)}(|\Delta X_s|)\\
\label{cons-e14}
&\text{(small jumps)} &    J_t^{[1/n,1)}    &= \sum_{s\leq t} \Delta X_s \I_{[1/n,1)}(|\Delta X_s|)\\
\label{cons-e16}
&\text{(compensated small jumps)} & \widetilde{J_t}^{[1/n,1)}    &= J_t^{[1/n,1)} - \Ee J_t^{[1/n,1)}\\
\notag
&&                                &= \sum_{s\leq t} \Delta X_s \I_{[1/n,1)}(|\Delta X_s|) - t\!\!\!\int\limits_{\frac 1n \leq |y| < 1} \!\!\!y\,\nu(dy).&&
\end{flalign}
are L\'evy processes and $J^{[1,\infty)}\bbot \widetilde{J}^{[1/n,1)}$.
\begin{corollary}\label{cons-13}
\index{L\'evy--It\^o decomposition}%
    Let $\psi$ be a characteristic exponent given by \eqref{cons-e02} and let $X$ be the L\'evy process constructed in Theorem~\ref{cons-11}. Then
    $$\begin{aligned}
        X_t
        &= \sqrt Q W_t \;+\: \lim_{n\to\infty}\bigg(\sum_{s\leq t} \Delta X_s \I_{[1/n,1)}(|\Delta X_s|) - t\!\!\!\!\int\limits_{[1/n,1)} \!\!\!\!y\,\nu(dy)\bigg)
        &&\pmb{\bigg]}\text{\ $=: M_t$, \ $L^2$-martingale}\\
        &\phantom{=}\underbracket{\;\,\mbox{}\quad tl\quad\, \vphantom{\sum_{s\leq t}\Delta X_s\I_{\{u\::\: |\Delta X_u|\geq 1\}}(s)}}_{\mathclap{\begin{gathered}\text{continuous}\\[-5pt]\text{Gaussian}\end{gathered}}}
        \;+ \;\; \underbracket{\quad\sum_{s\leq t}\Delta X_s\I_{[1,\infty)}(|\Delta X_s|).\qquad\qquad\qquad\qquad}_{\begin{gathered}\text{pure jump part}\end{gathered}}
        &&\pmb{\bigg]}\text{\ $=: A_t$, \ bdd.\ variation}
    \end{aligned}$$
    where all appearing processes are independent.
\end{corollary}
\begin{proof}
    The decomposition follows directly from the construction in Theorem~\ref{cons-11}. By Lemma~\ref{cons-09}, $\lim_{n\to\infty} X^{1/n,1}$ is an $L^2(\Pp)$-martingale, and since the (independent!) Wiener process $W$ is also an $L^2(\Pp)$-martingale, so is their sum $M$.

    The paths $t\mapsto A_t(\omega)$ are a.s.\ of bounded variation since, by construction, on any time-interval $[0,t]$ there are $N_t^0(\omega)$ jumps of size $\geq 1$. Since $N_t^0(\omega)<\infty$ a.s., the total variation of $A_t(\omega)$ is less or equal than $|l| t + \sum_{s\leq t} |\Delta X_s| \I_{[1,\infty)}(|\Delta X_s|) < \infty$ a.s.
\end{proof}

\begin{remark}\label{cons-21}
    A word of \emphh{caution}: Theorem~\ref{cons-11} associates with any $\psi$ given by the L\'evy--Khin\-tchi\-ne formula \eqref{cons-e02} a L\'evy process. Unless we know that all characteristic exponents are of this form (this was proved in Theorem~\ref{gen-21}), it does not follow that we have constructed \emphh{all} L\'evy processes.

    On the other hand, Theorem~\ref{cons-11} shows that the L\'evy triplet determining $\psi$ is \emphh{unique}. Indeed, assume that $(l,Q,\nu)$ and $(l',Q',\nu')$ are two L\'evy triplets which yield the same exponent $\psi$. Now we can associate, using Theorem~\ref{cons-11}, with each triplet a L\'evy process $X$ and $X'$ such that
    $$
        \Ee\,\eup^{\iup \xi\cdot X_t} = \eup^{-t\psi(\xi)} = \Ee\,\eup^{\iup \xi\cdot X_t'}.
    $$
    Thus, $X\sim X'$ and so these processes have (in law) the same pathwise decomposition, i.e.\ the same drift, diffusion and jump behaviour. This, however, means that $(l,Q,\nu)=(l',Q',\nu')$.
\end{remark}

\chapter{Two special L\'evy processes}\label{spe}
We will now study the structure of the paths of a L\'evy process. We begin with two extreme cases: L\'evy processes which only grow by jumps of size $1$ and L\'evy processes with continuous paths.

Throughout this chapter we assume that
\begin{centering}
    all paths $[0,\infty) \ni t\mapsto X_t(\omega)$ are right-continuous with finite left-hand limits (c\`adl\`ag).
\end{centering}
This is a bit stronger than \eqref{Lcont}, but it is always possible to construct a c\`adl\`ag version of a L\'evy process (see the discussion on page \pageref{levy-e04}). This allows us to consider the \emphh{jumps} of the process $X$
$$
    \Delta X_t := X_t - X_{t-} = X_t - \lim_{s\uparrow t}X_s.
$$

\begin{theorem}\label{spe-11}
\index{Poisson process!characterization among LP}
    Let $X$ be a one-dimensional L\'evy process which moves only by jumps of size $1$. Then $X$ is a Poisson process.
\end{theorem}
\begin{proof}
    Set $\Fscr_t^X := \sigma(X_s,\,s\leq t)$ and let $T_1 = \inf\{t>0\,:\, \Delta X_t = 1\}$ be the time of the first jump. Since $\{T_1 > t\} = \{X_t = 0\}\in \Fscr_t^X$, $T_1$ is a stopping time.

    Let $T_0=0$ and $T_k = \inf\{t>T_{k-1}\,:\, \Delta X_t = 1\}$, be the time of the $k$th jump; this is also a stopping time.
    By the Markov property (\eqref{lpmp-e10} and Lemma~\ref{lpmp-09}),
    \begin{align*}
        \Pp(T_1 > s+t)
        &= \Pp(T_1>s,\; T_1>s+t)\\
        &= \Ee\big[\I_{\{T_1>s\}} \Pp^{X_s}(T_1>t)\big]\\
        &= \Ee\big[\I_{\{T_1>s\}} \Pp^{0}(T_1>t)\big]\\
        &= \Pp(T_1>s)\Pp(T_1>t)
    \end{align*}
    where we use that $X_s = 0$ if $T_1>s$ (the process hasn't yet moved!) and $\Pp = \Pp^0$.

    Since $t\mapsto \Pp(T_1>t)$ is right-continuous, $\Pp(T_1>t)=\exp[t\log\Pp(T_1>1)]$ is the unique solution of this functional equation (Theorem~\ref{app-31}). Thus, the sequence of inter-jump times $\sigma_k := T_k-T_{k-1}$, $k\in\nat$, is an iid sequence of exponential times. This follows immediately from the strong Markov property (Theorem~\ref{lpmp-31}) for L\'evy processes and the observation that
    $$
        T_{k+1} - T_k = T_1^{Y}\quad\text{where}\quad Y = (Y_{t+T_k}-Y_{T_k})_{t\geq 0}
    $$
    and $T_1^Y$ is the first jump time of the process $Y$.

    Obviously, $X_t = \sum_{k=1}^\infty \I_{[0,t]}(T_k)$, $T_k = \sigma_1+\dots+\sigma_k$, and Example~\ref{exlp-05}.c) (and Theorem~\ref{exlp-13}) show that $X$ is a Poisson process.
\end{proof}

A L\'evy process with uniformly bounded jumps admits moments of all orders.
\begin{lemma}\label{spe-15}
\index{L\'evy process!moments}%
    Let $\Xt$ be a L\'evy process such that $|\Delta X_t(\omega)|\leq c$ for all $t\geq 0$ and some constant $c>0$. Then $\Ee(|X_t|^p) < \infty$ for all $p\geq 0$.
\end{lemma}
\begin{proof}
    Let $\Fscr_t^X := \sigma(X_s,\,s\leq t)$ and define the stopping times
    $$
        \tau_0 := 0,\quad \tau_n := \inf\left\{t>\tau_{n-1}\,:\, |X_t-X_{\tau_{n-1}}|\geq c\right\}.
    $$
    Since $X$ has c\`adl\`ag paths, $\tau_0<\tau_1<\tau_2<\dots$. Let us show that $\tau_1<\infty$ a.s. For fixed $t>0$ and $n\in\nat$ we have
    \begin{align*}
        \Pp(\tau_1=\infty)
        \leq \Pp(\tau_1\geq nt)
        &\pomu{\eqref{Lindep}}{\leq}{\eqref{LindepAlt}} \Pp(|X_{kt}-X_{(k-1)t}|\leq 2c,\;\;\forall k=1,\dots,n)\\
        &\omu{\eqref{Lindep}}{=}{\eqref{LindepAlt}} \prod_{k=1}^n \Pp(|X_{kt}-X_{(k-1)t}|\leq 2c)
        \omu{\eqref{Lstat}}{=}{\phantom{\eqref{LindepAlt}}} \Pp(|X_t|\leq 2c)^n.
    \end{align*}
    Letting $n\to\infty$ we see that $\Pp(\tau_1=\infty)=0$ if $\Pp(|X_t|\leq 2c)<1$ for some $t>0$. (In the alternative case, we have $\Pp(|X_t|\leq 2c)=1$ for all $t>0$ which makes the lemma trivial.)

    By the strong Markov property (Theorem~\ref{lpmp-31}) $\tau_n - \tau_{n-1}\sim \tau_1$ and $\tau_n-\tau_{n-1}\bbot\Fscr_{\tau_{n-1}}^X$, i.e.\ $(\tau_n-\tau_{n-1})_{n\in\nat}$ is an iid sequence. Therefore,
    $$
        \Ee\,\eup^{-\tau_n} = \left(\Ee\,\eup^{-\tau_1}\right)^n = q^n
    $$
    for some $q\in [0,1)$. From the very definition of the stoppping times we infer
    $$
        |X_{t\wedge\tau_n}|
        \leq \sum_{k=1}^n |X_{\tau_k}-X_{\tau_{k-1}}|
        \leq \sum_{k=1}^n \big(\underbrace{|\Delta X_{\tau_k}|}_{\leq c} + \underbrace{|X_{\tau_k-}-X_{\tau_{k-1}}|}_{\leq c}\big)
        \leq 2nc.
    $$
    Thus, $|X_t|>2nc$ implies that $\tau_n<t$, and by Markov's inequality
    $$
        \Pp(|X_t|>2nc)
        \leq \Pp(\tau_n < t)
        \leq \eup^t\Ee\,\eup^{-\tau_n}
        = \eup^t \,q^n.
    $$
    Finally,
    \begin{align*}
        \Ee\big(|X_t|^p\big)
        &= \sum_{n=0}^\infty \Ee\big(|X_t|^p \I_{\{ 2nc< |X_t|\leq 2(n+1)c\}}\big)\\
        &\leq (2c)^p\sum_{n=0}^\infty (n+1)^p\Pp(|X_t|> 2nc)
        \leq (2c)^p\eup^t\sum_{n=0}^\infty (n+1)^p q^n
        < \infty.
    \qedhere
    \end{align*}
\end{proof}

Recall that a Brownian motion $(W_t)_{t\geq 0}$ on $\rd$ is a L\'evy process such that $W_t$ is a normal random variable with mean $0$ and covariance matrix $t\id$. We will need Paul L\'evy's characterization of Brownian motion which we state without proof. An elementary proof can be found in \cite[Chapter 9.4]{schilling-bm}.
\begin{theorem}[L\'evy]\label{spe-17}
\index{Brownian motion!martingale characterization}%
\index{L\'evy's theorem}%
    Let $M=(M_t,\Fscr_t)$, $M_0=0$, be a one-dimensional martingale with continuous sample paths such that $(M_t^2-t,\Fscr_t)_{t\geq 0}$ is also a martingale. Then $M$ is a one-dimensional standard Brownian motion.
\end{theorem}

\begin{theorem}\label{spe-21}
\index{Brownian motion!characterization among LP}%
    Let $\Xt$ be a L\'evy process in $\rd$ whose sample paths are a.s.\ continuous. Then $X_t \sim tl + \sqrt{Q}W_t$ where $l\in\rd$, $Q$ is a positive semidefinite symmetric matrix, and $W$ is a standard Brownian motion in $\rd$.
\end{theorem}
We will give two proofs of this result.
\begin{proof}[Proof \textup{(}using Theorem~\ref{spe-17}\textup{)}]
    By Lemma~\ref{spe-15}, the process $\Xt$ has moments of all orders. Therefore, $M_t := M_t^\xi := \xi\cdot(X_t - \Ee X_t)$ exists for any $\xi\in\rd$ and is a martingale for the canonical filtration $\Fscr_t := \sigma(X_s,\,s\leq t)$. Indeed, for all $s\leq t$
    $$
        \Ee(M_t\mid\Fscr_s)
        =\Ee(M_t-M_s\mid\Fscr_s) - M_s
        \omu{\eqref{Lindep}}{=}{\eqref{Lstat}} \Ee M_{t-s} + M_s = M_s.
    $$
    Moreover
    \begin{align*}
        \Ee(M_t^2-M_s^2 \mid \Fscr_s)
        &\pomu{\text{Lemma}~\ref{exlp-51}}{=}{} \Ee((M_t-M_s)^2 +2M_s(M_t-M_s) \mid \Fscr_s)\\
        &\omu{\eqref{Lindep}}{=}{\phantom{\text{Lemma}~\ref{exlp-51}}} \Ee((M_t-M_s)^2) + 2M_s \Ee(M_t-M_s)\\
        &\omu{\text{Lemma}~\ref{exlp-51}}{=}{} (t-s)\Ee M_1^2
        = (t-s)\Vv M_1,
    \end{align*}
    and so $(M_t^2 - t\Vv M_1)_{t\geq 0}$ and $(M_t)_{t\geq 0}$ are martingales with continuous paths.

    Now we can use Theorem~\ref{spe-17} and deduce that $\xi\cdot(X_t - \Ee X_t)$ is a one-dimensional Brownian motion with variance $\xi\cdot Q\xi$ where $tQ$ is the covariance matrix of $X_t$ (cf.\ the proof of Lemma~\ref{cons-03} or Lemma~\ref{exlp-51}). Thus, $X_t - \Ee X_t = \sqrt Q W_t$ where $W_t$ is a $d$-dimensional standard Brownian motion. Finally, $\Ee X_t = t\Ee X_1 =: tl$.
\end{proof}

\begin{proof}[Standard proof \textup{(}using the CLT\textup{)}]
    Fix $\xi\in\rd$ and set $M(t) := \xi\cdot(X_t-\Ee X_t)$. Since $X$ has moments of all orders, $M$ is well-defined and it is again a L\'evy process. Moreover,
    $$
        \Ee M(t) = 0
        \et
        t\sigma^2 = \Vv M(t) = \Ee [(\xi\cdot(X_t-\Ee X_t))^2] = t\xi\cdot Q\xi
    $$
    where $Q$ is the covariance matrix of $X$, cf.\ the proof of Lemma~\ref{cons-03}. We proceed as in the proof of the CLT: Using a Taylor expansion we get
    \begin{align*}
        \Ee\,\eup^{\iup M(t)}
        = \Ee\,\eup^{\iup \sum_{k=1}^n \left[M(tk/n)-M(t(k-1)/n)\right]}
        \omu{\eqref{LindepAlt}}{=}{\eqref{Lstat}} \left(\Ee\,\eup^{\iup M(t/n)}\right)^n
        = \left(1-\frac 12\,\Ee M^2(t/n) + R_n\right)^n.
    \end{align*}
    The remainder term $R_n$ is estimated by $\frac 16 \Ee|M^3(\tfrac tn)|$. If we can show that $|R_n| \leq \epsilon\frac tn$ for large $n=n(\epsilon)$ and any $\epsilon>0$, we get because of $\Ee M^2(\tfrac tn) = \frac tn\sigma^2$
    \begin{align*}
        \Ee\,\eup^{\iup M(t)}
        = \lim_{n\to\infty} \left(1- \frac 12 (\sigma^2 + 2\epsilon)\frac tn\right)^n
        = \eup^{-\frac 12(\sigma^2+2\epsilon) t}
        \xrightarrow[\;\epsilon\to 0\;]{} \eup^{-\frac 12\sigma^2t}.
    \end{align*}
    This shows that $\xi\cdot (X_t-\Ee X_t)$ is a centered Gaussian random variable with variance $\sigma^2$. Since $\Ee X_t = t\Ee X_1$ we conclude that $X_t$ is Gaussian with mean $tl$ and covariance $tQ$.

    We will now estimate $\Ee|M^3(\frac tn)|$. For every $\epsilon>0$ we can use the uniform continuity of $s\mapsto M(s)$ on $[0,t]$ to get
    $$
        \lim_{n\to\infty} \max_{1\leq k\leq n} |M(\tfrac kn t) - M(\tfrac{k-1}{n}t)| = 0.
    $$
    Thus, we have for all $\epsilon>0$
    \begin{align*}
        1
        &\pomu{\eqref{Lindep}}{=}{\eqref{Lstat}}
            \lim_{n\to\infty} \Pp\Big( \max_{1\leq k\leq n} |M(\tfrac kn t) - M(\tfrac{k-1}{n}t)| \leq \epsilon\Big)\\
        &\pomu{\eqref{Lindep}}{=}{\eqref{Lstat}}
            \lim_{n\to\infty} \Pp\bigg( \bigcap_{k=1}^{n} \big\{|M(\tfrac kn t) - M(\tfrac{k-1}{n}t)| \leq \epsilon\big\}\bigg)\\
        &\omu{\eqref{Lindep}}{=}{\eqref{Lstat}}
            \lim_{n\to\infty} \prod_{k=1}^{n}\Pp(|M(\tfrac tn)|\leq \epsilon)\\
        &\pomu{\eqref{Lindep}}{=}{\eqref{Lstat}}
            \lim_{n\to\infty}\big[1-\Pp(|M(\tfrac tn)|> \epsilon)\big]^n\\
        &\pomu{\eqref{Lindep}}{\leq}{\eqref{Lstat}}
            \lim_{n\to\infty} \eup^{-n\,\Pp(|M(t/n)|> \epsilon)}
        \;\leq\; 1
    \end{align*}
    where we use the inequality $1+x\leq \eup^x$. This proves $\lim_{n\to\infty} n\,\Pp(|M(t/n)|> \epsilon)=0$.
    Therefore,
    \begin{align*}
        \Ee |M^3(\tfrac tn)|
        &\leq \epsilon \Ee M^2(\tfrac tn) + \int_{|M(t/n)| > \epsilon} |M^3(\tfrac tn)|\,d\Pp\\
        &\leq \epsilon \frac tn\,\sigma^2 + \sqrt{\Pp(|M(\tfrac tn)| > \epsilon)\vphantom{M^6(\tfrac tn)}} \sqrt{\Ee M^6(\tfrac tn)}.
    \end{align*}
    It is not hard to see that $\Ee M^6(s) = a_1s +\dots + a_6 s^6$ (differentiate $\Ee \eup^{\iup u M(s)} = \eup^{-s \psi(u)}$ six times at $u=0$), and so
    \begin{gather*}
        \Ee |M^3(\tfrac tn)|
        \leq \epsilon \frac tn\,\sigma^2 + c\frac tn \sqrt{\frac{\Pp(|M(t/n)| > \epsilon)}{t/n}}
        = \frac tn \big(\epsilon\sigma^2 + o(1)\big)
    \qedhere
    \end{gather*}
\end{proof}

\index{Brownian motion!construction|(}%
We close this chapter with Paul L\'evy's construction of a standard Brownian motion $(W_t)_{t\geq 0}$. Since $W$ is a L\'evy process which has the Markov property, it is enough to construct a Brownian motion $W(t)$ only for $t\in [0,1]$, then produce independent copies $(W^n(t))_{t\in [0,1]}, n=0,1,2,\dots$, and join them continuously:
$$
    W_t :=
    \begin{cases}
        W^0(t), & t\in [0,1),\\
        W^0(1)+\dots + W^{n-1}(1) + W^n(t-n), & t\in [n,n+1).
    \end{cases}
$$
Since each $W_t$ is normally distributed with mean $0$ and variance $t$, we will get a L\'evy process with characteristic exponent $\frac 12\xi^2$, $\xi\in\real$.
In the same vein we get a $d$-dimensional Brownian motion by making a vector $(W_t^{(1)}, \dots, W^{(d)}_t)_{t\geq 0}$ of $d$ independent copies of $(W_t)_{t\geq 0}$. This yields a L\'evy process with exponent $\frac 12(\xi_1^2+\dots+\xi_d^2)$, $\xi_1,\dots,\xi_d \in \real$.

Denote a one-dimensional normal distribution with mean $m$ and variance $\sigma^2$ as $\mathsf{N}(m,\sigma^2)$. The motivation for the construction is the observation that a Brownian motion satisfies the following mid-point law (cf.\ \cite[Chapter 3.4]{schilling-bm}):
$$
    \Pp(W_{(s+t)/2}\in \bullet \mid W_s = x, W_t = y) = \mathsf N\big(\tfrac 12(x+y),\: \tfrac14(t-s)\big),\quad s\leq t,\; x,y\in\real.
$$
This can be turned into the following construction method:
\begin{algorithm}
    Set $W(0)=0$ and let $W(1)\sim \Normal(0,1)$. Let $n\geq 1$ and assume that the random variables $W(k2^{-n})$, $k=1,\ldots, 2^n-1$ have already been constructed. Then
    $$
        W(l{2^{-n-1}})
        := \begin{cases}
            W(k2^{-n}), & l = 2k,\\
            \frac 12 \big(W(k2^{-n}) + W((k+1)2^{-n})\big) + \Gamma_{2^n+k}, & l = 2k+1,
        \end{cases}
    $$
    where $\Gamma_{2^n+k}$ is an independent (of everything else) $\mathsf{N}(0,2^{-n}/4)$ Gaussia random variable, cf.\ Figure~\ref{fig-bm-construction}.
    In-between the nodes we use piecewise linear interpolation:
    $$
        W_{2^n}(t,\omega)
        := \text{Linear interpolation of\ }\big(W(k2^{-n},\omega),\; k=0,1,\ldots, 2^n\big),
        \quad n\geq 1.
    $$
\end{algorithm}

At the dyadic points $t=k2^{-j}$ we get the `true' value of $W(t,\omega)$, while the linear interpolation is an approximation, see Figure~\ref{fig-bm-construction}.
\begin{figure}
    \includegraphics[height = 0.3\textheight]{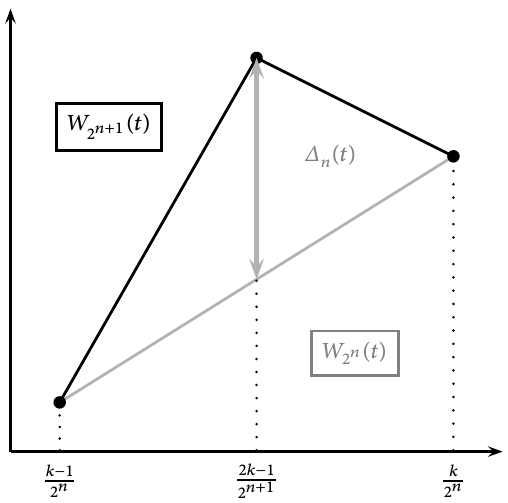}\hfill
    \includegraphics[height = 0.3\textheight]{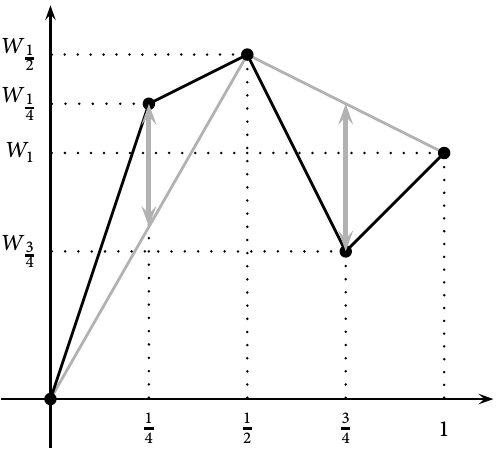}
    \caption{Interpolation of order four in L\'evy's construction of Brownian motion.}\label{fig-bm-construction}
\end{figure}

\begin{theorem}[L\'evy 1940]\label{spe-51}
The series
$$
    W(t,\omega) := \sum_{n=0}^\infty \big(W_{2^{n+1}}(t,\omega) - W_{2^{n}}(t,\omega)\big) + W_1(t,\omega),\quad t\in [0,1],
$$
converges a.s.\ uniformly. In particular $(W(t))_{t\in [0,1]}$ is a one-dimensional Brownian motion.
\end{theorem}
\begin{proof}
Set $\Delta_n(t,\omega):=W_{2^{n+1}}(t,\omega) - W_{2^{n}}(t,\omega)$. By construction,
$$
            \Delta_n\big((2k-1)2^{-n-1},\omega)
            = \Gamma_{2^n+(k-1)}(\omega),
            \quad k=1,2,\ldots, 2^{n},
$$
are iid $\Normal(0,2^{-(n+2)})$ distributed random variables. Therefore,
$$
            \Pp\left(\max_{1\leq k\leq 2^{n}} \big|\Delta_n\big((2k-1)2^{-n-1}\big)\big| > \frac{x_n}{\sqrt{2^{n+2}}}\right)
            \leq 2^{n} \Pp\left(\big|\sqrt{2^{n+2}}\, \Delta_n\big(2^{-n-1}\big)\big| > x_n\right),
$$
and the right-hand side equals
$$
            \frac{2\cdot 2^{n}}{\sqrt{2\pi}} \int_{x_n}^\infty \eup^{-r^2/2}\,dr
            \leq \frac{2^{n+1}}{\sqrt{2\pi}} \int_{x_n}^\infty  \frac{r}{x_n}\,\eup^{-r^2/2}\,dr
            = \frac{2^{n+1}}{x_n\,\sqrt{2\pi}}\,\eup^{-x_n^2/2}.
$$
Choose $c>1$ and $x_n := c\sqrt{2n\log 2}$. Then
\begin{align*}
            \sum_{n=1}^\infty \Pp\left(\max_{1\leq k\leq 2^{n}} \big|\Delta_n\big((2k-1)2^{-n-1}\big)\big| > \frac{x_n}{\sqrt{2^{n+2}}}\right)
            &\leq \sum_{n=1}^\infty \frac{2^{n+1}}{c\sqrt{2\pi}}\,\eup^{-c^2\log 2^n}\\
            &= \frac{2}{c\sqrt{2\pi}}\sum_{n=1}^\infty 2^{-(c^2-1)n}
            < \infty.
\end{align*}
Using the Borel--Cantelli lemma we find a set $\Omega_0\subset\Omega$ with $\Pp(\Omega_0)=1$ such that for every $\omega\in\Omega_0$ there is some $N(\omega)\geq 1$ with
$$
            \max_{1\leq k\leq 2^{n}}\big|\Delta_n\big((2k-1)2^{-n-1}\big)\big|
            \leq c\sqrt{\frac{n\log 2}{2^{n+1}}}
            \fa n\geq N(\omega).
$$
$\Delta_n(t)$ is the distance between the polygonal arcs $W_{2^{n+1}}(t)$ and $W_{2^n}(t)$; the maximum is attained at one of the midpoints of the intervals $[(k-1)2^{-n},k2^{-n}]$, $k=1,\ldots,2^n$, see Figure~\ref{fig-bm-construction}. Thus,
\begin{align*}
    \sup_{0\leq t\leq 1}\big|W_{2^{n+1}}(t,\omega) - W_{2^{n}}(t,\omega)\big|
            \leq \max_{1\leq k\leq 2^{n}}\big|\Delta_n\big((2k-1)2^{-n-1},\omega\big)\big|
            \leq c\sqrt{\frac{n\log 2}{2^{n+1}}},
\end{align*}
for all $n\geq N(\omega)$ which means that the limit
$$
            W(t,\omega) := \lim_{N\to\infty} W_{2^N}(t,\omega) = \sum_{n=0}^\infty \big(W_{2^{n+1}}(t,\omega) - W_{2^{n}}(t,\omega)\big) + W_1(t,\omega)
$$
exists for all $\omega\in\Omega_0$ uniformly in $t\in [0,1]$. Therefore, $t\mapsto W(t,\omega)$, $\omega\in\Omega_0$, inherits the continuity of the polygonal arcs $t\mapsto W_{2^n}(t,\omega)$. Set
$$
            \widetilde W(t,\omega) := W(t,\omega)\I_{\Omega_0}(\omega).
$$
By construction, we find for all $0\leq k\leq l\leq 2^n$
\begin{align*}
            \widetilde W(l2^{-n}) - \widetilde W(k2^{-n})
            &\pomu{\text{iid}}{=}{} W_{2^n}(l2^{-n}) - W_{2^n}(k2^{-n})\\
            &\pomu{\text{iid}}{=}{} \sum_{l= k+1}^l \big(W_{2^n}(l 2^{-n}) - W_{2^n}((l-1)2^{-n})\big)\\
            &\omu{\text{iid}}{\sim}{} \Normal(0,(l-k)2^{-n}).
\end{align*}
Since $t\mapsto \widetilde W(t)$ is continuous and the dyadic numbers are dense in $[0,t]$, we conclude that the increments $\widetilde W(t_k)-\widetilde W(t_{k-1})$, $0= t_0 < t_1 < \cdots < t_N \leq 1$ are independent $\Normal(0,t_k-t_{k-1})$ distributed random variables. This shows that $(\widetilde W(t))_{t\in [0,1]}$ is a Brownian motion.
\end{proof}
\index{Brownian motion!construction|)}%

\chapter{Random measures}\label{rm}

We continue our investigations of the paths of c\`adl\`ag L\'evy processes. Independently of Chapters~\ref{semi} and~\ref{gen} we will show in Theorem~\ref{rm-51} that the processes constructed in Theorem~\ref{cons-11} are indeed \emphh{all} L\'evy processes; this gives also a new proof of the L\'evy--Khintchine formula, cf.\ Corollary~\ref{rm-53}. As before, we denote the jumps of $\Xt$ by
$$
    \Delta X_t := X_t - X_{t-} = X_t - \lim_{s\uparrow t}X_s.
$$
\begin{definition}\label{rm-03}
\index{jump measure}%
\index{L\'evy process!jump measure}%
    Let $X$ be a L\'evy process. The counting measure
    \begin{equation}\label{rm-e04}
        N_t(B,\omega) := \#\left\{s\in (0,t] \,:\, \Delta X_s(\omega)\in B\right\},\quad B\in\Bscr(\rd\setminus\{0\})
    \end{equation}
    is called the \emphh{jump measure} of the process $X$.
\end{definition}
Since a c\`adl\`ag function $x:[0,\infty)\to\rd$ has on any compact interval $[a,b]$ at most finitely many jumps $|\Delta x_t|>\epsilon$ exceeding a fixed size,\footnote{Otherwise we would get an accumulation point of jumps within $[a,b]$, and $x$ would be unbounded on $[a,b]$.} we see that
$$
    N_t(B,\omega)<\infty\quad\forall t>0,\; B\in\Bscr(\rd)\text{\ \ such that\ \ } 0\notin \overline{B}.
$$
Notice that $0\notin\overline{B}$ is equivalent to $B_\epsilon(0)\cap B=\emptyset$ for some $\epsilon>0$. Thus, $B\mapsto N_t(B,\omega)$ is for every $\omega$ a locally finite Borel measure on $\rd\setminus\{0\}$.
\begin{definition}\label{rm-05}
    Let $N_t(B)$ be the jump measure of the L\'evy process $X$. For every Borel function $f:\rd\to\real$ with $0\notin\supp f$ we define
    \begin{equation}\label{rm-e06}
        N_t(f,\omega) := \int f(y)\,N_t(dy,\omega).
    \end{equation}
\end{definition}
Since $0\notin\supp f$, it is clear that $N_t(\supp f,\omega)<\infty$, and for every $\omega$
\begin{gather}\label{rm-e06-alt}\tag{$\ref{rm-e06}'$}
    N_t(f,\omega) = \sum_{0<s\leq t} f(\Delta X_s(\omega)) = \sum_{n=1}^\infty f(\Delta X_{\tau_n}(\omega))\I_{(0,t]}(\tau_n(\omega)).
\end{gather}
Both sums are finite sums, extending only over those $s$ where $\Delta X_s(\omega)\neq 0$. This is obvious in the second sum where $\tau_1(\omega), \tau_2(\omega), \tau_3(\omega) \dots$ are the jump times of $X$.

\begin{lemma}\label{rm-11}
    Let $N_t(\cdot)$ be the jump measure of a L\'evy process $X$, $f\in\cont_c(\rd\setminus\{0\},\real^m)$ \textup{(}i.e.\ $f$ takes values in $\real^m$\textup{)}, $s<t$ and $t_{k,n}:= s+ \frac kn(t-s)$. Then
    \begin{equation}\label{rm-e12}
        N_t(f,\omega)-N_s(f,\omega)
        = \sum_{s<u\leq t} f\big(\Delta X_u(\omega)\big)
        = \lim_{n\to\infty} \sum_{k=0}^{n-1} f\big(X_{t_{k+1,n}}(\omega) - X_{t_{k,n}}(\omega)\big).
    \end{equation}
\end{lemma}
\begin{proof}
    Throughout the proof $\omega$ is fixed and we will omit it in our notation. Since $0\notin\supp f$, there is some $\epsilon>0$ such that $B_\epsilon(0)\cap\supp f=\emptyset$; therefore, we need only consider jumps of size $|\Delta X_t|\geq \epsilon$. Denote by $J=\{\tau_1, \dots, \tau_N\}$ those jumps. For sufficiently large $n$ we can achieve that
    \begin{itemize}
    \item $\#\big(J\cap (t_{k,n},t_{k+1,n}]\big) \leq 1$ for all $k=0,\dots, n-1$;
    \item $|X_{t_{\kappa+1,n}} - X_{t_{\kappa,n}}|<\epsilon$ if $\kappa$ is such that $J\cap (t_{\kappa,n},t_{\kappa+1,n}] = \emptyset$.

        \noindent
        \emph{Indeed:} Assume this is not the case, then we could find sequences $s < s_k<t_k\leq t$ such that $t_k-s_k\to 0$, $J\cap (s_k,t_k]=\emptyset$ and $|X_{t_k}-X_{s_k}|\geq\epsilon$. Without loss of generality we may assume that $s_k\uparrow u$ and $t_k\downarrow u$ for some $u\in (s,t]$; $u=s$ can be ruled out because of right-continuity. By the c\`adl\`ag property of the paths, $|\Delta X_u|\geq\epsilon$, i.e.\ $u\in J$, which is a contradiction.
    \end{itemize}
    Since we have $f(X_{t_{\kappa+1,n}} - X_{t_{\kappa,n}})=0$ for intervals of the `second kind', only the intervals containing some jump contribute to the (finite!) sum \eqref{rm-e12}, and the claim follows.
\end{proof}

\begin{lemma}\label{rm-13}
\index{Poisson process!jump measure of LP}
    Let $N_t(\cdot)$ be the jump measure of a L\'evy process $X$.
    \begin{enumerate}
    \item[\upshape a)] $(N_t(f))_{t\geq 0}$ is a L\'evy process on $\real^m$ for all $f\in\cont_c(\rd\setminus\{0\},\real^m)$.
    \item[\upshape b)] $(N_t(B))_{t\geq 0}$ is a Poisson process for all $B\in\Bscr(\rd)$ such that $0\notin\overline{B}$.
    \item[\upshape c)] $\nu(B) := \Ee N_1(B)$ is a locally finite measure on $\rd\setminus\{0\}$.
    \end{enumerate}
\end{lemma}
\begin{proof}
    Set $\Fscr_t := \sigma(X_s,\,s\leq t)$.

\smallskip\noindent
    a)\ \
    Let $f\in\cont_c(\rd\setminus\{0\},\real^m)$. From Lemma~\ref{rm-11} and \eqref{LindepAlt} we see that $N_t(f)$ is $\Fscr_t$ measurable and $N_t(f)-N_s(f)\bbot\Fscr_s$, $s\leq t$. Moreover, if $N_t^Y(\cdot)$ denotes the jump measure of the L\'evy process $Y = (X_{t+s}-X_s)_{t\geq 0}$, we see that $N_t(f)-N_s(f)=N_{t-s}^Y(f)$. By the Markov property (Theorem~\ref{lpmp-13}), $X\sim Y$, and we get $N^Y_{t-s}(f)\sim N_{t-s}(f)$. Since $t\mapsto N_t(f)$ is c\`adl\`ag, $(N_t(f))_{t\geq 0}$ is a L\'evy process.

\medskip\noindent
    b)\ \
    By definition, $N_0(B)=0$ and $t\mapsto N_t(B)$ is c\`adl\`ag. Since $X$ is a L\'evy process, we see as in the proof of Theorem~\ref{spe-11} that the jump times
    \begin{align*}
        \tau_0 := 0,\quad
        \tau_1 := \inf\big\{t>0 : \Delta X_t \in B\big\},\quad
        \tau_k := \inf\big\{t>\tau_{k-1} : \Delta X_t \in B\big\}
    \end{align*}
    satisfy $\tau_1 \sim \Exp(\nu(B))$, and the inter-jump times $(\tau_k-\tau_{k-1})_{k\in\nat}$ are an iid sequence. The condition $0\notin\overline{B}$ ensures that $N_t(B)<\infty$ a.s., which means that the intensity $\nu(B)$ is finite. Indeed, we have
    $$
        1-\eup^{-t\nu(B)}
        = \Pp(\tau_1 \leq t)
        = \Pp(N_t(B) > 0)
        \xrightarrow[t\to 0]{} 0;
    $$
    this shows that $\nu(B)<\infty$. Thus,
    $$
        N_t(B) = \sum_{k=1}^\infty \I_{(0,t]}(\tau_k)
    $$
    is a Poisson process (Example~\ref{exlp-05}) and, in particular, a L\'evy process (Theorem~\ref{exlp-13}).

\medskip\noindent
    c)\ \
    The intensity of $(N_t(B))_{t\geq 0}$ is $\nu(B) = \Ee N_1(B)$. By Fubini's theorem it is clear that $\nu$ is a measure.
\end{proof}

\begin{definition}\label{rm-17}
    Let $N_t(\cdot)$ be the jump measure of a L\'evy process $X$. The \emphh{intensity measure} is the measure $\nu(B):=\Ee N_1(B)$ from Lemma~\ref{rm-13}.
\index{intensity measure}%
\index{L\'evy process!intensity measure}%
\end{definition}
We will see in Corollary~\ref{rm-53} that $\nu$ is the L\'evy measure of $\Xt$ appearing in the L\'evy--Khintchine formula.

\begin{lemma}\label{rm-19}
    Let $N_t(\cdot)$ be the jump measure of a L\'evy process $X$ and $\nu$ the intensity measure. For every $f\in L^1(\nu)$, $f:\rd\to\real^m$, the random variable $N_t(f) := \int f(y)\,N_t(dy)$ exists as $L^1$-limit of integrals of simple functions and satisfies
    \begin{equation}\label{rm-e18}
        \Ee N_t(f) = \Ee \int f(y)\,N_t(dy) = t\int_{y\neq 0} f(y)\,\nu(dy).
    \end{equation}
\end{lemma}
\begin{proof}
    For any step function $f$ of the form $f(y) = \sum_{k=1}^M \phi_k\I_{B_k}(y)$ with $0\notin\overline{B}_k$ the formula \eqref{rm-e18} follows from $\Ee N_t(B_k) = t\nu(B_k)$ and the linearity of the integral.

    Since $\nu$ is defined on $\rd\setminus\{0\}$, any $f\in L^1(\nu)$ can be approximated by a sequence of step functions $(f_n)_{n\in\nat}$ in $L^1(\nu)$-sense, and we get
    $$
        \Ee|N_t(f_n)-N_t(f_m)| \leq t\int |f_n-f_m|\,d\nu \xrightarrow[m,n\to\infty]{}0.
    $$
    Because of the completeness of $L^1(\Pp)$, the limit $\lim_{n\to\infty} N_t(f_n)$ exists, and with a routine argument we see that it is independent of the approximating sequence $f_n \to f\in L^1(\nu)$. This allows us to define $N_t(f)$ for $f\in L^1(\nu)$ as $L^1(\Pp)$-limit of stochastic integrals of simple functions; obviously, \eqref{rm-e18} is preserved under this limiting procedure.
\end{proof}

\begin{theorem}\label{rm-21}
    Let $N_t(\cdot)$ be the jump measure of a L\'evy process $X$ and $\nu$ the intensity measure.
    \begin{enumerate}
    \item[\upshape a)]
    $N_t(f) := \int f(y)\,N_t(dy)$ is a L\'evy process for every $f\in L^1(\nu)$, $f:\rd\to\real^m$. In particular, $(N_t(B))_{t\geq 0}$ is a Poisson process for every $B\in\Bscr(\rd)$ such that $0\notin\overline{B}$.

    \item[\upshape b)]
    $X^B_t := N_t(y\I_B(y))$ and $X_t-X_t^B$ are for every $B\in\Bscr(\rd)$, $0\notin\overline{B}$, L\'evy processes.
    \end{enumerate}
\end{theorem}
\begin{proof}
    a)\ \
    Note that $\nu$ is a locally finite measure on $\rd\setminus\{0\}$. This means that, by standard density results from integration theory, the family $\cont_c(\rd\setminus\{0\})$ is dense in $L^1(\nu)$. Fix $f\in L^1(\nu)$ and choose $f_n\in\cont_c(\rd\setminus\{0\})$ such that $f_n\to f$ in $L^1(\nu)$. Then, as in Lemma~\ref{rm-19},
    $$
        \Ee|N_t(f)-N_t(f_n)| \leq t\int |f-f_n|\,d\nu \xrightarrow[n\to\infty]{}0.
    $$
    Since $\Pp(|N_t(f)|>\epsilon)\leq \frac t\epsilon\int |f|\,d\nu\to 0$ for every $\epsilon>0$ as $t\to 0$, the process $N_t(f)$ is continuous in probability. Moreover, it is the limit (in $L^1$, hence in probability) of the L\'evy processes $N_t(f_n)$ (Lemma~\ref{rm-13}); therefore it is itself a L\'evy process, see Lemma~\ref{cons-07}.

    In view of Lemma~\ref{rm-13}.c), the indicator function $\I_{B}\in L^1(\nu)$ whenever $B$ is a Borel set satisfying $0\notin\overline{B}$. Thus, $N_t(B)=N_t(\I_B)$ is a L\'evy process which has only jumps of unit size, i.e.\ it is by Theorem~\ref{spe-11} a Poisson process.

\medskip\noindent
    b)\ \
    Set $f(y):=y\I_{B}(y)$ and $B_n := B\cap B_n(0)$. Then $f_n(y) = y\I_{B_n}(y)$ is bounded and $0\notin\supp f_n$, hence $f_n\in L^1(\nu)$. This means that $N_t(f_n)$ is for every $n\in\nat$ a L\'evy process. Moreover,
    $$
        N_t(f_n) = \int_{B_n} y\,N_t(dy) \xrightarrow[n\to\infty]{} \int_{B} y\,N_t(dy) = N_t(f)\quad\text{a.s.}
    $$
    Since $N_t(f)$ changes its value only by jumps,
    \begin{align*}
        \Pp(|N_t(f)|>\epsilon)
        &\leq \Pp(X\text{\ has at least one jump of size $B$ in $[0,t]$})\\
        &= \Pp(N_t(B)>0)
        = 1-\eup^{-t\nu(B)},
    \end{align*}
    which proves that the process $N_t(f)$ is continuous in probability. Lemma~\ref{cons-07} shows that $N_t(f)$ is a L\'evy process.

    Finally, approximate $f(y):= y\I_B(y)$ by a sequence $\phi_l\in\cont_c(\rd\setminus\{0\},\rd)$. Now we can use Lemma~\ref{rm-11} to get
    $$
        X_t - N_t(\phi_l)
        = \lim_{n\to\infty} \sum_{k=0}^{n-1} \big[(X_{t_{k+1,n}} - X_{t_{k,n}}) - \phi_l(X_{t_{k+1,n}} - X_{t_{k,n}})\big],
    $$
    The increments of $X$ are stationary and independent, and so we conclude from the above formula that $X-N(\phi_l)$ has also stationary and independent increments. Since both $X$ and $N(\phi_l)$ are continuous in probability, so is their difference, i.e.\ $X-N(\phi_l)$ is a L\'evy process. Finally,
    $$
        N_t(\phi_l) \xrightarrow[l\to\infty]{} N_t(f)
        \et
        X_t - N_t(\phi_l) \xrightarrow[l\to\infty]{} X_t - N_t(f),
    $$
    and since $X$ and $N(f)$ are continuous in probability, Lemma~\ref{cons-07} tells us that $X-N(f)$ is a L\'evy process.
\end{proof}

We will now show that L\'evy processes with `disjoint jump heights' are independent. For this we need the following immediate consequence of Theorem~\ref{exlp-03}:
\begin{lemma}[Exponential martingale]\label{rm-31}
\index{exponential martingale}%
    Let $\Xt$ be a L\'evy process. Then
    $$
        M_t := \frac{\eup^{\iup \xi\cdot X_t}}{\Ee\,\eup^{\iup \xi\cdot X_t}} = \eup^{\iup \xi\cdot X_t} \eup^{t\psi(\xi)},\quad t\geq 0,
    $$
    is a martingale for the filtration $\Fscr_t^X = \sigma(X_s,\,s\leq t)$ such that $\sup_{s\leq t}|M_s|\leq \eup^{t\Re\psi(\xi)}$.
\end{lemma}

\begin{theorem}\label{rm-33}
\index{L\'evy process!independent components}%
    Let $N_t(\cdot)$ be the jump measure of a L\'evy process $X$ and $U,V\in\Bscr(\rd)$, $0\notin\overline{U}, 0\notin\overline{V}$ and $U\cap V=\emptyset$. Then the processes
    $$
        X_t^U := N_t(y\I_U(y)),\quad
        X_t^V := N_t(y\I_V(y)),\quad
        X_t-X_t^{U\cup V}
    $$
    are independent L\'evy processes in $\rd$.
\end{theorem}
\begin{proof}
    Set $W:=U\cup V$. By Theorem~\ref{rm-21}, $X^U, X^V$ and $X-X^W$ are L\'evy processes. In fact, a slight variation of that argument even shows that $(X^U,X^V,X-X^W)$ is a L\'evy process in $\real^{3d}$.

    In order to see their independence, fix $s>0$ and define for $t>s$ and $\xi, \eta, \theta \in \rd$ the processes
    \begin{gather*}
        C_t := \frac{\eup^{\iup \xi\cdot(X_t^U - X_s^U)}}{\Ee\big[\eup^{\iup \xi\cdot(X_t^U - X_s^U)}\big]} - 1,\qquad
        D_t := \frac{\eup^{\iup \eta\cdot(X_t^V - X_s^V)}}{\Ee\big[\eup^{\iup \eta\cdot(X_t^V - X_s^V)}\big]} - 1, \\
        E_t := \frac{\eup^{\iup \theta\cdot(X_t - X_t^W - X_s + X_s^W)}}{\Ee \big[\eup^{\iup \theta\cdot (X_t - X_t^W - X_s + X_s^W)}\big]} - 1.
    \end{gather*}
    By Lemma~\ref{rm-31}, these processes are bounded martingales 
    satisfying $\Ee C_t = \Ee D_t = \Ee E_t = 0$. Set $t_{k,n} = s+\tfrac kn\,(t-s)$. Observe that
    \begin{align*}
        \Ee(C_t D_t E_t)
        &= \Ee\left( \sum_{k,l,m = 0}^{n-1} (C_{t_{k+1,n}}-C_{t_{k,n}}) (D_{t_{l+1,n}}-D_{t_{l,n}})(E_{t_{m+1,n}}-E_{t_{m,n}}) \right) \\
        &= \Ee\left( \sum_{k = 0}^{n-1} (C_{t_{k+1,n}}-C_{t_{k,n}}) (D_{t_{k+1,n}}-D_{t_{k,n}})(E_{t_{k+1,n}}-E_{t_{k,n}}) \right). 
    \end{align*}
    In the second equality we use that martingale increments $C_t-C_s, D_t-D_s, E_t-E_s$ are independent of $\Fscr_s^X$, and by the tower property
    $$
        \Ee\big[(C_{t_{k+1,n}}-C_{t_{k,n}}) (D_{t_{l+1,n}}-D_{t_{l,n}}) (E_{t_{m+1,n}}-E_{t_{m,n}})\big] = 0 \quad\text{unless $k=l=m$}.
    $$
    An argument along the lines of Lemma~\ref{rm-11} gives
    $$
        \Ee(C_t D_t E_t)
        =
        \Ee\bigg( \sum_{s<u\leq t} \underbrace{\Delta C_u \,\Delta D_u\,\Delta E_u}_{=\,0} \bigg)
        = 0
    $$
    as $X_t^U$, $X_t^V$ and $Y_t := X_t-X_t^W$ cannot jump simultaneously since $U$, $V$ and $\rd\setminus W$ are mutually disjoint. Thus,
    \begin{equation}\label{rm-e30}\begin{aligned}
        \Ee\Big[&\eup^{\iup \xi\cdot(X_t^U - X_s^U)} \eup^{\iup\eta\cdot(X_t^V - X_s^V)} \eup^{\iup\theta\cdot(Y_t - Y_s)}\Big]\\
        &=
        \Ee\Big[\eup^{\iup\xi\cdot(X_t^U - X_s^U)}\Big]\cdot
        \Ee\Big[\eup^{\iup\eta\cdot(X_t^V - X_s^V)}\Big]\cdot
        \Ee\Big[\eup^{\iup\theta\cdot(Y_t-Y_s)}\Big].
    \end{aligned}\end{equation}
    Since all processes are L\'evy processes, \eqref{rm-e30} already proves the independence of $X^U$, $X^V$ and $Y = X-X^W$. Indeed, we find for
    $0 = t_0 < t_1 < \ldots < t_m = t$ and $\xi_k, \eta_k, \theta_k \in \rd$
    \begin{align*}
    \Ee \Big( &\eup^{\iup\sum_k \xi_k\cdot(X_{t_{k+1}}^U- X_{t_{k}}^U)}
                \eup^{\iup\sum_k \eta_k\cdot(X_{t_{k+1}}^V- X_{t_{k}}^V)}
                \eup^{\iup\sum_k \theta_k\cdot(Y_{t_{k+1}}- Y_{t_{k}})}\Big) \\
    &\qquad\pomu{\eqref{LindepAlt}}{=}{\eqref{rm-e30}} \Ee\Big(\textstyle\prod\limits_k
                \eup^{\iup\xi_k\cdot(X_{t_{k+1}}^U- X_{t_{k}}^U)}
                \eup^{\iup\eta_k\cdot(X_{t_{k+1}}^V- X_{t_{k}}^V)}
                \eup^{\iup\theta_k\cdot(Y_{t_{k+1}}- Y_{t_{k}})} \Big) \\
    &\qquad\omu{\eqref{LindepAlt}}{=}{\phantom{\eqref{rm-e30}}} \textstyle\prod\limits_k\Ee\Big(
                \eup^{\iup\xi_k\cdot(X_{t_{k+1}}^U- X_{t_{k}}^U)}
                \eup^{\iup\eta_k\cdot(X_{t_{k+1}}^V- X_{t_{k}}^V)}
                \eup^{\iup\theta_k\cdot(Y_{t_{k+1}}- Y_{t_{k}})} \Big) \\
    &\qquad\omu{\eqref{rm-e30}}{=}{\phantom{\eqref{LindepAlt}}} \textstyle\prod\limits_k
                \Ee\Big( \eup^{\iup\xi_k\cdot(X_{t_{k+1}}^U- X_{t_{k}}^U)}\Big)
                \Ee\Big( \eup^{\iup\eta_k\cdot(X_{t_{k+1}}^V- X_{t_{k}}^V)}\Big)
                \Ee\Big(\eup^{\iup\theta_k\cdot(Y_{t_{k+1}}- Y_{t_{k}})} \Big).
    \end{align*}
    The last equality follows from \eqref{rm-e30}; the second equality uses \eqref{LindepAlt} for the L\'evy process $$(X_t^U, X_t^V, X_t-X_t^W).$$

    This shows that the families $(X_{t_{k+1}}^U - X_{t_{k}}^U)_{k}$, $(X_{t_{k+1}}^V - X_{t_{k}}^V)_{k}$ and $(Y_{t_{k+1}} - Y_{t_{k}})_{k}$ are independent, hence the $\sigma$-algebras $\sigma(X_t^U,\,t\geq 0)$, $\sigma(X_t^V,\,t\geq 0)$ and $\sigma(X_t-X_t^W,\,t\geq 0)$ are independent.
\end{proof}

\begin{corollary}\label{rm-35}
\index{Poisson process!jump measure of LP}%
    Let $N_t(\cdot)$ be the jump measure of a L\'evy process $X$ and $\nu$ the intensity measure.
    \begin{enumerate}
    \item[\upshape a)]
            $(N_t(U))_{t\geq 0} \bbot (N_t(V))_{t\geq 0}$ \ for \ $U,V\in\Bscr(\rd)$, $0\notin\overline{U}$, $0\notin\overline{V}$, $U\cap V=\emptyset$.
    \item[\upshape b)]
        For all measurable $f:\rd\to\real^m$ satisfying $f(0)=0$ and $f\in L^1(\nu)$
        \begin{equation}\label{rm-e32}
            \Ee \left(\eup^{\iup\xi\cdot  N_t(f)}\right)
            = \Ee \left(\eup^{\iup \int \xi\cdot f(y)\,N_t(dy)}\right)
            = \eup^{-t \int_{y\neq 0} \left[1-\eup^{\iup \xi\cdot f(y)}\right]\,\nu(dy)}.
        \end{equation}
    \item[\upshape c)]
        $\smash[t]{\displaystyle\int_{y\neq 0} \left(|y|^2\wedge 1\right)\,\nu(dy)<\infty}$.
    \end{enumerate}
\end{corollary}
\begin{proof}
    a) Since $(N_t(U))_{t\geq 0}$ and $(N_t(V))_{t\geq 0}$ are completely determined by the independent processes $(X_t^U)_{t\geq 0}$ and $(X_t^V)_{t\geq 0}$, cf.\ Theorem~\ref{rm-33}, the independence is clear.

\medskip\noindent
    b) Let us first prove \eqref{rm-e32} for step functions $f(x) = \sum_{k=1}^n \phi_k\I_{U_k}(x)$ with $\phi_k\in\real^m$ and disjoint sets $U_1, \ldots, U_n\in\Bscr(\rd)$ such that $0\not\in\overline U_k$. Then
    \begin{align*}
        \Ee\exp\left[\iup\int \xi\cdot f(y)\,N_t(dy)\right]
        &\pomu{\text{\ref{rm-21}.a)}}{=}{} \Ee \exp\bigg[\iup \sum_{k=1}^n \int \xi\cdot \phi_k\I_{U_k}(y)\,N_t(dy)\bigg]\\
        &\omu{\text{a)}}{=}{\phantom{\text{\ref{rm-21}.a)}}} \prod_{k=1}^n \Ee \exp\left[\iup \xi\cdot \phi_k \,N_t(U_k)\right]\\
        &\omu{\text{\ref{rm-21}.a)}}{=}{} \prod_{k=1}^n \exp\left[t\nu(U_k)\big[\eup^{\iup\xi\cdot \phi_k}-1\big]\right]\\
        &\pomu{\text{\ref{rm-21}.a)}}{=}{} \exp\bigg[t \sum_{k=1}^n\big[\eup^{\iup\xi\cdot \phi_k}-1\big]\nu(U_k)\bigg]\\
        &\pomu{\text{\ref{rm-21}.a)}}{=}{} \exp\bigg[-t \int \big[1-\eup^{\iup\xi\cdot f(y)}\big]\,\nu(dy)\bigg].
    \end{align*}
    For any $f\in L^1(\nu)$ the integral on the right-hand side of \eqref{rm-e32} exists. Indeed, the elementary inequality $|1-\eup^{\iup u}| \leq |u|\wedge 2$ and $\nu\{|y|\geq 1\}<\infty$ (Lemma~\ref{rm-13}.c)) yield
    \begin{align*}
        \left|\int_{y\neq 0} \big[1-\eup^{\iup\xi\cdot f(y)}\big]\,\nu(dy)\right|
        &\leq |\xi| \int_{0< |y|< 1} |f(y)|\,\nu(dy) + 2\int_{|y|\geq 1}\nu(dy)
        < \infty.
    \end{align*}
    Therefore, \eqref{rm-e32} follows with a standard approximation argument and dominated convergence.

\medskip\noindent
    c) We have already seen in Lemma~\ref{rm-13}.c) that $\nu\{|y|\geq 1\}<\infty$.

    Let us show that $\int_{0<|y|< 1}|y|^2\,\nu(dy)<\infty$. For this we take $U = \{\delta<|y|< 1\}$. Again by Theorem~\ref{rm-33}, the processes $X_t^U$ and $X_t-X_t^U$ are independent, and we get
    $$
        0<
        \big|\Ee\,\eup^{\iup \xi\cdot X_t}\big|
        = \big|\Ee\,\eup^{\iup \xi\cdot X_t^U}\big|\cdot \big|\Ee\,\eup^{\iup \xi\cdot (X_t-X_t^U)}\big|
        \leq \big|\Ee\,\eup^{\iup \xi\cdot X_t^U}\big|.
    $$
    Since $X_t^U$ is a compound Poisson process---use part b) with $f(y)=y\I_U(y)$---we get for all $|\xi|\leq 1$
\index{compound Poisson process!building block of LP}%
    $$
        0
        < \big|\Ee\,\eup^{\iup \xi\cdot X_t}\big|
        \leq \big|\Ee\,\eup^{\iup \xi\cdot X_t^U}\big|
        = \eup^{-t\int_U (1-\cos\xi\cdot y)\,\nu(dy)}
        \leq \eup^{-t\int_{\delta < |y|\leq 1} \frac 14 (\xi\cdot y)^2\,\nu(dy)}.
    $$
    For the equality we use $|\eup^z| = \eup^{\Re z}$, the inequality follows 
    $\frac 14 u^2\leq 1-\cos u$ if $|u|\leq 1$. Letting $\delta\to 0$ we see that $\int_{0<|y|< 1} |y|^2\,\nu(dy)<\infty$.
\end{proof}

\begin{corollary}\label{rm-36}
    Let $N_t(\cdot)$ be the jump measure of a L\'evy process $X$ and $\nu$ the intensity measure. For all $f:\rd\to\real^m$ satisfying $f(0)=0$ and $f\in L^2(\nu)$ we have\footnote{This is a special case of an It\^o isometry, cf.\ \eqref{si-e22} in the following chapter.}
    \begin{equation}\label{rm-e36}
        \Ee\bigg( \left|\int f(y)\,\big[N_t(dy)-t\nu(dy)\big]\right|^2 \bigg)
        = t\int_{y\neq 0} |f(y)|^2\,\nu(dy).
    \end{equation}
\end{corollary}
\begin{proof}
    It is clearly enough to show \eqref{rm-e36} for step functions of the form
    $$
        f(x) = \sum_{k=1}^n \phi_k\I_{B_k}(x),\quad
        B_k\text{\ disjoint},\; 0\not\in\overline{B}_k,\;\phi_k\in\real^m,
    $$
    and then use an approximation argument.

    Since the processes $N_t(B_k,\cdot)$ are independent Poisson processes with mean $\Ee N_t(B_k) = t\nu(B_k)$ and variance $\Vv N_t(B_k) = t\nu(B_k)$, we find
    \begin{align*}
        \Ee\big[&(N_t(B_k)-t\nu(B_k))(N_t(B_l)-t\nu(B_l))\big]\\
        &=  \left\{
            \begin{array}{ll}
                0, &\text{if\ } B_k\cap B_l=\emptyset,\text{\ i.e.\ }k\neq l,\\[\medskipamount]
                \Vv N_t(B_k) = t\nu(B_k), &\text{if\ } k=l,
            \end{array}\right\}\\
        &=t\nu(B_k\cap B_l).
    \end{align*}
    Therefore,
    \begin{align*}
        \Ee\bigg( &\bigg|\int f(y)\,\big(N_t(dy)-t\nu(dy)\big)\bigg|^2 \bigg)\\
        &= \Ee\bigg(\iint f(y)f(z)\,\big(N_t(dy)-t\nu(dy)\big)\big(N_t(dz)-t\nu(dz)\big) \bigg)\\
        &= \sum_{k,l=1}^n \phi_k\phi_l
           \underbrace{\Ee\Big(\big(N_t(B_k)-t\nu(B_k)\big)\big(N_t(B_l)-t\nu(B_l)\big) \Big)}_{=\,t\nu(B_k\cap B_l)}\\
        &= t\sum_{k=1}^n |\phi_k|^2\,\nu(B_k)
         = t\int |f(y)|^2\,\nu(dy).
    \qedhere
    \end{align*}
\end{proof}

In contrast to Corollary~\ref{cons-13} the following theorem does not need (but constructs) the L\'evy triplet $(l,Q,\nu)$.
\index{L\'evy--It\^o decomposition}%
\begin{theorem}[L\'evy--It\^o decomposition]\label{rm-51}
    Let $X$ be a L\'evy process and denote by $N_t(\cdot)$ and $\nu$ the jump and intensity measures. Then
    \begin{equation}\label{rm-e52}\begin{aligned}
        X_t
        &= \sqrt Q W_t \;+\: \int_{0<|y|<1} y \big(N_t(dy)-t\nu(dy)\big)
        &&\pmb{\bigg]}\text{\ $=: M_t$, \ $L^2$-martingale}\\
        &\phantom{=}\underbracket{\;\mbox{}\quad\; tl\quad \vphantom{\int_{|y|\geq 1}}}_{\mathclap{\begin{gathered}\text{continuous}\\[-5pt]\text{Gaussian}\end{gathered}}}
        \;+ \;\; \underbracket{\int_{|y|\geq 1} y\,N_t(dy).\qquad\qquad\quad}_{\begin{gathered}\text{pure jump part}\end{gathered}}
        &&\pmb{\bigg]}\text{\ $=: A_t$, \ bdd.\ variation}
    \end{aligned}\end{equation}
    where $l\in\rd$ and $Q\in\real^{d\times d}$ is a positive semidefinite symmetric matrix and $W$ is a standard Brownian motion in $\rd$. The processes on the right-hand side of \eqref{rm-e52} are independent L\'evy processes.
\end{theorem}
\begin{proof}
    $1^\circ$\ \
    Set $U_n := \{\frac 1n <|y|<1\}$, $V = \{|y|\geq 1\}$, $W_n := U_n\dcup V$ and define
    $$
        X^V_t := \int_V y\,N_t(dy)
        \et
        \widetilde X^{U_n}_t := \int_{U_n} y\,N_t(dy) - t\int_{U_n} y\,\nu(dy).
    $$
    By Theorem~\ref{rm-33} $(X^V_t)_{t\geq 0}$, $(\widetilde X^{U_n}_t)_{t\geq 0}$ and $\big(X_t-X_t^{W_n}+t\int_{U_n}y\,\nu(dy)\big)_{t\geq 0}$ are independent L\'evy processes. Since
    $$
        X = (X - \widetilde X^{U_n} - X^V) + \widetilde X^{U_n} + X^V,
    $$
    the theorem follows if we can show that the three terms on the right-hand side converge separately as $n\to\infty$.

\medskip\noindent
    $2^\circ$\ \
    Lemma~\ref{rm-19} shows $\Ee \widetilde X^{U_n}_t = 0$; since $\widetilde X^{U_n}$ is a L\'evy process, it is a martingale: for $s\leq t$
    \begin{align*}
        \Ee\left( \widetilde X^{U_n}_t \:\middle|\:\Fscr_s\right)
        &\pomu{\eqref{Lindep}}{=}{\eqref{Lstat}}
            \Ee\left(\widetilde X^{U_n}_t - \widetilde X^{U_n}_s \:\middle|\:\Fscr_s\right) + \widetilde X^{U_n}_s\\
        &\omu{\eqref{Lindep}}{=}{\eqref{Lstat}}
            \Ee\left(\widetilde X^{U_n}_{t-s}\right) + \widetilde X^{U_n}_s
        = \widetilde X^{U_n}_s.
    \end{align*}
    ($\Fscr_s$ can be taken as the natural filtration of $X^{U_n}$ or $X$).
    By Doob's $L^2$ martingale inequality we find for any $t>0$ and $m< n$
    \begin{align*}
        \Ee\left(\sup_{s\leq t} \big|\widetilde X^{U_n}_s - \widetilde X^{U_m}_s\big|^2\right)
        &\leq 4\Ee\left(\big|\widetilde X^{U_n}_t - \widetilde X^{U_m}_t\big|^2\right)\\
        &= 4t\int_{\frac 1n <|y|\leq \frac 1m} |y|^2\,\nu(dy)
        \xrightarrow[m,n\to\infty]{}0.
    \end{align*}
    Therefore, the limit $\int_{0<|y|<1} y\,\big(N_t(dy)-t\nu(dy)\big) = L^2$-$\lim_{n\to\infty}\widetilde X^{U_n}_t$ exists locally uniformly (in $t$). The limit is still an $L^2$ martingale with c\`adl\`ag paths (take a locally uniformly a.s.\ convergent subsequence) and, by Lemma~\ref{cons-07}, also a L\'evy process.

\medskip\noindent
    $3^\circ$\ \
    Observe that
    $$
        (X - \widetilde X^{U_n} - X^V) - (X - \widetilde X^{U_m} - X^V)
        = \widetilde X^{U_m} - \widetilde X^{U_n},
    $$
    and so $X_t^c := L^2\text{-}\lim_{n\to\infty}(X_t - \widetilde X^{U_n}_t - X^V_t)$ exists locally uniformly (in $t$) Since, by construction $|\Delta(X_t - \widetilde X^{U_n}_t - X^V_t)|\leq \frac 1n$, it is clear that $X^c$ has a.s.\ continuous sample paths. By Lemma~\ref{cons-07} it is a L\'evy process. From Theorem~\ref{spe-21} we know that all L\'evy processes with continuous sample paths are of the form $tl + \sqrt{Q}W_t$ where $W$ is a Brownian motion, $Q\in\real^{d\times d}$ a symmetric positive semidefinite matrix and $l\in\rd$.

\medskip\noindent
    $4^\circ$\ \
    Since independence is preserved under $L^2$-limits, the decomposition \eqref{rm-e52} follows. Finally,
    $$
        \int_{|y|\geq 1} y\,N_t(dy,\omega)
        = \sum_{0<s\leq t} \Delta X_s(\omega)\I_{\{\Delta X_s(\omega) \geq 1\}} - \sum_{0<s\leq t} |\Delta X_s(\omega)|\I_{\{\Delta X_s(\omega) \leq -1\}}
    $$
    is the difference of two increasing processes, i.e.\ it is of bounded variation.
\end{proof}

\index{L\'evy--Khintchine formula}%
\begin{corollary}[L\'evy--Khintchine formula]\label{rm-53}
    Let $X$ be a L\'evy process. Then the characteristic exponent $\psi$ is given by
    \begin{equation}\label{rm-e54}
        \psi(\xi)
        = -\iup l\cdot\xi + \frac 12\xi\cdot Q\xi + \int_{y\neq 0}\left[1-\eup^{\iup y\cdot\xi}+\iup\xi\cdot y\I_{(0,1)}(|y|)\right]\nu(dy)
    \end{equation}
    where $\nu$ is the intensity measure, $l\in\rd$ and $Q\in\real^{d\times d}$ is symmetric and positive semidefinite.
\end{corollary}
\begin{proof}
    Since the processes appearing in the L\'evy--It\^o decomposition \eqref{rm-e52} are independent, we see
    $$
        \eup^{-\psi(\xi)}
        = \Ee\,\eup^{\iup \xi\cdot X_1}
        = \Ee\,\eup^{\iup \xi\cdot (-l + \sqrt Q W_1)}
            \cdot \Ee\,\eup^{\iup\int_{0<|y|<1} \xi\cdot y (N_1(dy)-\nu(dy))}
            \cdot \Ee\,\eup^{\iup \int_{|y|\geq 1} \xi\cdot y\,N_1(dy)}.
    $$
    Since $W$ is a standard Brownian motion,
    \begin{align*}
        \Ee\,\eup^{\iup \xi\cdot (l + \sqrt Q W_1)}
        &= \eup^{\iup l\cdot\xi - \frac 12 \xi\cdot Q\xi}.
    \intertext{Using \eqref{rm-e32} with $f(y) = y\I_{U_n}(y)$, $U_n = \{\frac 1n< |y|< 1\}$, subtracting $\int_{U_n}y\,\nu(dy)$ and letting $n\to\infty$ we get}
        \Ee\exp\left[\iup\int_{0<|y|<1} \xi\cdot y (N_1(dy)-\nu(dy))\right]
        &= \exp\left[-\int_{0<|y|<1}\left[1-\eup^{\iup y\cdot\xi}+\iup\xi\cdot y\right]\nu(dy)\right];
    \intertext{finally, \eqref{rm-e32} with $f(y) = y\I_{V}(y)$, $V= \{|y|\geq 1\}$, once again yields}
        \Ee\exp\left[\iup\int_{|y|\geq 1} \xi\cdot y\,N_1(dy)\right]
        &= \exp\left[-\int_{|y|\geq 1}\left[1-\eup^{\iup y\cdot\xi}\right]\nu(dy)\right]
    \end{align*}
    finishing the proof.
\end{proof}

\chapter{A digression: stochastic integrals}\label{si}

In this chapter we explain how one can integrate with respect to (a certain class of) random measures. Our approach is based on the notion of \emphh{random orthogonal measures} and it will include the classical It\^o integral with respect to square-integrable martingales. Throughout this chapter, $(\Omega,\Ascr,\Pp)$ is a probability space, $(\Fscr_t)_{t\geq 0}$ some filtration, $(E,\Escr)$ is a measurable space and $\Ncont$ is a (positive) measure on $(E,\Escr)$. Moreover, $\Rscr\subset\Escr$ is a semiring, i.e.\ a family of sets such that $\emptyset\in\Rscr$, for all $R,S\in\Rscr$ we have $R\cap S\in\Rscr$, and $R\setminus S$ can be represented as a finite union of disjoint sets from $\Rscr$, cf.\ \cite[Chapter 6]{schilling-mims} or \cite[Definition 5.1]{schilling-mi}. It is not difficult to check that $\Rscr_0 := \{R\in\Rscr\::\: \Ncont(R)<\infty\}$ is again a semiring.

\begin{definition}\label{si-03}
    Let $\Rscr$ be a semiring on the measure space $(E,\Escr,\Ncont)$. A \emphh{random orthogonal measure} with \emphh{control measure} $\Ncont$
\index{random orthogonal measure}%
\index{control measure}%
    is a family of random variables $N(\omega,R)\in\real$, $R\in\Rscr_0$, such that
    \begin{gather}
    \label{si-e02}
        \Ee \left[|N(\cdot,R)|^2\right] < \infty\quad \forall R\in\Rscr_0\\
    \label{si-e04}
        \Ee \left[N(\cdot,R)N(\cdot,S)\right] = \Ncont(R\cap S)\quad \forall R,S\in\Rscr_0.
    \end{gather}
\end{definition}

The following Lemma explains why $N(R)=N(\omega,R)$ is called a (random) measure.
\begin{lemma}\label{si-05}
    The random set function $R\mapsto N(R):=N(\omega,R)$, $R\in\Rscr_0$, is countably additive in $L^2$, i.e.
    \begin{equation}\label{si-e06}
        N\left(\bigdcup_{n=1}^\infty R_n\right) = L^2\text{-}\lim_{n\to\infty}\sum_{k=1}^n N(R_k)\quad\text{a.s.}
    \end{equation}
    for every sequence $(R_n)_{n\in\nat}\subset\Rscr_0$ of mutually disjoint sets such that $R:=\bigdcup_{n=1}^\infty R_n\in\Rscr_0$. In particular, $N(R\dcup S)=N(R)+N(S)$ a.s.\ for disjoint $R,S\in\Rscr_0$ such that $R\dcup S\in\Rscr_0$ and $N(\emptyset)=0$ a.s. \textup{(}notice that the exceptional set may depend on the sets $R,S$\textup{)}.
\end{lemma}
\begin{proof}
    From $R=S=\emptyset$ and $\Ee[N(\emptyset)^2]=\Ncont(\emptyset)=0$ we get $N(\emptyset)=0$ a.s. It is enough to prove \eqref{si-e06} as finite additivity follows if we take $(R_1,R_2,R_3,R_4\dots)=(R,S,\emptyset,\emptyset,\dots)$. If $R_n\in\Rscr_0$ are mutually disjoint sets such that $R:=\bigdcup_{n=1}^\infty R_n\in\Rscr_0$, then
    \begin{align*}
        \Ee &\left[ \Big(N(R) - \sum_{k=1}^n N(R_k)\Big)^2\right]\\
        &\pomu{\eqref{si-e04}}{=}{} \Ee N^2(R) + \sum_{k=1}^n \Ee N^2(R_k) - 2\sum_{k=1}^n \Ee\left[N(R)N(R_k)\right] + \sum_{j\neq k, j,k=1}^n \Ee\left[N(R_j)N(R_k)\right]\\
        &\omu{\eqref{si-e04}}{=}{} \Ncont(R) - \sum_{k=1}^n \Ncont(R_k) \xrightarrow[n\to\infty]{}0
    \end{align*}
    where we use the $\sigma$-additivity of the measure $\Ncont$.
\end{proof}

\begin{example}\label{si-07}
\textup{a)}\ \
\index{Brownian motion!white noise}\index{white noise}%
    \emphh{(White noise)} Let $\Rscr=\{(s,t] \::\: 0\leq s<t<\infty\}$ and $\Ncont = \lambda$ be Lebesgue measure on $(0,\infty)$. Clearly, $\Rscr=\Rscr_0$ is a semiring. Let $W=(W_t)_{t\geq 0}$ be a one-dimensional standard Brownian motion. The random set function
    $$
        N(\omega,(s,t]) := W_t(\omega)-W_s(\omega),\quad 0\leq s<t<\infty
    $$
    is a random orthogonal measure with control measure $\lambda$. This follows at once from
    $$
        \Ee \big[(W_t-W_s)(W_v-W_u)\big] = t\wedge v - s\vee u= \lambda \big[(s,t]\cap (u,v]\big]
    $$
    for all $0\leq s<t<\infty$ and $0\leq u<v<\infty$.

    Mind, however, that $N$ is \emphh{not $\sigma$-additive}. To see this, take $R_n := (1/(n+1),1/n]$, where $n\in\nat$, and observe that $\bigdcup_n R_n = (0,1]$. Since $W$ has stationary and independent increments, and scales like $W_t \sim \sqrt t W_1$, we have
    \begin{align*}
        \Ee \exp\left[-\sum\nolimits_{n=1}^\infty |N(R_n)|\right]
        &\pomu{\text{Jensen's}}{=}{\text{ineq.}} \Ee \exp\left[-\sum\nolimits_{n=1}^\infty |W_{1/(n+1)} - W_{1/n}|\right]\\
        &\pomu{\text{Jensen's}}{=}{\text{ineq.}} \prod\nolimits_{n=1}^\infty \Ee \exp\left[-(n(n+1))^{-1/2}|W_1|\right]\\
        &\omu{\text{Jensen's}}{\leq}{\text{ineq.}}
        \prod\nolimits_{n=1}^\infty \alpha^{(n(n+1))^{-1/2}},\quad \alpha:= \Ee \eup^{-|W_1|}\in (0,1).
    \end{align*}
    As the series $\sum_{n=1}^\infty (n(n+1))^{-1/2}$ diverges, we get $\Ee \exp\left[-\sum_{n=1}^\infty |N(R_n)|\right]=0$ which means that $\sum_{n=1}^\infty |N(\omega,R_n)|=\infty$ for almost all $\omega$. This shows that $N(\cdot)$ cannot be countably additive. Indeed, countable additivity implies that the series
\index{white noise!is not $\sigma$-additive}%
    $$
        N\left(\omega,\bigcup_{n=1}^\infty R_n\right) = \sum_{n=1}^\infty N(\omega,R_n)
    $$
    converges. The left-hand side, hence the summation, is independent under rearrangements. This, however, entails absolute convergence of the series $\sum_{n=1}^\infty |N(\omega,R_n)|<\infty$ which does not hold as we have seen above.

\medskip\noindent\textup{b)}\ \
\index{second-order process}%
    \emphh{(2nd order orthogonal noise)} Let $X = (X_t)_{t\in T}$ be a complex-valued stochastic process defined on a bounded or unbounded interval $T\subset\real$. We assume that $X$ has a.s.\ c\`adl\`ag paths. If $\Ee(|X_t|^2)<\infty$, we call $X$ a \emphh{second-order process}; many properties of $X$ are characterized by the \emphh{correlation function}
\index{correlation function}%
    $K(s,t)=\Ee\big(X_s\overline X_t\big)$, $s,t\in T$.

    If $\Ee\big[(X_t-X_s)(\overline X_v - \overline X_u)\big]=0$ for all $s\leq t\leq u\leq v$, $s,t,u,v\in T$, then $X$ is said to have \emphh{orthogonal increments}.
\index{orthogonal increments}%
    Fix $t_0\in T$ and define for all $t\in T$
    $$
        F(t) :=
        \begin{cases}
            \Ee(|X_t-X_{t_0}|^2), &\text{if \ } t\geq t_0,\\
            -\Ee(|X_{t_0}-X_t|^2), &\text{if \ } t\leq t_0.
        \end{cases}
    $$
    Clearly, $F$ is increasing and, since $t\mapsto X_t$ is a.s.\ right-continuous, it is also right-continuous. Moreover,
    \begin{equation}
        F(t)-F(s) = \Ee(|X_t-X_s|^2)\quad\fa s\leq t,\; s,t\in T.
    \end{equation}
    To see this, we assume without loss of generality that $s\leq t_0\leq t$. We have
    \begin{align*}
        F(t)-F(s)
        &\pomu{\text{orth.}}{=}{\text{incr.}} \Ee(|X_t-X_{t_0}|^2) - \Ee(|X_s-X_{t_0}|^2)\\
        &\pomu{\text{orth.}}{=}{\text{incr.}} \Ee(|(X_t-X_s) + (X_s - X_{t_0})|^2) - \Ee(|X_s-X_{t_0}|^2)\\
        &\omu{\text{orth.}}{=}{\text{incr.}} \Ee(|X_t-X_s|^2).
    \end{align*}
    This shows that $\mu(s,t] := F(t)-F(s)$ defines a measure on $\Rscr=\Rscr_0 = \{(s,t] \::\: -\infty < s<t < \infty, \, s,t\in T\}$, which is the control measure of $N(\omega,(s,t]) := X_t(\omega)-X_s(\omega)$. In fact, for $s <t, u<v$, $s,t,u,v\in T$, we have
    \begin{align*}
        X_t-X_s &=
            \big(X_t-X_{t\wedge v}\big)
            +\big(X_{t\wedge v}-X_{s\vee u}\big)
            +\big(X_{s\vee u}-X_s\big)\\
        \overline X_v-\overline X_u &=
            \big(\overline X_u-\overline X_{t\wedge v}\big)
            +\big(\overline X_{t\wedge v}-\overline X_{s\vee u}\big)
            +\big(\overline X_{s\vee u}-\overline X_u\big).
    \end{align*}
    Using the orthogonality of the increments we get
    \begin{align*}
        \Ee\big[(X_t-X_s)(\overline X_v - \overline X_u)\big]
        &= \Ee\big[\big(X_{t\wedge v}-X_{s\vee u}\big) \big(\overline X_{t\wedge v}-\overline X_{s\vee u}\big)\big]\\
        &= F(t\wedge v)-F(s\vee u) = \mu\big((s,t]\cap (u,v]\big),
    \end{align*}
    i.e.\ $N(\omega,\bullet)$ is a random orthogonal measure.

\medskip\noindent\textup{c)}\ \
    \emphh{(Martingale noise)} Let $M=(M_t)_{t\geq 0}$ be a square-integrable martingale with respect to the filtration $(\Fscr_t)_{t\geq 0}$, $M_0=0$, and with c\`adl\`ag paths.
\index{martingale noise}%
    Denote by $\langle M\rangle$ the predictable quadratic variation, i.e.\ the unique $(\langle M\rangle_0:=0)$ increasing predictable process such that $M^2-\langle M\rangle$ is a martingale. The random set function
    $$
        N(\omega,(s,t]):= M_t(\omega)-M_s(\omega), \quad s\leq t,
    $$
    is a random orthogonal measure on $\Rscr=\{(s,t] \::\: 0\leq s<t<\infty\}$ with control measure $\Ncont(s,t] = \Ee (\langle M\rangle_t-\langle M\rangle_s)$. This follows immediately from the tower property of conditional expectation
    $$
        \Ee [M_tM_v]
        \omu{\text{tower}}{=}{} \Ee [M_t \Ee(M_v \mid \Fscr_t)]
        = \Ee [M_t^2]
        = \Ee \langle M\rangle_t
        \text{\ \ if\ \ } t\leq v
    $$
    which, in turn, gives for all $0\leq s<t$ and $0\leq u< v$
    \begin{align*}
        \Ee[(M_t-M_s)(M_v-M_u)]
        &= \Ee\langle M\rangle_{t\wedge v} - \Ee\langle M\rangle_{s\wedge v} - \Ee\langle M\rangle_{t\wedge u} + \Ee\langle M\rangle_{s\wedge u}\\
        &= \Ncont\big((s,t]\cap (0,v]\big) - \Ncont\big((s,t]\cap (0,u]\big)\\
        &= \Ncont\big((s,t]\cap (u,v]\big).
    \end{align*}

\medskip\noindent\textup{d)}\ \
    \emphh{(Poisson random measure)} Let $X$ be a $d$-dimensional L\'evy process,
\index{Poisson~random measure}%
    $$
        \Sscr := \big\{B\in\Bscr(\rd)\::\: 0\notin\overline{B}\big\},\quad
        \Rscr := \big\{(s,t]\times B \::\: 0\leq s < t <\infty,\; B\in\Sscr\big\},
    $$
    and $N_t(B)$ the jump measure (Definition~\ref{rm-03}). The random set function
    $$
        \Ntilde(\omega,(s,t]\times B) := [N_t(\omega,B) - t\nu(B)] - [N_s(\omega,B) - s\nu(B)] ,\quad R = (s,t]\times B\in\Rscr,
    $$
    is a random orthogonal measure with control measure $\lambda\times\nu$ where $\lambda$ is Lebesgue measure on $(0,\infty)$ and $\nu$ is the L\'evy measure of $X$. Indeed, by definition $\Rscr = \Rscr_0$, and it is not hard to see that $\Rscr$ is a semiring\footnote{Both $\Sscr$ and $\Iscr := \{(s,t]\::\: 0\leq s<t<\infty\}$ are semirings, and so is their cartesian product $\Rscr = \Iscr\times\Sscr$,
    see \cite[Lemma 13.1]{schilling-mims} or \cite[Lemma 15.1]{schilling-mi} for the straightforward proof.}.

    Set $\Ntilde_t(B):= \Ntilde((0,t]\times B)$ and let $B,C\in\Sscr$, $t,v\geq 0$.
    As in the proof of Corollary~\ref{rm-36} we have
    \begin{equation*}
        \Ee \big[\Ntilde_t(B)\Ntilde_t(C)\big] = t\nu(B\cap C).
    \end{equation*}
    Since $\Sscr$ is a semiring, we get $B = (B\cap C)\dcup (B\setminus C) = (B\cap C)\dcup B_1\dcup\dots\dcup B_n$ with finitely many mutually disjoint $B_k\in\Sscr$ such that $B_k\subset B\setminus C$.

    The processes $\Ntilde(B_k)$ and $\Ntilde(C)$ are independent (Corollary~\ref{rm-35}) and centered. Therefore we have for $t\leq v$
    \begin{align*}
        \Ee \big[\Ntilde_t(B)\Ntilde_v(C)\big]
        &= \Ee \big[\Ntilde_t(B\cap C)\Ntilde_v(C)\big] + \sum_{k=1}^n\Ee \big[\Ntilde_t(B_k)\Ntilde_v(C)\big]\\
        &= \Ee \big[\Ntilde_t(B\cap C) \Ntilde_v(C)\big] + \sum_{k=1}^n\Ee \Ntilde_t(B_k)\cdot \Ee\Ntilde_v(C)\\
        &= \Ee \big[\Ntilde_t(B\cap C)\Ntilde_v(C)\big].
    \end{align*}
    Use the same argument over again, as well as the fact that $\Ntilde_t(B\cap C)$ has independent and centered increments (Lemma~\ref{rm-13}), to get
    \begin{equation}\label{si-e08}\begin{aligned}
        \Ee \big[&\Ntilde_t(B)\Ntilde_v(C)\big]\\
        &= \Ee \big[\Ntilde_t(B\cap C)\Ntilde_v(B\cap C)\big]\\
        &= \smash[t]{\Ee \big[\Ntilde_t(B\cap C)\Ntilde_t(B\cap C)\big] + \overbrace{\Ee \big[\Ntilde_t(B\cap C)\big\{\Ntilde_v(B\cap C)-\Ntilde_t(B\cap C)\big\}\big]}^{=\Ee\Ntilde_t(B\cap C)\, \Ee[\Ntilde_v(B\cap C)-\Ntilde_t(B\cap C)] = 0}}\\
        &= \Ee \big[\Ntilde_t(B\cap C)\Ntilde_t(B\cap C)\big]\\
        &= t\nu(B\cap C)\\
        &= \lambda((0,t]\cap (0,v])\nu(B\cap C).
    \end{aligned}\end{equation}
    For $s\leq t$, $u\leq v$ and $B,C\in\Sscr$ a lengthy, but otherwise completely elementary, calculation based on \eqref{si-e08} shows
    $$
        \Ee\big[\Ntilde((s,t]\times B)\Ntilde((u,v]\times C)\big]
        = \lambda((s,t]\cap(u,v])\nu(B\cap C).
    $$

\medskip\noindent\textup{e)}\ \
    \emphh{(Space-time white noise)} Let $\Rscr := \left\{(0,t]\times B\::\: t> 0,\; B\in\Bscr(\rd)\right\}$ and $\Ncont = \lambda$ Le\-bes\-gue measure on the half-space $\mathds H^+ := [0,\infty)\times\rd$.
\index{space-time noise}%
\index{white noise}%

    Consider the mean-zero, real-valued Gaussian process $(\mathds W(R))_{R\in\Bscr(\mathds H^+)}$ whose covariance function is given by $\Cov(\mathds W(R)\mathds W(S)) = \lambda(R\cap S)$.\footnote{The map $(R,S)\mapsto\lambda(R\cap S)$ is positive semidefinite, i.e.\ for $R_1,\dots, R_n\in\Bscr(\mathds H^+)$ and $\xi_1, \dots,\xi_n \in\real$
    $$
        \sum_{j,k=1}^n \xi_j\xi_k \lambda(R_j\cap R_k)
        =\sum_{j,k=1}^n \int \xi_j\I_{R_j}(x) \xi_k\I_{R_k}(x)\,\lambda(dx)
        = \int\Big(\sum_{k=1}^n  \xi_k\I_{R_k}(x)\Big)^2 \lambda(dx)
        \geq 0.
    $$}
    By its very definition $\mathds W(R)$ is a random orthogonal measure on $\Rscr_0$ with control measure $\lambda$.
\end{example}
We will now define a stochastic integral in the spirit of It\^o's original construction.
\begin{definition}\label{si-11}
\index{stochastic integral}\index{random orthogonal measure!stochastic integral}%
    Let $\Rscr$ be a semiring and $\Rscr_0 = \{R\in\Rscr\::\: \Ncont(R)<\infty\}$. A \emphh{simple function} is a deterministic function of the form
    \begin{equation}\label{si-e10}
        f(x) = \sum_{k=1}^n c_k\I_{R_k}(x),\quad n\in\nat,\; c_k\in\real,\; R_k\in\Rscr_0.
    \end{equation}
\end{definition}
Intuitively, $I_N(f) = \sum_{k=1}^n c_k N(R_k)$ should be the stochastic integral of a simple function $f$.
The only problem is the well-definedness. Since a random orthogonal measure is a.s.\ finitely additive, the following lemma has exactly the same proof as the usual well-definedness result for the Lebesgue integral of a step function, see e.g.\ Schilling~\cite[Lemma~9.1]{schilling-mims} or \cite[Lemma~8.1]{schilling-mi}; note that finite unions of null sets are again null sets.
\begin{lemma}\label{si-13}
    Let $f$ be a simple function and assume that $f = \sum_{k=1}^n c_k\I_{R_k} = \sum_{j=1}^m b_j\I_{S_j}$ has two representations as step-function. Then
    $$
        \sum_{k=1}^n c_k N(R_k)
        = \sum_{j=1}^m b_j N(S_j)
        \quad\text{a.s.}
    $$
\end{lemma}

\begin{definition}\label{si-15}
    Let $N(R)$, $R\in\Rscr_0$, be a random orthogonal measure with control measure $\Ncont$. The \emphh{stochastic integral} of a simple function $f$ given by \eqref{si-e10} is the random variable
    \begin{equation}\label{si-e12}
        I_N(\omega,f) := \sum_{k=1}^n c_k N(\omega,R_k).
    \end{equation}
\end{definition}
The following properties of the stochastic integral are more or less immediate from the definition.
\begin{lemma}\label{si-17}
    Let $N(R)$, $R\in\Rscr_0$, be a random orthogonal measure with control measure $\Ncont$, $f,g$ simple functions, and $\alpha, \beta\in \real$.
    \begin{enumerate}
    \item[\textup{a)}]
        $I_N(\I_R) = N(R)$ for all $R\in\Rscr_0$;
    \item[\textup{b)}]
        $S\mapsto I_N(\I_S)$ extends $N$ uniquely to $S\in\rho(\Rscr_0)$, the ring generated by $\Rscr_0$;\footnote{A ring is a family of sets which contains $\emptyset$ and which is stable under unions and differences of finitely many sets. Since $R\cap S = R\setminus (R\setminus S)$, it is automatically stable under finite intersections. The ring generated by $\Rscr_0$ is the smallest ring containing $\Rscr_0$.}
    \item[\textup{c)}]
        $I_N(\alpha f + \beta g) = \alpha I_N(f) + \beta I_N(g)$; \hfill \textup{(}linearity\textup{)}
    \item[\textup{d)}]
        $\Ee \left[I_N(f)^2\right] = \int f^2\,d\Ncont$; \hfill \textup{(}It\^o's isometry\textup{)}
    \end{enumerate}
\end{lemma}
\begin{proof}
    The properties a) and c) are clear. For b) we note that $\rho(\Rscr_0)$ can be constructed from $\Rscr_0$ by adding all possible finite unions of (disjoint) sets (see e.g.\ \cite[Proof of Theorem 6.1, Step 2]{schilling-mims}). In order to see d), we use \eqref{si-e10} and the orthogonality relation $\Ee \left[N(R_j)N(R_k)\right] = \Ncont(R_j\cap R_k)$ to get
    \begin{align*}
        \Ee \big[I_N(f)^2\big]
        &= \sum_{j,k=1}^n c_j c_k \Ee \left[N(R_j)N(R_k)\right]\\
        &= \sum_{j,k=1}^n c_j c_k \Ncont(R_j\cap R_k)\\
        &= \int \sum_{j,k=1}^n c_j \I_{R_j}(x)\, c_k \I_{R_k}(x) \,\Ncont(dx)\\
        &= \int f^2(x)\,\Ncont(dx).
    \qedhere
    \end{align*}
\end{proof}
It\^o's isometry now allows us to extend the stochastic integral to the $L^2(\Ncont)$-closure of the simple functions:  $L^2(E,\sigma(\Rscr),\Ncont)$. For this take $f\in L^2(E,\sigma(\Rscr),\Ncont)$ and any approximating sequence $(f_n)_{n\in\nat}$ of simple functions, i.e.
$$
    \lim_{n\to\infty} \int |f-f_n|^2\,d\Ncont = 0.
$$
In particular, $(f_n)_{n\in\nat}$ is an $L^2(\Ncont)$ Cauchy sequence, and It\^o's isometry shows that the random variables $(I_N(f_n))_{n\in\nat}$ are a Cauchy sequence in $L^2(\Pp)$:
$$
    \Ee\left[(I_N(f_n)-I_N(f_m))^2\right]
    = \Ee\left[ I_N(f_n-f_m)^2\right]
    = \int (f_n-f_m)^2\,d\Ncont
    \xrightarrow[m,n\to\infty]{}0.
$$
Because of the completeness of $L^2(\Pp)$, the limit $\lim_{n\to\infty} I_N(f_n)$ exists and, by a standard argument, it does not depend on the approximating sequence.
\begin{definition}\label{si-19}
    Let $N(R)$, $R\in\Rscr_0$, be a random orthogonal measure with control measure $\Ncont$. The \emphh{stochastic integral} of a function $f\in L^2(E,\sigma(\Rscr),\Ncont)$ is the random variable
\index{stochastic integral}%
\index{random orthogonal measure!stochastic integral}%
    \begin{equation}\label{si-e20}
        \int f(x)\,N(\omega,dx) := L^2(\Pp)\text{-}\lim_{n\to\infty} I_N(\omega,f_n)
    \end{equation}
    where $(f_n)_{n\in\nat}$ is any sequence of simple functions which approximate $f$ in $L^2(\Ncont)$.
\end{definition}
It is immediate from the definition of the stochastic integral, that $f\mapsto \int f\,dN$ is linear and enjoys \emphh{It\^o's isometry}
\index{It\^o isometry}%
\begin{equation}\label{si-e22}
    \Ee \left[\Big(\int f(x)\,N(dx)\Big)^2\right] = \int f^2(x)\,\Ncont(dx).
\end{equation}
\begin{remark}\label{si-21}
\index{space-time noise}%
    Assume that the random orthogonal measure $N$ is of \emphh{space-time type}, i.e.\ $E = (0,\infty)\times X$ where $(X,\Xscr)$ is some measurable space, and $\Rscr = \{(0,t]\times B\::\: B\in\Sscr\}$ where $\Sscr$ is a semiring in $\Xscr$. If for $B\in\Sscr$ the stochastic process $N_t(B):= N((0,t]\times B)$ is a martingale w.r.t.\ the filtration $\Fscr_t := \sigma(N((0,s]\times B), \:s\leq t, B\in\Sscr)$, then
    $$
        N_t(f) := \iint \I_{(0,t]}(s)f(x)\,N(ds,dx),\quad \I_{(0,t]}\otimes f\in L^2(\Ncont),\; t\geq 0,
    $$
    is again a(n $L^2$-)martingale. For simple functions $f$ this follows immediately from the fact that sums and differences of finitely many martingales (with a common filtration) are again a martingale. Since $L^2(\Pp)$-limits preserve the martingale property, the claim follows.
\qedhere
\end{remark}
At first sight, the stochastic integral defined in \ref{si-19} looks rather restrictive since we can only integrate deterministic functions $f$. As all randomness can be put into the random orthogonal measure, we have considerable flexibility, and the following construction shows that Definition~\ref{si-19} covers pretty much the most general stochastic integrals.

\bigskip\noindent
\textbf{From now on we assume} that
\begin{itemize}
\item the random measure $N(dt,dx)$ on $(E,\Escr)=((0,\infty)\times X,\:\Bscr(\rd)\otimes\Xscr)$ is of space-time type, cf.\ Remark~\ref{si-21}, with control measure $\Ncont(dt,dx)$;
\item $(\Fscr_t)_{t\geq 0}$ is some filtration in $(\Omega,\Ascr,\Pp)$.
\end{itemize}
Let $\tau$ be a stopping time; the set $\rrbracket 0,\tau\rrbracket := \{(\omega,t)\::\: 0<t\leq\tau(\omega)\}$ is called \emphh{stochastic interval}.
\index{stochastic interval}%
We define
\begin{itemize}
    \item $\Ecirc:= \Omega\times (0,\infty)\times X$;
    \item $\Eecirc := \Pscr\otimes\Xscr$ where $\Pscr$ is the predictable $\sigma$-algebra in $\Omega\times(0,\infty)$, see Definition~\ref{app-61} in the appendix;
    \item $\Rrcirc := \left\{\rrbracket 0,\tau\rrbracket\times B\::\:\tau\text{\ bounded stopping time},\; B\in\Sscr\right\}$;
    \item $\Ncirccont(d\omega,dt,dx) := \Pp(d\omega)\Ncont(dt,dx)$ as control measure; 
    \item $\Ncirc\big(\omega,\rrbracket 0,\tau\rrbracket\times B\big) := N\big(\omega,(0,\tau(\omega)]\times B\big)$ as random orthogonal measure\footnote{To see that it is indeed a random orthogonal measure, use a discrete approximation of $\tau$.}.
\end{itemize}

\begin{lemma}\label{si-41}
    Let $\Ncirc$, $\Rrcirc_0$ and $\Ncirccont$ be as above. The $\Rrcirc_0$-simple processes
    $$
        f(\omega,t,x) := \sum_{k=1}^n c_k \I_{\rrbracket 0,\tau_k\rrbracket}(\omega,t)\I_{B_k}(x),
        \qquad c_k\in\real,\;\; \rrbracket 0,\tau_k\rrbracket\times B_k\in\Rrcirc_0
    $$
    are $L^2(\Ncirccont)$-dense in $L^2(\Ecirc,\Pscr\otimes\sigma(\Sscr),\Ncirccont)$.
\end{lemma}
\begin{proof}
    This follows from standard arguments from measure and integration; notice that the predictable $\sigma$-algebra $\Pscr\otimes\sigma(\Sscr)$ is generated by sets of the form  
    $\rrbracket 0,\tau\rrbracket\times B$ where $\tau$ is a bounded stopping time and $B\in\Sscr$, cf.\ Theorem~\ref{app-63} in the appendix.
\end{proof}
Observe that for simple processes appearing in Lemma~\ref{si-41}
$$
    \iint f(\omega,t,x)\,N(\omega,dt,dx)
    := \sum_{k=1}^n c_k \,\Ncirc(\omega,\rrbracket 0,\tau_k\rrbracket\times B_k)
    = \sum_{k=1}^n c_k \,N(\omega,(0,\tau_k(\omega)]\times B_k)
$$
is a stochastic integral which satisfies
$$
    \Ee\left[\Big(\iint f(\cdot,t,x)\,N(\cdot,dt,dx)\Big)^2\right]
    = \sum_{k=1}^n c_k^2 \,\Ncirccont(\rrbracket 0,\tau_k\rrbracket\times B_k)
    = \sum_{k=1}^n c_k^2 \,\Ee \Ncont((0,\tau_k]\times B_k).
$$
Just as above we can now extend the stochastic integral to $L^2(\Ecirc,\Pscr\otimes\sigma(\Sscr),\Ncirccont)$.
\begin{corollary}\label{si-43}
    Let $N(\omega,dt,dx)$ be a random orthogonal measure on $E$ of space-time type \textup{(}cf.\ Remark~\ref{si-21}\textup{)} with control measure $\Ncont(dt,dx)$ and $f:\Omega\times(0,\infty)\times X\to\real$ be an element of $L^2(\Ecirc,\Pscr\otimes\sigma(\Sscr),\Ncirccont)$. Then the stochastic integral
    $$
        \iint f(\omega,t,x)\,N(\omega,dt,dx)
    $$
    exists and satisfies the following \emphh{It\^o isometry}
\index{It\^o isometry}%
    \begin{align}\label{si-e42}
        \Ee \left[\Big(\iint f(\cdot,t,x)\,N(\cdot,dt,dx)\Big)^2\right]
        &= \iint \Ee f^2(\cdot,t,x)\,\Ncont(dt,dx).
    \end{align}
\end{corollary}

Let us show that the stochastic integral w.r.t.\ a space-time random orthogonal measure extends the usual It\^o integral. To do so we need the following auxiliary result.
\begin{lemma}\label{si-51}
    Let $N(\omega,dt,dx)$ be a random orthogonal measure on $E$ of space-time type \textup{(}cf.\ Remark~\ref{si-21}\textup{)} with control measure $\Ncont(dt,dx)$ and $\tau$ a stopping time. Then
    \begin{equation}\label{si-e52}\begin{aligned}
        \iint \phi(\omega) &\I_{\rrbracket \tau,\infty\llbracket}(\omega,t) f(\omega,t,x)\,N(\omega,dt,dx)\\
        &= \phi(\omega)\iint \I_{\rrbracket \tau,\infty\llbracket}(\omega,t) f(\omega,t,x)\,N(\omega,dt,dx)
    \end{aligned}\end{equation}
    for all $\phi\in L^\infty(\Fscr_\tau)$ and $f\in L^2(\Ecirc,\Pscr\otimes\sigma(\Sscr),\Ncirccont)$.
\end{lemma}
\begin{proof}
    Since $t\mapsto\I_{\rrbracket \tau,\infty\llbracket}(\omega,t)$ is adapted to the filtration $(\Fscr_t)_{t\geq 0}$ and left-continuous, the integrands appearing in \eqref{si-e52} are predictable, hence all stochastic integrals are well-defined.

\medskip\noindent $1^\circ$ \
    Assume that $\phi(\omega) = \I_F(\omega)$ for some $F\in\Fscr_\tau$ and $f(\omega,t,x) = \I_{\rrbracket 0,\sigma\rrbracket}(\omega,t)\I_B(x)$ for some bounded stopping time $\sigma$ and $\rrbracket 0,\sigma\rrbracket \times B\in\Rscr_0^\circ$. Define
    $$
        N_\sigma(\omega,B):= N(\omega, (0,\sigma(\omega)]\times B)= \Ncirc(\omega, \rrbracket 0,\sigma\rrbracket\times B)
    $$
    for any $\rrbracket 0,\sigma\rrbracket \times B\in\Rscr_0^\circ$. The random time $\tau_F := \tau\I_F + \infty\I_{F^c}$ is a stopping time\footnote{Indeed,
    $
        \{\tau_F\leq t\}
        = \{\tau\leq t\}\cap F
        = \left\{
            \begin{aligned}
                \emptyset, &\quad \tau>t\\
                F, &\quad \tau\leq t
            \end{aligned}
        \right\}\in\Fscr_t
    $ for all $t\geq 0$.},
    and we have
    $$
        \phi\I_{\rrbracket \tau,\infty\llbracket} f
        = \I_F \I_{\rrbracket \tau,\infty\llbracket} \I_{\rrbracket 0,\sigma\rrbracket} \I_B
        = \I_{\rrbracket \tau_F,\infty\llbracket} \I_{\rrbracket 0,\sigma\rrbracket} \I_B
        = \I_{\rrbracket \tau_F\wedge\sigma,\sigma\rrbracket} \I_B.
    $$
    From this we get \eqref{si-e52} for our choice of $\phi$ and $f$:
    \begin{align*}
        \iint \phi\I_{\rrbracket \tau,\infty\llbracket}(t) f(t,x)\,N(dt,dx)
        &= \iint \I_{\rrbracket \tau_F\wedge\sigma,\sigma\rrbracket}(t)\I_B(x)\, N(dt,dx)\\
        &= N_\sigma(B) - N_{\tau_F\wedge\sigma}(B) \\
        &= \I_F\cdot(N_\sigma(B)-N_{\tau\wedge\sigma}(B))\\
        &= \I_F \iint \I_{\rrbracket \tau\wedge\sigma,\sigma\rrbracket}(t)\I_B(x)\,N(dt,dx)\\
        &= \I_F \iint \I_{\rrbracket \tau,\infty\llbracket} \I_{\rrbracket 0,\sigma\rrbracket}(t)\I_B(x)\,N(dt,dx)\\
        &= \phi \iint \I_{\rrbracket \tau,\infty\llbracket}(t) f(t,x) \,N(dt,dx).
    \end{align*}

\noindent $2^\circ$ \
    If $\phi = \I_F$ for some $F\in\Fscr_\tau$ and $f$ is a simple process, then \eqref{si-e52} follows from $1^\circ$ because of the linearity of the stochastic integral.

\medskip\noindent $3^\circ$ \
    If $\phi = \I_F$ for some $F\in\Fscr_\tau$ and $f\in L^2(\Ncirccont)$, then \eqref{si-e52} follows from $2^\circ$ and It\^o's isometry: Let $f_n$ be a sequence of simple processes which approximate $f$. Then
    \begin{align*}
        &\Ee\left[ \Big(\iint \big(\phi \I_{\rrbracket \tau,\infty\llbracket}(t) f_n(t,x) - \phi \I_{\rrbracket \tau,\infty\llbracket}(t) f(t,x)\big)\,N(dt,dx)\Big)^2\right]\\
        &\qquad=
        \iint \Ee\left[\big(\phi \I_{\rrbracket \tau,\infty\llbracket}(t) f_n(t,x) - \phi \I_{\rrbracket \tau,\infty\llbracket}(t) f(t,x)\big)^2\right]\,\Ncont(dt,dx)\\
        &\qquad\leq
        \iint \Ee\left[\big(f_n(t,x) - f(t,x)\big)^2\right]\,\Ncont(dt,dx)
        \xrightarrow[n\to\infty]{}0.
    \end{align*}

\noindent $4^\circ$ \
    If $\phi$ is an $\Fscr_\tau$ measurable step-function and $f\in L^2(\Ncirccont)$, then \eqref{si-e52} follows from $3^\circ$ because of the linearity of the stochastic integral.

\medskip\noindent $5^\circ$ \
    Since we can approximate $\phi\in L^\infty(\Fscr_\tau)$ uniformly by $\Fscr_\tau$ measurable step functions $\phi_n$, \eqref{si-e52} follows from $4^\circ$ and It\^o's isometry because of the following inequality:
    \begin{align*}
        \Ee&\left[\Big(\iint \left[\phi_n \I_{\rrbracket\tau,\infty\llbracket}(t) f(t,x) - \phi \I_{\rrbracket\tau,\infty\llbracket}(t) f(t,x)\right]\,N(dt,dx)\Big)^2\right]\\
        &\qquad=
        \iint \Ee\left(\left[\phi_n \I_{\rrbracket\tau,\infty\llbracket}(t) f(t,x) - \phi \I_{\rrbracket\tau,\infty\llbracket}(t) f(t,x)\right]^2\right) \Ncont(dt,dx)\\
        &\qquad\leq \|\phi_n-\phi\|^2_{L^\infty(\Pp)} \iint \Ee \big[f^2(t,x)\big]\,\Ncont(dt,dx).
    \qedhere
    \end{align*}
\end{proof}

We will now consider `martingale noise' random orthogonal measures, see Example~\ref{si-07}.c), which are given by (the predictable quadratic variation of) a square-integrable martingale $M$. For these random measures our definition of the stochastic integral coincides with It\^o's definition. Recall that the It\^o integral driven by $M$ is first defined for simple, left-continuous processes of the form
\begin{equation}\label{si-e54}
    f(\omega,t):= \sum_{k=1}^n \phi_k(\omega)\I_{\rrbracket \tau_k, \tau_{k+1}\rrbracket}(\omega,t), \quad t\geq 0,
\end{equation}
where $0\leq \tau_1\leq \tau_2\leq \dots\leq \tau_{n+1}$ are bounded stopping times and $\phi_k$ bounded $\Fscr_{\tau_k}$ measurable random variables. The It\^o integral for such simple processes is
$$
    \int f(\omega,t)\,dM_t(\omega) := \sum_{k=1}^n \phi_k(\omega) \big(M_{\tau_{k+1}}(\omega) - M_{\tau_k}(\omega)\big)
$$
and it is extended by It\^o's isometry to all integrands from $L^2(\Omega\times(0,\infty),\Pscr, d\Pp\otimes d\langle M\rangle_t)$. For details we refer to any standard text on It\^o integration, e.g.\ Protter \cite[Chapter II]{protter04} or Revuz \& Yor \cite[Chapter IV]{rev-yor05}.

We will now use Lemma~\ref{si-51} in the particular situation where the space component $dx$ is not present.
\begin{theorem}\label{si-53}
    Let $N(dt)$ be a `martingale noise' random orthogonal measure induced by the square-integrable martingale $M$ \textup{(}Example~\ref{si-07}\textup{)}. The stochastic integral w.r.t.\ the random orthogonal measure $N(dt)$ and It\^o's stochastic integral w.r.t.\ $M$ coincide.
\end{theorem}
\begin{proof}
    Let $0\leq \tau_1\leq \tau_2\leq \dots\leq \tau_{n+1}$ be bounded stopping times, $\phi_k\in L^\infty(\Fscr_{\tau_k})$ bounded random variables and $f(\omega,t)$ be a simple stochastic process of the form \eqref{si-e54}.
    From Lemma~\ref{si-51} we get
    \begin{align*}
        \int f(t)\,N(dt)
        &= \int \sum_{k=1}^n \phi_k\I_{\rrbracket\tau_k,\tau_{k+1}\rrbracket}(t)\,N(dt)\\
        &= \sum_{k=1}^n \int  \phi_k\I_{\rrbracket\tau_k,\infty\llbracket}(t) \I_{\rrbracket0,\tau_{k+1}\rrbracket}(t)\,N(dt)\\
        &= \sum_{k=1}^n \phi_k \int  \I_{\rrbracket\tau_k,\infty\llbracket}(t) \I_{\rrbracket0,\tau_{k+1}\rrbracket}(t)\,N(dt)\\
        &= \sum_{k=1}^n \phi_k (M_{\tau_{k+1}}-M_{\tau_k}).
    \end{align*}
    This means that both stochastic integrals coincide on the simple stochastic processes.
    Since both integrals are extended by It\^o's isometry, the assertion follows.
\end{proof}

\begin{example}\label{si-55}
Using random orthogonal measures we can re-state the L\'evy-It\^o decomposition appearing in Theorem~\ref{rm-51}. For this, let $\widetilde N(dt,dx)$ be the Poisson random orthogonal measure (Example~\ref{si-07}.d) on $E=(0,\infty)\times(\rd\setminus\{0\})$ with control measure $dt\times\nu(dx)$ ($\nu$ is a L\'evy measure). Additionally, we \emphh{define} for all deterministic functions $h:(0,\infty)\times\rd\to\real$
$$
    \iint h(s,x)\,N(\omega,ds,dx) := \sum_{0<s<\infty} h(s,\Delta X_s(\omega))\quad\forall\omega\in\Omega
$$
provided that the sum $\sum_{0<s<\infty} |h(s,\Delta X_s(\omega))|<\infty$ for each $\omega$.\footnote{This is essentially an $\omega$-wise Riemann--Stieltjes integral. A sufficient condition for the absolute convergence is, e.g.\ that $h$ is continuous and $h(t,\cdot)$ vanishes uniformly in $t$ in some neighbourhood of $x=0$. The reason for this is the fact that $N_t(\omega,B_\epsilon^c(0))=N(\omega,(0,t]\times B_\epsilon^c(0))<\infty$, i.e.\ there are at most finitely many jumps of size exceeding $\epsilon>0$.} If $X$ is a L\'evy process with characteristic exponent $\psi$ and L\'evy triplet $(l,Q,\nu)$, then
\index{L\'evy--It\^o decomposition}%
    \begin{equation*}
    \begin{aligned}
        X_t
        &= \sqrt Q W_t \;+\: \iint \I_{(0,t]}(s) y\I_{(0,1)}(|y|)\, \Ntilde(ds,dy)
        &&\pmb{\bigg]}\text{\ $=: M_t$, \ $L^2$-martingale}\\
        &\phantom{=}\underbracket{\;\mbox{}\quad\; tl\quad \vphantom{\int}}_{\mathclap{\begin{gathered}\text{continuous}\\[-5pt]\text{Gaussian}\end{gathered}}}
        \;+ \;\; \underbracket{\iint \I_{(0,t]}(s) y\I_{\{|y|\geq 1\}}\,N(ds,dy).\quad}_{\begin{gathered}\text{pure jump part}\end{gathered}}
        &&\pmb{\bigg]}\text{\ $=: A_t$, \ bdd.\ variation}
    \end{aligned}
    \end{equation*}
\end{example}

Example~\ref{si-55} is quite particular in the sense that $N(\cdot,dt,dx)$ is a bona fide positive measure, and the control measure $\Ncont(dt,dx)$ is also the compensator, i.e.\ a measure such that $$\Ntilde((0,t]\times B) = N((0,t]\times B)-\Ncont((0,t]\times B)$$ is a square-integrable martingale.

\bigskip
Following Ikeda \& Watanabe \cite[Chapter II.4]{ikeda-watanabe89} we can generalize the set-up of Example~\ref{si-55} in the following way: Let $N(\omega,dt,dx)$ be for each $\omega$ a positive measure of space-time type. Since $t\mapsto N(\omega,(0,t]\times B)$ is increasing, there is a unique \emphh{compensator} $\Nhat(\omega,dt,dx)$ such that for all $B$ with $\Ee\Nhat((0,t]\times B)<\infty$
\index{Poisson~random measure!compensator}%
$$
    \Ntilde(\omega,(0,t]\times B) := N(\omega,(0,t]\times B)-\Nhat(\omega,(0,t]\times B),\quad t\geq 0,
$$
is a square-integrable martingale. If $t\mapsto\Nhat((0,t]\times B)$ is continuous and $B\mapsto\Nhat((0,t]\times B)$ a $\sigma$-finite measure, then one can show that the angle bracket satisfies
$$
    \big\langle\: \Ntilde((0,\cdot]\times B),\:\Ntilde((0,\cdot]\times C)\:\big\rangle_t = \Nhat\big((0,t]\times(B\cap C)\big).
$$
This means, in particular, that $\Ntilde(\omega,dt,dx)$ is a random orthogonal measure with control measure $\mu((0,t]\times B) = \Ee \Nhat((0,t]\times B)$, and we are back in the theory which we have developed in the first part of this chapter.

It is possible to develop a fully-fledged stochastic calculus for this kind of random measures.
\begin{definition}\label{si-71}
\index{semimartingale}
    Let $(\Omega,\Ascr,\Pp)$ be a probability space with a filtration $(\Fscr_t)_{t\geq 0}$. A \emphh{semimartingale} is a stochastic process $X$ of the form
    $$
       \smash[b]{X_t = X_0 + A_t + M_t + \int_0^t\!\!\int f(\cdot,s,x)\,\Ntilde(ds,dx) + \int_0^t\!\!\int g(\cdot,s,x)\,N(ds,dx)}
    $$
    ($\int_0^t := \int_{(0,t]}$) where
    \begin{itemize}
    \item $X_0$ is an $\Fscr_0$ measurable random variable,
    \item $M$ is a continuous square-integrable local martingale (w.r.t.\ $\Fscr_t$),
    \item $A$ is a continuous $\Fscr_t$ adapted process of bounded variation,
    \item $N(ds,dx)$, $\Ntilde(ds,dx)$ and $\Nhat(ds,dx)$ are as described above,
    \item $f\I_{\rrbracket 0,\tau_n\rrbracket}\in L^2(\Omega\times (0,\infty)\times X,\Pscr\otimes\Xscr,\Ncirccont)$ for some increasing sequence $\tau_n\uparrow\infty$ of bounded stopping times,
    \item $g$ is such that $\int_0^t\!\!\int g(\omega,s,x)\,N(\omega,ds,dx)$ exists as an $\omega$-wise integral,
    \item $f(\cdot,s,x)g(\cdot,s,x)\equiv 0$.
    \end{itemize}
\end{definition}

\noindent
In this case, we even have \emphh{It\^o's formula}, \index{It\^o's formula}%
see \cite[Chapter II.5]{ikeda-watanabe89}, for any $F\in \cont^2(\real,\real)$:
\begin{align*}
    F(X_t)-F(X_0)
    &= \int_0^t F'(X_{s-})\,dA_s + \int_0^t F'(X_{s-})\,dM_s + \frac 12\int_0^t F''(X_{s-})\,d\langle M\rangle_s\\
    &\quad\mbox{}+ \int_0^t\!\!\int \big[F(X_{s-}+f(s,x)) - F(X_{s-})\big]\,\Ntilde(ds,dx)\\
    &\quad\mbox{}+ \int_0^t\!\!\int \big[F(X_{s-}+g(s,x)) - F(X_{s-})\big]\,N(ds,dx)\\
    &\quad\mbox{}+ \int_0^t\!\!\int \big[F(X_{s}+f(s,x))- F(X_{s}) - f(s,x)F'(X_{s})\big]\,\Nhat(ds,dx)
\end{align*}
where we use again the convention that $\int_0^t := \int_{(0,t]}$.

\chapter{From L\'evy to Feller processes}\label{feller}

We have seen in Lemma~\ref{lpmp-23} that the semigroup $P_t f(x) := \Ee^xf(X_t) = \Ee f(X_t+x)$ of a L\'evy process $\Xt$ is a Feller semigroup. Moreover, the convolution structure of the transition semigroup $\Ee f(X_t+x) = \int f(x+y)\Pp(X_t\in dy)$ is a consequence of the spatial homogeneity (translation invariance) of the L\'evy process, see Remark~\ref{lpmp-11} and the characterization of translation invariant linear functionals (Theorem~\ref{app-91}). Lemma~\ref{lpmp-09} shows that the translation invariance of a L\'evy process is due to the assumptions \eqref{Lstat} and \eqref{Lindep}.

It is, therefore, a natural question to ask what we get if we consider stochastic processes whose semigroups are Feller semigroups which are \emphh{not} translation invariant. Since every Feller semigroup admits a Markov transition kernel (Lemma~\ref{semi-05}), we can use Kolmogorov's construction to obtain a Markov process. Thus, the following definition makes sense.
\begin{definition}\label{feller-03}
\index{Feller process}%
    A \emphh{Feller process} is a c\`adl\`ag Markov process $\Xt$, $X_t : \Omega\to\rd$, $t\geq 0$, whose transition semigroup $P_tf(x)=\Ee^xf(X_t)$ is a Feller semigroup.
\end{definition}

\begin{remark}\label{feller-04}
\index{Feller process!c\`adl\`ag modification}%
\index{L\'evy process!c\`adl\`ag modification}%
It is no restriction to require that a Feller process has c\`adl\`ag paths. By a fundamental result in the theory of stochastic processes we can construct such modifications. Usually, one argues like this: It is enough to study the coordinate processes, i.e.\ $d=1$. Rather than looking at $t\mapsto X_t$ we consider a (countable, point-separating) family of functions $u:\real\to\real$ and show that each $t\mapsto u(X_t)$ has a c\`adl\`ag modification. One way of achieving this is to use martingale regularization techniques (e.g.\ Revuz \& Yor \cite[Chapter II.2]{rev-yor05}) which means that we should pick $u$ in such a way that $u(X_t)$ is a supermartingale. The usual candidate for this is the resolvent $\eup^{-\lambda t} R_\lambda f(X_t)$ for some $f\in\cont^+_\infty(\real)$. Indeed, if $\Fscr_t = \sigma(X_s,\,s\leq t)$ is the natural filtration, $f\geq 0$ and $s\leq t$, then
\begin{align*}
    \Ee^x\big[R_\lambda f(X_t) \mid \Fscr_s\big]
    &= \Ee^{X_s} \int_0^\infty \eup^{-\lambda r} P_r f(X_{t-s}) \,dr
    =  \int_0^\infty \eup^{-\lambda r} P_{r}P_{t-s} f(X_s)\,dr \\
    &= \eup^{\lambda (t-s)}  \int_{t-s}^\infty \eup^{-\lambda u} P_{u} f(X_s)\,du
    \leq \eup^{\lambda (t-s)}  \int_{0}^\infty \eup^{-\lambda u} P_{u} f(X_s)\,du \\
    &= \eup^{\lambda (t-s)}  R_\lambda f(X_s).
\end{align*}
\end{remark}

Let $\Fscr_t = \Fscr_t^X := \sigma(X_s,\,s\leq t)$  be the canonical filtration.
\begin{lemma}\label{feller-05}
\index{Feller process!strong Markov property}%
\index{strong Markov property!of a Feller process}%
\index{Markov process!strong Markov property}%
    Every Feller process $\Xt$ is a \emphh{strong Markov process}, i.e.
    \begin{equation}\label{feller-e04}
        \Ee^x \big[f(X_{t+\tau}) \mid \Fscr_\tau\big]
        = \Ee^{X_\tau}f(X_t),\quad
        \Pp^x\text{-a.s.\ on $\{\tau<\infty\},\; t\geq 0$},
    \end{equation}
    holds for any stopping time $\tau$, $\Fscr_\tau := \{F\in\Fscr_\infty\::\: F\cap\{\tau\leq t\}\in\Fscr_t\:\:\forall t\geq 0\}$ and $f\in\cont_\infty(\rd)$.
\end{lemma}
    A routine approximation argument shows that \eqref{feller-e04} extends to $f(y)=\I_K(y)$ (where $K$ is a compact set) and then, by a Dynkin-class argument, to any $f(y)=\I_B(y)$ where $B\in\Bscr(\rd)$.
\begin{proof}
    To prove \eqref{feller-e04}, approximate $\tau$ from above by discrete stopping times $\tau_n = \big(\entier{2^n\tau}+1\big)2^{-n}$ and observe that for $F\in\Fscr_\tau \cap \{\tau<\infty\}$
    \begin{align*}
        \Ee^x [\I_F f(X_{t+\tau})]
        \omu{\text{(i)}}{=}{} \lim_{n\to\infty} \Ee^x [\I_F f(X_{t+\tau_n})]
        \omu{\text{(ii)}}{=}{} \lim_{n\to\infty} \Ee^x \big[\I_F \Ee^{X_{\tau_n}}f(X_t)\big]
        \omu{\text{(iii)}}{=}{} \Ee^x\big[\I_F \Ee^{X_\tau}f(X_t)\big].
    \end{align*}
    Here we use that $t\mapsto X_t$ is right-continuous, plus (i) dominated convergence and (iii) the Feller continuity~\ref{lpmp-21}.f); (ii) is the strong Markov property for discrete stopping times which follows directly from the Markov property: Since $\{\tau_n<\infty\}=\{\tau<\infty\}$, we get
    \begin{align*}
        \Ee^x[\I_F f(X_{t+\tau_n})]
        &= \sum_{k=1}^\infty \Ee^x \big[\I_{F\cap \{\tau_n = k2^{-n}\}} f(X_{t+ k2^{-n}})\big]\\
        &= \sum_{k=1}^\infty \Ee^x\big[\I_{F\cap \{\tau_n = k2^{-n}\}}\Ee^{X_{k2^{-n}}} f(X_t)\big]\\
        &= \Ee^x\big[\I_F \Ee^{X_{\tau_n}}f(X_t)\big].
    \end{align*}
    In the last calculation we use that $F\cap\{\tau_n=k2^{-n}\}\in\Fscr_{k2^{-n}}$ for all $F\in\Fscr_\tau$.
\end{proof}

Once we know the generator of a Feller process, we can construct many important martingales with respect to the canonical filtration of the process.
\begin{corollary}\label{feller-07}
    Let $\Xt$ be a Feller process with generator $(A,\dom(A))$ and semigroup $(P_t)_{t\geq 0}$. For every $f\in\dom(A)$ the process
    \begin{equation}\label{feller-e06}
        M_t^{[f]} := f(X_t) - \int_0^t Af(X_r)\,dr,\quad t\geq 0,
    \end{equation}
    is a martingale for the canonical filtration $\Fcal_t^X :=\sigma(X_s,\;s\leq t)$ and any $\Pp^x$, $x\in\rd$.
\end{corollary}
\begin{proof}
    Let $s\leq t$, $f\in\dom(A)$ and write, for short, $\Fcal_s := \Fcal_s^X$ and $M_t := M_t^{[f]}$.  By the Markov property
    \begin{align*}
        \Ee^x\left[M_t-M_s\mid\Fcal_s\right]
        &= \Ee^x\left[f(X_t)-f(X_s) -\int_s^t Af(X_r)\,dr\;\middle|\;\Fcal_s\right]\\
        &= \Ee^{X_s}f(X_{t-s}) - f(X_s) - \int_0^{t-s} \Ee^{X_s} Af(X_u)\,du.
    \end{align*}
    On the other hand, we get from the semigroup identity \eqref{semi-e10}
    $$
        \int_0^{t-s}\Ee^{X_s}Af(X_u)\,du
        = \int_0^{t-s} P_uAf(X_s)\,du
        = P_{t-s}f(X_s) - f(X_s)
        = \Ee^{X_s}f(X_{t-s}) - f(X_s)
    $$
    which shows that $\Ee^x\left[M_t-M_s\mid\Fcal_s\right] = 0$.
\end{proof}

Our approach from Chapter~\ref{gen} to prove the structure of a L\'evy generator `only' uses the positive maximum principle. Therefore, it can be adapted to Feller processes provided that the domain $\dom(A)$ is rich in the sense that $\cont_c^\infty(\rd)\subset\dom(A)$. All we have to do is to take into account that Feller processes are not any longer invariant under translations. The following theorem is due to Courr\`ege \cite{courrege65} and von Waldenfels \cite{waldenfels61,waldenfels65}.
\begin{theorem}[von Waldenfels, Courr\`ege]\label{feller-11}
\index{Feller process!generator}%
\index{generator!Feller process@of a Feller process}%
\index{Courr\`ege--von Waldenfels theorem}%
\index{pseudo differential operator}%
\index{symbol!of a Feller process}%
    Let $(A,\dom(A))$ be the generator of a Feller process such that $\cont_c^{\infty}(\rd) \subset \dom(A)$. Then $A|_{\cont_c^{\infty}(\rd)}$ is a \emphh{pseudo differential operator}
    \begin{equation}\label{feller-e12}
		Au(x) = -q(x,D)u(x) := - \int q(x,\xi) \widehat{u}(\xi) \eup^{\iup   x\cdot \xi} \, d\xi
	\end{equation}
	whose \emphh{symbol} $q:\rd\times\rd\to\comp$ is a measurable function of the form
    \begin{equation}\label{feller-e14}
		q(x,\xi)
        = \underbrace{q(x,0)}_{\geq 0}
             -\iup l(x)\cdot\xi + \frac 12 \xi \cdot Q(x)\xi + \int_{y\neq 0} \big[1-\eup^{\iup   y\cdot\xi}+ \iup  y\cdot\xi \I_{(0,1)}(|y|)\big] \, \nu(x,dy)
	\end{equation}
	and $(l(x),Q(x),\nu(x,dy))$ is a L\'evy triplet\footnote{Cf.\ Definition~\ref{gen-25}} for every fixed $x \in \rd$.
\end{theorem}
If we insert \eqref{feller-e14} into \eqref{feller-e12} and invert the Fourier transform we obtain the following \emphh{integro-differential representation} of the Feller generator $A$:
\index{generator!Feller process@of a Feller process}%
\begin{equation}\label{feller-e15}\begin{aligned}
		Af(x)
        &= l(x)\cdot \nabla f(x) + \frac{1}{2} \nabla\cdot Q(x) \nabla f(x) \\
        &\quad\mbox{}+ \int\limits_{y \neq 0} \big[f(x+y)-f(x)- \nabla f(x)\cdot y\I_{(0,1)}(|y|)\big]\, \nu(x,dy).
\end{aligned}\end{equation}
This formula obviously extends to all functions $f\in \cont_b^2(\rd)$. In particular, we may use the function $f(x) = \eup_\xi(x) = \eup^{\iup x\cdot\xi}$, and get\footnote{This should be compared with Definition~\ref{gen-13} and the subsequent comments.}
\begin{equation}\label{feller-e16}
		\eup_{-\xi}(x)A\eup_\xi(x)
        = -q(x,\xi).
\end{equation}

\begin{proof}[Proof of Theorem~\ref{feller-11} \textup{(}sketch\textup{)}. For a worked-out version see\ \mbox{\upshape\cite[Chapter 2.3]{schilling-lm}}]
    In the proof of Theorem~\ref{gen-21} use, instead of $A_0$ and $A_{00}$
    $$
        A_0f \rightsquigarrow A_xf := (Af)(x)
        \et
        A_{00}f \rightsquigarrow A_{xx}f := A_x(|\cdot - x|^2 f)
    $$
    for every $x\in\rd$. This is needed since $P_t$ and $A$ are not any longer translation invariant, i.e.\ we cannot shift $A_0f$ to get $A_xf$. Then follow the steps $1^\circ$--$\,4^\circ$ to get $\nu(dy) \rightsquigarrow \nu(x,dy)$ and $6^\circ$--$\,9^\circ$ for $(l(x),Q(x))$. Remark~\ref{gen-23} shows that the term $q(x,0)$ is non-negative.

    The key observation is, as in the proof of Theorem~\ref{gen-21}, that we can use in steps $3^\circ$ and $7^\circ$ the positive maximum principle\footnote{To be precise: its weakened form \eqref{PP}, cf.\ page \pageref{PP}.} to make sure that $A_xf$ is a distribution of order $2$, i.e.\
    $$
        |A_xf| = |L_xf + S_xf| \leq C_K\|f\|_{(2)}\fa f\in\cont_c^\infty(K) \text{\ \ and all compact sets\ \ } K\subset\rd.
    $$
    Here $L_x$ is the local part with support in $\{x\}$ accounting for $(q(x,0),l(x),Q(x))$, and $S_x$ is the non-local part supported in $\rd\setminus\{x\}$ giving $\nu(x,dy)$.
\end{proof}

With some abstract functional analysis we can show some (local) boundedness properties of $x\mapsto Af(x)$ and $(x,\xi)\mapsto q(x,\xi)$.
\begin{corollary}\label{feller-12}
    In the situation of Theorem~\ref{feller-11}, the condition \eqref{PP} shows that
    \begin{gather}\label{feller-e17}
        \sup_{|x|\leq r} |Af(x)| \leq C_r\|f\|_{(2)} \fa f\in \cont_c^\infty(\overline{B_r(0)})
    \intertext{and the positive maximum principle \eqref{PMP} gives}
    \label{feller-e18}
        \sup_{|x|\leq r} |Af(x)| \leq C_{r,A}\|f\|_{(2)} \fa f\in \cont_c^\infty(\rd),\; r> 0.
    \end{gather}
\end{corollary}
\begin{proof}
	In the above sketched analogue of the proof of Theorem~\ref{gen-21} we have seen that the family of linear functionals
    $$
		\left\{\cont_c^{\infty}(\overline{B_r(0)}) \ni f \mapsto A_x f\::\: x \in \overline{B_r(0)}\right\}
	$$
	where $A_xf := (Af)(x)$ satisfies
    $$
		|A_x f| \leq c_{r,x} \|f\|_{(2)}, \quad f \in \cont_c^{\infty}(\overline{B_r(0)}),
	$$
    i.e.\ $A_x: (\cont_b^2(\overline{B_r(0)}),\|\cdot\|_{(2)}) \to(\real,|\cdot|)$ is bounded. By the Banach--Steinhaus theorem (uniform boundedness principle)
    $$
		\sup_{|x| \leq r} |A_x f| \leq C_r \|f\|_{(2)}.
	$$

    Since $A$ also satisfies the positive maximum principle \eqref{PMP}, we know from step $4^\circ$ of the (suitably adapted) proof of Theorem~\ref{gen-21} that
    $$
        \int_{|y|>1}\nu(x,dy) \leq A\phi_0(x)
        \text{\ \ for some\ \ } \phi_0\in\cont_c(\overline{B_1(0)}).
    $$
    Let $r>1$, pick $\chi=\chi_r\in \cont_c^\infty(\rd)$ such that $\I_{B_{2r}(0)}\leq\chi\leq\I_{B_{3r}(0)}$. We get for $|x|\leq r$
    \begin{align*}
        Af(x)
        &= A[\chi f](x) + A[(1-\chi)f](x)\\
        &= \smash[b]{A[\chi f](x) + \int_{|y|\geq r} \big[(1-\chi(x+y))f(x+y)-\underbrace{(1-\chi(x))f(x)}_{=0}\big]\,\nu(x,dy)},
    \end{align*}
    and so
    \begin{gather*}
        \sup_{|x|\leq r}|Af(x)|
        \leq C_r\|\chi f\|_{(2)} + \|f\|_\infty \|A\phi_0\|_\infty
        \leq C_{r,A}\|f\|_{(2)}.
    \qedhere
    \end{gather*}
\end{proof}
\begin{corollary}\label{feller-13}
    In the situation of Theorem~\ref{feller-11} there exists a locally bounded nonnegative function $\gamma:\rd\to[0,\infty)$ such that
    \begin{equation}\label{feller-e19}
	   |q(x,\xi)| \leq \gamma(x) (1+|\xi|^2), \quad x,\xi \in \rd.
	\end{equation}
\end{corollary}
\begin{proof}
    Using \eqref{feller-e18} we can extend $A$ by continuity to $\cont_b^2(\rd)$ and, therefore,
    $$
		-q(x,\xi) = \eup_{-\xi}(x) A \eup_{\xi}(x),\quad \eup_\xi(x) = \eup^{\iup x\cdot\xi}
	$$
    makes sense. Moreover, we have $\sup_{|x|\leq r} |A\eup_\xi(x)|\leq C_{r,A}\|\eup_\xi\|_{(2)}$ for any $r\geq 1$; since $\|\eup_\xi\|_{(2)}$ is a polynomial of order $2$ in the variable $\xi$, the claim follows.
\end{proof}

For a L\'evy process we have $\psi(0)=0$ since $\Pp(X_t\in\rd)=1$ for all $t\geq 0$, i.e.\ the process $X_t$ does not explode in finite time. For Feller processes the situation is more complicated. We need the following technical lemmas.

\begin{lemma}\label{feller-15}
    Let $q(x,\xi)$ be the symbol of \textup{(}the generator of\textup{)} a Feller process as in Theorem~\ref{feller-11} and $F\subset \rd$ be a closed set. Then the following assertions are equivalent.
    \begin{enumerate}
	\item[\upshape a)]
        $\displaystyle |q(x,\xi)| \leq C (1+|\xi|^2)$ \ for all $x,\xi \in \rd$ where $\displaystyle C = 2\sup_{|\xi| \leq 1}\sup_{x \in F} |q(x,\xi)|$.
    \item[\upshape b)]
        $\displaystyle\sup_{x \in F} q(x,0) + \sup_{x \in F} |l(x)|+ \sup_{x \in F} \|Q(x)\|+\sup_{x \in F} \int_{y \neq 0} \frac{|y|^2}{1+|y|^2} \, \nu(x,dy)<\infty$.
	\end{enumerate}
\end{lemma}
If
\index{bounded coefficients}%
\index{symbol!bounded coefficients}%
$F=\rd$, then the equivalent properties of Lemma~\ref{feller-15} are often referred to as `the symbol has \emphh{bounded coefficients}'.
\begin{proof}[Outline of the proof \textup{(}see \textup{\cite[\emph{Appendix}]{schilling-schnurr}} for a complete proof\textup{)}]
    The direction b)$\Rightarrow$a) is proved as Theorem~\ref{gen-05}. Observe that $\xi\mapsto q(x,\xi)$ is for fixed $x$ the characteristic exponent of a L\'evy process. The finiteness of the constant $C$ follows from the assumption b) and the L\'evy--Khintchine formula \eqref{feller-e14}.

    For the converse a)$\Rightarrow$b) we note that the integrand appearing in \eqref{feller-e14} can be estimated by $c\frac{|y|^2}{1+|y|^2}$ which is itself a L\'evy exponent:
    $$
		\frac{|y|^2}{1+|y|^2} = \int \left[1-\cos(y\cdot \xi)\right] g(\xi) \, d\xi,\quad
        g(\xi) = \frac 12\int_0^\infty (2\pi\lambda)^{-d/2} \eup^{-|\xi|^2/2\lambda} \eup^{-\lambda/2}\,d\lambda.
	$$
	Therefore, by Tonelli's theorem,
    \begin{align*}
		\int_{y \neq 0} \frac{|y|^2}{1+|y|^2} \, \nu(x,dy)
		&= \iint_{y\neq 0} \left[1-\cos(y\cdot \xi)\right] \,\nu(x,dy) \, g(\xi) \, d\xi \\
		&= \int g(\xi) \left( \Re q(x,\xi)- \frac{1}{2} \xi \cdot Q(x) \xi - q(x,0) \right) \, d\xi.
    \qedhere
	\end{align*}
\end{proof}

\begin{lemma}\label{feller-17}
    Let $A$ be the generator of a Feller process, assume that $\cont_c^\infty(\rd)\subset\dom(A)$ and denote by $q(x,\xi)$ the symbol of $A$.
    For any cut-off function $\chi\in \cont_c^\infty(\rd)$ satisfying $\I_{B_1(0)}\leq\chi\leq\I_{B_2(0)}$ and $\chi_r(x) := \chi(x/r)$ one has
    \begin{gather}
    \label{feller-e20}
        \left|q(x,D)(\chi_r\eup_\xi)(x)\right|
        \leq 4\sup_{|\eta|\leq 1}|q(x,\eta)| \int_{\rd} \left[1+ r^{-2}|\rho|^2+|\xi|^2\right] \big|\widehat\chi(\rho)\big|\,d\rho,\\
    \label{feller-e22}
        \lim_{r\to\infty} A(\chi_r \eup_\xi)(x) = -\eup_\xi(x)q(x,\xi)\quad\text{for all \ } x,\xi\in\rd.
    \end{gather}
\end{lemma}
\begin{proof}
    Observe that
    $
        \widehat{\chi_r \eup_\xi}(\eta) = r^d\widehat\chi(r(\eta-\xi))
    $
    and
    \begin{equation}\label{feller-e24}
    \begin{aligned}
        q(x,D)(\chi_r\eup_\xi)(x)
        &= \int q(x,\eta) \,\eup^{\iup x\cdot\eta} \,\widehat{\chi_r \eup_\xi}(\eta)\,d\eta\\
        &= \int q(x,\eta) \,\eup^{\iup x\cdot\eta} \,r^d\,\widehat\chi(r(\eta-\xi))\,d\eta\\
        &= \int q(x,\xi + r^{-1}\rho) \,\eup^{\iup x\cdot(\xi + \rho/r)} \,\widehat\chi(\rho)\,d\rho.
    \end{aligned}
    \end{equation}
    Therefore we can use the estimate \eqref{feller-e19} with the optimal constant $\gamma(x) = 2\sup_{|\eta|\leq 1}|q(x,\eta)|$ and the elementary estimate $(a+b)^2\leq 2(a^2+b^2)$ to obtain
    \begin{align*}
        \left|q(x,D)(\chi_r\eup_\xi)(x)\right|
        &\leq \int \big|q(x,\xi+r^{-1}\rho)\big| \big|\widehat\chi(\rho)\big|\,d\rho\\
        &\leq 4\sup_{|\eta|\leq 1}|q(x,\eta)| \int\left[1+r^{-2}|\rho|^2+|\xi|^2\right] \big|\widehat\chi(\rho)\big|\,d\rho.
    \end{align*}
    This proves \eqref{feller-e20}; it also allows us to use dominated convergence in \eqref{feller-e24} to get \eqref{feller-e22}. Just observe that $\widehat\chi\in\mathcal S(\rd)$ and $\int \widehat\chi(\rho)\,d\rho = \chi(0) = 1$.
\end{proof}

\begin{lemma}\label{feller-19}
\index{symbol!continuous in $x$}%
    Let $q(x,\xi)$ be the symbol of \textup{(}the generator of\textup{)} a Feller process. Then the following assertions are equivalent:
    \begin{enumerate}
	\item[\upshape a)]
		$x \mapsto q(x,\xi)$ is continuous for all $\xi$.
	\item[\upshape b)]
        $x \mapsto q(x,0)$ is continuous.
    \item[\upshape c)]
        Tightness: $\displaystyle \lim_{r \to \infty} \sup_{x \in K}  \nu(x,\rd \setminus B_r(0)) = 0$
		\ for all compact sets $K \subset \rd$.
	\item[\upshape d)]
        Uniform continuity at the origin: $\displaystyle\lim_{|\xi| \to 0} \sup_{x \in K} |q(x,\xi)-q(x,0)| = 0$
		\ for all compact $K \subset \rd$.
	\end{enumerate}
\end{lemma}
\begin{proof}
    Let $\chi:[0,\infty)\to [0,1]$ be a decreasing $\cont^\infty$-function satisfying $\I_{[0,1)}\leq \chi \leq \I_{[0,4)}$. The functions $\chi_n(x) := \chi(|x|^2/n^2)$, $x\in\rd$ and $n\in\nat$, are radially symmetric, smooth functions with $\I_{B_n(0)}\leq\chi_n\leq\I_{B_{2n}(0)}$. Fix any compact set $K\subset\rd$ and some sufficiently large $n_0$ such that $K\subset B_{n_0}(0)$.

\medskip\noindent
    a)$\Rightarrow$b) is obvious.

\medskip\noindent
    b)$\Rightarrow$c) For $m\geq n\geq 2n_0$ the positive maximum principle implies
    $$
        \I_K(x) A(\chi_n-\chi_m)(x) \geq 0.
    $$
    Therefore, $-\I_K(x)q(x,0) = \lim_{n\to\infty}\I_K(x)A\chi_{n+n_0}(x)$ is a decreasing limit of continuous functions. Since the limit function $q(x,0)$ is continuous, Dini's theorem implies that the limit is uniform on the set $K$. From the integro-differential representation \eqref{feller-e15} of the generator we get
    $$
        \I_K(x)|A\chi_m(x)-A\chi_n(x)|
        = \I_K(x) \int_{n-n_0\leq |y|\leq 2m+n_0} \big(\chi_m(x+y)-\chi_n(x+y)\big)\,\nu(x,dy).
    $$
    Letting $m\to\infty$ yields
    \begin{align*}
        \I_K(x)|q(x,0)+A\chi_n(x)|
        &= \I_K(x) \int_{|y|\geq n-n_0} \big(1-\chi_n(x+y)\big)\,\nu(x,dy)\\
        &\geq \I_K(x) \int_{|y|\geq n-n_0} \big(1-\I_{B_{2n+n_0}(0)}(y)\big)\,\nu(x,dy)\\
        &\geq \I_K(x)\nu\left(x,B_{2n+n_0}^c(0)\right).
    \end{align*}
    where we use that $K\subset B_{n_0}(0)$ and
    $$
        \chi_n(x+y)
        \leq \I_{B_{2n}(0)}(x+y)
        = \I_{B_{2n}(0)-x}(y)
        \leq \I_{B_{2n+n_0}(0)}(y).
    $$
    Since the left-hand side converges uniformly to $0$ as $n\to\infty$, c) follows.

\medskip\noindent
    c)$\Rightarrow$d)
    Since the function $x\mapsto q(x,\xi)$ is locally bounded, we conclude from Lemma~\ref{feller-15} that $\sup_{x\in K}|l(x)|+\sup_{x\in K}\|Q(x)\|<\infty$. Thus, $\lim_{|\xi|\to 0} \sup_{x\in K}(|l(x)\cdot\xi|+|\xi\cdot Q(x)\xi|)=0$, and we may safely assume that $l\equiv 0$ and $Q\equiv 0$. If $|\xi|\leq 1$ we find, using \eqref{feller-e14} and Taylor's formula for the integrand,
    \begin{align*}
        |q(x,&\xi)-q(x,0)|\\
        &= \left|\int_{y\neq 0}\left[1-\eup^{\iup y\cdot\xi} +\iup y\cdot\xi \I_{(0,1)}(y)\right]\,\nu(x,dy)\right|\\
        &\leq \int_{0<|y|^2<1} \frac 12|y|^2|\xi|^2\,\nu(x,dy)
            + \int_{1\leq |y|^2<1/|\xi|} |y||\xi|\,\nu(x,dy)
            + \int_{|y|^2\geq 1/|\xi|} 2\,\nu(x,dy)\\
        &\leq \int_{0<|y|^2 <1/|\xi|} \frac{|y|^2}{1+|y|^2}\,\nu(x,dy) \big(1+|\xi|^{-1/2}\big)|\xi| + 2 \nu\big(x,\{y\,:\,|y|^2\geq 1/|\xi|\}\big).
    \end{align*}
    Since this estimate is uniform for $x\in K$, we get d) as $|\xi|\to 0$.

\medskip\noindent
    d)$\Rightarrow$c)
    As before, we may assume that $l\equiv 0$ and $Q\equiv 0$. For every $r>0$
    \begin{align*}
        \frac 12 \nu(x,B_r^c(0))
        &\leq \int_{|y|\geq r} \frac{|y/r|^2}{1+|y/r|^2}\,\nu(x,dy)\\
        &= \int_{|y|\geq r}\int_{\rd} \left(1-\cos\frac{\eta\cdot y}{r}\right) g(\eta)\,d\eta\,\nu(x,dy)\\
        &\leq \int_{\rd} \left[\Re q\big(x,\tfrac\eta r\big)-q(x,0)\right]\,g(\eta)\,d\eta,
    \end{align*}
    where $g(\eta)$ is as in the proof of Lemma~\ref{feller-15}. Since $\int (1+|\eta|^2) g(\eta)\,d\eta < \infty$, we can use \eqref{feller-e19} and find
    $$
        \nu(x,B_r^c(0))
        \leq c_g \sup_{|\eta|\leq 1/r} \big|\Re q(x,\eta)-q(x,0)\big|.
    $$
    Taking the supremum over all $x\in K$ and letting $r\to\infty$ proves c).

\medskip\noindent
    c)$\Rightarrow$a) From Lemma~\ref{feller-17} we know that $\lim_{n\to\infty} \eup_{-\xi}(x)A[\chi_n \eup_\xi](x) = -q(x,\xi)$. Let us show that this convergence is uniform for $x\in K$. Let $m\geq n\geq 2n_0$. For $x\in K$
    \begin{align*}
        \big|\eup_{-\xi}(x) &A[\eup_\xi \chi_n](x) - \eup_{-\xi}(x) A[\eup_\xi \chi_m](x)\big|\\
        &= \left|\int_{y\neq 0} \big[\eup_\xi(y)\chi_n(x+y) - \eup_\xi(y)\chi_m(x+y)\big]\,\nu(x,dy)\right|\\
        &\leq \int_{y\neq 0} \big[\chi_m(x+y) - \chi_n(x+y)\big]\,\nu(x,dy)\\
        &\leq \int_{n-n_0\leq |y|\leq 2m+n_0} \big[\chi_m(x+y) - \chi_n(x+y)\big]\,\nu(x,dy)\\
        &\leq \nu(x,B_{n-n_0}^c(0)).
    \end{align*}
    In the penultimate step we use that, because of the definition of the functions $\chi_n$,
    $$
        \supp \big(\chi_m(x+\cdot)-\chi_n(x+\cdot)\big)
        \subset B_{2m}(x)\setminus B_n(x)
        \subset B_{2m+n_0}(0)\setminus B_{n-n_0}(0)
    $$
    for all $x\in K\subset B_{n_0}(0)$.
    The right-hand side tends to $0$ uniformly for $x\in K$ as $n\to\infty$, hence $m\to\infty$.
\end{proof}

\begin{remark}\label{feller-20}
\index{symbol!upper semicontinuous in $x$}%
    The argument used in the first three lines of the step b)$\Rightarrow$c) in the proof of Lemma \ref{feller-19} shows, incidentally, that
    $$
        x\mapsto q(x,\xi)\quad\text{is always upper semicontinuous}
    $$
    since it is (locally) a decreasing limit of continuous functions.
\end{remark}

\begin{remark}\label{feller-21}
    Let $A$ be the generator of a Feller process, and assume that $\cont_c^\infty(\rd)\subset\dom(A)$; although $A$ maps $\dom(A)$ into $\cont_\infty(\rd)$, this is not enough to guarantee that the symbol $q(x,\xi)$ is continuous in the variable $x$. On the other hand, if the Feller process $X$ has only bounded jumps, i.e.\ if the support of the L\'evy measure $\nu(x,\cdot)$ is uniformly bounded, then $q(\cdot,\xi)$ is continuous. This is, in particular, true for diffusions.

    This follows immediately from Lemma~\ref{feller-19}.c) which holds if $\nu(x,B_r^c(0))=0$ for some $r>0$ and all $x\in\rd$.

    We can also give a direct argument: pick $\chi\in\cont_c^\infty(\rd)$ satisfying $\I_{B_{3r}(0)}\leq\chi\leq 1$. From the representation \eqref{feller-e15} it is not hard to see that
    $$
        Af(x) = A[\chi f](x)\quad\text{for all\ \ } f\in \cont_b^2(\rd) \text{\ \ and \ \ } x\in B_r(0);
    $$
    in particular, $A[\chi f]$ is continuous.

    If we take $f(x):=\eup_\xi(x)$, then, by \eqref{feller-e16},  $-q(x,\xi) = \eup_{-\xi}(x)A\eup_\xi(x) = \eup_{-\xi}(x)A[\chi\eup_\xi](x)$ which proves that $x\mapsto q(x,\xi)$ is continuous on every ball $B_r(0)$, hence everywhere.
\end{remark}

We can now discuss the role of $q(x,\xi)$ for the conservativeness of a Feller process.
\begin{theorem}\label{feller-22}
\index{symbol!continuous in $x$}%
\index{Feller process!conservativeness}%
    Let $(X_t)_{t \geq 0}$ be a Feller process with infinitesimal generator $(A,\dom(A))$ such that $\cont_c^{\infty}(\rd) \subset \dom(A)$, symbol $q(x,\xi)$ and semigroup $(P_t)_{t\geq 0}$.
    \begin{enumerate}
	\item[\upshape a)]
        If $x \mapsto q(x,\xi)$ is continuous for all $\xi \in \rd$ and $P_t 1 = 1$, then $q(x,0)=0$.
    \item[\upshape b)]
	    If $q(x,\xi)$ has bounded coefficients and $q(x,0)=0$, then $x \mapsto q(x,\xi)$ is continuous for all $\xi\in\rd$ and $P_t1=1$.
	\end{enumerate}
\end{theorem}
\begin{proof}
    Let $\chi$ and $\chi_r$, $r\in\nat$, be as in Lemma~\ref{feller-17}. Then $\eup_\xi\chi_r\in\dom(A)$ and, by Corollary~\ref{feller-07},
    $$
        M_t := \eup_\xi\chi_r(X_t) - \eup_\xi\chi_r(x) - \int_{[0,t)} A(\eup_\xi\chi_r)(X_s)\,ds,\quad t\geq 0,
    $$
    is a martingale. Using optional stopping for the stopping time
    $$
        \tau := \tau_R^x := \inf\{s\geq 0\,:\, |X_s - x|\geq R\},\quad x\in\rd,\;R>0,
    $$
    the stopped process $(M_{t\wedge\tau})_{t\geq 0}$ is still a martingale.
    Since $\Ee^x M_{t\wedge\tau}=0$, we get
    $$
        \Ee^x (\chi_r\eup_\xi)(X_{t\wedge\tau}) - \chi_r\eup_\xi(x) = \Ee^x \int_{[0,t\wedge\tau)} A (\chi_r\eup_\xi)(X_s)\,ds.
    $$
    Note that the integrand is evaluated only for times $s<t\wedge \tau$ where $|X_s|\leq R+x$. Since $A(\eup_\xi\chi_r)(x)$ is locally bounded, we can use dominated convergence and Lemma~\ref{feller-17} and we find, as $r\to\infty$,
    $$
        \Ee^x \eup_\xi(X_{t\wedge\tau}) - \eup_\xi(x) = -\Ee^x \int_{[0,t\wedge\tau)} \eup_\xi(X_s) q(X_s,\xi)\,ds.
    $$

\medskip\noindent
    a) Set $\xi=0$ and observe that $P_t1=1$ implies that $\tau=\tau_R^x\to\infty$ a.s.\ as $R\to\infty$. Therefore,
    $$
        \Pp^x(X_{t\wedge\tau}\in\rd) - 1 = -\Ee^x\int_{[0,\tau\wedge t)} q(X_s,0)\,ds,
    $$
    and with Fatou's Lemma we can let $R\to\infty$ to get
    $$
        0
        = \liminf_{R\to\infty} \Ee^x\int_{[0,\tau\wedge t)} q(X_s,0)\,ds
        \geq  \Ee^x\left[\liminf_{R\to\infty}\int_{[0,\tau\wedge t)} q(X_s,0)\,ds\right]
        = \Ee^x\int_{0}^t q(X_s,0)\,ds.
    $$
    Since $x\mapsto q(x,0)$ is continuous and $q(x,0)$ non-negative, we conclude with Tonelli's theorem that
    $$
        q(x,0) = \lim_{t\to 0} \frac 1t\int_0^t \Ee^x q(X_s,0)\,ds = 0.
    $$

\medskip\noindent
    b) Set $\xi=0$ and observe that the boundedness of the coefficients implies that
    $$
        \Pp^x(X_{t}\in\rd) - 1 = -\Ee^x\int_{[0,t)} q(X_s,0)\,ds
    $$
    as $R\to\infty$. Since the right-hand side is $0$, we get $P_t1 = \Pp^x(X_t\in\rd) = 1$.
\end{proof}

\begin{remark}\label{feller-23}
    The boundedness of the coefficients in Theorem~\ref{feller-22}.b) is important. If the coefficients of $q(x,\xi)$ grow too rapidly, we may observe explosion in finite time even if $q(x,0)=0$. A typical example in dimension $3$ is the diffusion process given by the generator
    $$
        Lf(x) = \frac 12 a(x)\Delta f(x)
    $$
    where $a(x)$ is continuous, rotationally symmetric $a(x) = \alpha(|x|)$ for a suitable function $\alpha(r)$, and satisfies $\int_1^\infty 1/\alpha(\sqrt r)\,dr < \infty$, see Stroock \& Varadhan \cite[p.\ 260, 10.3.3]{stroock-varadhan97}; the corresponding symbol is $q(x,\xi) = \frac 12 a(x)|\xi|^2$.
    This process explodes in finite time. Since this is essentially a time-changed Brownian motion
    (see B\"{o}ttcher, Schilling \& Wang \cite[Chapter 4.1]{schilling-lm}), this example works only if Brownian motion is transient, i.e.\ in dimensions $d=3$ and higher. A sufficient criterion for conservativeness in terms of the symbol is
    $$
        \liminf_{r\to\infty} \sup_{|y-x|\leq 2r} \sup_{|\eta|\leq 1/r} |q(y,\eta)| < \infty
        \quad\text{for all \ $x\in\rd$},
    $$
    see \cite[Theorem 2.34]{schilling-lm}.
\end{remark}

\chapter{Symbols and semimartingales}\label{symb}

So far, we have been treating the symbol $q(x,\xi)$ of (the generator of) a Feller process $X$ as an analytic object. On the other hand, Theorem~\ref{feller-22} indicates, that there should be some probabilistic consequences. In this chapter we want to follow this lead, show a probabilistic method to calculate the symbol and link it to the semimartingale characteristics of a Feller process. The blueprint for this is the relation of the L\'evy--It\^o decomposition (which is the semimartingale decomposition of a L\'evy process, cf.\ Theorem~\ref{rm-51}) with the L\'evy--Khintchine formula for the characteristic exponent (which coincides with the symbol of a L\'evy process, cf.\ Corollary~\ref{rm-53}).

For a L\'evy process $X_t$ with semigroup $P_tf(x) = \Ee^xf(X_t) = \Ee f(X_t+x)$ the symbol can be calculated in the following way:
\begin{equation}\label{symb-e06}
    \lim_{t \to 0} \frac{\eup_{-\xi}(x) P_t \eup_{\xi}(x)-1}{t}
	= \lim_{t \to 0} \frac{\Ee^x \eup^{\iup   \xi\cdot(X_t-x)}-1}{t}
	= \lim_{t \to 0} \frac{\eup^{-t \psi(\xi)}-1}{t}
	= - \psi(\xi).
\end{equation}

For a Feller process a similar formula is true.
\begin{theorem}\label{symb-03}
\index{symbol!probabilistic formula}%
    Let $X=(X_t)_{t \geq 0}$ be a Feller process with transition semigroup $(P_t)_{t\geq 0}$ and generator $(A,\dom(A))$ such that $\cont_c^{\infty}(\rd) \subset \dom(A)$. If $x \mapsto q(x,\xi)$ is continuous and $q$ has bounded coefficients \textup{(}Lemma~\ref{feller-15} with $F=\rd$\textup{)}, then
    \begin{equation}\label{symb-e08}
		-q(x,\xi) = \lim_{t \to 0} \frac{\Ee^x \eup^{\iup   \xi\cdot (X_t-x)}-1}{t}.
	\end{equation}
\end{theorem}
\begin{proof}
	Pick $\chi \in \cont_c^{\infty}(\rd)$, $\I_{B_1(0)} \leq \chi \leq \I_{B_2(0)}$ and set $\chi_n(x) := \chi \left( \frac{x}{n} \right)$.
	Obviously, $\chi_n \to 1$ as $n \to \infty$. By Lemma~\ref{semi-09},
    \begin{align*}
		P_t[\chi_n \eup_{\xi}](x)-\chi_n(x) \eup_{\xi}(x)
		&= \int_0^t AP_s[\chi_n \eup_{\xi}](x) \, ds \\
		&= \Ee^x \int_0^t A[\chi_n \eup_{\xi}](X_s) \, ds \\
        &= -\int_0^t\int_{\rd} \Ee^x \Big( \eup_{\eta}(X_s) q(X_s,\eta) \underbrace{\widehat{\chi_n \eup_{\xi}}(\eta)}_{\mathclap{= 
        n^d\widehat{\chi}(n(\eta-\xi))}} \Big) \,d\eta\,ds \\
		&\xrightarrow[n \to \infty]{} - \int_0^t P_s\big[\eup_{\xi} q(\cdot,\xi)\big](x) \, ds.
	\end{align*}
    Since $x\mapsto q(x,\xi)$ is continuous, we can divide by $t$ and let $t\to 0$; this yields
    \begin{gather*}
        \lim_{t\to 0} \frac 1t\left(P_t\eup_{\xi}(x)-\eup_{\xi}(x)\right) = -\eup_\xi(x)q(x,\xi).
    \qedhere
    \end{gather*}
\end{proof}

Theorem~\ref{symb-03} is a relatively simple probabilistic formula to calculate the symbol. We want to relax the boundedness and continuity assumptions. Here Dynkin's characteristic operator becomes useful.
\begin{lemma}[Dynkin's formula]
\index{Dynkin's formula}%
    Let $\Xt$ be a Feller process with semigroup $(P_t)_{t\geq 0}$ and generator $(A,\dom(A))$. For every stopping time $\sigma$ with $\Ee^x\sigma<\infty$ one has
    \begin{equation}\label{symb-e10}
        \Ee^x f(X_\sigma)-f(x) = \Ee^x\int_{[0,\sigma)} Af(X_s)\,ds,\quad
        f\in\dom(A).
    \end{equation}
\end{lemma}
\begin{proof}
    From Corollary~\ref{feller-07} we know that $M^{[f]}_t := f(X_t) - f(X_0) - \int_0^t Af(X_s)\,ds$ is a martingale; thus \eqref{symb-e10} follows from the optional stopping theorem.
\end{proof}

\begin{definition}\label{symb-05}
\index{absorbing point}%
    Let $(X_t)_{t \geq 0}$ be a Feller process. A point $a\in\rd$ is an \emphh{absorbing point}, if
    $$
        \Pp^a(X_t = a,\;\forall t\geq 0) = 1.
    $$
\end{definition}

Denote by $\tau_r := \inf\{t > 0\::\: |X_t-X_0|\geq r\}$ the first exit time from the ball $B_r(x)$ centered at the starting position $x=X_0$.
\begin{lemma}\label{symb-07}
    Let $\Xt$ be a Feller process and assume that $b\in\rd$ is not absorbing. Then there exists some $r>0$ such that $\Ee^b \tau_r < \infty$.
\end{lemma}
\begin{proof}
$1^\circ$ \
    If $b$ is not absorbing, then there is some $f\in\dom(A)$ such that $Af(b)\neq 0$. Assume the contrary, i.e.\
    $$
        Af(b) = 0\fa f\in\dom(A).
    $$
    By Lemma~\ref{semi-09}, $P_sf\in\dom(A)$ for all $s\geq 0$, and
    $$
        P_tf(b) - f(b) = \int_0^t A(P_sf)(b)\,ds = 0.
    $$
    So, $P_tf(b)=f(b)$ for all $f\in\dom(A)$. Since the domain $\dom(A)$ is dense in $\cont_\infty(\rd)$ (Remark~\ref{semi-11}), we get $P_tf(b)=f(b)$ for all $f\in\cont_\infty(\rd)$, hence $\Pp^b(X_t=b)=1$ for any $t\geq 0$. Therefore,
    $$
        \Pp^b(X_q = b,\;\forall q\in\mathds{Q},\; q\geq 0) = 1
        \et
        \Pp^b(X_t=b,\;\forall t\geq 0)=1,
    $$
    because of the right-continuity of the sample paths. This means that $b$ is an absorbing point, contradicting our assumption.

\medskip\noindent $2^\circ$ \
    Pick $f\in\dom(A)$ such that $Af(b)>0$. Since $Af$ is continuous, there exist $\epsilon>0$ and $r>0$ such that $Af|_{B_r(b)}\geq \epsilon > 0$. From Dynkin's formula \eqref{symb-e10} with $\sigma=\tau_r\wedge n$, $n\geq 1$, we deduce
    \begin{gather*}
        \epsilon\Ee^b (\tau_r\wedge n)
        \leq \Ee^b \int_{[0,\tau_r\wedge n)} Af(X_s)\,ds
        = \Ee^b f(X_{\tau_r\wedge n}) - f(b)
        \leq 2\|f\|_\infty.
    \end{gather*}
    Finally, monotone convergence shows that $\Ee^b \tau_r \leq 2\|f\|_\infty/\epsilon < \infty$.
\end{proof}

\begin{definition}[Dynkin's operator]\label{symb-09}
\index{characteristic operator}%
\index{Dynkin's operator}%
\index{generator!Dynkin's operator}%
    Let $X$ be a Feller process. The linear operator $(\Afrak,\dom(\Afrak))$ defined by
    \begin{align*}
		\Afrak f(x) &:=
        \begin{cases}
            \displaystyle\lim_{r \to 0} \frac{\Ee^x f(X_{\tau_r})-f(x)}{\Ee^x\tau_r}, &\text{if $x$ is not absorbing},\\
            \displaystyle 0, & \text{otherwise},
        \end{cases} \\
		\dom(\Afrak) &:= \left\{f \in \cont_{\infty}(\rd)\::\: \text{the above limit exists pointwise}\right\},
	\end{align*}
    is called \emphh{Dynkin's (characteristic) operator}.
\end{definition}

\begin{lemma}\label{symb-11}
    Let $\Xt$ be a Feller process with generator $(A,\dom(A))$ and characteristic operator $(\Afrak,\dom(\Afrak))$.
    \begin{enumerate}
    \item[\upshape a)] $\Afrak$ is an extension of $A$, i.e.\ $\dom(A)\subset \dom(\Afrak)$ and $\Afrak|_{\dom(A)}=A$.
    \item[\upshape b)] $(\Afrak,\dom)=(A,\dom(A))$ \ if \ $\dom = \{f\in \dom(\Afrak)\,:\, f,\: \Afrak f\in \cont_\infty(\rd)\}$.
    \end{enumerate}
\end{lemma}
\begin{proof}
a)
    Let $f\in\dom(A)$ and assume that $x\in\rd$ is not absorbing. By Lemma~\ref{symb-07} there is some $r=r(x)>0$ with $\Ee^x \tau_r < \infty$. Since we have $Af\in \cont_\infty(\rd)$, there exists for every $\epsilon>0$ some $\delta>0$ such that
    $$
        |Af(y)-Af(x)| < \epsilon\fa y\in B_\delta(x).
    $$
    Without loss of generality let $\delta < r$. Using Dynkin's formula \eqref{symb-e10} with $\sigma = \tau_\delta$, we see
    \begin{align*}
        \big| \Ee^x f(X_{\tau_\delta}) - f(x) - Af(x)\Ee^x \tau_\delta\big|
        \leq \Ee^x \int_{[0,\tau_\delta)} \underbrace{|Af(X_s)-Af(x)|}_{\leq \epsilon}\,ds
        \leq \epsilon \Ee^x \tau_\delta.
    \end{align*}
    Thus, $\lim_{r\to 0} \big(\Ee^x f(X_{\tau_r}) - f(x)\big)\big/\Ee^x \tau_r = Af(x)$.

\smallskip
    If $x$ is absorbing and $f\in\dom(A)$, then $Af(x)=0$ and so $Af(x)=\Afrak f(x)$.

\medskip\noindent
b)
    Since $(\Afrak,\dom)$ satisfies the 
    \eqref{PMP}, the claim follows from Lemma~\ref{semi-29}.
\end{proof}

\begin{theorem}\label{symb-31}
\index{symbol!probabilistic formula}%
    Let $(X_t)_{t \geq 0}$ be a Feller process with infinitesimal generator $(A,\dom(A))$ such that $\cont_c^{\infty}(\rd) \subset \dom(A)$ and $x\mapsto q(x,\xi)$ is continuous\footnote{For instance, if $X$ has bounded jumps, see Lemma~\ref{feller-19} and Remark~\ref{feller-21}. Our proof will show that it is actually enough to assume that $s\mapsto q(X_s,\xi)$ is right-continuous.}. Then
    \begin{equation}\label{symb-e32}
		-q(x,\xi) = \lim_{r \to 0} \frac{\Ee^x \eup^{\iup(X_{\tau_r}-x)\cdot\xi}-1}{\Ee^x \tau_r}
	\end{equation}
	for all $x\in\rd$ \textup{(}as usual, $\frac 1\infty := 0$\textup{)}. In particular, $q(a,\xi)=0$ for all absorbing states $a\in\rd$.
\end{theorem}

\begin{proof}
    Let $\chi_n\in\cont_c^\infty(\rd)$ such that $\I_{B_n(0)}\leq\chi_n\leq 1$. By Dynkin's formula \eqref{symb-e10}
    \begin{align*}
		\eup_{-\xi}(x) \Ee^x[\chi_n(X_{\tau_r\wedge t}) \eup_{\xi}(X_{\tau_r\wedge t})] - \chi_n(x)
        &= \Ee^x \int_{[0,\tau_r\wedge t)} \eup_{-\xi}(x) A[\chi_n \eup_\xi](X_s)\,ds.
    \end{align*}
    Observe that $A[\chi_n \eup_\xi](X_s)$ is bounded if $s<\tau_r$, see Corollary~\ref{feller-12}. Using the dominated convergence theorem, we can let $n\to\infty$ to get
    \begin{equation}\label{symb-e34}
    \begin{aligned}
		\eup_{-\xi}(x) \Ee^x \eup_{\xi}(X_{\tau_r\wedge t}) - 1
        &\pomu{\eqref{feller-e16}}{=}{} \Ee^x\int_{[0,\tau_r\wedge t)} \eup_{-\xi}(x) A \eup_\xi(X_s)\,ds\\
        &\omu{\eqref{feller-e16}}{=}{}  -\Ee^x\int_{[0,\tau_r\wedge t)} \eup_\xi(X_s-x)q(X_s,\xi)\,ds.
    \end{aligned}
    \end{equation}
    If $x$ is absorbing, we have $q(x,\xi)=0$, and \eqref{symb-e32} holds trivially. For non-absorbing $x$, we pass to the limit $t\to \infty$ and get, using $\Ee^x \tau_r< \infty$ (see Lemma~\ref{symb-07}),
    \begin{align*}
		\frac{\eup_{-\xi}(x) \Ee^x \eup_{\xi}(X_{\tau_r}) - 1}{\Ee^x \tau_r}
        = -\frac 1{\Ee^x\tau_r} \Ee^x \int_{[0,\tau_r)} \eup_\xi(X_s-x)q(X_s,\xi)\,ds.
    \end{align*}
    Since $s\mapsto q(X_s,\xi)$ is right-continuous at $s=0$, the limit $r\to 0$ exists, and \eqref{symb-e32} follows.
\end{proof}

A small variation of the above proof yields
\begin{corollary}[Schilling, Schnurr \cite{schilling-schnurr}]\label{symb-33}
\index{symbol!probabilistic formula}%
    Let $X$ be a Feller process with generator $(A,\dom(A))$ such that $\cont_c^{\infty}(\rd) \subset \dom(A)$ and $x\mapsto q(x,\xi)$ is continuous. Then
    \begin{equation}\label{symb-e36}
		-q(x,\xi) = \lim_{t \to 0} \frac{\Ee^x \eup^{\iup (X_{t\wedge\tau_r} - x)\cdot\xi}-1}{t}
	\end{equation}
	for all $x\in\rd$ and $r>0$.
\end{corollary}
\begin{proof}
    We follow the proof of Theorem~\ref{symb-31} up to \eqref{symb-e34}. This relation can be rewritten as
    $$
        \frac{\Ee^x \eup^{\iup (X_{t\wedge\tau_r}-x)\cdot\xi}-1}{t}
        = -\frac 1t\Ee^x\int_0^t \eup_\xi(X_s-x)q(X_s,\xi)\I_{[0,\tau_r)}(s)\,ds.
    $$
    Observe that $X_s$ is bounded if $s<\tau_r$ and that $s\mapsto q(X_s,\xi)$ is right-continuous. Therefore, the limit $t\to 0$ exists and yields \eqref{symb-e36}.
\end{proof}

Every Feller process $\Xt$ such that $\cont_c^\infty(\rd)\subset\dom(A)$ is a semimartingale. Moreover, the semimartingale characteristics can be expressed in terms of the L\'evy triplet $(l(x),Q(x),\nu(x,dy))$ of the symbol. Recall that a ($d$-dimensional) \emphh{semimartingale} is a stochastic process of the
\index{semimartingale}%
form
$$
        X_t = X_0 + X_t^{\textsf{c}} + \int_0^t\!\!\int y\I_{(0,1)}(|y|)\big[\mu^X(\cdot,ds,dy)-\nu(\cdot,ds,dy)\big] + \sum_{s\leq t} \I_{[1,\infty)}(|\Delta X_s|)\Delta X_s + B_t
$$
where $X^{\textsf{c}}$ is the continuous martingale part, $B$ is a previsible process with paths of finite variation (on compact time intervals) and with the jump measure
$$
    \mu^X(\omega,ds,dy) = \sum_{s\,:\,\Delta X_s(\omega)\neq 0} \delta_{(s,\Delta X_s(\omega))}(ds,dy)
$$
whose compensator is $\nu(\omega,ds,dy)$. The triplet $(B,C,\nu)$ with the (predictable) quadratic variation $C=[X^\textsf{c},X^\textsf{c}]$ of $X^\textsf{c}$ is called the \emphh{semimartingale characteristics}.
\begin{theorem}[Schilling \cite{schilling-ptrf}, Schnurr \cite{schnurr12}]\label{symb-35}
\index{Feller process!semimartingale characteristics}
    Let $(X_t)_{t \geq 0}$ be a Feller process with infinitesimal generator $(A,\dom(A))$ such that $\cont_c^{\infty}(\rd) \subset \dom(A)$ and symbol $q(x,\xi)$ given by \eqref{feller-e14}. If $q(x,0)=0$,\footnote{A sufficient condition is, for example, that $X$ has infinite life-time and $x\mapsto q(x,\xi)$ is continuous (either for all $\xi$ or just for $\xi=0$), cf.\ Theorem~\ref{feller-22} and Lemma~\ref{feller-19}.} then $X$ is a semimartingale whose semimartingale characteristics can be expressed by the L\'evy triplet $(l(x),Q(x),\nu(x,dy))$
    \begin{align*}
        B_t = \int_0^t l(X_s)\,ds,\quad
        C_t = \int_0^t Q(X_s)\,ds,\quad
        \nu(\cdot,ds,dy) = \nu(X_s(\cdot),dy)\,ds.
    \end{align*}
\end{theorem}
\begin{proof}
    $1^\circ$ \ Combining Corollary~\ref{feller-07} with the integro-differential representation \eqref{feller-e15} of the generator shows that
    \begin{align*}
        M_t^{[f]}
        &= f(X_t) - \int_{[0,t)} Af(X_s)\,ds\\
        &= f(X_t)
            - \int_{[0,t)} l(X_s)\cdot\nabla f(X_s)\,ds
            - \frac 12\int_{[0,t)} \nabla\cdot Q(X_s)\nabla f(X_s)\,ds\\
        &\qquad\mbox{} - \int_{[0,t)} \int_{y\neq 0} \left[f(X_s+y)-f(X_s)-\nabla f(X_s)\cdot y\I_{(0,1)}(|y|)\right]\nu(X_s,dy)\,ds
    \end{align*}
    is for all $f\in\cont_b^2(\rd)\cap\dom(A)$ a martingale.

\medskip\noindent $2^\circ$ \
    We claim that $\cont_c^2(\rd)\subset\dom(A)$.
    Indeed, let $f\in \cont_c^2(\rd)$ with $\supp f\subset B_r(0)$ for some $r>0$ and pick a sequence $f_n\in \cont_c^\infty(B_{2r}(0))$ such that $\lim_{n\to\infty}\|f-f_n\|_{(2)} = 0$. Using \eqref{feller-e17} we get
    $$
        \sup_{|x|\leq 3r}|Af_n(x) - Af_m(x)| \leq c_{3r}\|f_n-f_m\|_{(2)} \xrightarrow[m,n\to\infty]{}0.
    $$
    Since $\supp f_n\subset B_{2r}(0)$ and $f_n \to f$ uniformly, there is some $u\in \cont_c^\infty(B_{3r}(0))$ with $|f_n(x)|\leq u(x)$. Therefore, we get for $|x|\geq 2r$
    \begin{align*}
        |Af_n(x)-Af_m(x)|
        &\leq \int_{y\neq 0} |f_n(x+y)-f_m(x+y)|\,\nu(x,dy)\\
        &\leq 2\int_{y\neq 0} u(x+y)\,\nu(x,dy)
        = 2Au(x)
        \xrightarrow[|x|\to\infty]{}0
    \end{align*}
    uniformly for all $m,n$. This shows that $(Af_n)_{n\in\nat}$ is a Cauchy sequence in $\cont_\infty(\rd)$. By the closedness of $(A,\dom(A))$ we gather that $f\in\dom(A)$ and $Af = \lim_{n\to\infty}Af_n$.

\medskip\noindent $3^\circ$ \
    Fix $r>1$, pick $\chi_r \in \cont_c^\infty(\rd)$ such that $\I_{B_{3r}(0)}\leq\chi\leq 1$, and set
    $$
        \sigma = \sigma_r := \inf\{t>0\::\: |X_t-X_0|\geq r\}\wedge\inf\{t>0\::\: |\Delta X_t|\geq r\}.
    $$
    Since $\Xt$ has c\`adl\`ag paths and infinite life-time, $\sigma_r$ is a family of stopping times with $\sigma_r \uparrow \infty$.
    For any $f\in \cont^2(\rd)\cap \cont_b(\rd)$ and $x\in\rd$ we set $f_r := \chi_r f$, $f^x := f(\cdot-x)$, $f_{r}^{x}(\cdot-x)$, and consider
    \begin{align*}
        M_{t\wedge\sigma}^{[f^x]}
        &= f^x(X_{t\wedge\sigma})
            - \int_{[0,{t\wedge\sigma})} l(X_s)\cdot\nabla f^x(X_s)\,ds
            - \frac 12\int_{[0,{t\wedge\sigma})} \nabla\cdot Q(X_s)\nabla f^x(X_s)\,ds\\
        &\qquad\mbox{} - \int_{[0,{t\wedge\sigma})} \int_{y\neq 0} \left[f^x(X_s+y)-f^x(X_s)-\nabla f^x(X_s)\cdot y\I_{(0,1)}(|y|)\right]\nu(X_s,dy)\,ds\\
        &= f^x_r(X_{t\wedge\sigma})
            - \int_{[0,{t\wedge\sigma})} l(X_s)\cdot\nabla f^x_r(X_s)\,ds
            - \frac 12\int_{[0,{t\wedge\sigma})} \nabla\cdot Q(X_s)\nabla f^x_r(X_s)\,ds\\
        &\qquad\mbox{} - \int_{[0,{t\wedge\sigma})} \int_{0<|y|< r} \left[f^x_r(X_s+y)-f^x_r(X_s)-\nabla f^x_r(X_s)\cdot y\I_{(0,1)}(|y|)\right]\nu(X_s,dy)\,ds\\
        &= M_{t\wedge\sigma}^{[f^x_r]}.
    \end{align*}
    Since $f_r\in \cont_c^2(\rd)\subset\dom(A)$, we see that $M_t^{[f^x]}$ is a local martingale (with reducing sequence $\sigma_r$, $r>0$), and by a general result of Jacod \& Shiryaev \cite[Theorem II.2.42]{jacod-shiryaev87}, it follows that $X$ is a semimartingale with the characteristics mentioned in the theorem.
\end{proof}

We close this chapter with the discussion of a very important example: L\'evy driven stochastic differential equations. From now on we assume that
\begin{gather*}
    \Phi : \rd \to \real^{d\times n}\text{\ \ is a matrix-valued, measurable function},\\
    L = (L_t)_{t\geq 0} \text{\ \ is an $n  $-dimensional L\'evy process with exponent $\psi(\xi)$},
\end{gather*}
and consider the following It\^o stochastic differential equation (SDE)
\begin{equation}\label{symb-e51}
    dX_t = \Phi(X_{t-})\,dL_t,\quad X_0=x.
\end{equation}
If $\Phi$ is globally Lipschitz continuous, then the SDE \eqref{symb-e51} has a unique solution which is a strong Markov process, see Protter~\cite[Chapter V, Theorem 32]{protter04}\footnote{Protter requires that $L$ has independent coordinate processes, but this is not needed in his proof. For existence and uniqueness the local Lipschitz and linear growth conditions are enough; the strong Lipschitz condition is used for the Markov nature of the solution.}. If we write $X^x_t$ for the solution of \eqref{symb-e51} with initial condition $X_0= x=X_0^x$, then the flow $x\mapsto X_t^x$ is continuous \cite[Chapter V, Theorem 38]{protter04}.

If we use $L_t = (t,W_t,J_t)^\top$ as driving L\'evy process where $W$ is a Brownian motion and $J$ is a pure-jump L\'evy process (we assume\footnote{This assumption can be relaxed under suitable assumptions on the (joint) filtration, see for example Ikeda \& Watanabe \cite[Theorem II.6.3]{ikeda-watanabe89}} that $W\bbot J$), and if $\Phi$ is a block-matrix, then we see that \eqref{symb-e51} covers also SDEs of the form
$$
    dX_t = f(X_{t-})\,dt + F(X_{t-})\,dW_t + G(X_{t-})\,dJ_t.
$$

\begin{lemma}\label{symb-52}
    Let $\Phi$ be bounded and Lipschitz, $X$ the unique solution of the SDE \eqref{symb-e51}, and denote by $A$ the generator of $X$. Then $\cont_c^{\infty}(\rd) \subset \dom(A)$.
\end{lemma}
\begin{proof}
    Because of Theorem~\ref{semi-31} (this theorem does not only hold for Feller semigroups, but for any strongly continuous contraction semigroup satisfying the positive maximum principle), it suffices to show
    $$
		\lim_{t\to 0} \frac{1}{t} \big( \Ee^x f(X_t)-f(x) \big) = g(x)
        \et g\in \cont_{\infty}(\rd)
	$$
	for any $f \in \cont_c^{\infty}(\rd)$. For this we use, as in the proof of the following theorem, It\^o's formula to get
    $$
		\Ee^x f(X_t)-f(x) = \Ee^x \int_0^t Af(X_s) \, ds,
	$$
    and a calculation similar to the one in the proof of the next theorem. A complete proof can be found in Schilling \& Schnurr
    \cite[Theorem 3.5]{schilling-schnurr}.
\end{proof}

\begin{theorem}\label{symb-55}
\index{symbol!of an SDE}%
    Let $(L_t)_{t \geq 0}$ be an $n$-dimensional L\'evy process with exponent $\psi$ and assume that $\Phi$ is Lipschitz continuous. Then the unique Markov solution $X$ of the SDE \eqref{symb-e51} admits a generalized symbol in the sense that
    $$
		\lim_{t\to 0} \frac{\Ee^x \eup^{\iup (X_{t\wedge\tau_r}-x)\cdot\xi}-1}{t} = -\psi(\Phi^\top(x) \xi),\quad r>0.
	$$
    If $\Phi \in \cont_b^1(\rd)$, then $X$ is Feller, $\cont_c^{\infty}(\rd) \subset \dom(A)$ and $q(x,\xi)=\psi(\Phi^\top(x)\xi)$ is the symbol of the generator.
\end{theorem}

Theorem~\ref{symb-55} indicates that the formulae \eqref{symb-e32} and \eqref{symb-e36} may be used to associate symbols not only with Feller processes but with more general Markov semimartingales. This has been investigated by Schnurr~\cite{schnurr-phd} and \cite{schnurr12} who shows that the class of It\^o processes with jumps is essentially the largest class that can be described by symbols; see also \cite[Chapters 2.4--5]{schilling-lm}. This opens up the way to analyze rather general semimartingales using the symbol. Let us also point out that the boundedness of $\Phi$ is only needed to ensure that $X$ is a Feller process.
\begin{proof}
    Let $\tau = \tau_r$ be the first exit time for the process $X$ from the ball $B_r(x)$ centered at $x=X_0$. We use the L\'evy--It\^o decomposition \eqref{rm-e52} of $L$. From It\^o's formula for jump processes (see, e.g.\ Protter~\cite[Chapter II, Theorem 33]{protter04}) we get
    \begin{align*}
		\frac{1}{t} \Ee^x &\left(\eup_{\xi}(X_{t\wedge\tau}-x)-1\right)\\
		&= \frac{1}{t} \Ee^x \bigg[ \int_0^{t \wedge \tau} \iup\eup_{\xi}(X_{s-}-x)\xi \, dX_s
            - \frac{1}{2} \int_0^{t \wedge \tau} \eup_{\xi}(X_{s-}-x) \xi\cdot d[X,X]_s^\textsf{c} \xi \\
		&\qquad\qquad\qquad\mbox{} + \eup_{-\xi}(x)\sum_{s \leq \tau \wedge t} \big(\eup_{\xi}(X_s)-\eup_{\xi}(X_{s-})-\iup \eup_{\xi}(X_{s-})
            \xi\cdot \Delta X_s\big)\bigg] \\
		&=: I_1+I_2+I_3.
	\end{align*}
	We consider the three terms separately.
    \begin{align*}
		I_1
		&= \frac{1}{t} \Ee^x \int_0^{t\wedge\tau} \iup\eup_{\xi}(X_{s-}-x) \xi\, dX_s\\
		&= \frac{1}{t} \Ee^x \int_0^{t\wedge\tau} \iup\eup_{\xi}(X_{s-}-x) \xi\cdot\Phi(X_{s-})\,dL_s\\
        &= \frac{1}{t} \Ee^x \int_0^{t\wedge\tau} \iup\eup_{\xi}(X_{s-}-x) \xi\cdot\Phi(X_{s-}) \,d_s\!\left[ls + \int_{|y|\geq 1} y\,N_s(dy)\right]\\
		&=: I_{11} + I_{12}
	\end{align*}
	where we use that the diffusion part and the compensated small jumps of a L\'evy process are a martingale, cf.\ Theorem~\ref{rm-51}. Further,
    \begin{align*}
	I_3&+I_{12}\\
        &= \frac{1}{t} \Ee^x \int_0^{t\wedge\tau}\int_{y\neq 0} \eup_{\xi}(X_{s-}-x) \left[ \eup_{\xi}(\Phi(X_{s-})y)-1- \iup  \xi\cdot \Phi(X_{s-}) y \I_{(0,1)}(|y|) \right]\,d_sN_s(dy)\\
        &= \frac{1}{t} \Ee^x \int_0^{t \wedge \tau} \int_{y\neq 0} \eup_{\xi}(X_{s-}-x) \left[ \eup_{\xi}(\Phi(X_{s-}) y)-1-\iup  \xi \Phi(X_{s-})y \I_{(0,1)}(|y|)\right] \, \nu(dy) \, ds\\
		&\xrightarrow[t \to 0]{} \int_{y\neq 0} \big[\eup^{\iup \xi\cdot \Phi(x) y}-1-\iup  \xi\cdot \Phi(x)y \I_{(0,1)}(|y|)\big] \, \nu(dy).
	\end{align*}
    Here we use that $\nu(dy)\,ds$ is the compensator of $d_sN_s(dy)$, see Lemma~\ref{rm-e18}. This requires that the integrand is $\nu$-integrable, but this is ensured by the local boundedness of $\Phi(\cdot)$ and the fact that $\int_{y\neq 0}\min\{|y|^2,1\}\,\nu(dy)<\infty$. Moreover,
    \begin{align*}
        I_{11}
        = \frac{1}{t} \Ee^x \int_0^{t\wedge\tau} \iup\eup_{\xi}(X_{s-}-x) \xi\cdot\Phi(X_{s-}) l\,ds
        \xrightarrow[t\to 0]{} \iup\xi\cdot\Phi(x)l,
    \end{align*}
    and, finally, we observe that
    \begin{align*}
        [X,X]^{\textsf{c}}
        &= \Big[\int\Phi(X_{s-})\,dL_s,\: \int\Phi(X_{s-})\,dL_s\Big]^{\textsf{c}}\\
        &= \int \Phi(X_{s-})\,d[L,L]_s^\textsf{c}\,\Phi^\top(X_{s-})\\
        &= \int \Phi(X_{s-}) \,Q\, \Phi^\top(X_{s-})\, ds
        \:\in\:\real^{d\times d}
    \end{align*}
    which gives
    \begin{align*}
        I_{2}
        &= - \frac{1}{2t} \Ee^x \int_0^{t \wedge \tau} \eup_{\xi}(X_{s-}-x) \xi\cdot d[X,X]_s^\textsf{c} \xi \\
        &= - \frac{1}{2t} \Ee^x \int_0^{t \wedge \tau} \eup_{\xi}(X_{s-}-x) \xi\cdot\Phi(X_{s-})Q\Phi^\top(X_{s-})\xi\,ds\\
        &\xrightarrow[t\to 0]{} -\frac 12 \xi\cdot\Phi(x)Q\Phi^\top(x)\xi.
    \end{align*}
    This proves
    \begin{align*}
        q(x,\xi)
        &= -\iup l\cdot\Phi^\top(x)\xi + \frac 12 \xi\cdot \Phi(x) Q \Phi^\top(x)\xi\\
        &\qquad\mbox{} + \int_{y\neq 0} \big[1-\eup^{\iup y\cdot\Phi^\top(x)\xi}+\iup y\cdot\Phi^\top(x)\xi\I_{(0,1)}(|y|)\big]\,\nu(dy)\\
        &= \psi(\Phi^\top(x)\xi).
    \end{align*}
    For the second part of the theorem we use Lemma~\ref{symb-52} to see that $\cont_c^\infty(\rd)\subset\dom(A)$. The continuity of the flow $x\mapsto X_t^x$ (Protter \cite[Chapter V, Theorem 38]{protter04})---$X^x$ is the unique solution of the SDE with initial condition $X_0^x=x$---ensures that the semigroup $P_tf(x):=\Ee^x f(X_t)=\Ee f(X_t^x)$ maps $f\in\cont_\infty(\rd)$ to $\cont_b(\rd)$. In order to see that $P_tf\in\cont_\infty(\rd)$ we need that $\lim_{|x|\to\infty} |X_t^x|=\infty$ a.s. This requires some longer calculations, see e.g.\ Schnurr \cite[Theorem 2.49]{schnurr-phd} or Kunita \cite[Proof of Theorem 3.5, p.~353, line 13 from below]{kunita04} (for Brownian motion this argument is fully worked out in Schilling \& Partzsch \cite[Corollary 19.31]{schilling-bm}).
\end{proof}

\chapter{D\'enouement}\label{den}

It is well known that the characteristic exponent $\psi(\xi)$ of a L\'evy process $L=(L_t)_{t\geq 0}$ can be used to describe many probabilistic properties of the process. The key is the formula
\begin{equation}\label{den-e02}
    \Ee^x \eup^{\iup \xi\cdot (L_t-x)}
    =     \Ee \eup^{\iup \xi\cdot L_t}
    = \eup^{-t\psi(\xi)}
\end{equation}
which gives access to the Fourier transform of the transition function $\Pp^x(L_t\in dy) = \Pp(L_t+x\in dy)$. Although it is not any longer true that the symbol $q(x,\xi)$ of a Feller process $X=\Xt$ is the characteristic \emph{exponent}, we may interpret formulae like \eqref{symb-e32}
$$
		-q(x,\xi) = \lim_{r \to 0} \frac{\Ee^x \eup^{\iup(X_{\tau_r}-x)\cdot\xi}-1}{\Ee^x \tau_r}
$$
as infinitesimal versions of the relation \eqref{den-e02}. What is more, both $\psi(\xi)$ and $q(x,\xi)$ are the Fourier symbols of the generators of the processes. We have already used these facts to discuss the conservativeness of Feller processes (Theorem~\ref{feller-22}) and the semimartingale decomposition of Feller processes (Theorem~\ref{symb-35}).

It is indeed possible to investigate further path properties of a Feller process using its symbol $q(x,\xi)$. Below we will, mostly without proof, give some examples which are taken from B\"{o}ttcher, Schilling \& Wang \cite{schilling-lm}. Let us point out the two guiding principles.
\begin{quote}\itshape
$\blacktriangleright$
    For sample path properties, the symbol $q(x,\xi)$ of a Feller process assumes same role as the characteristic exponent $\psi(\xi)$ of a L\'evy process.

\noindent
$\blacktriangleright\blacktriangleright$
    A Feller process is `locally L\'evy', i.e.\ for short-time path properties the Feller process, started at $x_0$, behaves like the L\'evy process $(L_t+x_0)_{t\geq 0}$ with characteristic exponent $\psi(\xi):=q(x_0,\xi)$.
\end{quote}
The latter property is the reason why such Feller processes are often called \emphh{L\'evy-type processes}.
\index{L\'evy-type process}%
The model case is the stable-like
\index{stable-like process}%
process whose symbol is given by $q(x,\xi) = |\xi|^{\alpha(x)}$ where $\alpha:\rd\to(0,2)$ is sufficiently smooth\footnote{In dimension $d=1$ Lipschitz or even Dini continuity is enough (see Bass~\cite{bass88}), in higher dimensions we need something like $\cont^{5d+3}$-smoothness, cf.~Hoh~\cite{hoh00}. Meanwhile, K\"{u}hn \cite{kuehn-diss} established the existence for $d\geq 1$ with $\alpha(x)$ satisfying any H\"{o}lder condition.}. This process behaves locally, and for short times $t\ll 1$, like an $\alpha(x)$-stable process, only that the index now depends on the starting point $X_0=x$.

The key to many path properties are the following \emphh{maximal estimates}
which were first proved in \cite{schilling-ptrf}. The present proof is taken from \cite{schilling-lm}, the observation that we may use a random time $\tau$ instead of a fixed time $t$ is due to F.\ K\"{u}hn.
\begin{theorem}\label{den-11}
\index{Feller process!maximal estimate}%
    Let $(X_t)_{t \geq 0}$ be a Feller process with generator $A$, $\cont_c^{\infty}(\rd) \subset \dom(A)$ and symbol $q(x,\xi)$. If $\tau$ is an integrable stopping time, then
    \begin{equation}\label{den-e10}
		\Pp^x \left( \sup_{s \leq \tau} |X_s-x| > r\right)
        \leq c \,\Ee^x\tau \,\sup_{|y-x| \leq r} \sup_{|\xi| \leq r^{-1}} |q(y,\xi)|.
	\end{equation}
\end{theorem}

\begin{proof}
	Denote by $\sigma_r = \sigma_r^x$ the first exit time from the closed ball $\overline{B_r(x)}$. Clearly,
    $$
		\left\{\sigma_r^x<\tau\right\}
        \subset \left\{ \sup_{s \leq \tau} |X_s-x| > r \right\}
        \subset \left\{\sigma_r^x \leq \tau\right\}.
	$$
	Pick $u \in \cont_c^{\infty}(\rd)$, $0 \leq u \leq 1$, $u(0)=1$, $\supp u \subset B_1(0)$, and set
    $$
		u_r^x(y) := u \left( \frac{y-x}{r} \right).
	$$
	In particular, $u_r^x|_{B_r(x)^c} = 0$. Hence,
    $$
		\I_{\{\sigma_r^x \leq \tau\}} \leq 1- u_r^x(X_{\tau\wedge\sigma_r^x}).
	$$
	Now we can use \eqref{semi-e10} or \eqref{symb-e10} to get
    \begin{align*}
		\Pp ^x(\sigma_r^x \leq \tau)
		&\pomu{\eqref{feller-e19}}{\leq}{\text{L~\ref{feller-15}}}
            \Ee^x \left[1-u_r^x(X_{\tau\wedge\sigma_r^x})\right]\\
		&\pomu{\eqref{feller-e19}}{=}{\text{L~\ref{feller-15}}}
            \Ee^x \int_{[0,\tau \wedge \sigma_r^x)} q(X_s,D) u_r^x(X_s) \, ds \\
		&\pomu{\eqref{feller-e19}}{=}{\text{L~\ref{feller-15}}}
            \Ee^x \int_{[0,\tau \wedge \sigma_r^x)} \int \I_{B_r(x)}(X_s) \eup_{\xi}(X_{s}) q(X_s,\xi) \widehat{u}_r^x(\xi) \, d\xi\,ds\\
		&\pomu{\eqref{feller-e19}}{\leq}{\text{L~\ref{feller-15}}}
            \Ee^x \int_{[0,\tau \wedge \sigma_r^x)}\int \sup_{|y-x| \leq r} |q(y,\xi)| |\widehat{u}_r^x(\xi)| \, d\xi\,  ds\\
		&\pomu{\eqref{feller-e19}}{=}{\text{L~\ref{feller-15}}}
            \Ee^x\left[\tau \wedge \sigma_r^x\right] \int \sup_{|y-x| \leq r} |q(y,\xi)| |\widehat{u}_r^x(\xi)| \, d\xi \\
		&\omu{\eqref{feller-e19}}{\leq}{\text{L~\ref{feller-15}}}
            c \,\Ee^x\tau \sup_{|y-x| \leq r}\, \sup_{|\xi| \leq r^{-1}} |q(y,\xi)| \int (1+|\xi|^2) |\widehat{u}(\xi)| \, d\xi.
    \qedhere
	\end{align*}
\end{proof}

There is a also a probability estimate for $\sup_{s \leq t} |X_s-x| \leq r$, but this requires some \emphh{sector condition} for the symbol, that
\index{symbol!sector condition}%
is an estimate of the form
\begin{equation}\label{den-e12}
		|\Im q(x,\xi)|
        \leq \kappa \Re q(x,\xi), \qquad x,\xi \in \rd.
\end{equation}
One consequence of \eqref{den-e12} is that the drift (which is contained in the imaginary part of the symbol) is not dominant. This is a familiar assumption in the study of path properties of L\'evy processes, see e.g.\ Blumenthal \& Getoor \cite{blu-get61}; a typical example where this condition is violated are (L\'evy) symbols of the form $\psi(\xi) = \iup  \xi + |\xi|^{\frac{1}{2}}$. For a L\'evy process the sector condition on $\psi$ coincides with the sector condition for the generator and the associated non-symmetric Dirichlet form, see Jacob \cite[Volume 1, 4.7.32--33]{jacob-pseudo}.

With some more effort (see \cite[Theorem 5.5]{schilling-lm}), we may replace the sector condition by imposing conditions on the expression
$$
    \sup_{|y-x|\leq r}\frac{\Re q(x,\xi)}{|\xi|\,|\Im q(y,\xi)|}
    \quad\text{as $r\to\infty$}.
$$
\begin{theorem}[see \mbox{\cite[pp.\ 117--119]{schilling-lm}}]\label{den-13}
    Let $(X_t)_{t \geq 0}$ be a Feller process with infinitesimal generator $A$, $\cont_c^{\infty}(\rd) \subset \dom(A)$ and symbol $q(x,\xi)$ satisfying the sector condition \eqref{den-e12}. Then
	\begin{equation}\label{den-e14}
		\Pp^x \left( \sup_{s \leq t} |X_s-x| < r\right)
        \leq \frac{c_{\kappa,d}}{t \sup_{|\xi| \leq r^{-1}} \inf_{|y-x| \leq 2r} |q(y,\xi)|}.
	\end{equation}
\index{Feller process!maximal estimate}%
\end{theorem}

The maximal estimates \eqref{den-e10} and \eqref{den-e14} are quite useful tools. With them we can estimate the mean exit time from balls ($X$ and $q(x,\xi)$ are as in Theorem~\ref{den-13}):
$$
    \frac{c}{\sup_{|\xi|\leq 1/r} \inf_{|y-x|\leq r} |q(y,\xi)|}
    \leq
    \Ee^x\sigma_r^x
    \leq
    \frac{c_\kappa}{\sup_{|\xi|\leq k^*/r} \inf_{|y-x|\leq r} |q(y,\xi)|}
$$
for all $x\in\rd$ and $r>0$ and with $k^* = \arccos\sqrt{2/3}\approx 0.615$.

\medskip
Recently, K\"{u}hn~\cite{kuehn} studied the existence of and estimates for generalized moments; a typical result is contained in the following theorem.
\begin{theorem}\label{den-21}
\index{Feller process!moments}%
    Let $X=\Xt$ be a Feller process with infinitesimal generator $(A,\dom(A))$ and $\cont_c^\infty(\rd)\subset\dom(A)$. Assume that the symbol $q(x,\xi)$, given by \eqref{feller-e14}, satisfies $q(x,0)=0$ and has $(l(x),Q(x),\nu(x,dy))$ as $x$-dependent L\'evy triplet. If $f:\rd\to[0,\infty)$ is \textup{(}comparable to\textup{)} a twice continuously differentiable submultiplicative function such that
    $$
        \sup_{x\in K} \int f(y)\,\nu(x,dy)<\infty \quad\text{for a compact set $K\subset\rd$},
    $$
    then the generalized moment $\sup_{x\in K}\sup_{s\leq t}\Ee^x f(X_{s \wedge \tau_K}-x) <\infty$ exists and
    $$
        \Ee^x f(X_{t\wedge\tau_K})\leq c f(x)\eup^{C(M_1+M_2) t};
    $$
    here $\tau_K = \inf\{t>0\,:\, X_t\notin K\}$ is the first exit time from $K$, $C=C_K$ is some absolute constant and
    $$
        M_1 = \sup_{x\in K}\Big(|l(x)|+|Q(x)|+\int_{y\neq 0} (|y|^2\wedge 1)\,\nu(x,dy)\Big),\quad
        M_2 = \sup_{x\in K} \int_{|y|\geq 1} f(y)\,\nu(x,dy).
    $$
    If $X$ has bounded coefficients \textup{(}Lemma~\ref{feller-15}\textup{)}, then $K=\rd$ is admissible.
\end{theorem}

There are also counterparts for the Blumenthal--Getoor and Pruitt indices. Below we give two representatives, for a full discussion we refer to \cite{schilling-lm}.
\index{Blumenthal--Getoor--Pruitt index}%
\begin{definition}\label{den-31}
    Let $q(x,\xi)$ be the symbol of a Feller process. Then
    \begin{align}
    \label{den-e32}
        \beta_\infty^x &:= \inf\left\{\lambda>0\::\: \lim_{|\xi|\to\infty} \frac{\sup_{|\eta|\leq |\xi|} \sup_{|y-x|\leq 1/|\xi|} |q(y,\eta)|}{|\xi|^\lambda} = 0\right\},\\
    \label{den-e34}
        \delta_\infty^x &:= \sup\left\{\lambda>0\::\: \lim_{|\xi|\to\infty} \frac{\inf_{|\eta|\leq |\xi|} \inf_{|y-x|\leq 1/|\xi|} |q(y,\eta)|}{|\xi|^\lambda} = \infty\right\}.
    \end{align}
\end{definition}
By definition, $0\leq\delta_\infty^x\leq\beta_\infty^x\leq 2$. For example, if $q(x,\xi)=|\xi|^{\alpha(x)}$ with a smooth exponent function $\alpha(x)$, then $\beta_\infty^x=\delta_\infty^x=\alpha(x)$; in general, however, we cannot expect that the two indices coincide.

As for L\'evy processes, these indices can be used to describe the path behaviour.
\begin{theorem}\label{den-33}
\index{Feller process!Hausdorff dimension}%
    Let $\Xt$ be a $d$-dimensional Feller process with the generator $A$ such that $\cont_c^\infty(\rd)\subset \dom(A)$ and symbol $q(x,\xi)$.
    For every bounded analytic set $E\subset[0,\infty)$, the Hausdorff dimension
    \begin{equation}\label{den-e36}
        \mathrm{dim}\,  \{X_t\,:\, t\in E\}
        \le \min\left\{d,\: \sup_{x\in\rd}\beta_\infty^x \, \mathrm{dim}\, E\right\}.
    \end{equation}
\end{theorem}
A proof can be found in \cite[Theorem 5.15]{schilling-lm}. It is instructive to observe that we have to take the supremum w.r.t.\ the space variable $x$, as we do not know how the process $X$ moves while we observe it during $t\in E$. This shows that we can only expect to get `exact' results if $t\to 0$. Here is such an example.
\begin{theorem}\label{den-35}
\index{Feller process!asymptotics as $t\to 0$}%
    Let $\Xt$ be a $d$-dimensional Feller process with symbol $q(x,\xi)$ satisfying the sector condition. Then, $\Pp^x$-a.s.
    \begin{align}
    \label{den-e38}
        &\lim_{t\to 0} \frac{\sup_{0\leq s\leq t}|X_s - x|}{t^{1/\lambda}} = 0
        \qquad\forall \lambda > \beta^x_\infty,\\
    \label{den-e40}
        &\lim_{t\to 0} \frac{\sup_{0\leq s\leq t}|X_s - x|}{t^{1/\lambda}} = \infty
        \qquad\forall \lambda < \delta^x_\infty.
    \end{align}
\end{theorem}
As one would expect, these results are proved using the maximal estimates \eqref{den-e10} and \eqref{den-e14} in conjunction with the Borel--Cantelli lemma, see \cite[Theorem 5.16]{schilling-lm}.

\medskip
If we are interested in the long-term behaviour, one could introduce indices `at zero', where we replace in Definition~\ref{den-31} the limit $|\xi|\to \infty$ by $|\xi|\to 0$, but we will always have to pay the price that we loose the influence of the starting point $X_0=x$, i.e.\ we will have to take the supremum or infimum for all $x\in\rd$.

\medskip
With the machinery we have developed here, one can also study further path properties, such as invariant measures, ergodicity, transience and recurrence etc. For this we refer to the monograph \cite{schilling-lm} as well as recent developments by Behme \& Schnurr \cite{behme-schnurr} and Sandri\'c \cite{sandric15,sandric14}.

\appendix
\renewcommand{\chaptername}{Appendix}

\chapter{Some classical results}\label{app}

In this appendix we collect some classical results from (or needed for) probability theory which are not always contained in a standard course.

\subsection*{The Cauchy--Abel functional equation}

Below we reproduce the standard proof for \emphh{continuous} functions which, however, works also for right-continuous (or monotone) functions.
\index{Cauchy--Abel functional equation}%
\begin{theorem}\label{app-31}
    Let $\phi:[0,\infty)\to\comp$ be a right-continuous function satisfying the functional equation $\phi(s+t)=\phi(s)\phi(t)$. Then $\phi(t)=\phi(1)^t$.
\end{theorem}
\begin{proof}
    Assume that $\phi(a)=0$ for some $a\geq 0$. Then we find for all $t\geq 0$
    $$
        \phi(a+t)=\phi(a)\phi(t)=0
        \implies
        \phi|_{[a,\infty)}\equiv 0.
    $$
    To the left of $a$ we find for all $n\in\nat$
    $$
        0 = \phi(a) = \big[\phi\big(\tfrac an\big)\big]^n
        \implies
        \phi\big(\tfrac an\big)=0.
    $$
    Since $\phi$ is right-continuous, we infer that $\phi|_{[0,\infty)}\equiv 0$, and $\phi(t)=\phi(1)^t$ holds.

    Now assume that $\phi(1)\neq 0$. Setting $f(t):=\phi(t)\phi(1)^{-t}$ we get
    $$
        f(s+t)=\phi(s+t)\phi(1)^{-(s+t)} = \phi(s)\phi(1)^{-s}\phi(t)\phi(1)^{-t} = f(s)f(t)
    $$
    as well as $f(1)=1$. Applying the functional equation $k$ times we conclude that
    $$
        f\big(\tfrac kn\big) = \big[f\big(\tfrac 1n\big)\big]^k\fa k,n\in\nat.
    $$
    The same calculation done backwards yields
    $$
        \big[f\big(\tfrac 1n\big)\big]^k
        = \big[f\big(\tfrac 1n\big)\big]^{n\frac{k}{n}}
        = \big[f\big(\tfrac nn\big)\big]^{\frac{k}{n}}
        = \big[f(1)\big]^{\frac{k}{n}}
        = 1.
    $$
    Hence, $f|_{\mathds{Q}_+}\equiv 1$. Since $\phi$, hence $f$, is right-continuous, we see that $f\equiv 1$ or, equivalently, $\phi(t)=[\phi(1)]^t$ for all $t>0$.
\end{proof}

\subsection*{Characteristic functions and moments}
\begin{theorem}[Even moments and characteristic functions]\label{app-21}
    Let $Y=(Y^{(1)}, \dots , Y^{(d)})$ be a random variable in $\rd$ and let $\chi(\xi)=\Ee\,\eup^{\iup \xi\cdot Y}$ be its characteristic function. Then $\Ee (|Y|^2)$ exists if, and only if, the second derivatives $\frac{\partial^2}{\partial\xi_k^2}\chi(0)$, $k=1,\dots, d$, exist and are finite. In this case all mixed second derivatives exist and
    \begin{gather}\label{app-e22}
    	\Ee Y^{(k)}
        = \frac{1}{\iup}\frac{\partial\chi(0)}{\partial \xi_k}
    \et
		\Ee(Y^{(k)} Y^{(l)})
        = - \frac{\partial^2\chi(0)}{\partial \xi_k \partial \xi_l}.
    \end{gather}
\end{theorem}
\begin{proof}
    In order to keep the notation simple, we consider only $d=1$. If $\Ee(Y^2)<\infty$, then the formulae \eqref{app-e22} are routine applications of the differentiation lemma for parameter-dependent integrals, see e.g.\ \cite[Theorem 11.5]{schilling-mims} or
    \cite[Satz 12.2]{schilling-mi}. Moreover, $\chi$ is twice continuously differentiable.

    Let us prove that $\Ee(Y^2)\leq -\chi''(0)$. An application of l'Hospital's rule gives
    \begin{align*}
        \chi''(0)
        &= \lim_{h\to 0}\frac 12\left(\frac{\chi'(2h)-\chi'(0)}{2h} + \frac{\chi'(0)-\chi'(-2h)}{2h}\right)\\
        &= \lim_{h\to 0} \frac{\chi'(2h)-\chi'(-2h)}{4h}\\
        &= \lim_{h\to 0} \frac{\chi(2h)-2\chi(0)+\chi(-2h)}{4h^2}\\
        &= \lim_{h\to 0} \Ee \left[\left(\frac{\eup^{\iup h Y}-\eup^{-\iup hY}}{2h}\right)^2\right]\\
        &= -\lim_{h\to 0} \Ee \left[\left(\frac{\sin hY}{h}\right)^2\right].
    \end{align*}
    From Fatou's lemma we get
    \begin{gather*}
        \chi''(0) \leq -\Ee\left[ \lim_{h\to 0} \left(\frac{\sin hY}{h}\right)^2\right] = -\Ee\big[Y^2\big].
    \end{gather*}
    In the multivariate case observe that $\Ee |Y^{(k)}Y^{(l)}| \leq \Ee \big[(Y^{(k)})^2\big] + \Ee \big[(Y^{(l)})^2\big]$.
\end{proof}

\subsection*{Vague and weak convergence of measures}
A sequence of locally finite\footnote{I.e.\ every compact set $K$ has finite measure.} Borel measures $(\mu_n)_{n\in\nat}$ on $\rd$ converges \emphh{vaguely} to a locally finite measure $\mu$ if
\index{vague convergence}%
\begin{equation}\label{app-e82}
    \lim_{n\to\infty}\int \phi\,d\mu_n = \int\phi\,d\mu
    \fa \phi\in \cont_c(\rd).
\end{equation}
Since the compactly supported continuous functions $\cont_c(\rd)$ are dense in the space of continuous functions vanishing at infinity $\cont_\infty(\rd) = \{\phi\in \cont(\rd)\,:\, \lim_{|x|\to\infty}\phi(x)=0\}$, we can replace in \eqref{app-e82} the set $\cont_c(\rd)$ with $\cont_\infty(\rd)$. The following theorem guarantees that a family of Borel measures is sequentially relatively compact\footnote{Note that compactness and sequential compactness need not coincide!} for the vague convergence.
\begin{theorem}\label{app-81}
    Let $(\mu_t)_{t\geq 0}$ be a family of measures on $\rd$ which is uniformly bounded, in the sense that $\sup_{t\geq 0}\mu_t(\rd)<\infty$. Then every sequence $(\mu_{t_n})_{n\in\nat}$ has a vaguely convergent subsequence.
\end{theorem}

If we test in \eqref{app-e82} against all bounded continuous functions $\phi\in \cont_b(\rd)$, we get \emphh{weak convergence} of the sequence $\mu_n\to\mu$. One has
\begin{theorem}\label{app-83}
    A sequence of measures $(\mu_n)_{n\in\nat}$ converges weakly to $\mu$ if, and only if, $\mu_n$ converges vaguely to $\mu$ and $\lim_{n\to\infty}\mu_n(\rd)=\mu(\rd)$ \textup{(}preservation of mass\textup{)}. In particular, weak and vague convergence coincide for sequences of probability measures.
\end{theorem}
Proofs and a full discussion of vague and weak convergence can be found in Malliavin \cite[Chapter~III.6]{malliavin95} or Schilling \cite[Chapter 25]{schilling-mi}.

For any finite measure $\mu$ on $\rd$ we denote by $\widecheck\mu(\xi):=\int_{\rd} \eup^{\iup \xi\cdot y}\,\mu(dy)$ its characteristic function.
\begin{theorem}[L\'evy's continuity theorem]\label{app-85}
\index{L\'evy's continuity theorem}%
    Let $(\mu_n)_{n\in\nat}$ be a sequence of finite measures on $\rd$. If $\mu_n\to\mu$ weakly, then the characteristic functions $\widecheck\mu_n(\xi)$ converge locally uniformly to $\widecheck\mu(\xi)$.

    Conversely, if the limit $\lim_{n\to\infty}\widecheck\mu_n(\xi)=\chi(\xi)$ exists for all $\xi\in\rd$ and defines a function $\chi$ which is continuous at $\xi=0$, then there exists a finite measure $\mu$ such that $\widecheck\mu(\xi)=\chi(\xi)$ and $\mu_n\to\mu$ weakly.
\end{theorem}
A proof in one dimension is contained in the monograph by Breiman \cite{breiman68}, $d$-di\-men\-sion\-al versions can be found in Bauer \cite[Chapter 23]{bauer96} and Malliavin \cite[Chapter IV.4]{malliavin95}.

\subsection*{Convergence in distribution}
By $\xrightarrow[\quad]{d}$ we denote convergence in distribution.
\begin{theorem}[Convergence of types]\label{app-03}
\index{convergence of types theorem}%
    Let $(Y_n)_{n\in\nat}$, $Y$ and $Y'$ be random variables and suppose that there are constants $a_n>0$, $c_n\in\real$ such that
    $$
        Y_n \xrightarrow[n\to\infty]{d}Y
        \et
        a_nY_n + c_n \xrightarrow[n\to\infty]{d}Y'.
    $$
    If $Y$ and $Y'$ are non-degenerate, then the limits $a=\lim\limits_{n\to\infty}a_n$ and $c=\lim\limits_{n\to\infty}c_n$ exist and $Y'\sim aY+c$.
\end{theorem}
\begin{proof}
    Write $\chi_Z(\xi):=\Ee \eup^{iup \xi Z}$ for the characteristic function of the random variable $Z$.

\medskip\noindent
    $1^\circ$\ \ By L\'evy's continuity theorem (Theorem~\ref{app-85}) convergence in distribution ensures that
    $$
        \chi_{a_nY_n+c_n}(\xi) = \eup^{\iup c_n\cdot\xi}\chi_{Y_n}(a_n\xi) \xrightarrow[n\to\infty]{\text{locally unif.}} \chi_{Y'}(\xi)
        \et
        \chi_{Y_n}(\xi) \xrightarrow[n\to\infty]{\text{locally unif.}} \chi_{Y}(\xi).
    $$
    Take some subsequence $(a_{n(k)})_{k\in\nat}\subset (a_n)_{n\in\nat}$ such that $\lim_{k\to\infty} a_{n(k)} = a\in [0,\infty]$.

\medskip\noindent
    $2^\circ$\ \ Claim: $a>0$. Assume, on the contrary, that $a=0$.
    $$
        |\chi_{a_{n(k)}Y_{n(k)}+c_{n(k)}}(\xi)| = |\chi_{Y_{n(k)}}(a_{n(k)}\xi)| \xrightarrow[k\to\infty]{} |\chi_{Y}(0)|=1.
    $$
    Thus, $|\chi_{Y'}|\equiv 1$ which means that $Y'$ would be degenerate, contradicting our assumption.

\medskip\noindent
    $3^\circ$\ \ Claim: $a<\infty$. If $a=\infty$, we use $Y_n = (a_n)^{-1} (Y_n - c_n)$ and the argument from step $1^\circ$ to reach the contradiction
    $$
        (a_{n(k)})^{-1} \xrightarrow[k\to\infty]{} a^{-1} > 0 \iff a<\infty.
    $$

\medskip\noindent
    $4^\circ$\ \ Claim: There exists a unique $a\in [0,\infty)$ such that $\lim_{n\to\infty}a_n = a$. Assume that there were two different subsequential limits $a_{n(k)}\to a$, $a_{m(k)}\to a'$ and $a\neq a'$. Then
    \begin{align*}
        |\chi_{a_{n(k)}Y_{n(k)}+c_{n(k)}}(\xi)| = |\chi_{a_{n(k)}Y_{n(k)}}(\xi)| &\xrightarrow[k\to\infty]{} \chi_Y(a\xi),\\
        |\chi_{a_{m(k)}Y_{m(k)}+c_{m(k)}}(\xi)| = |\chi_{a_{m(k)}Y_{m(k)}}(\xi)| &\xrightarrow[k\to\infty]{} \chi_Y(a'\xi).
    \end{align*}
    On the other hand, $\chi_Y(a\xi) = \chi_Y(a'\xi) = \chi_{Y'}(\xi)$. If $a'<a$, we get by iteration
    $$
        |\chi_Y(\xi)|
        = \big|\chi_Y\big(\tfrac{a'}{a}\xi\big)\big|
        = \dots
        = \big|\chi_Y\big(\big(\tfrac{a'}{a}\big)^N\xi\big)\big|
        \xrightarrow[N\to\infty]{} |\chi_Y(0)| = 1.
    $$
    Thus, $|\chi|\equiv 1$ and $Y$ is a.s.\ constant. Since $a, a'$ can be interchanged, we conclude that $a=a'$.

\medskip\noindent
    $5^\circ$\ \ We have
    $$
        \eup^{\iup c_n\cdot\xi}
        = \frac{\chi_{a_nY_n+c_n}(\xi)}{\chi_{a_nY_n}(\xi)}
        = \frac{\chi_{a_nY_n+c_n}(\xi)}{\chi_{Y_n}(a_n\xi)}
        \xrightarrow[n\to\infty]{} \frac{\chi_{Y'}(\xi)}{\chi_Y(a\xi)}.
    $$
    Since $\chi_Y$ is continuous and $\chi_Y(0)=1$, the limit $\lim_{n\to\infty}\eup^{\iup c_n\cdot\xi}$ exists for all $|\xi|\leq \delta$ and some small $\delta$. For $\xi = t\xi_0$ with $|\xi_0|=1$, we get
    \begin{align*}
        0
        < \left|\int_0^\delta \frac{\chi_{Y'}(t\xi_0)}{\chi_{Y}(ta\xi_0)}\,dt\right|
        = \left|\lim_{n\to\infty}\int_0^\delta \eup^{\iup t c_n\cdot\xi_0}\,dt\right|
        = \left|\lim_{n\to\infty}\frac{\eup^{\iup \delta c_n\cdot\xi_0}-1}{\iup c_n\cdot\xi_0}\right|
        \leq \liminf_{n\to\infty}\frac{2}{|c_n\cdot\xi_0|},
    \end{align*}
    and so $\limsup_{n\to\infty}|c_n|<\infty$; if there were two limit points $c\neq c'$, then $\eup^{\iup c\cdot\xi} = \eup^{\iup c'\cdot\xi}$ for all $|\xi|\leq\delta$. This gives $c=c'$, and we finally see that $c_n\to c$, $\eup^{\iup c_n\cdot\xi} \to \eup^{\iup c\cdot\xi}$, as well as
    \begin{gather*}
        \chi_{Y'}(\xi) = \eup^{\iup c\cdot\xi}\chi_Y(a\xi).
    \qedhere
    \end{gather*}
\end{proof}

A random variable is called \emphh{symmetric} if $Y\sim-Y$.
\begin{theorem}[Symmetrization inequality]\label{app-13}
    Let $Y_1, \dots, Y_n$ be independent symmetric random variables. Then the partial sum $S_n=Y_1+\dots+Y_n$ is again symmetric and
    \begin{align}
    \label{app-e12}
        \Pp(|Y_1+\dots+Y_n|>u) &\geq \tfrac 12\Pp\Big(\max_{1\leq k\leq n}|Y_k|>u\Big).
    \intertext{If the $Y_k$ are iid with $Y_1\sim\mu$, then}
    \label{app-e14}
        \Pp(|Y_1+\dots+Y_n|>u) &\geq \tfrac 12\big(1-\eup^{-n\Pp(|Y_1|> u)}\big).
    \end{align}
\end{theorem}
\begin{proof}
    By independence, $S_n = Y_1+\dots+Y_n \sim -Y_1-\dots-Y_n = -S_n$.

    Let $\tau = \min\{1\leq k\leq n\,:\, |Y_k| = \max_{1\leq l\leq n}|Y_l|\}$ and set $Y_{n,\tau} = S_n - Y_\tau$. Then the four (counting all possible $\pm$ combinations) random variables $(\pm Y_\tau, \pm Y_{n,\tau})$ have the same law. Moreover,
    $$
        \Pp(Y_\tau > u)
        \leq \Pp(Y_\tau > u,\: Y_{n,\tau}\geq 0) + \Pp(Y_\tau > u,\: Y_{n,\tau}\leq 0)
        = 2\Pp(Y_\tau > u,\: Y_{n,\tau}\geq 0),
    $$
    and so
    $$
        \Pp(S_n > u)
        = \Pp(Y_\tau+Y_{n,\tau} > u)
        \geq \Pp(Y_\tau>u,\:Y_{n,\tau}\geq 0)
        \geq \tfrac 12\Pp(Y_\tau>u).
    $$
    By symmetry, this implies \eqref{app-e12}. In order to see \eqref{app-e14}, we use that the $Y_k$ are iid, hence
    $$
        \Pp(\max_{1\leq k\leq n}|Y_k|\leq u)
        = \Pp(|Y_1|\leq u)^n
        \leq \eup^{-n\Pp(|Y_1|>u)},
    $$
    along with the elementary inequality $1-p\leq \eup^{-p}$ for $0\leq p\leq 1$. This proves \eqref{app-e14}.
\end{proof}

\subsection*{The predictable $\sigma$-algebra}
Let $(\Omega,\Ascr,\Pp)$ be a probability space and $(\Fscr_t)_{t\geq 0}$ some filtration. A stochastic process $\Xt$ is called \emphh{adapted},
\index{adapted}%
if for every $t\geq 0$ the random variable $X_t$ is $\Fscr_t$ measurable.
\begin{definition}\label{app-61}
\index{predictable}%
    The \emphh{predictable $\sigma$-algebra} $\Pscr$ is the smallest $\sigma$-algebra on $\Omega\times (0,\infty)$ such that all left-continuous adapted stochastic processes $(\omega,t)\mapsto X_t(\omega)$ are measurable. A $\Pscr$ measurable process $X$ is called a \emphh{predictable process}.
\end{definition}
For a stopping time $\tau$ we denote by
\begin{equation}\label{app-e62}
    \rrbracket 0,\tau\rrbracket := \{(\omega,t) \::\: 0 < t\leq \tau(\omega)\}
    \et
    \rrbracket \tau,\infty\llbracket := \{(\omega,t) \::\: t>\tau(\omega)\}
\end{equation}
the left-open \emphh{stochastic intervals}.
\index{stochastic interval}%
The following characterization of the predictable $\sigma$-algebra is essentially from Jacod \& Shiryaev \cite[Theorem I.2.2]{jacod-shiryaev87}.
\begin{theorem}\label{app-63}
    The predictable $\sigma$-algebra $\Pscr$ is generated by any one of the following families of random sets
    \begin{enumerate}
    \item[\textup{a)}]
        $\rrbracket 0,\tau\rrbracket$ where 
        $\tau$ is any bounded stopping time;
    \item[\textup{b)}]
        $F_s\times (s,t]$ where 
        $F_s\in\Fscr_s$ and $0\leq s< t$.
    \end{enumerate}
\end{theorem}
\begin{proof}
    We write $\Pscr_a$ and $\Pscr_b$ for the $\sigma$-algebras generated by the families listed in a) and b), respectively.

\medskip\noindent
    $1^\circ$ \ Pick $0\leq s<t$, $F=F_s\in\Fscr_s$ and let $n>t$. Observe that $s_F := s\I_F + n\I_{F^c}$ is a bounded stopping time\footnote{Indeed,
    $
        \{s_F\leq t\}
        = \{s\leq t\}\cap F
        = \left\{
            \begin{aligned}
                \emptyset, &\quad s>t\\
                F, &\quad s\leq t
            \end{aligned}
        \right\}\in\Fscr_t
    $ for all $t\geq 0$.} and $F\times (s,t] = \rrbracket s_F,t_F\rrbracket$.
    Therefore, $\rrbracket s_F,t_F\rrbracket = \rrbracket0,t_F\rrbracket\setminus \rrbracket0,s_F\rrbracket \in\Pscr_a$, and we conclude that $\Pscr_b\subset\Pscr_a$.

\medskip\noindent
    $2^\circ$ \ Let 
    $\tau$ be a bounded stopping time. Since 
    $t\mapsto\I_{\rrbracket0,\tau\rrbracket}(\omega,t)$ 
    is adapted and left-continuous, we have $\Pscr_a\subset\Pscr$.

\medskip\noindent
    $3^\circ$ \ Let $X$ be an adapted and left-continuous process and define for every $n\in\nat$
    $$
        X^n_t:= 
        \sum_{k=0}^\infty X_{k2^{-n}}\I_{\rrbracket k2^{-n},\,(k+1)2^{-n}\rrbracket}(t).
    $$
    Obviously, $X^n = (X^n_t)_{t\geq 0}$ is $\Pscr_b$ measurable; because of the left-continuity of $t\mapsto X_t$, the limit $\lim_{n\to\infty} X_t^n = X_t$ exists, and we conclude that $X$ is $\Pscr_b$ measurable; consequently, $\Pscr\subset\Pscr_b$.
\end{proof}

\subsection*{The structure of translation invariant operators}
Let $\vartheta_x f(y):= f(y+x)$ be the translation operator and $\widetilde f(x):= f(-x)$. A linear operator $L : \cont_c^\infty(\rd) \to \cont(\rd)$ is called \emphh{translation invariant} if
\index{translation invariant}%
\begin{equation}\label{app-e92}
    \vartheta_x (Lf) = L(\vartheta_x f).
\end{equation}
    A \emphh{distribution} $\lambda$ is an element of the topological dual $(\cont_c^\infty(\rd))'$, i.e.\ a continuous linear functional $\lambda:\cont_c^\infty(\rd)\to\real$. The \emphh{convolution of a distribution} with a function $f\in\cont_c^\infty(\rd)$ is defined as
    $$
        f*\lambda(x) := \lambda(\vartheta_{-x}\widetilde f),\quad \lambda\in\big(\cont_c^\infty(\rd)\big)',\; f\in\cont_c^\infty(\rd),\; x\in\rd.
    $$
    If $\lambda=\mu$ is a measure, this formula generalizes the `usual' convolution
    $$f*\mu(x) = \int f(x-y)\,\mu(dy) = \int \widetilde f(y-x)\,\mu(dy) = \mu(\vartheta_{-x} \widetilde f).$$

\begin{theorem}\label{app-91}
    If $\lambda\in(\cont_c^\infty(\rd))'$ is a distribution, then $Lf(x):= f*\lambda(x)$ defines a translation invariant continuous linear map $L:\cont_c^\infty(\rd)\to\cont(\rd)$.

    Conversely, every translation-invariant continuous linear map $L:\cont_c^\infty(\rd)\to\cont(\rd)$ is of the form $Lf = f*\lambda$ for some unique distribution $\lambda\in (\cont_c^\infty(\rd))'$.
\end{theorem}
\begin{proof}
    Let $Lf = f*\lambda$. From the very definition of the convolution we get
    \begin{align*}
        (\vartheta_{-x}f)*\lambda
        = \lambda(\vartheta_{-x}\widetilde{\vartheta_{-x}f})
        &= \lambda(\vartheta_{-x} [f(x-\cdot)])\\
        &= \vartheta_{-x}\lambda(\vartheta_{-x} [f(-\,\cdot)])\\
        &= \vartheta_{-x} (f*\lambda).
    \end{align*}
    For any sequence $x_n\to x$ in $\rd$ and $f\in\cont_c^\infty(\rd)$ we know that $\vartheta_{-x_n}\widetilde{f} \to \vartheta_{-x}\widetilde{f}$ in $\cont_c^\infty(\rd)$. Since $\lambda$ is a continuous linear functional on $\cont_c^\infty(\rd)$, we conclude that
    $$
        \lim_{n\to\infty} Lf(x_n)
        = \lim_{n\to\infty} \lambda(\vartheta_{-x_n}\widetilde f)
        = \lambda(\vartheta_{-x}\widetilde f)
        = Lf(x)
    $$
    which shows that $Lf\in\cont(\rd)$. (With a similar argument based on difference quotients we could even show that $Lf\in\cont^\infty(\rd)$.)

    In order to prove the continuity of $L$, it is enough to show that $L:\cont_c^\infty(K)\to\cont(\rd)$ is continuous for every compact set $K\subset\rd$ (this is because of the definition of the topology in $\cont_c^\infty(\rd)$). We will use the closed graph theorem: Assume that $f_n\to f$ in $\cont_c^\infty(\rd)$ and $f_n*\lambda \to g$ in $\cont(\rd)$, then we have to show that $g=f*\lambda$.

    For every $x\in\rd$ we have $\vartheta_{-x}\widetilde f_n \to \vartheta_{-x}\widetilde f$ in $\cont_c^\infty(\rd)$, and so
    $$
        g(x)
        = \lim_{n\to\infty} (f_n*\lambda)(x)
        = \lim_{n\to\infty} \lambda(\vartheta_{-x}\widetilde f_n)
        = \lambda(\vartheta_{-x}\widetilde f)
        = (f*\lambda)(x).
    $$

    Assume now, that $L$ is translation invariant and continuous. Define $\lambda(f):= (L\widetilde f)(0)$. Since $L$ is linear and continuous, and $f\mapsto\widetilde f$ and the evaluation at $x=0$ are continuous operations, $\lambda$ is a continuous linear map on $\cont_c^\infty(\rd)$. Because of the translation invariance of $L$ we get
    \begin{align*}
        (Lf)(x)
        = (\vartheta_x Lf)(0)
        = L(\vartheta_x f)(0)
        = \lambda(\widetilde{\vartheta_x f})
        = \lambda(\vartheta_{-x}\widetilde f)
        = (f*\lambda)(x).
    \end{align*}
    If $\mu$ is a further distribution with $Lf(0) = f*\mu(0)$, we see
    $$
        (\mu-\lambda)(\widetilde f) = f*(\mu-\lambda)(0) = f*\mu(0) - f*\lambda(0) = 0
        \fa f\in\cont_c^\infty(\rd)
    $$
    which proves $\mu=\lambda$.
\end{proof}





\index{infinitesimal generator|see{generator}}%
\index{transition semigroup|see{semigroup}}%
\index{Feller~semigroup|see{\emph{also} Feller process}}%
\index{Poisson~process|see{\emph{also} compound -----}}%

\fancyhead[LO]{\sffamily Index}

\printindex


\begin{thebibliography}{99}
\fancyhead[LO]{\sffamily Bibliography}
\addcontentsline{toc}{chapter}{Bibliography}
\bibitem{aczel66}
    Acz\'el, J.: \emph{Lectures on Functional Equations and Their Applications}. Academic Press, New York (NY) 1966.

\bibitem{applebaum09}
    Applebaum, D.: \emph{L\'evy Processes and Stochastic Calculus}. Cambridge University Press, Cambridge 2009 (2nd ed).

\bibitem{barndorff01}
    Barndorff-Nielsen, O.E.\ et al.: \emph{L\'evy Processes. Theory and Applications}. Birk\-h\"{a}user, Boston (MA) 2001.

\bibitem{bass88}
    Bass, R.: Uniqueness in law for pure-jump Markov processes. \emph{Probab.\ Theor.\ Relat.\ Fields} \textbf{79} (1988) 271--287.

\bibitem{bauer96}
    Bauer, H.: \emph{Probability Theory}. De Gruyter, Berlin 1996.

\bibitem{bawly36}
    Bawly, G.M.: \"{U}ber einige Verallgemeinerungen der Grenzwerts\"{a}tze der Wahr\-schein\-lich\-keits\-rech\-nung. \emph{Mat.\ Sbornik} (\emph{R\'ec.\ Math.\ Moscou, N.S.}) \textbf{1} (1936) 917--930.

\bibitem{behme-schnurr}
    Behme, A., Schnurr, A.: A criterion for invariant measures of It\^{o} processes based on the symbol. \emph{Bernoulli} \textbf{21} (2015) 1697--1718.

\bibitem{blu-get61}
    Blumenthal, R.M., Getoor, R.K.: Sample functions of stochastic processes with stationary independent increments. \emph{J.\ Math.\ Mech.} \textbf{10} (1961) 493--516.


\bibitem{schilling-lm}
    B\"{o}ttcher, B., Schilling, R.L., Wang, J.: \emph{L\'evy-type Processes: Construction, Approximation and Sample Path Properties}. Lecture Notes in Mathematics \textbf{2099} (L\'evy Matters III), Springer, Cham 2014.

\bibitem{breiman68}
    Breiman, L.: \emph{Probability}. Addison--Wesley, Reading (MA) 1968 (several reprints by SIAM, Philadelphia).

\bibitem{bretagnolle73}
    Bretagnolle, J.L.: Processus \`a accroissements independants. In: Bretagnolle, J.L.\ et al.: \emph{\'Ecole d'\'Et\'e de Probabilit\'es: Processus Stochastiques}. Springer, Lecture Notes in Mathematics \textbf{307}, Berlin 1973, pp.~1--26.

\bibitem{cinlar75}
    \c{C}inlar, E.: \emph{Introduction to Stochastic Processes}. Prentice--Hall, Englewood Cliffs (NJ) 1975 (reprinted by Dover, Mineola).

\bibitem{courrege64}
    Courr\`ege, P.: G\'en\'erateur infinit\'esimal d'un semi-groupe de convolution sur $\real^n$, et formule de L\'evy--Khintchine. \emph{Bull.\ Sci.\ Math.} \textbf{88} (1964) 3--30.

\bibitem{courrege65}
    Courr\`ege, P.: Sur la forme int\'egro diff\'erentielle des op\'erateurs de $C_k^\infty$ dans $C$ satisfaisant au principe du maximum. In: \emph{S\'eminaire Brelot--Choquet--Deny. Th\'eorie du potentiel} \textbf{10} (1965/66) expos\'e 2, pp.\ 1--38.

\bibitem{definetti29}
    de Finetti, B.: Sulle funzioni ad incremento aleatorio. \emph{Rend.\ Accad.\ Lincei Ser.\ VI} \textbf{10} (1929) 163--168.

\bibitem{dieudonne69}
    Dieudonn\'e, J.: \emph{Foundations of Modern Analysis}. Academic Press, New York (NY) 1969.



\bibitem{eth-kur86}
    Ethier, S.N., Kurtz, T.G.: \emph{Markov Processes. Characterization and Convergence}. John Wiley \& Sons, New York (NY) 1986.

\bibitem{gik-sko69}
    Gikhman, I.I., Skorokhod, A.V.: \emph{Introduction to the Theory of Random Processes}. W.B.\ Saunders, Philadelphia (PA) 1969 (reprinted by Dover, Mineola 1996).

\bibitem{grimvall72}
    Grimvall, A.: A theorem on convergence to a L\'evy process. \emph{Math.\ Scand.} \textbf{30} (1972) 339--349.

\bibitem{herz64}
    Herz, C.S.: Th\'eorie \'el\'ementaire des distributions de Beurling. \emph{Publ.\ Math.\ Orsay} no.\ \textbf{5}, 2\`eme ann\'ee 1962/63, Paris 1964.

\bibitem{hoh00}
    Hoh, W.: Pseudo differential operators with negative definite symbols of variable order. \emph{Rev.\ Mat.\ Iberoamericana} \textbf{16} (2000) 219--241.

\bibitem{ikeda-watanabe89}
    Ikeda, N., Watanabe S.: \emph{Stochastic Differential Equations and Diffusion Processes}. North-Holland Publishing Co./Kodansha, Amsterdam and Tokyo 1989 (2nd ed).

\bibitem{ito42}
    It\^o, K.: On stochastic processes. I. (Infinitely divisible laws of probability). \emph{Japanese J.\ Math.} \textbf{XVIII} (1942) 261--302.

\bibitem{ito61}
    It\^o, K.: \emph{Lectures on Stochastic Processes}. Tata Institute of Fundamental Research, Bombay 1961 (reprinted by Springer, Berlin 1984). \newline \url{http://www.math.tifr.res.in/~publ/ln/tifr24.pdf}

\bibitem{ito93}
    It\^o, K.: Semigroups in probability theory. In: \emph{Functional Analysis and Related Topics}, Proceedings in Memory of K.\ Yosida (Kyoto 1991). Springer, Lecture Notes in Mathemtics \textbf{1540}, Berlin 1993, 69--83.

\bibitem{jacob-pseudo}
    Jacob, N.: \emph{Pseudo Differential Operators and Markov Processes \textup{(}3 volumes\textup{)}}. Imperial College Press, London 2001--2005.

\bibitem{jacod-shiryaev87}
    Jacod, J., Shiryaev, A.N.: \emph{Limit Theorems for Stochastic Processes}. Springer, Berlin 1987 (2nd ed Springer, Berlin 2003).

\bibitem{khintchine37}
    Khintchine, A.Ya.: A new derivation of a formula by P.\ L\'evy. \emph{Bull.\ Moscow State Univ.} \textbf{1} (1937) 1--5 (Russian; an English translation is contained in Appendix~2.1, pp.~44--49 of \cite{mai-rog10}).

\bibitem{khintchine37b}
    Khintchine, A.Ya.: Zur Theorie der unbeschr\"{a}nkt teilbaren Verteilungsgesetze. \emph{Mat.\ Sbornik} (\emph{R\'ec.\ Math.\ Moscou, N.S.}) \textbf{2} (1937) 79--119 (German; an English translation is contained in Appendix~3.3, pp.~79--125 of \cite{mai-rog10}).

\bibitem{kno-sch13}
    Knopova, V., Schilling, R.L.: A note on the existence of transition probability densities for L\'evy processes. \emph{Forum Math.} \textbf{25} (2013) 125--149.

\bibitem{kolmogorov32}
    Kolmogorov, A.N.: Sulla forma generale di un processo stocastico omogeneo (Un problema die Bruno de Finetti). \emph{Rend.\ Accad.\ Lincei Ser.\ VI} \textbf{15} (1932) 805--808 and 866--869 (an English translation is contained in: Shiryayev, A.N.~(ed.): \emph{Selected Works of A.N.\ Kolmogorov}. Kluwer, Dordrecht 1992, vol.\ 2, pp.121--127).

\bibitem{kuehn}
    K\"{u}hn, F.: Existence and estimates of moments for L\'evy-type processes. \emph{Stoch.\ Proc.\ Appl.} (in press). DOI: \url{10.1016/j.spa.2016.07.008}

\bibitem{kuehn-diss}
    K\"{u}hn, F.: \emph{Probability and Heat Kernel Estimates for L\'evy\textup{(}-Type\textup{)} Processes}.
    PhD Thesis, Technische Universit\"{a}t Dresden 2016.

\bibitem{kunita04}
    Kunita, H.: Stochastic differential equations based on L\'evy processes and stochastic flows of diffeomorphisms. In: Rao, M.M.\ (ed.): \emph{Real and Stochastic Analysis}. Birkh\"{a}user, Boston (MA) 2004, pp.\ 305--373.

\bibitem{kun-wat67}
    Kunita, H., Watanabe, S.: On square integrable martingales. \emph{Nagoya Math.\ J.} \textbf{30} (1967) 209--245.

\bibitem{kyprianou06}
    Kyprianou, A.: \emph{Introductory Lectures on Fluctuations and L\'evy Processes with Applications}. Springer, Berlin 2006.

\bibitem{levy34}
    L\'evy, P.: Sur les int\'egrales dont les \'elements sont des variables al\'eatoires in\-d\'ep\-en\-dantes. \emph{Ann.\ Sc.\ Norm.\ Sup.\ Pisa} \textbf{3} (1934) 337--366.

\bibitem{levy37}
    L\'evy, P.: \emph{Th\'eorie de l'addition des variables al\'eatoires}. Gauthier--Villars, Paris 1937.

\bibitem{mai-rog06}
    Mainardi, F., Rogosin, S.V.: The origin of infinitely divisible distributions: from de Finetti's problem to L\'evy--Khintchine formula. \emph{Math.\ Methods Economics Finance} \textbf{1} (2006) 37--55.

\bibitem{malliavin95}
    Malliavin, P.: \emph{Integration and Probability}. Springer, New York (NY) 1995.

\bibitem{pazy83}
    Pazy, A.: \emph{Semigroups of Linear Operators and Applications to Partial Differential Equations}. Springer, New York (NY) 1983.

\bibitem{prokhorov56}
    Prokhorov, Yu.V.: Convergence of random processes and limit theorems in probability theory. \emph{Theor.\ Probab.\ Appl.} \textbf{1} (1956) 157--214.

\bibitem{protter04}
    Protter, P.E.: \emph{Stochastic Integration and Differential Equations}. Springer, Berlin 2004 (2nd ed).

\bibitem{rev-yor05}
    Revuz, D., Yor, M.: \emph{Continuous Martingales and Brownian Motion}. Springer, Berlin 2005 (3rd printing of the 3rd ed).

\bibitem{mai-rog10}
    Rogosin, S.V., Mainardi, F.: \emph{The Legacy of A.Ya.\ Khintchine's Work in Probability Theory}. Cambridge Scientific Publishers, Cambridge 2010.

\bibitem{rosinski90}
    Rosi\'nski, J.: On series representations of infinitely divisible random vectors. \emph{Ann.\ Probab.} \textbf{18} (1990) 405--430.

\bibitem{rosinski01}
    Rosi\'nski, J.: Series representations of L\'evy processes from the perspective of point processes. In: Barndorff-Nielsen et al.\ \cite{barndorff01} (2001) 401--415.

\bibitem{sam-taq94}
    Samorodnitsky, G., Taqqu, M.S.: \emph{Stable Non-Gaussian Random Processes}. Chapman \& Hall, New York (NY) 1994.

\bibitem{sandric15}
    Sandri\'c, N.: On recurrence and transience of two-dimensional L\'{e}vy and L\'{e}vy-type processes. \emph{Stoch.\ Proc.\ Appl.} \textbf{126} (2016) 414--438.

\bibitem{sandric14}
    Sandri\'c, N.: Long-time behavior for a class of Feller processes. \emph{Trans.\ Am.\ Math.\ Soc.} \textbf{368} (2016) 1871--1910.

\bibitem{sato99}
    Sato, K.: \emph{L\'evy Processes and Infinitely Divisible Distributions}. Cambridge University Press, Cambridge 1999 (2nd ed 2013).

\bibitem{schilling-positivity}
    Schilling, R.L.: Conservativeness and extensions of Feller semigroups. \emph{Positivity} \textbf{2} (1998) 239--256.

\bibitem{schilling-ptrf}
    Schilling, R.L.: Growth and H\"{o}lder conditions for the sample paths of a Feller process. \emph{Probab.\ Theor.\ Relat.\ Fields} \textbf{112} (1998) 565--611.

\bibitem{schilling-mims}
    Schilling, R.L.: \emph{Measures, Integrals and Martingales}. Cambridge University Press, Cambridge 2011 (3rd printing).

\bibitem{schilling-mi}
    Schilling, R.L.: \emph{Ma{\ss} und Integral}. De Gruyter, Berlin 2015.

\bibitem{schilling-bm}
    Schilling, R.L., Partzsch, L.: \emph{Brownian Motion. An Introduction to Stochastic Processes}. De Gruyter, Berlin 2014 (2nd ed).

\bibitem{schilling-schnurr}
    Schilling, R.L., Schnurr, A.: The symbol associated with the solution of a stochastic differential equation. \emph{El.\ J.\ Probab.} \textbf{15} (2010) 1369--1393.

\bibitem{schnurr-phd}
    Schnurr, A.: \emph{The Symbol of a Markov Semimartingale}. PhD Thesis, Technische Universit\"{a}t Dresden 2009. Shaker-Verlag, Aachen 2009.

\bibitem{schnurr12}
    Schnurr, A.: On the semimartingale nature of Feller processes with killing. \emph{Stoch.\ Proc.\ Appl.} \textbf{122} (2012) 2758--2780.


\bibitem{stroock-varadhan97}
    Stroock, D.W., Varadhan, S.R.S.: \emph{Multidimensional Stochastic Processes}. Springer, Berlin 1997 (2nd corrected printing).

\bibitem{waldenfels61}
    von Waldenfels, W.: \emph{Eine Klasse station\"{a}rer Markowprozesse}. Kernforschungsanlage J\"{u}lich, Institut f\"{u}r Plasmaphysik, J\"{u}lich 1961.

\bibitem{waldenfels65}
    von Waldenfels, W.: Fast positive Operatoren. \emph{Z.\ Wahrscheinlichkeitstheorie verw.\ Geb.} \textbf{4} (1965) 159--174.

\end{thebibliography}
\end{document}